\documentclass[11pt, twoside]{article}
\usepackage{cite}
\usepackage{amssymb}
\usepackage{mathrsfs}
\usepackage{amsmath}
\usepackage{amsthm}
\usepackage{amsfonts}
\usepackage{latexsym}
\usepackage{indentfirst}
\usepackage{color}
\usepackage{enumerate}
\usepackage[english]{babel}
\usepackage[colorlinks=true,
linkcolor=blue,
citecolor=red,
urlcolor=magenta, backref=page]{hyperref}

\usepackage{txfonts}
\usepackage{anysize}

\allowdisplaybreaks

\newtheorem{theorem}{Theorem}[section]
\newtheorem{lemma}[theorem]{Lemma}
\newtheorem{corollary}[theorem]{Corollary}
\newtheorem{proposition}[theorem]{Proposition}

\theoremstyle{definition}
\newtheorem{remark}[theorem]{Remark}
\newtheorem{definition}[theorem]{Definition}

\newtheorem{myenv}{Theorem}

\numberwithin{equation}{section}

\begin{document}

\title{\bf\Large Real-Variable Characterizations and
Their Applications of Anisotropic Besov Spaces with Matrix $\mathcal A_\infty$ Weights
\footnotetext{\hspace{-0.35cm} 2020 {\it Mathematics Subject Classification}.
Primary 46E35; Secondary 47A56, 42B25, 42B35, 35S05.\endgraf
{\it Key words and phrases.}
matrix weight,
anisotropic Besov space,
$\mathcal A_{p,\infty}$-dimension,
$\varphi$-transform,
almost diagonal operator,
molecule,
pseudo-differential operator.\endgraf
This project is partially supported by
the National Natural Science Foundation of China
(Grant Nos. 12431006, 12371093, and 12271041),
the Beijing Natural Science Foundation (Grant No. 1262011), and
the Fundamental Research Funds for the Central Universities (Grant No. 2253200028).}}
\date{}
\author{Fan Bu, Shuaijun Feng, Qingying Xue,
Dachun Yang\footnote{Corresponding author, E-mail:
\texttt{dcyang@bnu.edu.cn}/{\color{red} \today}/Newest version.}
\ and Wen Yuan}

\maketitle

\vspace{-0.8cm}

\begin{center}
\begin{minipage}{13cm}
{\small {\bf Abstract}\quad
Let $\alpha\in\mathbb{R}$, $p\in(0,\infty)$, and $q\in(0,\infty]$. In this article, we develop a theory of matrix-weighted anisotropic
Besov spaces associated with an expansive matrix $A$ and an
$\mathcal A_{p,\infty}$-matrix weight $W$.
We first introduce the homogeneous spaces
$\dot B_{p,q}^{\alpha}(A,W)$ and establish their
$\varphi$-transform characterization. Then we construct counterexamples
to show that
the assumption $W\in\mathcal A_{p,\infty}$ in this characterization
cannot be relaxed to $W\in\bigcup_{r\in(0,\infty)}\mathcal A_r$.
The same counterexamples also show that
this weaker condition $W\in\bigcup_{r\in(0,\infty)}\mathcal A_r$
is insufficient to ensure the well-definedness of
$\dot B_{p,q}^{\alpha}(A,W)$.
Next we characterize $\mathcal A_{p,\infty}$-matrix weights
via the rescaled maximal operator, which leads naturally to a new concept of the critical rescaling
index that quantitatively captures the self-improving behavior
of matrix weights.
In terms of this index, we obtain optimal boundedness
for almost diagonal operators on the associated sequence spaces
$\dot b_{p,q}^{\alpha}(A,W)$.
Based on these, we further establish the molecular characterization of
$\dot B_{p,q}^{\alpha}(A,W)$ and some sharp boundedness results for
pseudo-differential operators on these spaces.
}
\end{minipage}
\end{center}


\tableofcontents

\vspace{0.1cm}

\section{Introduction}

The investigation of Besov spaces originates from
Bern\v{s}te\v{\i}n \cite{b47} and Zygmund \cite{z45}.
Building upon their contributions,
Nikol'ski\u{\i} \cite{n51} introduced the family of function spaces
now denoted by $B^\alpha_{p,\infty}$.
By introducing the third index $q$,
Besov \cite{b59, b61} subsequently generalized this framework to the spaces $B^\alpha_{p,q}$.
The theory was further advanced by Peetre \cite{p73, p76},
who extended the ranges of admissible parameters $p$ and $q$ to include values less than one.
The theory of Besov spaces has since been widely applied to various branches of analysis;
for systematic studies of these spaces, we refer to
the monographs of Triebel \cite{t83,t92,t06,t13a,t14} and Sawano \cite{Sa18, Sa20}.
Over the past two decades, there has been significant interest in
extending these spaces to various new frameworks
(see, for instance, \cite{GKP21,GN16,whhy21,YY08,YY10,YY13,YZ10,YHMSY15,YHSY15,YHSY152}).

In particular, one notable extension involves studying
Euclidean spaces endowed with non-isotropic dilation structures.
This direction can be traced to Calder\'on and Torchinsky \cite{CT1,CT2},
who introduced and investigated the parabolic Hardy space
associated with specific one-parameter dilation groups.
Subsequently, Folland and Stein \cite{FS} extended this research
to Hardy spaces on a class of homogeneous groups.
A notable development was later achieved by Bownik \cite{Bownik}, who
introduced anisotropic Hardy spaces associated with a general group of dilations and established the real-variable theory of these spaces.
Since then, the theory of anisotropic Besov spaces
has attracted considerable attention
(see, for instance, \cite{gjabownik05,LBYY,LBYY2,cf20,suv}).
For further developments of anisotropic function spaces,
we refer the reader to \cite{gjabownik07,gjabownik06,BLL,BLYZ,BLYZ10,LBY,LBYZ,lwyy}.
Very recently,  anisotropic Besov spaces were used to study the degenerate McKean--Vlasov equations in \cite{iprt}.

The investigation of function spaces with matrix weights
can be traced back to Wiener and Masani \cite[\S 4]{wm58},
who introduced matrix-weighted Lebesgue spaces
during their development of the prediction theory
for multivariate stochastic processes.
Later, Treil and Volberg \cite{tv97} found that
several problems arising in multivariate stationary stochastic processes
and the theory of Toeplitz operators
can be transformed into the boundedness of the Hilbert transform
on the matrix-weighted Lebesgue space $L^2(W)$.
In the same work, they introduced matrix $\mathcal A_2$ weights and proved that
the Hilbert transform is bounded on $L^2(W)$
if and only if $W$ is a matrix $\mathcal A_2$ weight.
This characterization was subsequently independently extended to
the general case $p\in(1,\infty)$ by
Nazarov and Treil \cite{nt96} and Volberg \cite{v97}.
Furthermore, Bownik \cite{b01} showed that matrix weights do not possess the self-improving property.
Christ and Goldberg \cite{cg01,g03}
introduced the matrix-weighted maximal operator
and established the boundedness of both this operator and
the convolution-type singular integral operator on
$L^p(W)$ with matrix $\mathcal A_p$ weights.
The theory of $L^p(W)$ spaces has grown considerably in recent years,
encompassing diverse areas such as extrapolation \cite{BCarx,CS25,Nieraeth2025},
sparse domination \cite{bbdpw,cdo18,DHL20,DPTVarx,NPTV17,DLY,IKP17,kn2026},
and sharp weighted inequalities \cite{hpv2019,LLOR24,LLOR242}.
For further studies concerning $L^p(W)$ spaces, we refer the reader to
\cite{bpw16,lyz2023,N2010,N2012,N2013,nh2021,nh2025a,nh2026,zz25,nr18}.

Alongside the developments concerning $L^p(W)$ spaces,
the matrix-weighted Besov spaces $\dot B^\alpha_{p,q}(W)$
were introduced and studied by Roudenko \cite{ro03} for $p\in (1,\infty)$
and by Frazier and Roudenko \cite{fr04} for $p\in (0,1]$.
Subsequently, the duality, traces, and extensions of $\dot B^\alpha_{p,q}(W)$
were further investigated by the same authors in \cite{fr08,ro04}.
Later, Frazier and Roudenko \cite{fr21} developed the real-variable theory of
matrix-weighted Triebel--Lizorkin spaces.
Very recently, Bu et al. \cite{bf3,bf6,bf5,bchyy26}
generalized this framework by introducing more general
matrix-weighted Besov--Triebel--Lizorkin-type spaces.
For other matrix-weighted function spaces, we refer the reader to
\cite{CMR16} for Sobolev spaces,
\cite{KS1,KS2,ipt22} for BMO spaces,
\cite{Chen2025,bcyy25} for Hardy spaces,
\cite{bxarXiv} for Bourgain--Morrey spaces,
\cite{Wgx2025,nie25a} for modulation spaces, and
\cite{N2025,lyy24,lyy,bf,bgx,bx24,bx242,mx25,wyy24}
for further extensions of Besov--Triebel--Lizorkin spaces.
Notably, Nielsen \cite{NarX1,NarX2} recently investigated
matrix-weighted Besov spaces associated with a one-parameter group of dilations,
which is a weaker anisotropic setting than the one introduced by Bownik \cite{Bownik}.
The function spaces mentioned above are all equipped with
$\mathcal A_p$-matrix weights.
However, for certain function spaces,
particularly Besov--Triebel--Lizorkin spaces,
it is more natural to consider $\mathcal A_{p,\infty}$-matrix weights.
This weight class, introduced independently by
Nazarov and Treil \cite{nt96} and Volberg \cite{v97},
is the matrix-valued analogue of the scalar $\mathcal A_\infty$ class.
Recently, Bu et al. \cite{bf4} provided an equivalent characterization of
$\mathcal A_{p,\infty}$-matrix weights
and systematically studied this weight class.
Building on this work, Bu et al. \cite{bf2} developed the real-variable theory of
Besov--Triebel--Lizorkin-type spaces with matrix $\mathcal{A}_\infty$ weights.
Motivated by these contributions,
Besov--Triebel--Lizorkin spaces of optimal scale \cite{byyz2026,yyz1}
and Hardy spaces \cite{cyyz}
associated with matrix $\mathcal{A}_\infty$ weights have also been investigated.

Despite these developments, several fundamental problems remain open.
First, Frazier and Roudenko \cite[Theorem 7.1]{fr04} proved that
the doubling condition on $W$ is insufficient
to guarantee the $\varphi$-transform characterization of
matrix-weighted Besov spaces.
Although it is now known that the condition
$W\in\mathcal A_{p,\infty}$ is sufficient for this characterization,
it remains unclear whether this condition is necessary
or can be relaxed to the weaker assumption
$W\in\bigcup_{r\in(0,\infty)}\mathcal A_r$.
Second, the sharp boundedness condition for almost diagonal operators
have been established for $\mathcal A_p$-matrix weights
with $p\in(0,\infty)$ (see \cite[Theorem 7.1]{bf6})
and for $\mathcal A_{p,\infty}$-matrix weights
with $p\in(0,1]$ (see \cite[Lemma 4.13]{bf2}),
whereas the corresponding sharp condition
for $\mathcal A_{p,\infty}$-matrix weights with $p\in(1,\infty)$ remains unknown.
Finally, although the boundedness of pseudo-differential operators on
matrix-weighted Besov spaces has been studied in several articles
(see, for instance, \cite{bf2,bf5,yyz1}),
the sharpness of these results remains unknown.

Motivated by these problems, we develop a real-variable theory of
matrix-weighted anisotropic Besov spaces
$\dot B_{p,q}^{\alpha}(A,W)$
associated with an expansive matrix $A$ and an
$\mathcal A_{p,\infty}$-matrix weight $W$.
We first establish their $\varphi$-transform characterization and show, by constructing a new class of matrix weights, that the assumption
$W\in\mathcal A_{p,\infty}$ is necessary in general for this characterization to hold.
As a consequence of the $\varphi$-transform characterization,
the spaces $\dot B_{p,q}^{\alpha}(A,W)$ are well-defined.
Moreover, we show that the assumption
$W\in\mathcal A_{p,\infty}$ is necessary in general for
$\dot B_{p,q}^{\alpha}(A,W)$ to be well-defined.
We also characterize $\mathcal A_{p,\infty}$-matrix weights in terms of a rescaled maximal operator and introduce the critical rescaling index to quantify their self-improving behavior.
Using this index, we obtain optimal boundedness condition
for almost diagonal operators on the corresponding sequence spaces
$\dot b_{p,q}^{\alpha}(A,W)$.
Combining the $\varphi$-transform characterization with this boundedness result, we establish the molecular characterization of
$\dot B_{p,q}^{\alpha}(A,W)$ and subsequently obtain
sharp boundedness results of pseudo-differential operators
from $\dot B_{p,q}^{\alpha+u}(A,W)$ to
$\dot B_{p,q}^{\alpha}(A,W)$.

Specifically, the main novelties of these results are highlighted as follows.

\begin{itemize}
\item In the proof of the $\varphi$-transform characterization of
$\dot B^{\alpha}_{p,q}(A,W)$,
the  method using the operators $\sup_{\mathbb A,\varphi}$
and $\inf_{\mathbb A,\varphi,N}$ from previous articles
\cite{bf2,byyz2026} is no longer applicable
due to the insufficient decay rate of
the diameters of dyadic cubes in the anisotropic setting.
To overcome this difficulty, we bypass the operator
$\inf_{\mathbb A,\varphi,N}$ and use the self-improving property
of matrix weights to establish the $\varphi$-transform
characterization of $\dot B^{\alpha}_{p,q}(A,W)$ (see Theorem \ref{dl1103}).

\item Recall that, for scalar weights, $A_\infty = \bigcup_{r\in(1,\infty)} A_r$.
However, for matrix weights of size $m\geq 2$,
the class $\mathcal A_{p,\infty}$
is a finer classification of $\bigcup_{r\in(0,\infty)}\mathcal A_r$,
that is $\bigcup_{p\in(0,\infty)}\mathcal A_{p,\infty}
=\bigcup_{r\in(0,\infty)}\mathcal A_r$
and, for any fixed $p\in(0,\infty)$, $\mathcal A_{p,\infty}\subsetneqq
\bigcup_{r\in(0,\infty)}\mathcal A_r$
(see \cite[(vi) and (vii) of Proposition 2.26]{byyz25}).
Therefore, a natural question is whether we can study
$\dot B^{\alpha}_{p,q}(A,W)$ under the weaker assumption that
$W \in \bigcup_{r\in(0,\infty)} \mathcal A_r$.
To address this question, we construct a special class of matrix weights
(see Proposition \ref{WAp})
and use it to prove that, for any $\varepsilon\in(0,\infty)$,
there exists $W \in \mathcal A_{p+\varepsilon,\infty}
\subsetneqq \bigcup_{r\in(0,\infty)} \mathcal A_r$
such that the $\varphi$-transform characterization of $\dot B^{\alpha}_{p,q}(A,W)$
fails to hold (see Theorem \ref{boundedfail_p})
and $\dot B^{\alpha}_{p,q}(A,W)$ fail to be well-defined
(see Theorem \ref{well_defined_fail}).
In this sense, the assumption $W \in \mathcal A_{p,\infty}$ is necessary in general.

\item
Based on a characterization of $\mathcal A_{p,\infty}$-matrix weights in terms of the rescaled maximal operator established in Theorem \ref{bounded B1}, we introduce a new concept,
the critical rescaling index of matrix weights,
to quantify a different self-improving property of matrix weights
(noting that Bownik \cite[Remark 5.4]{b01} proved that
the classical self-improving property
$A_p=\bigcup_{q\in[1,p)}A_q$ for all $p\in(1,\infty]$
no longer holds for matrix weights).
Using the critical rescaling index of matrix weights,
we obtain the boundedness of
almost diagonal operators on $\dot b_{p,q}^{\alpha}(A,W)$
(see Theorem \ref{ad FJ-YY}),
which is sharp when $W$ is a power matrix weight.
Moreover, this result not only recovers existing optimal results but also
improves upon all other known results (see Subsection \ref{compare}).

\item Park \cite[Theorem 2.6]{p20} proved that for
inhomogeneous unweighted Besov spaces $B^\alpha_{p,q}$,
the pseudo-differential operator of type $(1,1)$ and order $u$ is bounded $B^{\alpha+u}_{p,q}$ to $B^\alpha_{p,q}$
if and only if $\alpha>\frac{n}{1\wedge p}-n$.
We not only provide a counterpart to this result
in the $|\cdot|^\lambda$-weighted homogeneous anisotropic setting
(see Theorem \ref{sharppse2}),
but also prove that when $\alpha\leq \max \{0,\frac{1}{p}-1,
\frac{1+\lambda}{p}-1 \}$,
the order of the vanishing assumption on
the pseudo-differential operator is sharp (see Theorem \ref{sharppse1}).
\end{itemize}

The remainder of this article is organized as follows.

In Section \ref{MWABS}, we introduce the matrix-weighted
anisotropic Besov spaces $\dot{B}_{p,q}^{\alpha}(A,W)$
associated with $\mathcal{A}_{p,\infty}$-matrix weights,
and state their $\varphi$-transform characterization
(Theorem \ref{dl1103}),
whose proof is given in Section \ref{phi}.
In particular, we show that the assumption $W\in \mathcal{A}_{p,\infty}$
is necessary in general, not only for this characterization to hold (see Theorem \ref{boundedfail_p}),
but also for the space $\dot{B}_{p,q}^{\alpha}(A,W)$
to be well-defined (see Theorem \ref{well_defined_fail}).
The proof of this necessity relies on a special construction of
matrix weights (see Proposition \ref{WAp}).

In Section \ref{key}, we investigate various properties of
$\mathcal{A}_{p,\infty}$-matrix weights on spaces of homogeneous type,
of which the anisotropic setting is a special case.
This general framework is necessary because anisotropic dyadic cubes
fail to form a nested structure. In Subsection \ref{self},
we prove the self-improving property of
$\mathcal{A}_{p,\infty}$-matrix weights (see Proposition \ref{improve}).
In Subsection \ref{rescaled}, we introduce
the rescaled maximal operator $\mathcal{M}_{W,p}^{(v)}$
and characterize $\mathcal{A}_{p,\infty}$-matrix weights
via its $L^p$-boundedness (see Theorem \ref{bounded B1}),
further show that the exponent range $v\in(0,p)$ is sharp
(see Proposition \ref{bounded fail}).
In Subsection \ref{critical}, we introduce
the critical rescaling index, which quantitatively
characterizes the self-improving property of the matrix weight;
we also analyze its properties and compute it explicitly for power weights
(see Proposition \ref{prop of sp(W)} and Lemma \ref{value}).
In Subsection \ref{keylemma}, we return to the anisotropic setting,
as our analysis requires the geometric condition that
the measure of a ball is comparable to its radius.
Within this framework, we introduce the upper and lower dimensions of
$\mathcal{A}_{p,\infty}$-matrix weights (see Definition \ref{dim}),
which provide a finer classification of this weight class.
Using these concepts, we establish the sharp estimate of
the composition of reducing operators (see Lemma \ref{fuhe}).
This estimate is essential for both the $\varphi$-transform
characterization and the boundedness of almost diagonal operators.

In Section \ref{phi}, we prove the $\varphi$-transform characterization
of $\dot{B}_{p,q}^{\alpha}(A,W)$ (Theorem \ref{dl1103})
by using the self-improving property
of $\mathcal{A}_{p,\infty}$-matrix weights
and the sharp estimate for the composition of reducing operators
obtained in Section \ref{key}.
In Subsection \ref{BB}, we introduce the averaging matrix-weighted
anisotropic Besov spaces $\dot{B}_{p,q}^{\alpha}(A,\mathbb{A})$
and prove their equivalence to $\dot{B}_{p,q}^{\alpha}(A,W)$ (see Theorem \ref{dl111301}).
Since the diameters of anisotropic dyadic cubes decay slowly,
existing methods from \cite{bf2,byyz2026} are no longer valid.
To overcome this difficulty, the proof of Theorem \ref{dl111301}
crucially exploits the self-improving property
of $\mathcal{A}_{p,\infty}$-matrix weights.
In Subsection \ref{equibb}, we introduce the corresponding sequence spaces
$\dot{b}_{p,q}^{\alpha}(A,\mathbb{A})$
and establish their equivalence to $\dot{b}_{p,q}^{\alpha}(A,W)$ (see Theorem \ref{dj}).
Finally, in Subsection \ref{PTC},
we combine these equivalences with the aforementioned sharp estimate
to complete the proof of Theorem \ref{dl1103}.
In particular, we show that the space $\dot{B}_{p,q}^{\alpha}(A,W)$ is well-defined.

In Section \ref{jh dj}, we investigate the boundedness of almost diagonal operators
on the sequence space $\dot{b}_{p,q}^{\alpha}(A,W)$.
In Subsection \ref{BADO}, we establish this boundedness
on $\dot{b}_{p,q}^{\alpha}(A,W)$ (Theorem \ref{ad FJ-YY})
by using the sharp estimate for the composition of reducing operators
and the boundedness of rescaled maximal operator $\mathcal{M}_{W,p}^{(v)}$
obtained in Section \ref{key}.
In Subsection \ref{SADO}, we prove the sharpness of this result
(see Theorem \ref{ad Besov sharp}).
Finally, in Subsection \ref{compare},
we show that Theorem \ref{ad FJ-YY}
not only recovers existing optimal results
but also improves upon all other known bounds.

In Section \ref{f z}, as an application of the $\varphi$-transform characterization
of $\dot{B}_{p,q}^{\alpha}(A,W)$
and the boundedness of almost diagonal operators,
we establish the molecular characterization
of $\dot{B}_{p,q}^{\alpha}(A,W)$.
This result improves upon the work of Bownik
\cite[Theorems 5.5 and 5.7]{gjabownik05} (see Remark \ref{Bownik}).
Furthermore, using this molecular characterization,
we prove that $(\mathcal S_\infty)^m$ continuously embeds into
$\dot{B}_{p,q}^{\alpha}(A,W)$ (see Proposition \ref{distri}).

Finally, in Subsection \ref{7.1}, we apply
the $\varphi$-transform and molecular characterizations
to establish the boundedness of pseudo-differential operators
from $\dot{B}_{p,q}^{\alpha+u}(A,W)$ to $\dot{B}_{p,q}^{\alpha}(A,W)$
(Theorem \ref{pseudo}).
As detailed in Remark \ref{bijiao},
this result improves upon the earlier work of
B\'enyi and Bownik \cite[Theorem 4.8]{bb10}.
Subsequently, in Subsection \ref{7.2},
we prove the sharpness of Theorem \ref{pseudo}.

At the end of this introduction,
we make some conventions on notation.
For any $ r \in \mathbb{R} $, $ r_+ $ is defined as $ r_+ := \max \{0, r\} $
and $ r_- $ is defined as $ r_- := \max \{0, -r\} $. For any $t\in (0,\infty)$, $\log_{+}t:=\max\{0,\log t\}$.
For any $ a, b \in \mathbb{R} $, $ a \wedge b := \min \{a, b\} $
and $ a \vee b := \max \{a, b\} $.
The symbol $ C $ denotes a positive constant which is independent
of the main parameters involved,
but may vary from line to line.
The symbol $ A \lesssim B $ means
that $ A \leq CB $ for some positive constant $ C $,
while $ A \sim B $ means $ A \lesssim B \lesssim A $.
If $ f \leq Cg $ and $ g = h $ or $ g \leq h $,
we then write $ f \lesssim g = h $ or $ f \lesssim g \leq h $,
\emph{rather than} $ f \lesssim g \sim h $ or $ f \lesssim g \lesssim h $.
Let $\mathbb N:=\{1,2,\ldots\}$,
$\mathbb Z_+:=\mathbb N\cup\{0\}$,
and $\mathbb Z_+^n:=(\mathbb Z_+)^n$.
For any multi-index $\gamma:=(\gamma_1,\ldots,\gamma_n)\in
\mathbb Z_+^n$ and any $x:=(x_1,\ldots,x_n)\in{\mathbb{R}^n}$,
let $|\gamma|:=\gamma_1+\cdots+\gamma_n$,
$x^\gamma:=x_1^{\gamma_1}\cdots x_n^{\gamma_n}$,
and $\partial^\gamma:=
(\frac{\partial}{\partial x_1})^{\gamma_1}\cdots
(\frac{\partial}{\partial x_n})^{\gamma_n}$.
We use $\mathbf{0}$ to denote the
\emph{origin} of $\mathbb{R}^n$.
For any set $ E \subset \mathbb{R}^n $,
we use $ \mathbf 1_E $ to denote its \emph{characteristic function}.
For any $p\in(0,\infty]$, the \emph{Lebesgue space} $L^p$ has the usual meaning,
and the \emph{local Lebesgue space}
$ L^p_{\mathrm{loc}}$ is defined to be the set of
all measurable functions $f$ on $ \mathbb{R}^n $ such that
$$
\left\| f \right\|_{L^p(E)}
:= \left\| f \mathbf{1}_E \right\|_{L^p} < \infty
$$
for any bounded measurable set $E$.
For any $p\in(0,\infty)$, let $p':=\frac{p}{p-1}$
if $p\in(1,\infty)$ and let $p':= \infty$ if $p\in(0,1]$
be the conjugate index of $p$.
For any measurable function $ w $ on $\mathbb{R}^n $
and any measurable set $ E \subset \mathbb{R}^n $, let
$$
w(E) := \int_E w(x) \, dx.
$$
For any measurable function $ f $ on $ \mathbb{R}^n $
and any measurable set
$ E \subset \mathbb{R}^n $ with $ |E| \in (0, \infty) $,
let $\fint_E f(x) \, dx := \frac{1}{|E|} \int_E f(x) \, dx$
and $\widetilde{\mathbf{1}}_{E}:=\mathbf{1}_E|E|^{-\frac12}$.
For any space $ X $, the product space
$ X^m $ with $ m \in \mathbb{N} $
is defined by setting
\begin{align*} 
X^m := \left\{ \vec f := (f_1, \ldots, f_m)^{T} :\
\text{for any } i \in \{1, \ldots, m\},\ f_i \in X \right\}.
\end{align*}
Finally, in all proofs we consistently retain the notation
introduced in the original theorem (or related statement).

\section{Matrix-Weighted Anisotropic Besov Spaces}\label{MWABS}

In this section, we introduce matrix-weighted anisotropic Besov spaces
and present their $\varphi$-transform characterization.
Moreover, we construct an example showing that
the condition $W\in \mathcal A_{p,\infty}$
is necessary   to guarantee the validity of this characterization.
We begin by recalling the concept of expansive
matrices from \cite[Definition 2.1]{Bownik}.

\begin{definition}
A real $n\times n$ matrix $A$ is called an \emph{expansive matrix} if
$
\min_{\lambda \in \sigma(A)} |\lambda|>1,
$
where $\sigma (A)$ denotes the set of all eigenvalues of $A$.
\end{definition}

In general, the anisotropic dilation associated with $A$ is not homogeneous
with respect to the Euclidean norm.
To overcome this drawback, Bownik \cite[Definition 2.3]{Bownik}
introduced the homogeneous quasi-norm associated with $A$.

\begin{definition}\label{quasi-norm}
A \emph{homogeneous quasi-norm
associated with an expansive matrix $A$}
is a measurable mapping
$\rho_A:\ \mathbb{R}^n\to[0,\infty)$ such that
\begin{enumerate}[(i)]
\item $\rho_A (x)=0\Longleftrightarrow x=\bf{0}$;

\item for any $x\in\mathbb{R}^n$,
$\rho_A(Ax)=b\rho_A(x)$, here and thereafter, $b:=|\det A|$;

\item there exists $H\in[1,\infty)$ such that,
for any $x,y\in\mathbb{R}^n$,
$\rho_A(x+y)\leq H\,[\rho_A(x)+\rho_A(y)]$.
\end{enumerate}
\end{definition}

For a fixed expansive matrix $A$, Bownik \cite[Lemma 2.4]{Bownik}
showed that any two homogeneous quasi-norms associated with $A$ are equivalent.
Based on this, we always use a special quasi-norm associated with $A$.
Before recalling its definition,
we need the following lemma, which is precisely \cite[Lemma 2.2]{Bownik}.

\begin{lemma}\label{byl2d2}
Let $A$ be an expansive matrix.
Then there exists an invertible $n\times n$ matrix $P$ and $r_A\in(1,\infty)$
such that the set $\Delta:=\{x\in{\mathbb{R}^n}:\ |Px|<1\}$ has measure $1$ and
\begin{align*}
\Delta\subset r_A\Delta\subset A\Delta.
\end{align*}
\end{lemma}

Based on Lemma \ref{byl2d2}, for any $k\in\mathbb{Z}$,
we define $B_k := A^k \Delta$,
which forms a family of anisotropic balls centered at the origin.
Using the definition of $B_k$ and Lemma \ref{byl2d2},
we conclude that, for any $k\in\mathbb{Z}$,
\begin{align}\label{B_k}
B_k\subset r_A B_k\subset B_{k+1}\ \ \text{and}\ \ |B_k|=b^k.
\end{align}
Let ${\mathcal B}:=\{x+B_k :\ x\in\mathbb{R}^n,\ k\in{\mathbb Z}\}$
be the set of all \emph{anisotropic balls} in $\mathbb R^n$.

\begin{definition}
The \emph{step homogeneous quasi-norm $\rho_A$
associated with an expansive matrix $A$} is defined by setting,
for any $x\in\mathbb R^n$,
\begin{align*}
\rho_A(x):= \left\lbrace
\begin{aligned}
&b^k\  &&{\mathrm {if}}\  x \in B_{k+1}\setminus B_k,\\
&0\  &&{\mathrm {if}}\  x=\bf{0}.
\end{aligned}
\right.
\end{align*}
\end{definition}

In what follows, we denote $\rho_A$ simply by $\rho$.
It is easy to verify that $\rho$ is a quasi-norm associated with $A$
and throughout this article we always work with this quasi-norm.
Using $\rho$, we can give another form of anisotropic balls.
For any $x\in\mathbb{R}^n$ and $r\in(0,\infty)$, define
\begin{align*}
B_{\rho}(x,r):=\{y\in\mathbb{R}^n:\ \rho(y-x)<r\}.
\end{align*}
Then $\mathcal B=\{B_{\rho}(x,r):\ x\in\mathbb{R}^n,\ r\in(0,\infty)\}$ and
\begin{align}\label{ballmea}
r \leq \left|B_{\rho}(x,r)\right| \leq b r.
\end{align}
For any ball $B$, let $c_B$ be its center and $r_B$ its radius.
Now, we recall the concept of dyadic cubes.
For any $j\in\mathbb{Z}$, let
$$
\mathcal{D}_{j}:=\{Q_{j,k}:=A^{-j}([0,1)^n+k):\ k\in{\mathbb{Z}}^{n}\}
\quad \text{and} \quad
\mathcal{D}:= \bigcup_{j\in\mathbb Z} \mathcal{D}_{j}.
$$
For any $Q:=Q_{j,k}\in\mathcal{D}_j$,
let $j_Q:=j$, $x_{Q}:=A^{-j}k$, and
$c_Q:=A^{-j}[k+(\frac{1}{2},\ldots,\frac{1}{2})]$.
For any complex-valued function $\varphi$,  $j\in\mathbb Z$, and $x\in\mathbb R^n$,
define $\varphi_j(x):=b^j\varphi(A^j x)$ and
$$
\varphi_Q(x)
:=|Q|^{\frac{1}{2}}\varphi_j(x-x_Q)
=b^{\frac{j}{2}}\varphi(A^jx-k).
$$

The \emph{Schwartz space}   $\mathcal S$
is defined to be the set of all $C^\infty$ functions $f$
such that, for any $\beta\in{\mathbb Z}_+$ and
any multi-index $\gamma\in{\mathbb Z}_+^n$,
\begin{align}\label{S}
\left\|f\right\|_{\beta,\gamma} := \sup_{x \in \mathbb{R}^n}
\left[1+\rho (x)\right]^\beta|\partial^\gamma f(x)|<\infty.
\end{align}
We denote the \emph{dual space} of ${\mathcal S}$ by ${\mathcal S}'$,
equipped with the weak-$*$ topology.
For any $f\in{\mathcal S}'$ and
$\varphi\in{\mathcal S}$, we denote
$\langle f,\varphi\rangle=f(\overline{\varphi})$.
For any $\xi \in\mathbb{R}^n$ and $f\in \mathcal S$, let
\begin{align*}
\widehat{f}(\xi):=\int_{\mathbb{R}^n}f(x)
e^{-ix\cdot\xi}\,dx
\end{align*}
denote its \emph{Fourier transform} and
$f^\vee(\xi):=(2\pi)^{-n}\widehat{f}(-\xi)$ its \emph{inverse Fourier transform}.
These definitions ensure that we still have $(\widehat{f})^\vee=f$.
For any complex-valued function $g$ on $\mathbb{R}^n$, let
\begin{align*}
\operatorname{supp}g:=\overline{\{x\in\mathbb{R}^n:\ g(x)\neq 0\}}.
\end{align*}
For any $f\in{\mathcal S}'$, let
\begin{align*}
\operatorname{supp}f:=\bigcap\left\{\text{closed set}\ K\subset\mathbb{R}^n:\
\langle f,\varphi\rangle=0,\
\text{for any }\ \varphi\in{\mathcal S}\
\text{with}\ \operatorname{supp}\varphi\subset\mathbb{R}^n\setminus K\right\},
\end{align*}
which can be found in \cite[Definition 2.3.16]{hsdyk01}.

The closed subspace ${\mathcal S}_\infty$
of the Schwartz class ${\mathcal S}$ is given by
\begin{align*}
\mathcal{S}_\infty
=\left\{\varphi\in\mathcal{S}:\
\int_{\mathbb{R}^n} x^\gamma \varphi(x)\,dx=0
\text{ for all } \gamma\in{\mathbb Z}_+^n\right\}.
\end{align*}
Following Triebel \cite{t83}, we view $\mathcal{S}_\infty$
as a subspace of $\mathcal{S}$ with  the same topology.
Thus, $\mathcal{S}_\infty$ is a complete metric space
(see, for example, \cite[(3.7)]{sw71}).
We denote the \emph{dual space} of $\mathcal S_\infty$
by $\mathcal S_\infty'$, equipped with the weak-$*$ topology.



For any $m,n\in\mathbb{N}$, the set of all $m\times n$
complex-valued matrices is denoted by $M_{m,n}({\mathbb C})$,
and $M_{m,m}({\mathbb C})$ is denoted simply by $M_m({\mathbb C})$.
The zero matrix in $M_{m,n}({\mathbb C})$ is denoted by
$O_{m,n}$ and $O_{m,m}$ is simply denoted by $O_m$.
For any matrix $M\in M_{m,n}({\mathbb C})$,
$M^*$ is the \emph{conjugate transpose} of $M$.
A matrix $M\in M_m({\mathbb C})$
is said to be \emph{positive definite}
(resp. \emph{nonnegative definite}) if,
for any $\vec{z}\in\mathbb{C}^m\setminus\{\vec{\mathbf{0}}\}$,
$(M\vec{z},\vec{z})>0$ (resp. $\geq 0$).
Furthermore, the operator norm of a matrix $M$ is defined by
\begin{align*}
\|M\|:=\sup_{\vec{z}\in\mathbb{C}^m,|\vec{z}|=1} |M\vec{z}|.
\end{align*}

Next, we recall the concept of matrix weights
(see, for instance, \cite{nt96,tv97, v97}).

\begin{definition} \label{MatrixWeight}
A matrix-valued function $W: \mathbb{R}^n\to M_m(\mathbb{C})$
is called a \emph{matrix weight} if $ W $ satisfies that,
for any $ x \in \mathbb{R}^n $, $ W(x) $ is nonnegative definite;
for almost every $ x \in \mathbb{R}^n $, $ W(x) $ is invertible;
the entries of $W$ are locally integrable.
When $m=1$, matrix weights reduce to \emph{scalar weights}.
\end{definition}


Let $p, q\in(0,\infty]$. For any sequence  $\{f_j\}_{j \in \mathbb{Z}}$ of measurable functions on $\mathbb{R}^n$,
let
\begin{align}\label{lqLp}
\left\|\{f_j\}_{j\in \mathbb Z}\right\|_{\ell^q L^p} := \left[\sum_{j \in \mathbb Z} \left\|f_j\right\|_{L^p}^q\right]^{\frac{1}{q}}
\end{align}
with the usual modification made when $q=\infty$.
Now, we introduce the homogeneous matrix-weighted anisotropic Besov space.

\begin{definition}
Let $\alpha\in\mathbb{R}$, $p\in(0,\infty)$,
$q\in(0,\,\infty]$,
and $W$ be a matrix weight.
Assume that $\varphi\in{\mathcal S}$ satisfies
\begin{align}\label{hs2}
\operatorname{supp}\widehat{\varphi}
\subset[-\pi,\pi]^n\setminus\{\mathbf0\}
\quad\text{and}\quad
\sup_{j\in\mathbb{Z}}
\left|\widehat{\varphi}
((A^{\ast})^{j}\xi)\right|>0
\text{ for all } \xi\in\mathbb{R}^n\setminus\{\mathbf 0\}.
\end{align}
The \emph{homogeneous matrix-weighted anisotropic Besov space}
$\dot{B}^{\alpha}_{p,q}(A,W,\varphi)$
is defined by setting
\begin{align*}
\dot{B}^{\alpha}_{p,q}(A,W,\varphi)
:=\left\{\vec{f}\in (\mathcal{S}'_\infty)^m:\
\left\|\vec{f}\right\|_{{\dot{B}^{\alpha}_{p,q}(A,W,\varphi)}}
<\infty\right\},
\end{align*}
where, for any
$\vec{f}\in (\mathcal{S}'_\infty)^m$,
\begin{align*}
\left\|\vec{f}\right\|_{{\dot{B}^{\alpha}_{p,q}(A,W,\varphi)}}
:= \left\| \left\{b^{j\alpha}\left|W^{\frac1p}\left( \varphi_j \ast\vec f\right) \right| \right\}_{j\in\mathbb Z} \right\|_{\ell^q L^p}.
\end{align*}
\end{definition}

In what follows, if there is no confusion,
we denote $\dot{B}^{\alpha}_{p,q}(A,W,\varphi)$ simply by $\dot{B}^{\alpha}_{p,q}(W,\varphi)$.

We also introduce the sequence space related to $\dot{B}^{\alpha}_{p,q}(W,\varphi)$.

\begin{definition}
Let $\alpha\in\mathbb{R}$, $p\in(0,\infty)$, $q\in(0,\,\infty]$,
and $W$ be a matrix weight.
The \emph{matrix-weighted anisotropic Besov sequence space}
$\dot{b}^{\alpha}_{p,q}(A,W)$ is defined to be the set of
all sequences $\vec{s}=\{\vec{s}_Q\}_{Q\in \mathcal{D}}$ in $\mathbb C^m$
such that
$\|\vec{s}\|_{\dot{b}^{\alpha}_{p,q}(A,W)}
:=\|\{b^{j\alpha}|W^{\frac1p}\vec s_j|\}_{j\in\mathbb Z}\|_{\ell^q L^p}
<\infty,$
where $\|\cdot\|_{\ell^q L^p}$ is as in \eqref{lqLp} and, for any $j\in\mathbb Z$,
\begin{align}\label{sj}
\vec s_j:=\sum_{Q\in\mathcal D_j} \widetilde{\mathbf1}_Q \vec s_Q.
\end{align}
In what follows, we denote $\dot{b}^{\alpha}_{p,q}(A,W)$ simply by $\dot{b}^{\alpha}_{p,q}(W)$ and
$\dot{b}^{\alpha}_{p,q}:=\dot{b}^{\alpha}_{p,q}(I_m)$,
where $I_m$ is the identity matrix.
\end{definition}

In this article, we study the space $\dot{B}^{\alpha}_{p,q}(W,\varphi)$ under the assumption that $W$ is a matrix $\mathcal A_{\infty}$ weight.
Recall that, in the scalar setting, Li et al. \cite{LBYY,LBYY2} studied
anisotropic Besov spaces with scalar $A_{\infty}$ weights.
Matrix $\mathcal A_{\infty}$ weights were first introduced by Volberg \cite{v97} and Nazarov and Treil \cite{nt96}.
Recently, Bu et al. \cite[Definition 3.1]{bf4} provided an equivalent characterization of matrix $\mathcal  A_{\infty}$ weights.
Using this equivalent characterization, we introduce matrix
$\mathcal A_{\infty}$ weights associated with expansive matrix $A$ as follows.

\begin{definition}
Let $p\in(0, \infty)$. A matrix weight $W$ on $\mathbb{R}^{n}$ is called an
\emph{$\mathcal A_{p,\infty}\left(\mathbb{R}^{n},\mathbb{C}^{m},A\right)$-matrix weight}
if, for any ball $B\in{\mathcal B}$,
$\log _{+}(\fint_{B}\|W^{\frac{1}{p}}(x) W^{-\frac{1}{p}}
(\cdot)\|^{p} \,dx) \in L^{1}(B)$
and
\begin{align*}
[W]_{\mathcal A_{p, \infty}\left(\mathbb{R}^{n}, \mathbb{C}^{m},A\right)}
:=\sup_{B\in{\mathcal B}}\,\, \exp \left(\fint_{B} \log \left(
\fint_{B}\left\|W^{\frac{1}{p}}(x) W^{-\frac{1}{p}}(y)\right\|^{p} \,dx
\right) \,dy\right)<\infty.
\end{align*}
In what follows, if there is no confusion,
we denote $\mathcal A_{p,\infty}(\mathbb{R}^n,\mathbb{C}^m,A)$ simply by $\mathcal A_{p,\infty}$.
\end{definition}

To show the space $\dot{B}^{\alpha}_{p,q}(W,\varphi)$ is well-defined, we establish its $\varphi$-transform characterization, which connects
$\dot{B}^{\alpha}_{p,q}(W,\varphi)$ to its related sequence space.
Recall that,  for $\varphi,\psi\in{\mathcal S}$ satisfying \eqref{hs2},
the \emph{$\varphi$-transform} is defined to be the
map taking each $\vec f\in({\mathcal S}_\infty')^m$
to the sequence $S_\varphi \vec{f}:=\{(S_\varphi \vec{f})_Q\}_{Q\in{\mathcal D}}$,
where $(S_\varphi \vec{f})_Q:=\langle \vec{f},\varphi_Q\rangle$
for any $Q\in{\mathcal D}$; the \emph{inverse $\varphi$-transform} is defined
to be the map taking a sequence
$\vec{s}:=\{\vec{s}_Q\}_{Q\in \mathcal{D}}\subset{\mathbb C}^m$
to $T_\psi \vec{s}:=\sum_{Q\in\mathcal{D}}
\vec{s}_Q\psi_Q$ whenever this series converges in
$({\mathcal S}_\infty')^m$ (see, for instance, \cite{ck110301,ck110302}).
For any complex-valued function $\varphi$ and $x\in \mathbb R^n$, let
$\widetilde{\varphi}(x):=\overline{\varphi(-x)}$.
Now, we state the  $\varphi$-transform characterization of $\dot{B}^\alpha_{p,q}(W,\varphi)$.

\begin{theorem}\label{dl1103}
Let $\alpha\in{\mathbb R},\,p\in(0,\infty),q\in(0,\infty]$, and $W\in \mathcal A_{p,\infty}$. Assume that
$\varphi,\psi\in{\mathcal S}$ satisfy \eqref{hs2}.
Then the following statements hold.
\begin{enumerate}[{\rm(i)}]
\item The operators
$S_\varphi:\dot{B}^\alpha_{p,q}(W,\widetilde{\varphi})
\to\dot{b}^\alpha_{p,q}(W)$
and $T_\psi:\dot{b}^\alpha_{p,q}(W)
\to\dot{B}^\alpha_{p,q}(W,\varphi)$ are bounded.

\item If $\varphi$ and $\psi$ further satisfy
\begin{align}\label{hs3}
\sum_{j\in\mathbb{Z}}
\overline{\widehat{\varphi}((A^*)^j\xi)}
\widehat{\psi}((A^*)^j\xi)=1
\text{ for all }
\xi\in\mathbb R^n\setminus\{\mathbf 0\},
\end{align}
then $T_\psi\circ S_\varphi$ is the identity on
$\dot{B}^\alpha_{p,q}(W,\widetilde{\varphi})$.

\item The space $\dot{B}^\alpha_{p,q}(W,\varphi)$ is independent of the choice of $\varphi$.
\end{enumerate}
\end{theorem}

To prove Theorem \ref{dl1103},
we first establish some properties of $\mathcal A_{p,\infty}$ in Section \ref{key},
and then give the proof in Section \ref{phi}.
Based on Theorem \ref{dl1103}(iii), if there exists no confusion,
we denote $\dot{B}^\alpha_{p,q}(W,\varphi)$ simply by $\dot{B}^\alpha_{p,q}(W)$.

Recall that the $\varphi$-transform characterization
was first established by Frazier and Jawerth \cite[Theorem 2.2]{FJ90}.
This characterization is useful because it allows us to
translate the boundedness of various operators on function spaces
into the boundedness of almost diagonal operators on sequence spaces
(see, for example, \cite{bl26,fsyy,ftw88,GJN17,syy,tor,YYZ14,YSY05,syy10}).

Note that for scalar weights, $A_\infty=\bigcup_{r\in[1,\infty)} A_r$.
However, for matrix weights with $m\geq 2$ and $p\in(0,\infty)$,
we have
$\mathcal A_p\subsetneqq\mathcal A_{p,\infty} \subsetneqq \bigcup_{r\in(0,\infty)} \mathcal A_r$
(see \cite[(vi) and (vii) of Proposition 2.26]{byyz25}).
A natural question is whether
the assumption $W\in \mathcal A_{p,\infty}$ in Theorem \ref{dl1103}
can be relaxed to $W\in \bigcup_{r\in(0,\infty)} \mathcal A_r$.
The following result shows that the assumption
$W\in \mathcal A_{p,\infty}$ is optimal
for Theorem \ref{dl1103} to hold.
To simplify the presentation, in the remainder of this section
we assume that $n=1$, $m=2$, $A=2$, and $\rho(\cdot)=|\cdot|$.

\begin{theorem}\label{boundedfail_p}
Let $p \in (0, \infty)$, $q\in (0,\infty]$, and $\alpha \in \mathbb{R}$.
Then, for any $r\in(p,\infty)$,
there exist $W\in \mathcal A_{r,\infty}\setminus\mathcal A_{p,\infty}$
and $\varphi\in\mathcal S$ satisfying \eqref{hs2}
such that
$S_\varphi:\dot{B}^\alpha_{p,q}(W,\widetilde{\varphi})
\to\dot{b}^\alpha_{p,q}(W)$ is unbounded.
\end{theorem}

By giving a doubling matrix weight which
is not in $\bigcup_{r\in(0,\infty)} \mathcal A_r$, Frazier and Roudenko \cite[Theorem 7.1]{fr04}, who attribute this example to Fedor Nazarov,
proved that the doubling condition is insufficient to guarantee the boundedness of $S_\varphi$. Although the proof of Theorem
\ref{boundedfail_p} is inspired by their argument,
the construction of the matrix weight is essentially different.
Specifically, their example is a doubling weight that
does not belong to $\bigcup_{r\in(0,\infty)} \mathcal A_r$,
whereas our setting requires $\mathcal A_{p,\infty}$-matrix weights.
Furthermore, some other known existing nontrivial matrix weights,
including those constructed by Bownik \cite[Proposition 5.3]{b01}
and Bickel et al. \cite[Remark 3.6]{blm17},
are not suitable for our purposes. Consequently,
the proof of Theorem \ref{boundedfail_p} requires
the construction of a new nontrivial matrix weight.

\begin{proposition}\label{WAp}
Let $r,p \in (0, \infty)$.
For any $t\in\mathbb R$, let
$$
W(t):= [O(t)]^T D(t) O(t),
$$
where
$$
O(t):=\frac{1}{\sqrt{1+t^2}} \begin{pmatrix}
1 & -t \\
t & 1 \end{pmatrix}
\quad \text{and} \quad
D(t):= \begin{pmatrix} 1 & 0 \\ 0 & |t|^r \end{pmatrix}.
$$
Then $W\in \mathcal A_{p,\infty}$
if and only if $r\leq p$.
\end{proposition}

\begin{proof}
Note that, for any $p\in[r,\infty)$,
$\mathcal A_{r,\infty}\subset \mathcal A_{p,\infty}$
(see \cite[Proposition 4.2(i)]{bf4}).
Therefore, to prove the present proposition,
it suffices to show that $W\in \mathcal A_{r,\infty}
\setminus \bigcup_{q\in(0,r)} \mathcal A_{q,\infty}$.

We first prove that $W\in \mathcal A_{r,\infty}$.
For any interval $J\subset\mathbb R$, let
$$
\mathcal I(J)
:=\exp\left(\fint_J\log\left(\fint_J\left\|W^{\frac1r}(t)W^{-\frac1r}(s)\right\|^r \, dt\right) \, ds\right).
$$
Observe that, for any $U:=[u_{ij}]_{i,j\in\{1,2\}}\in M_2(\mathbb C)$,
\begin{align*}
\|U\|
&=\sup_{\genfrac{}{}{0pt}{}{(z_1,z_2)\in\mathbb C^2}{|z_1|^2+|z_2|^2=1}}
\left(\left|u_{11}z_1+u_{12}z_2\right|^2+\left|u_{21}z_1+u_{22}z_2\right|^2\right)^{\frac12} \\
&\leq \left[(|u_{11}|+|u_{12}|)^2+(|u_{21}|+|u_{22}|)^2\right]^{\frac12}
\leq |u_{11}|+|u_{12}|+|u_{21}|+|u_{22}|.
\end{align*}
For any $t\in\mathbb R$, $O(t)$ is orthogonal,
which further implies that $\|O(t)U\|=\|UO(t)\|=\|U\|$ for all $U\in M_2(\mathbb C)$.
From these, we deduce that, for any $t\in\mathbb R$ and $s\in\mathbb R \setminus\{0\}$,
\begin{align} \label{desired}
\left\| W^{\frac1r}(t) W^{-\frac1r}(s) \right\|
&=\left\| D^{\frac1r}(t) O(t) O^T(s) D^{-\frac1r}(s) \right\| \notag\\
&= \left\| \frac{1}{\sqrt{1+t^2}\sqrt{1+s^2}}
\begin{pmatrix} 1+ts & -\frac{1}{|s|}(t-s) \notag\\
|t|(t-s) & \frac{|t|}{|s|} (1+ts) \end{pmatrix} \right\| \notag\\
&\leq \left( 1+\frac{|t|}{|s|}\right) \frac{|1+ts|}{\sqrt{1+t^2}\sqrt{1+s^2}}
+ \left( |t|+\frac{1}{|s|}\right) \frac{|t-s|}{\sqrt{1+t^2}\sqrt{1+s^2}} \notag\\
&=: \mathrm{J}_1 + \mathrm{J}_2.
\end{align}
For the term $\mathrm{J}_1$, a direct computation shows that
\begin{align*}
\mathrm{J}_1
= \left( 1+\frac{|t|}{|s|}\right)
\left[ \frac{1+2ts+(ts)^2}{1+t^2+s^2+(ts)^2} \right]^{\frac12}
\leq 1+\frac{|t|}{|s|}.
\end{align*}
As for $\mathrm{J}_2$, we can bound it by
\begin{align*}
\mathrm{J}_2
&= \frac{|t|}{\sqrt{1+t^2}} \frac{|t-s|}{\sqrt{1+s^2}}
+ \frac{|t-s|}{|s|} \frac{1}{\sqrt{1+t^2}\sqrt{1+s^2}} \\
&\leq \frac{|t-s|}{\sqrt{1+s^2}} + \frac{|t-s|}{|s|}
\leq 2\frac{|t|+|s|}{|s|}
= 2 \left( 1+\frac{|t|}{|s|}\right) .
\end{align*}
Substituting the above estimates of $\mathrm{J}_1$ and $\mathrm{J}_2$
into \eqref{desired}, we obtain, for any $t\in\mathbb R$
and $s\in \mathbb R \setminus\{0\}$,
$$
\left\|W^{\frac1r}(t)W^{-\frac1r}(s)\right\|
\leq 3 \left( 1+ \frac{|t|}{|s|}\right).
$$
From this and \cite[Lemma 2.40]{bf3}, we deduce that,
for any $s \in J\setminus \{0\}$,
\begin{align*}
\fint_J \left\|W^{\frac1r}(t)W^{-\frac1r}(s)\right\|^r\,dt
\lesssim \fint_J 1+ \frac{|t|^r}{|s|^r}\, dt
\sim 1 + \frac{(|c_J|+|J|)^r}{|s|^r}.
\end{align*}
Therefore,
\begin{align}\label{logint}
\mathcal I(J)
\lesssim \exp\left(\fint_J \log\left( 1+ \frac{(|c_J|+|J|)^r}{|s|^r} \right)\,ds\right).
\end{align}
Note that, for any $s\in J$,
\begin{align}\label{sJ}
|c_J|-\frac12|J|
\leq |s|
\leq |c_J|+\frac12|J|<|c_J|+|J|,
\end{align}
If $|c_J|\leq |J|$, it follows from \eqref{sJ} that $J\subset[-2|J|,2|J|]$, and hence
\begin{align*}
\fint_J \log\left( 1+ \frac{(|c_J|+|J|)^r}{|s|^r} \right)\,ds
&\le \frac{1}{|J|} \int_{[-2|J|,2|J|]}
\log\left( 1+ \frac{(2|J|)^r}{|s|^r} \right)\,ds \\
&= 2 \int_{-1}^1
\log\left( 1+ \frac 1{|t|^r} \right)\,dt.
\end{align*}
If $|c_J|> |J|$, then \eqref{sJ} implies that $|s|>\frac12|c_J|$, and hence
\begin{align*}
\fint_J \log\left( 1+ \frac{(|c_J|+|J|)^r}{|s|^r} \right)\,ds
< \fint_J \log(1+4^r) \,ds
= \log(1+4^r).
\end{align*}
Substituting the above two estimates into \eqref{logint},
we obtain $\mathcal I(J)\lesssim 1$,
which further implies that $W\in \mathcal A_{r,\infty}$.

Next, we prove that $W\notin\mathcal A_{q,\infty}$
for all $q\in(0,r)$.
Observe that the operator norm of a matrix is bounded below by
the absolute value of any of its entries.
Using this and the orthogonality of $O(t)$, we find that, for any $t,s\in(0,1)$,
\begin{align*}
\left\| W^{\frac1q}(t) W^{-\frac1q}(s) \right\|^q
&=\left\| D^{\frac1q}(t) O(t) O^T(s) D^{-\frac1q}(s) \right\|^q \\
&= \left\| \frac{1}{\sqrt{1+t^2}\sqrt{1+s^2}}
\begin{pmatrix} 1+ts & -|s|^{-\frac rq}(t-s) \notag\\
|t|^{\frac rq}(t-s) & (\frac{|t|}{|s|})^{\frac rq} (1+ts) \end{pmatrix} \right\|^q \\
&\geq \left[\frac{1}{\sqrt{1+t^2}\sqrt{1+s^2}}
\left| |s|^{-\frac rq} (t-s) \right|\right]^q
\geq \frac{|t-s|^q}{2^q s^r}.
\end{align*}
This, together with \cite[Lemma 2.40]{bf3}, further implies that,
for any $\varepsilon\in(0,1)$ and $s\in(0,\varepsilon)$,
\begin{align*}
\fint_{(0,\varepsilon)} \left\| W^{\frac1q}(t) W^{-\frac1q}(s) \right\|^q \,dt
\gtrsim \fint_{(0,\varepsilon)} \frac{|t-s|^q}{s^r} \,dt
\sim \frac{(s+\varepsilon)^q}{s^r}
\sim \frac{\varepsilon^q}{s^r}.
\end{align*}
Thus, for any $\varepsilon\in(0,1)$,
\begin{align*}
[W]_{\mathcal A_{q,\infty}}
\gtrsim \exp\left(\fint_{(0,\varepsilon)}
\log \frac{\varepsilon^q}{s^r} \,ds\right)
= e^r \varepsilon^{q-r}.
\end{align*}
Since $q\in(0,r)$, letting $\varepsilon \to 0^+$ yields
$[W]_{\mathcal A_{q,\infty}}=\infty$.
Therefore, $W\notin \mathcal A_{q,\infty}$.
This completes the proof of Proposition \ref{WAp}.
\end{proof}

Now, we  prove Theorem \ref{boundedfail_p}.

\begin{proof}[Proof of Theorem \ref{boundedfail_p}]
For any $r\in(p,\infty)$, let $W$ be as in Proposition \ref{WAp}.
Then $W\in \mathcal A_{r,\infty}\setminus\mathcal A_{p,\infty}$.
To construct $\varphi$, we first choose a radially decreasing function
$\psi \in C_{\mathrm c}^\infty$ such that
$\psi(\xi) = 1$ for $|\xi| \le \frac32$,
$\psi(\xi) = \frac12$ for $|\xi| =\frac74 $,
and $\psi(\xi) = 0$ for $|\xi| \ge 2$.
Let $\varphi$ be defined via its Fourier transform
$\widehat{\varphi}(\cdot):= \psi(\cdot)- \psi(2\cdot)$.
A straightforward calculation shows that $\varphi\in\mathcal S$ satisfies
\begin{enumerate}[{\rm(i)}]
\item $\operatorname{supp} \widehat\varphi
\subset [-2, -\frac34] \cup [\frac34, 2]$,

\item for any $\xi\in J:= [-\frac32, -1] \cup [1, \frac32]$,
we have $\widehat{\varphi}(\xi) = 1$
and $\widehat{\varphi}(2^v \xi) = 0$ for all $v\in\mathbb Z\setminus\{0\}$,

\item for any $\xi\in [-\frac74, -\frac78] \cup [\frac78, \frac74]$,
we have $\widehat{\varphi}(\xi) \geq \frac12$.
\end{enumerate}
Statements (i) and (iii) imply that $\varphi$ satisfies \eqref{hs2}.

Now, we prove $S_\varphi:\ \dot{B}^\alpha_{p,q}(W,\widetilde{\varphi})
\to\dot{b}^\alpha_{p,q}(W)$ is unbounded.
Let $\gamma\in\mathcal S$ satisfy $\operatorname{supp} \widehat{\gamma} \subset J$
and $\gamma\not\equiv 0$.
For any $j\in\mathbb Z$ and $t\in\mathbb R$,
define $\vec{f}_j(t):=\gamma(2^j t)(t,1)^T$.
We first estimate $\| \vec{f}_j \|_{\dot{B}^{\alpha}_{p,q}(W,\widetilde{\varphi})}$.
A straightforward calculation shows that $\operatorname{supp} \widehat{\vec f_j}\subset 2^j J$.
This, together with (ii), further implies that, for any $v\in\mathbb Z$,
\begin{align}\label{phifj}
\widetilde{\varphi}_v * \vec{f}_j
= \left( \widehat{\widetilde{\varphi}_v} \widehat{\vec{f}_j} \right)^{\vee}
= \left( \overline{\widehat{\varphi}(2^{-v} \cdot)} \widehat{\vec{f}_j} (\cdot) \right)^{\vee}
=\begin{cases}
\vec{f}_j & \text{if } v=j, \\
0 & \text{if } v\neq j.
\end{cases}
\end{align}
From \eqref{phifj}, we infer that
\begin{align}\label{Ej}
\left\| \vec{f}_j \right\|_{\dot{B}^{\alpha}_{p,q}(W,\widetilde{\varphi})}
= 2^{j \alpha} \left[ \int_{\mathbb{R}}
\left| W^{\frac1p}(t) \vec{f}_j(t) \right|^p \,dt \right]^{\frac1p}.
\end{align}
By the definitions of $W$ and $\vec f_j$, we conclude that
\begin{align*}
\left|W^{\frac1p}(t)\vec{f}_j(t)\right|
&=\frac{|\gamma(2^j t)|}{1+t^2}\left|\begin{pmatrix}1&t\\ -t&1\end{pmatrix}
\begin{pmatrix}1&0\\0& |t|^{\frac{r}{p}}\end{pmatrix}
\begin{pmatrix}1&-t\\ t&1\end{pmatrix}
\begin{pmatrix} t\\1\end{pmatrix}\right| \\
&=|t|^{\frac rp} \left(1+t^2\right)^{\frac12} |\gamma(2^j t)|
\leq |t|^{\frac rp} (1+|t|) |\gamma(2^j t)|,
\end{align*}
and hence
\begin{align} \label{82}
\int_{\mathbb R}\left| W^{\frac1p}(t)\vec{f}_j(t) \right|^p \,dt
&\leq \int_{\mathbb R} |t|^{r} (1+|t|)^p |\gamma(2^j t)|^p \,dt
=2^{-j(r+1)} \int_{\mathbb R} |s|^{r}
(1+|2^{-j}s|)^p |\gamma(s)|^p \, ds \notag \\
&\leq 2^{-j(r+1)} \int_{\mathbb R} |s|^{r}
(1+|s|)^p |\gamma(s)|^p \, ds
\sim 2^{-j(r+1)}.
\end{align}
Substituting this estimate into \eqref{Ej},
we conclude that, for any $j\in\mathbb N$,
\begin{align}\label{Bj}
\left\| \vec{f}_j \right\|_{\dot{B}^{\alpha}_{p,q}(W,\widetilde\varphi)}
\lesssim 2^{j(\alpha-\frac{r+1}{p})},
\end{align}
where the implicit constant  depends only on $\gamma$, $p$, and $r$.

We next estimate $\|S_\varphi \vec f_j \|_{\dot{b}^{\alpha}_{p,q}(W)}$.
Using \eqref{phifj}, we conclude that, for any $j\in\mathbb N$ and any $k\in\mathbb Z$,
\begin{align}\label{fjQ}
\left\langle \vec{f}_j, \varphi_{Q_{j,k}} \right\rangle
=|Q_{j,k}|^{\frac12}\left(\widetilde{\varphi}_j*\vec f_j\right)(x_{Q_{j,k}})
=2^{-\frac j2} \vec{f}_j(2^{-j} k).
\end{align}
Note that $\operatorname{supp} \widehat{\gamma}\subset J
\subset\{\xi\in\mathbb R:\ |\xi| \leq \frac32 \}$.
By this and \cite[Lemma 3.4]{bf2}, we conclude that
there exists $k_0 \in\mathbb Z$ such that $\gamma(k_0) \neq 0$.
Since $O(t)^T$ is an orthogonal matrix, we obtain
\begin{align} \label{822}
\left|W^{\frac1p}(t)\vec{f}_j(2^{-j} k_0)\right|
&= \left|O(t)^TD^{\frac1p}(t)O(t)\vec{f}_j(2^{-j} k_0) \right|
=\left|D^{\frac1p}(t)O(t)\vec{f}_j(2^{-j} k_0) \right| \notag \\
&= |\gamma(k_0)| \frac1{\sqrt{1+t^2}}
\left|\begin{pmatrix}1&0\\0& |t|^{\frac{r}{p}}\end{pmatrix}
\begin{pmatrix}1&-t\\ t&1\end{pmatrix}
\begin{pmatrix} 2^{-j} k_0\\1\end{pmatrix}\right| \notag \\
&= |\gamma(k_0)| \frac1{\sqrt{1+t^2}}
\left|\begin{pmatrix}t-2^{-j} k_0\\|t|^{\frac rp}(t(2^{-j} k_0)+1)\end{pmatrix} \right| \notag \\
&\ge \left|\gamma(k_0)\right|\left(1+t^2\right)^{-\frac12}|t-2^{-j} k_0|.
\end{align}
Applying this and \eqref{fjQ}, we conclude that, for any $j\in\mathbb N$,
\begin{align*}
\left\|S_\varphi \vec f_j\right\|_{\dot{b}^{\alpha}_{p,q}(W)}
&\ge 2^{j \alpha} \left[ \int_{Q_{j, k_0}}
\left| W^{\frac1p}(t) \vec{f}_j(2^{-j} k_0)
\right|^p \,dt \right]^{\frac1p} \\
&\ge 2^{j\alpha}|\gamma(k_0)|\left( \int_{Q_{j,k_0}}
\left( 1+t^2\right)^{-\frac p2}|t-2^{-j} k_0|^{p}\,dt
\right)^{\frac 1p} \\
&\sim 2^{j\alpha}|\gamma(k_0)|\left( \int_{Q_{j,k_0}}
|t-2^{-j} k_0|^{p}\,dt\right)^{\frac 1p}
\sim 2^{j(\alpha-\frac{p+1}p)}.
\end{align*}
From this, \eqref{Bj}, and the assumption $r > p$, we infer that, for any $j\in\mathbb N$,
$$
\frac{\|S_\varphi \vec f_j\|_{\dot{b}^{\alpha}_{p,q}(W)}}{\| \vec{f}_j \|_{\dot{B}^{\alpha}_{p,q}(W,\widetilde\varphi)}}
\gtrsim \frac{2^{j(\alpha-\frac{p+1}p)}}{2^{j(\alpha-\frac{r+1}{p})}}
=2^{j\frac{r-p}{p}}
\to\infty
$$
as $j\to\infty$.
This completes the proof of Theorem \ref{boundedfail_p}.
\end{proof}

Next, we prove that under the same conditions as Theorem \ref{boundedfail_p},
the matrix-weighted Besov spaces fail to be well-defined.

\begin{theorem}\label{well_defined_fail}
Let $p \in (0, \infty)$, $q\in (0,\infty]$, and $\alpha \in \mathbb{R}$.
Then, for any $r\in(p,\infty)$,
there exist $W\in \mathcal A_{r,\infty}\setminus\mathcal A_{p,\infty}$
and $\varphi,\phi\in\mathcal S$ satisfying \eqref{hs2}
such that
$\dot{B}^\alpha_{p,q}(W,\varphi)
\neq\dot{B}^\alpha_{p,q}(W,\phi)$.
\end{theorem}

\begin{proof}
Let $W,\varphi,\{\vec f_j\}_{j\in\mathbb Z}$
be as in the proof of Theorem \ref{boundedfail_p}
and $\phi(\cdot):=\varphi(\cdot-1)$.
Then $W\in \mathcal A_{r,\infty}\setminus\mathcal A_{p,\infty}$,
$\varphi,\phi\in\mathcal S$ satisfy \eqref{hs2}, and,
for any $v,j\in\mathbb Z$,
\begin{align}\label{phifj new}
\varphi_v * \vec{f}_j
=\begin{cases}
\vec{f}_j & \text{if } v=j, \\
0 & \text{if } v\neq j
\end{cases}
\quad\text{and}\quad
\phi_v * \vec{f}_j(\cdot)
=\begin{cases}
\vec{f}_j(\cdot-2^{-j}) & \text{if } v=j, \\
0 & \text{if } v\neq j.
\end{cases}
\end{align}
It remains to show $\dot{B}^\alpha_{p,q}(W,\varphi)
\neq\dot{B}^\alpha_{p,q}(W,\phi)$.
Let $\varepsilon\in(0,\frac rp-1)$,
$\delta:=-\alpha+\frac{r+1}p-\varepsilon$,
and $\vec g:= \sum_{j=1}^\infty 2^{j\delta} \vec f_j$.
Using Parseval's identity, we obtain,
for any $N\in\mathbb N$ and $h\in\mathcal S$,
\begin{align*}
\sum_{j=1}^\infty 2^{j\delta}
\left|\left\langle\vec f_j, h\right\rangle\right|
= \sum_{j=1}^\infty 2^{j\delta}
\left|\left\langle \widehat{\vec f_j}, \widehat{h}\right\rangle\right|
\leq \sum_{j=1}^\infty 2^{j\delta}
\left\| \widehat{\vec f_j} \right\|_{L^1}
\sup_{\xi\in \operatorname{supp} \widehat{\vec f_j}}
(1+|\xi|)^{-N} \|\widehat{h}\|_{N,0},
\end{align*}
where $\|\widehat{h}\|_{N,0}$ is as in \eqref{S}.
Note that $\operatorname{supp} \widehat{\vec f_j}
\subset 2^j ([-\frac32, -1] \cup [1, \frac32])$ and
\begin{align*}
\left\| \widehat{\vec f_j} \right\|_{L^1}
=\left\| \begin{pmatrix}
i2^{-2j}(\widehat\gamma)'(2^{-j}\cdot) \\
2^{-j}\widehat\gamma(2^{-j}\cdot)
\end{pmatrix} \right\|_{L^1}
\sim 2^{-j} \|(\widehat\gamma)'\|_{L^1} + \|\widehat\gamma\|_{L^1}.
\end{align*}
Thus, for any $N\in\mathbb N$ with $N>\delta$ and any $h\in\mathcal S$,
\begin{align*}
\sum_{j=1}^\infty 2^{j\delta}
\left|\left\langle\vec f_j, h\right\rangle\right|
\lesssim \sum_{j=1}^\infty 2^{j(\delta-N)} \|\widehat{h}\|_{N,0}
\sim \|\widehat{h}\|_{N,0},
\end{align*}
and hence $\vec g\in (\mathcal S')^2$.
Next, we show $\vec g\in\dot{B}^\alpha_{p,q}(W,\varphi)
\setminus\dot{B}^\alpha_{p,q}(W,\phi)$.
By \eqref{phifj new} and \eqref{82}, we conclude that
\begin{align*}
\left\|\vec g\right\|_{\dot{B}^\alpha_{p,q}(W,\varphi)}
&= \left\|\left\{2^{j(\alpha+\delta)}
\left[\int_{\mathbb R}\left| W^{\frac1p}(t)\vec{f}_j(t) \right|^p \,dt\right]^{\frac1p}
\right\}_{j=1}^\infty\right\|_{\ell^q} \\
&\lesssim \left\|\left\{2^{j\alpha} 2^{j\delta} 2^{-j\frac{r+1}p}
\right\}_{j=1}^\infty\right\|_{\ell^q}
= \left\|\left\{2^{-j\varepsilon}
\right\}_{j=1}^\infty\right\|_{\ell^q}<\infty.
\end{align*}
From \eqref{phifj new} and \eqref{822}
with $2^{-j}k_0$ replaced by $t-2^{-j}$,
it follows that, for any $j\in\mathbb Z$,
\begin{align*}
2^{j\alpha} \left\|\,\left| W^{\frac1p} \phi_j*\vec g \right|\,\right\|_{L^p}
&= 2^{j(\alpha+\delta)} \left[\int_{\mathbb R}
\left| W^{\frac1p}(t)\vec{f}_j(t-2^{-j}) \right|^p \,dt\right]^{\frac1p} \\
&\gtrsim 2^{j(\alpha+\delta)} \left[\int_{\mathbb R}
|\gamma(2^j t-1)|^p \left(1+t^2\right)^{-\frac p2} 2^{-jp} \,dt\right]^{\frac1p}.
\end{align*}
Note that $\operatorname{supp} \widehat{\gamma}
\subset\{\xi\in\mathbb R:\ |\xi| \leq \frac32 \}$.
By this and \cite[Lemma 3.4]{bf2}, we conclude that
$\gamma \neq 0$ almost everywhere.
Let $E\subset[-1,0]$ be a closed interval such that
the minimum of $|\gamma|$ on $E$ is greater than $0$.
Then
\begin{align*}
\int_{\mathbb R}
|\gamma(2^j t-1)|^p \left(1+t^2\right)^{-\frac p2} 2^{-jp} \,dt
\geq |2^{-j}(E+1)| \left[\min_{s\in E} |\gamma(s)|\right]^p
2^{-\frac p2} 2^{-jp}
\gtrsim 2^{-j(p+1)},
\end{align*}
and hence
\begin{align*}
2^{j\alpha} \left\|\,\left| W^{\frac1p} \phi_j*\vec g \right|\,\right\|_{L^p}
\gtrsim 2^{j(\alpha+\delta-1-\frac1p)}
=2^{j(\frac rp-1-\varepsilon)}.
\end{align*}
Note that $\frac rp-1-\varepsilon>0$. Thus,
\begin{align*}
\left\|\vec g\right\|_{\dot{B}^\alpha_{p,q}(W,\phi)}
&\gtrsim \left\|\left\{ 2^{j(\frac rp-1-\varepsilon)}
\right\}_{j=1}^\infty\right\|_{\ell^q}
=\infty.
\end{align*}
This completes the proof of Theorem \ref{well_defined_fail}.
\end{proof}

\section{$\mathcal A_{p,\infty}$-Matrix Weights on Spaces of Homogeneous Type} \label{key}

In this section, we investigate various properties of
$\mathcal A_{p,\infty}$-matrix weights on spaces of homogeneous type.
Since $(\mathbb{R}^n, \rho, dx)$ is a space of homogeneous type
(see \cite[Proposition 2.3]{gjabownik06}),
these results naturally yield the corresponding properties of
$\mathcal A_{p,\infty}$-matrix weights in the anisotropic setting,
thereby laying the foundation for the subsequent study of
matrix-weighted anisotropic Besov spaces.

A natural question is why we do not study anisotropic matrix weights directly.
This is because anisotropic dyadic cubes fail to form a nested structure,
a property that is fundamental in many standard arguments.
As a result, even when our interest lies in anisotropic weights,
it is often necessary to regard them as weights on a space of homogeneous type
and to treat them within this more general framework.

We first recall the concept of spaces of homogeneous type
in the sense of Coifman and Weiss \cite{CG,CG77}.

\begin{definition}\label{quasidis}
A \emph{quasi-metric space} $({\mathcal X}, d)$
is a non-empty set ${\mathcal X}$ equipped with
a \emph{quasi-metric} $d$,
namely, a non-negative function defined on
${\mathcal X} \times {\mathcal X}$ satisfying that,
for any $x,y,z\in{\mathcal X}$,
\begin{enumerate}[{\rm(i)}]
\item $d(x,y)=0$ if and only if $x=y$;

\item $d(x,y)=d(y,x)$;

\item there exists a constant
$A_0\in[1,\infty)$, independent of $x$,
$y$, and $z$, such that
\begin{align*}
d(x,z)\leq A_0[d(x,y)+d(y,z)].
\end{align*}
\end{enumerate}
\end{definition}
The ball $B$ of $\mathcal{X}$, centered at $x \in \mathcal{X}$ with radius $r \in (0, \infty)$, is defined by setting
$$B := \left\{y \in \mathcal{X} : d(y, x) < r\right\} =: B_d(x, r);$$
moreover, for any $\tau \in (0, \infty)$, $\tau B := B_d(x, \tau r)$.
For any ball $B$ in $\mathcal{X}$, let $c_B$ be its center.
For any point $x \in \mathcal{X}$, we assume that the balls $\{B_d(x, r)\}_{r \in (0, \infty)}$ form a basis of open neighborhoods of $x$.

\begin{definition}
Let $({\mathcal X}, d)$ be a quasi-metric space,
equipped with a Borel regular measure $\mu$, which
means that open sets are measurable and every measurable set
$F\subset{\mathcal X}$ is contained in a Borel set $E$ satisfying
that $\mu(F)=\mu(E)$. Moreover, assume that,
for any ball $B\subset{\mathcal X}$, $\mu(B)\in(0,\infty)$.
Then the triple $({\mathcal X}, d, \mu)$ is called
\emph{a space of homogeneous type} if $\mu$
satisfies the following doubling condition:
there exists a positive constant $C\in[1, \infty)$
such that, for any ball $B\subset{\mathcal X}$,
\begin{align}\label{double}
\mu(2B)\leq C\mu(B).
\end{align}
\end{definition}

Let $C_{\mu} := \sup_{\text{ball } B \subset \mathcal{X}} \mu(2B)/\mu(B)$.
Then $C_{\mu} \in [1, \infty)$ is the smallest constant such that \eqref{double} holds true and, for any ball $B \subset \mathcal{X}$ and any $\lambda \in [1, \infty)$,
\begin{align}\label{doublelambda}
\mu(\lambda B) \leq C_{\mu} \lambda^\omega \mu(B),
\end{align}
where $\omega := \log_2 C_{\mu}$.
Moreover, we always assume that $(\mathcal{X}, d, \mu)$ is nonatomic, that is, $\mu(\{x\}) = 0$ for any $x \in \mathcal{X}$.
Throughout this section, a positive constant $C$ is said to
depend on $\mathcal X$ if it depends on $A_0$ and $C_\mu$.

A matrix-valued function $W :\ \mathcal{X} \to M_m(\mathbb{C})$ is called
a \emph{matrix weight} if it satisfies all conditions
in Definition \ref{MatrixWeight} with $\mathbb R^n$ replaced by $\mathcal X$.

\begin{definition}
Let $p\in(0, \infty)$. A matrix weight $W$ on $\mathcal X$ is called an \emph{$\mathcal A_{p,\infty}\left(\mathcal X,\mathbb{C}^{m}\right)$-matrix weight}
if $ W $ satisfies that, for any ball $B\subset{\mathcal X}$,
$\log _{+}(\fint_{B}\|W^{\frac{1}{p}}(x) W^{-\frac{1}{p}}
(\cdot)\|^{p} \,d\mu(x)) \in L^{1}(B)$
and
\begin{align*}
[W]_{\mathcal A_{p, \infty}(\mathcal X, \mathbb{C}^{m})}
:=\sup_{\mathrm{ball}\,B\subset\mathcal X}
\exp \left(\fint_{B} \log \left(
\fint_{B}\left\|W^{\frac{1}{p}}(x) W^{-\frac{1}{p}}(y)\right\|^{p} \,d\mu(x)
\right) \,d\mu(y)\right)<\infty.
\end{align*}
The $\mathcal A_{p,\infty}(\mathcal X, \mathbb{C}^{m})$-matrix weights
reduce to $A_{\infty}(\mathcal X)$-weights when $m = 1$.
\end{definition}

The remainder of this section is organized as follows.
In Subsection \ref{self}, we establish the self-improvement
property of matrix weights.
In Subsection \ref{rescaled}, we study the boundedness of
rescaled matrix-weighted maximal operators.
Subsection \ref{critical} introduces the
critical rescaling index of matrix weights and investigates its properties.
Finally, in Subsection \ref{keylemma}, we establish a sharp estimate
for the composition of reducing operators.

\subsection{Self-Improving Property} \label{self}

In the Euclidean setting, Bu et al. \cite[Proposition 4.1]{bf4}
established the self-improving property of $\mathcal A_{p,\infty}$-matrix weights.
In this subsection, we extend this result to spaces of homogeneous type.
We begin with recalling the concept of
$\mathcal A_p(\mathcal X,\mathbb{C}^m)$-matrix weights
(see, for example, \cite[p.\,490]{fr21}).

\begin{definition} \label{Ap}
Let $p\in (0,\infty)$.
A matrix weight
$ W $ on $ \mathcal X $
is called an
$\mathcal A_p(\mathcal X,\mathbb{C}^m) $-\emph{matrix weight}
if $ W $ satisfies that when $ p \in (0, 1] $,
\begin{align*}
[W]_{\mathcal A_p(\mathcal X,\mathbb{C}^m)}
:= \sup_{\mathrm{ball}\, B\subset\mathcal X}\mathop{\operatorname{ess\,sup}}_{y\in B}
\fint_B \left\| W^{\frac{1}{p}}(x) W^{-\frac{1}{p}}(y) \right\|^p \, d\mu(x)
< \infty
\end{align*}
or that, when $p\in (1,\infty)$,
\begin{align*}
[W]_{\mathcal A_p(\mathcal X,\mathbb{C}^m)}
:= \sup_{\mathrm{ball}\, B\subset\mathcal X}
\fint_B \left[ \fint_B \left\| W^{\frac{1}{p}}(x) W^{-\frac{1}{p}}(y) \right\|^{p'}
\, d\mu(y) \right]^{\frac{p}{p'}} \, d\mu(x)
< \infty.
\end{align*}
In the case $m=1$, $\mathcal A_p(\mathcal X, \mathbb C^m)$-matrix weights
reduce to the classical scalar $A_{p\vee1}(\mathcal X)$-weights.
\end{definition}

Now, we introduce a new weight class $\mathcal A_{p,u}(\mathcal X,\mathbb{C}^m)$.

\begin{definition}
Let $p,u\in (0,\infty)$.
A matrix weight $ W $ on $ \mathcal X $ is called an
$\mathcal A_{p,u}(\mathcal X,\mathbb{C}^m)$-\emph{matrix weight}
if
\begin{align*}
[W]_{\mathcal A_{p,u}(\mathcal X,\mathbb{C}^m)}
:= \sup_{\mathrm{ball}\, B\subset\mathcal X}
\fint_B \left[ \fint_B \left\| W^{\frac{1}{p}}(x) W^{-\frac{1}{p}}(y) \right\|^u
\, d\mu(y) \right]^{\frac pu} \, d\mu(x)
< \infty.
\end{align*}
\end{definition}

The following self-improving property is the main result of this subsection.

\begin{proposition}\label{improve}
Let $p\in(0,\infty)$.
Then the following statements hold.
\begin{enumerate}[{\rm (i)}]
\item $\mathcal A_{p,\infty}(\mathcal X,\mathbb{C}^{m})
=\bigcup_{u\in(0,\infty)}\mathcal A_{p,u}(\mathcal X,\mathbb{C}^{m})$.

\item For any $u\in(0,\infty)$,
$\mathcal A_{p,u}(\mathcal X,\mathbb{C}^{m})
=\bigcup_{q\in(u,\infty)}\mathcal A_{p,q}(\mathcal X,\mathbb{C}^{m})$.

\item If $p\in(0,1]$ and $W$ is a matrix weight, then
$$
W\in\mathcal A_p(\mathcal X,\mathbb C^m)
\Longleftrightarrow
\sup_{u\in(0,\infty)}
[W]_{\mathcal A_{p,u}(\mathcal X,\mathbb{C}^{m})}<\infty.
$$

\item If $p\in(1,\infty)$, then $\mathcal A_p(\mathcal X,\mathbb C^m)
=\bigcup_{u\in(p',\infty)}\mathcal A_{p,u}(\mathcal X,\mathbb{C}^{m})$.
\end{enumerate}
\end{proposition}

\begin{remark}
When $\mathcal X=\mathbb R^n$, Proposition \ref{improve}(i)
is equivalent to \cite[Proposition 4.1]{bf4},
as can be seen from Lemma \ref{Apu} below.
\end{remark}

To prove Proposition \ref{improve}, we need several technical lemmas.
The following conclusion is precisely \cite[Theorem 4.1]{ta12},
which establishes the adjacent dyadic cube systems on $\mathcal X$.

\begin{lemma}\label{dcs}
Let $\delta \in (0,1)$ satisfy $96A_0^6\delta\leq1$.
Then there exist a positive integer $K$,
a countable set of points,
$\{z_{\tau}^{k,t}:\, k\in \mathbb{Z},
\,\tau\in\mathcal{A}_{k},
\,t\in\{1,\ldots,K\}\}\subset\mathcal X$
with $\mathcal{A} _{k}$ being a set of indices,
and a finite number of dyadic grids,
$\{\mathbb{D}^t:=\{Q_\tau^{k,t}:k\in\mathbb{Z},
\,\tau\in \mathcal{A}_{k}\} :\,t\in\{1,\ldots,K\}\}$, such that
\begin{enumerate}[\rm(a)]
\item for any $t\in\{1,\ldots,K\}$,
\begin{enumerate}[\rm(i)]
\item for any $k\in \mathbb{Z},
\bigcup_{\tau\in\mathcal{A}_k}Q_\tau^{k,t}=\mathcal X$ and $\{Q_\tau^{k,t}:\tau\in\mathcal{A}_k\}$ is disjoint;

\item if $k,l\in\mathbb{Z}$ and $k\leq l$, then,
for any $\tau\in\mathcal{A}_k$ and $\beta\in\mathcal{A}_l$, either $Q_\beta^{l,t}\subset Q_\tau^{k,t}$ or
$Q_{\beta}^{l,t}\cap Q_{\tau}^{k,t}=\emptyset;$

\item for any $k\in\mathbb{Z}$ and $\tau\in\mathcal{A}_k,B_d(z_{\tau}^{k,t},(12A_0^{4})^{-1}\delta^{k})
\subset Q_{\tau}^{k,t}\subset B_d(z_{\tau}^{k,t},4A_0^{2}\delta^{k})=:
B(Q_{\tau}^{k,t})$;
\end{enumerate}

\item there exists a  positive constant $C\in(1,\infty)$,
depending only on $A_0$ and $\delta$, such that,
for any ball
$B\subset\mathcal X$, there exist $t\in\{1,\ldots,K\}$ and $Q(B)\in \mathbb{D}^t$ such that $B\subset Q(B)$ and
$\sup\{d(x,y):\ x,y\in Q(B)\}\leq Cr_B.$
\end{enumerate}
\end{lemma}

The collection $\{\mathbb{D}^t\}_{t=1}^K$ in Lemma \ref{dcs}
is called an \emph{adjacent dyadic cube system} on $\mathcal X$.
Moreover, let
\begin{align*}
\mathbb{D}:=\bigcup_{t=1}^K\mathbb{D}^t.
\end{align*}
The following conclusion is exactly \cite[Corollary 7.4]{ta12}.

\begin{lemma}\label{ballcube}
The following statements hold.
\begin{enumerate}[\rm(i)]
\item For any $Q \in \mathbb{D}$, there exists a positive constant $C\in [1, \infty)$, depending only on $\mathcal X$, such that $\mu(B(Q)) \leq C\mu(Q)$, where $B(Q)$ is as in Lemma \ref{dcs}{(a)(iii)} with $Q_{\tau}^{k,t}$ replaced by $Q$.

\item For any ball $B\subset\mathcal X$,
there exist $Q(B) \in \mathbb{D}$ and a positive constant $C\in[1, \infty)$, depending only on $\mathcal X$ and $\delta$, such that $B\subset Q(B)$ and $\mu(Q(B)) \leq C\mu(B)$.
\end{enumerate}
\end{lemma}


The following lemma is well known (see, for instance, \cite[Lemma 2.3]{bf3}).

\begin{lemma}\label{exchange}
Let $A,B\in M_m({\mathbb C})$ be two nonnegative definite matrices.
Then $\|AB\|=\|BA\|$.
\end{lemma}

Now, we recall the concept of reducing operators (see \cite[(3.1)]{v97}).

\begin{definition}
Let $ p \in (0,\infty) $, $W$ be a matrix weight
and $E\subset\mathcal X$ be a measurable set satisfying $\mu(E) \in (0,\infty)$.
The positive matrix $A_{E}$ is called a \emph{reducing operator
of order $p$ for $W$} if,  for any $ \vec z \in \mathbb{C}^m$,
\begin{align} \label{equ_reduce}
\left| A_E \vec z \right|
\sim \left[ \fint_E \left| W^{\frac{1}{p}}(x)\vec z\right|^p \,d\mu(x)
\right]^{\frac{1}{p}},
\end{align}
where the positive equivalence constants depend only on $m$ and $p$.
\end{definition}

\begin{remark}
The existence of reducing operators in the Euclidean spaces was established in \cite[Proposition 1.2]{g03} for $p\in(1,\infty)$,
and in \cite[p.\,1237]{fr04} for $p\in(0,1]$.
Their arguments extend naturally to spaces of homogeneous type, as shown in \cite[Lemma 2.17]{bf}.
\end{remark}

Repeating the argument used in the proof of \cite[Proposition 2.14]{bcyy25},
we find that \eqref{equ_reduce} still holds with
$\vec z\in\mathbb{C}^m$ replaced by $M\in M_m(\mathbb{C})$;
we omit the details.

\begin{lemma} \label{reduceM}
Let $p\in (0,\infty)$, $W$ be a matrix weight,
and a measurable set $E\subset\mathcal X$
satisfy $ \mu(E) \in (0, \infty) $.
Then $A_E$ is a reducing operator of order $p$ for $W$ if and only if
for any matrix $M\in M_m(\mathbb{C})$,
\begin{align*}
\left\| A_E M \right\|\sim \left[ \fint_E
\left\| W^{\frac{1}{p}} (x) M \right\|^p \, d\mu(x) \right]^{\frac{1}{p}},
\end{align*}
where the positive equivalence constants depend only on $m$ and $p$.
\end{lemma}

We need the following equivalent characterization of
$\mathcal A_{p,\infty}(\mathcal X,\mathbb{C}^{m})$-matrix weights.

\begin{lemma}\label{Apcha}
Let $p\in (0,\infty)$ and $W$ be a matrix weight.
For any ball $B\subset\mathcal X$ and $Q\in\mathbb D$,
let $A_B$ and $A_Q$ be reducing operators of order $p$ for $W$.
Then the following conditions are equivalent:
\begin{enumerate}[\rm(i)]
\item $W\in \mathcal A_{p,\infty}(\mathcal X,\mathbb{C}^{m})$;

\item $\mathrm{I}:=\sup_{\mathrm{ball}\, B\subset\mathcal X}
\fint_B\log_+\|W^{-\frac{1}{p}}(x) A_B\|^p\,d\mu(x)<\infty$;

\item $\mathrm{II}:=\sup_{Q\in\mathbb D}
\fint_Q\log_+\|W^{-\frac{1}{p}}(x)A_Q\|^p\,d\mu(x)<\infty$.
\end{enumerate}
Moreover, we have
$$
[W]_{\mathcal A_{p,\infty}(\mathcal X,\mathbb{C}^{m})}
\sim \exp(\mathrm{I})\sim \exp(\mathrm{II}),$$
where the positive equivalence constants are independent of $W$.
\end{lemma}

\begin{proof}
We first prove (i) $\Longleftrightarrow$ (ii).
By Lemmas \ref{reduceM} and \ref{exchange} and the inequality $\log\leq\log_+$,
we obtain
\begin{align*}
[W]_{\mathcal A_{p,\infty}(\mathcal X,\mathbb{C}^{m})}
\sim \sup_{\mathrm{ball}\, B\subset\mathcal X}
\exp\left( \fint_B\log
\left\|W^{-\frac{1}{p}}(x) A_B\right\|^p\,d\mu(x) \right)
\leq \exp(\mathrm{I}).
\end{align*}
Conversely, using Lemmas \ref{reduceM} and \ref{exchange}
and the equality $\log_+ t = \log t + \log_+ t^{-1}$
for all $t\in (0,\infty)$, we conclude that
\begin{align} \label{2.13}
\exp(\mathrm{I})
&= \sup_{\mathrm{ball}\, B\subset\mathcal X}
\exp\left( \fint_B\log \left\|W^{-\frac{1}{p}}(x) A_B\right\|^p\,d\mu(x)
+\fint_B\log_+\left\|W^{-\frac{1}{p}}(x) A_B\right\|^{-p}\,d\mu(x) \right) \notag \\
&\lesssim [W]_{\mathcal A_{p,\infty}(\mathcal X,\mathbb{C}^{m})}
\sup_{\mathrm{ball}\, B\subset\mathcal X}
\exp\left( \fint_B\log_+\left\|A_B^{-1}W^{\frac{1}{p}}(x) \right\|^p\,d\mu(x) \right),
\end{align}
where the last inequality used the fact that
$1\leq \|W^{-\frac{1}{p}}(x) A_B\|^p \|A_B^{-1}W^{\frac{1}{p}}(x) \|^p$.
From Lemmas \ref{reduceM} and \ref{exchange} and the inequality $\log_+t<t$
for all $t\in(0,\infty)$, we infer that
\begin{align*}
\fint_B\log_+\left\|A_B^{-1}W^{\frac{1}{p}}(x) \right\|^p\,d\mu(x)
\leq  \fint_B \left\|W^{\frac{1}{p}}(x)A_B^{-1} \right\|^p\,d\mu(x)
\sim  \left\|A_BA_B^{-1} \right\|^p
= 1.
\end{align*}
This, together with \eqref{2.13}, further implies that
$\exp(\mathrm{I})\lesssim [W]_{\mathcal A_{p,\infty}(\mathcal X,\mathbb{C}^{m})}$.
Thus, $[W]_{\mathcal A_{p,\infty}(\mathcal X,\mathbb{C}^{m})}
\sim \exp(\mathrm{I})$ and (i) $\Longleftrightarrow$ (ii).

Next, we show that (ii) $\Longleftrightarrow$ (iii).
From Lemmas \ref{exchange}, \ref{reduceM}, and \ref{ballcube}(ii),
we infer that, for any ball $B\subset\mathcal X$ and almost every $x\in\mathcal X$,
\begin{align*}
\left\|W^{-\frac{1}{p}}(x)A_B\right\|^p
&= \left\|A_B W^{-\frac{1}{p}}(x)\right\|^p
\sim \fint_B \left\|W^{\frac{1}{p}}(y)W^{-\frac{1}{p}}(x)\right\|^p\,d\mu(y) \notag\\
&\lesssim \fint_{Q(B)}\left\|W^{\frac{1}{p}}(y)W^{-\frac{1}{p}}(x)\right\|^p\,d\mu(y)
\sim\left\|W^{-\frac{1}{p}}(x)A_{Q(B)}\right\|^p,
\end{align*}
where $Q(B)$ is as in Lemma \ref{dcs}(b).
By this and Lemma \ref{ballcube}(ii),
we find that, for any ball $B\subset\mathcal X$,
\begin{align*}
&\exp\left(\fint_B \log_+\left\|W^{-\frac{1}{p}}(x)A_B\right\|^p\,d\mu(x)\right) \\
&\quad\lesssim\exp\left( \fint_B \log_+\left\|
W^{-\frac{1}{p}}(x)A_{Q(B)}\right\|^p\,d\mu(x)\right)\notag\\
&\quad\lesssim\exp\left(\fint_{Q(B)} \log_+\left\|W^{-\frac{1}{p}}(x)A_{Q(B)}\right\|^p\,d\mu(x)\right)
\leq \exp(\mathrm{II}),
\end{align*}
and hence $\exp(\mathrm{I})\lesssim \exp(\mathrm{II})$.
Repeating the argument above with Lemma \ref{ballcube}(ii)
replaced by Lemma \ref{ballcube}(i),
we obtain the reverse inequality.
This completes the proof of Lemma \ref{Apcha}.
\end{proof}

The following lemma is a special case of \cite[Lemma 2.5(i)]{bchyy26}
extended to the setting of spaces of homogeneous type.
Although the original result was established in the Euclidean setting,
its proof carries over naturally to $\mathcal X$; we omit the details.

\begin{lemma}\label{intexchange}
Let $p,q\in(0,\infty]$.
Then, for any balls $B_1,B_2\subset\mathcal X$
and any matrix-valued functions $V_1,V_2:\ \mathcal X\to M_m(\mathbb C)$
satisfying $\|V_1\|\in L^p(B_1)$ and $\|V_2\|\in L^q(B_2)$,
\begin{align*}
&\left\{\fint_{B_2}\left[\fint_{B_1} \|V_1(x)V_2(y)\|^p \,d\mu(x)
\right]^{\frac qp} \,d\mu(y) \right\}^{\frac 1q} \\
&\quad\sim \left\{\fint_{B_1}\left[\fint_{B_2} \|V_1(x)V_2(y)\|^q \,d\mu(y)
\right]^{\frac pq} \,d\mu(x) \right\}^{\frac 1p},
\end{align*}
where the positive equivalence constants are independent of
$V_1$, $V_2$, $B_1$, and $B_2$.
\end{lemma}

We also need the following equivalent characterization of
$\mathcal A_{p,u}(\mathcal X,\mathbb{C}^{m})$-matrix weights.

\begin{lemma} \label{Apu}
Let $p,u\in (0,\infty)$ and $W$ be a matrix weight.
Then $W\in \mathcal A_{p,u}(\mathcal X,\mathbb{C}^{m})$
if and only if
\begin{align*}
[W]_{\mathcal A_{p,u}^*(\mathcal X,\mathbb{C}^{m})}
:=\sup_{\mathrm{ball}\,B\subset\mathcal X}
\left[ \fint_B \left\|A_B W^{-\frac1p}(y) \right\|^{u}\,d\mu(y)\right]^{\frac pu}<\infty.
\end{align*}
Moreover, for any matrix weight $W$,
$[W]_{\mathcal A_{p,u}(\mathcal X,\mathbb{C}^{m})}\sim
[W]_{\mathcal A_{p,u}^*(\mathcal X,\mathbb{C}^{m})}$,
where the implicit positive equivalence constants are independent of $W$.
\end{lemma}

\begin{proof}
By Lemma \ref{reduceM}, we conclude that,
for any ball $B\subset\mathcal X$,
\begin{align}\label{6.14}
\left[ \fint_B \left\|A_B W^{-\frac1p}(y) \right\|^{u}\,d\mu(y) \right]^{\frac1u}
\sim \left\{ \fint_B \left[ \fint_B
\left\|W^{\frac1p}(x) W^{-\frac1p}(y) \right\|^p \,d\mu(x)
\right]^{\frac up} \,d\mu(y) \right\}^{\frac1u}.
\end{align}
Note that $\|W^{\frac1p}\|^p=\|W\|\in L^1(B)$ (due to the local integrability of $W$)
and $\|W^{-\frac1p}\|^u\in L^1(B)$ (as a consequence of $[W]_{\mathcal A_{p,u}(\mathcal X,\mathbb{C}^{m})}<\infty$
or $[W]_{\mathcal A_{p,u}^*(\mathcal X,\mathbb{C}^{m})}<\infty$).
Applying Lemma \ref{intexchange} to the right-hand side of
\eqref{6.14}, we obtain
\begin{align*}
\left[ \fint_B \left\|A_B W^{-\frac1p}(y) \right\|^{u}\,d\mu(y) \right]^{\frac1u}
\sim \left\{ \fint_B \left[ \fint_B
\left\|W^{\frac1p}(x) W^{-\frac1p}(y) \right\|^u \,d\mu(y)
\right]^{\frac pu} \,d\mu(x) \right\}^{\frac1p}.
\end{align*}
Taking the $p$-th power on both sides of the above equality
and then taking the supremum over all balls $B\subset\mathcal X$,
we obtain $[W]_{\mathcal A_{p,u}(\mathcal X,\mathbb{C}^{m})}\sim
[W]_{\mathcal A_{p,u}^*(\mathcal X,\mathbb{C}^{m})}$.
This completes the proof of Lemma \ref{Apu}.
\end{proof}

\begin{lemma} \label{improve D}
Let $p,u\in(0,\infty)$ and $W$ be a matrix weight.
Then $W\in \mathcal A_{p,u}(\mathcal X,\mathbb{C}^{m})$ if and only if
\begin{align*}
\sup_{Q\in \mathbb D} \fint_Q \left\|A_Q W^{-\frac{1}{p}}(y)\right\|^u\,d\mu(y) <\infty.
\end{align*}
\end{lemma}

\begin{proof}
Repeating the proof of the equivalence between (ii) and (iii) in Lemma \ref{Apcha}
with the integrands $\log_+\|W^{-\frac{1}{p}}(x) A_Q\|$
and $\log_+\|W^{-\frac{1}{p}}(x) A_B\|$
replaced, respectively, by $\|A_Q W^{-\frac{1}{p}}(x)\|^u$
and $\|A_B W^{-\frac{1}{p}}(x)\|^u$, we obtain
\begin{equation*}
\sup_{Q\in \mathbb D}
\fint_Q \left\|A_Q W^{-\frac{1}{p}}(y)\right\|^u\,d\mu(y)
\sim \sup_{\mathrm{ball}\, B\subset \mathcal{X}}
\fint_B \left\|A_BW^{-\frac{1}{p}}(y)\right\|^u\,d\mu(y).
\end{equation*}
This, together with Lemma \ref{Apu},
then completes the proof of Lemma \ref{improve D}.
\end{proof}

Next, we recall the reverse H\"{o}lder inequality for matrix weights.
To this end, we first recall Wilson's $A_\infty(\mathcal X)$ constant.
For any scalar weight $w$,
\begin{align*}
[w]_{A_\infty(\mathcal X)}^*:=\sup_{{\mathrm {ball}}\, B \subset\mathcal{X}}\frac{1}{w(B)}\int_B\mathcal{M}(w\mathbf{1}_B)(x)\,d\mu(x).
\end{align*}

The following lemma was proved by Hyt\"onen and P\'erez
\cite[Proposition 2.2]{thcp} in the case $\mathcal X=\mathbb R^n$.
Furthermore, Hyt\"onen et al. \cite[p.\,3886]{tce} claim that
the same conclusion holds for general $\mathcal X$.

\begin{lemma} \label{two Ainfty}
For any scalar weight $w$ on $\mathcal X$,
we have $[w]_{A_\infty(\mathcal X)}^* \le C_{\mathcal X} [w]_{A_\infty(\mathcal X)}$,
where $C_{\mathcal X}$ is a positive constant depending only on $\mathcal X$.
\end{lemma}

Repeating the proof of \cite[Lemma 5.3]{bf4} with Lemma 5.2 therein
replaced by Lemma \ref{two Ainfty}, we have the
following lemma, which gives the relation between scalar and matrix weights;
we omit the details.

\begin{lemma}\label{scalar}
Let	$p\in(0,\infty)$ and $W\in \mathcal A_{p,\infty}(\mathcal X,\mathbb{C}^{m})$.
Then, for any nonzero matrix $M\in M_m(\mathbb{C})$, we have
$w_M:=\|W^{\frac{1}{p}}M\|^p\in A_\infty(\mathcal X)$ with
$\left[w_M\right]_{A_\infty(\mathcal X)}
\leq[W]_{\mathcal A_{p,\infty}(\mathcal X,\mathbb{C}^{m})}$, and
$$
[W]_{\mathcal A_{p,\infty}(\mathcal X,\mathbb{C}^m)}^{\mathrm{sc}}
:=\sup_{M\in M_m(\mathbb{C})\setminus\{O_m\}}
[w_M]_{A_\infty(\mathcal X)}^*
\leq C_{\mathcal X} [W]_{\mathcal A_{p,\infty}(\mathcal X,\mathbb{C}^m)},
$$
where $C_{\mathcal X}$ is as in Lemma \ref{two Ainfty}.
\end{lemma}

The following reverse H\"{o}lder inequality for scalar weights
is exactly \cite[Theorem 1.1]{tce}.

\begin{lemma} \label{scalarrh}
Let $w\in A_\infty(\mathcal X)$.
Then, for any ball $B\subset\mathcal X$,
\begin{align*}
\left[ \fint_B [w(x)]^{r(w)}\,d\mu(x) \right]^{\frac{1}{r(w)}}
\leq 2(4 A_0)^{\omega} \fint_{2 A_0 B} w(x)\,d\mu(x),
\end{align*}
where $\omega$ is as in \eqref{doublelambda} and
\begin{align*}
r(w):=1+\frac{1}{6[32 A_0^2 (4A_0^2+A_0)^2]^{\omega}
[w]_{A_\infty(\mathcal X)}^*}.
\end{align*}
\end{lemma}

Now we prove the reverse H\"older inequality
of $\mathcal A_{p,\infty}(\mathcal X,\mathbb C^m)$-matrix  weights.

\begin{proposition}\label{matrixrhi}
Let $p\in(0,\infty)$ and $W\in \mathcal A_{p,\infty}(\mathcal X,\mathbb{C}^{m})$.
Then, for any ball $B\subset\mathcal X$ and  $M\in M_m(\mathbb{C})$,
$$
\left[ \fint_B\left\|W^{\frac1p}(x)M\right\|^{p r(W)}\,d\mu(x) \right]^{\frac 1{r(W)}}
\leq2(4A_0)^{\omega}
\fint_{2A_0B}\left\| W^{\frac1p}(x)M\right\|^p\,d\mu(x),
$$
where $\omega$ is as in \eqref{doublelambda} and
\begin{align}\label{rW}
r(W):=1+\frac{1}{6[32 A_0^2 (4A_0^2+A_0)^2]^{\omega}
[W]^{\mathrm{sc}}_{\mathcal A_{p,\infty}(\mathcal X,\mathbb C^m)}}.
\end{align}
\end{proposition}

\begin{proof}
If $M$ is a zero matrix, then the present proposition is obvious,
so it is enough to consider the nonzero matrix $M$.
By Lemma \ref{scalar}, we find that the scalar weight
$w_M:=\|W^{\frac1p}M\|^p$ belongs to $A_\infty(\mathcal X)$.
From this,  H\"older's inequality, and Lemma \ref{scalarrh},
we infer that, for any ball $B\subset\mathcal X$,
\begin{align*}
\left\{\fint_B \left[w_M(x)\right]^{r(W)}\,d\mu(x)\right\}^{\frac 1{r(W)}}
&\leq \left\{\fint_B \left[w_M(x)\right]^{r(w_M)}\,d\mu(x) \right\}^{\frac 1{r(w_M)}}\\
&\leq2(4A_0)^{\omega} \fint_{2A_0B} w_M(x)\,d\mu(x).
\end{align*}
This completes the proof of Proposition \ref{matrixrhi}.
\end{proof}

The following conclusion is an analogue of Lemma \ref{scalar}.

\begin{lemma} \label{4.13}
Let $p,u\in(0,\infty)$ and $W\in \mathcal A_{p,u}(\mathcal X,\mathbb{C}^{m})$.
Then, for any $M\in M_m(\mathbb C)\setminus\{O_m\}$,
$w_M:=\|W^{-\frac1p}M\|^u\in A_{\infty}(\mathcal X)$.
Moreover, $[w_M]_{A_{\infty}(\mathcal X)}\leq [W]_{\mathcal A_{p,u}(\mathcal X,\mathbb{C}^{m})}^{\frac up}$.
\end{lemma}

\begin{proof}
For any ball $B\subset\mathcal X$ and almost every $x\in B$,
\begin{align*}
\log\left(\fint_B w_M(y)\,d\mu(y)\right)
\le\log\left(\fint_B \left\|W^{-\frac1p}(y)W^{\frac1p}(x)\right\|^{u}\,d\mu(y)\right)
+\log\left(\left\|W^{-\frac1p}(x)M \right\|^{u}\right).
\end{align*}
Taking the integral average over $x\in B$ on both sides of the above inequality, we obtain
\begin{align*}
&\log\left(\fint_B w_M(y)\,d\mu(y)\right)
+\fint_B\log\left([w_M(x)]^{-1}\right)\,d\mu(x) \\
&\quad \le\fint_B\log\left(\fint_B \left\|W^{-\frac1p}(y)
W^{\frac1p}(x)\right\|^{u}\,d\mu(y)\right)\,d\mu(x).
\end{align*}
Applying this, Lemma \ref{exchange},
and  Jensen's inequality, we conclude that
\begin{align*}
&\fint_B w_M(y)\,d\mu(y)
\exp\left(\fint_B\log \left([w_M(x)]^{-1}\right)\,d\mu(x)\right)\\
&\quad\le \exp\left(\fint_B\log\left(\fint_B
\left\|W^{\frac1p}(x) W^{-\frac1p}(y)
\right\|^{u}\,d\mu(y)\right)\,d\mu(x)\right)\\
&\quad=
\left\{\exp\left(\fint_B\log\left(\left[\fint_B
\left\|W^{\frac1p}(x) W^{-\frac1p}(y) \right\|^{u}
\,d\mu(y)\right]^{\frac p{u}}
\right)\,d\mu(x)\right)\right\}^{\frac {u}p}\\
&\quad\le
\left\{\fint_B \left[\fint_B\
\left\|W^{\frac1p}(x) W^{-\frac1p}(y) \right\|^{u}
\,d\mu(y)\right]^{\frac p{u}} \,d\mu(x)\right\}^{\frac {u}p}
\leq [W]_{\mathcal A_{p,u}(\mathcal X,\mathbb{C}^{m})}^{\frac up}.
\end{align*}
This completes the proof of Lemma \ref{4.13}.
\end{proof}

%

The following lemma shows that if two dyadic cubes from
adjacent generations intersect, then their measures are comparable.

\begin{lemma}\label{QQ}
Let  $k\in\mathbb Z$,  $\tau\in\mathcal A_k$, and $t\in\{1,\ldots,K\}$
be as in Lemma \ref{dcs}. If $Q_{\tau}^{k,t}\subset Q_{\beta}^{k-1,t}$
for some $\beta\in\mathcal A_{k-1}$, then $\mu(Q_{\tau}^{k,t})\sim\mu(Q_{\beta}^{k-1,t})$,
where the positive equivalence constants  depend only on $A_0,\delta$, and $C_{\mu}$.
\end{lemma}

\begin{proof}
Clearly $\mu(Q_{\tau}^{k,t})\le\mu(Q_{\beta}^{k-1,t})$.
We next prove $\mu(Q_{\beta}^{k-1,t})\lesssim\mu(Q_{\tau}^{k,t})$.
From Lemma \ref{dcs}(a)(iii), we infer that
\begin{align}\label{BQQB}
B\left(z_{\tau}^{k,t},(12A_0^{4})^{-1}\delta^{k}\right)
\subset Q_{\tau}^{k,t}\subset Q_{\beta}^{k-1,t}\subset B\left(z_{\beta}^{k-1,t},4A_0^{2}\delta^{k-1}\right).
\end{align}
Moreover, by Definition \ref{quasidis}(iii), we find that,
for any $x\in B_d(z_{\beta}^{k-1,t},4A_0^{2}\delta^{k-1})$,
\begin{align*}
d\left(x,z_{\tau}^{k,t}\right)
\le A_0\left[d\left(x,z_{\beta}^{k-1,t}\right)
+d\left(z_{\tau}^{k,t},z_{\beta}^{k-1,t}\right) \right]
\le 8A_0^{3}\delta^{k-1}.
\end{align*}
Thus, $B_d(z_{\beta}^{k-1,t},4A_0^{2}\delta^{k-1})\subset \lambda B_d(z_{\tau}^{k,t},(12A_0^{4})^{-1}\delta^{k}) $,
where $\lambda:=96A_0^7 \delta^{-1}$.
Using this and \eqref{doublelambda}, we conclude that,
\begin{align*}
\mu\left(B\left(z_{\beta}^{k-1,t},4A_0^{2}\delta^{k-1}\right)\right)
\le\mu\left(\lambda B\left(z_{\tau}^{k,t},(12A_0^{4})^{-1}\delta^{k}\right)\right)
\le C_{\mu} \lambda^\omega
\mu\left(B_d(z_{\tau}^{k,t},(12A_0^{4})^{-1}\delta^{k})\right),
\end{align*}
where $\omega := \log_2 C_{\mu}$.
From this and \eqref{BQQB}, we deduce that
\begin{align*}
\mu\left(Q_{\beta}^{k-1,t}\right)
\le \mu\left(B\left(z_{\beta}^{k-1,t},4A_0^{2}\delta^{k-1}\right)\right)
\le C_{\mu} \lambda^\omega
\mu\left(B\left(z_{\tau}^{k,t},(12A_0^{4})^{-1}\delta^{k}\right)\right)
\le C_{\mu} \lambda^\omega \mu\left(Q_{\tau}^{k,t}\right),
\end{align*}
which completes the proof of Lemma \ref{QQ}.
\end{proof}

Now, we   prove Proposition \ref{improve}.

\begin{proof}[Proof of Proposition \ref{improve}]
We first prove (i).
Repeating the argument used in the proof of
the equivalence between (ii) and (iii) in Lemma \ref{Apcha},
we obtain, for any $u\in(0,\infty)$,
\begin{align*}
K_{\mathbb D}(u)
&:=\sup_{Q\in\mathbb D}\fint_{Q}
\left\|W^{-\frac{1}{p}}(x)A_Q\right\|^u\,d\mu(x)
\sim\sup_{\mathrm{ball}\, B \subset\mathcal X}
\fint_{B}\left\|W^{-\frac{1}{p}}(x)A_B\right\|^u\,d\mu(x),
\end{align*}
where the positive equivalence constants are independent of $W$.
From this and Lemmas \ref{Apcha} and \ref{Apu}, it follows that
to prove (i), we need only to show that
\begin{align}\label{expQB}
M:=\sup_{Q\in\mathbb D}\fint_Q\log_+
\left\|W^{-\frac{1}{p}}(x)A_Q\right\|\,d\mu(x)<\infty
\end{align}
if and only if
there exists $u\in(0,\infty)$ such that $K_{\mathbb D}(u)<\infty$.

If $K_{\mathbb D}(u)<\infty$ for some $u\in(0,\infty)$,
then applying the inequality $\log_+t\leq t$ for all $t\in(0,\infty)$,
we obtain, for any $Q\in\mathbb D$,
\begin{align*}
\fint_Q\log_+\left\|W^{-\frac{1}{p}}(x)A_Q\right\|\,d\mu(x)
=\frac1u\fint_Q\log_+\left\|W^{-\frac{1}{p}}(x)A_Q\right\|^u\,d\mu(x)
\leq\frac1u K_{\mathbb D}(u),
\end{align*}
and hence \eqref{expQB} holds.

On the other hand, assume that \eqref{expQB} holds.
If $M=0$, then, for any $Q\in\mathbb D$ and almost every $x\in Q$, we have
$\|W^{-\frac{1}{p}}(x)A_Q\|\leq 1$, and hence $K_{\mathbb D}(u)\leq 1$.
It therefore remains only to consider $M\in(0,\infty)$.
Let $Q\in\mathbb D$ be fixed.
Then $Q\in\mathbb D^t$ for some $t\in\{1,\ldots,K\}$.
Applying Lemma
\ref{dcs}, we can choose  $\{Q_i\}_{i\in \mathscr I}\subset \mathbb D^t$
to be maximal adjacent dyadic cube within $Q$ such that,
for any $i\in \mathscr I$,
$$\fint_{Q_i}\log_+\left\|W^{-\frac{1}{p}}(x)A_Q\right\|\,d\mu(x)>2M.$$
From this and the definition of $M$, we infer that
\begin{align*}
\sum_{i\in \mathscr I}\mu(Q_i)&<\frac{1}{2M}\sum_{i\in \mathscr I}\int_{Q_i}\log_+\left\|
W^{-\frac{1}{p}}(x)A_Q\right\|\,d\mu(x)\notag\\
&\leq\frac{1}{2M}\mu(Q)\fint_Q\log_+\left\|W^{-\frac{1}{p}}(x)A_Q\right\|\,d\mu(x)
\le\frac{1}{2}\mu(Q).
\end{align*}
For any $x\in\mathcal X$,
\begin{align}\label{QQi}
\mathbf{1}_Q(x)\log_+\left\|W^{-\frac{1}{p}}(x)A_Q\right\|
&=\mathbf{1}_{Q\setminus\bigcup_{i\in\mathscr I}Q_i}(x)\log_+\left\|
W^{-\frac{1}{p}}(x)A_Q\right\|
+\sum_{i\in\mathscr I}\mathbf{1}_{Q_i}(x)
\log_+\left\|W^{-\frac{1}{p}}(x)A_Q\right\| \notag \\
&=:I(x)+II(x).
\end{align}

We first estimate $I(x)$.
Applying Lemma \ref{dcs}{(a)}, we conclude that,
for any $x\in \mathcal X$,
there exist
$\{Q^j\}_{j\in\mathbb N}\subset \mathbb D^t$ satisfying
$\bigcap_{j\in\mathbb N}Q^j=\{x\}$ and $\mu(Q^j)\to0$. Moreover, by \eqref{expQB},
we obtain $\log_+\|W^{-\frac1p}A_Q\|\mathbf 1_Q\in L_{\rm{loc}}^1$.
Using these and Lebesgue's differentiation theorem
(see, for instance, \cite[Corollary 2.6]{acm15}),
we conclude that, for almost every  $x\in\mathcal X$,
\begin{align}\label{LDT}
\log_+\big\|W^{-\frac{1}{p}}(x)A_Q\big\|
=\lim_{j\to\infty}\fint_{ Q^j}\log_+\left\|W^{-\frac{1}{p}}(y)A_Q\right\|\,d\mu(y)
.
\end{align}
From \eqref{LDT} and the choice of $\{Q_i\}_{i\in \mathscr I}$,
we deduce that,
for almost every
$x\in Q\setminus\bigcup_{i\in \mathscr I}Q_i$,
$\log_+\|W^{-\frac{1}{p}}(x)A_Q\|\le 2M$, and hence
\begin{align}\label{Q1}
I(x)\le 2M\mathbf{1}_{Q\setminus\bigcup_{i\in\mathscr I}Q_i}(x).
\end{align}
We next estimate $II(x)$. For any $i\in \mathscr I$ and $x\in Q_i$, we see that
\begin{align}\label{Q2}
\log_+\left\|W^{-\frac{1}{p}}(x)A_Q\right\|
\leq\log_+\left\|W^{-\frac{1}{p}}(x)A_{Q_i}\right\|
+\log_+\left\|A_{Q_i}^{-1}A_Q\right\|.
\end{align}
For any $y\in Q_i$,
$$\log_+\left\|A_{Q_i}^{-1}A_Q\right\|
\leq\log_+\left\|A_{Q_i}^{-1}W^{\frac{1}{p}}(y)\right\|
+\log_+\left\|W^{-\frac{1}{p}}(y)A_Q\right\|.$$
Taking the average over $y\in Q_i$,
we obtain
\begin{align}\label{Q3}
\log_+\left\|A_{Q_i}^{-1}A_Q\right\|
&\leq \fint_{Q_i}\log_+\left\|A_{Q_i}^{-1}W^{\frac{1}{p}}(y)\right\|\,d\mu(y)
+\fint_{Q_i}\log_+\left\|W^{-\frac{1}{p}}(y)A_Q\right\|\,d\mu(y) \notag \\
&=:\rm J_1+J_2.
\end{align}
Using the inequality $\log_+a\le(ep)^{-1}a^p$
with $a=\|A_{Q_i}^{-1}W^{\frac{1}{p}}(y)\|$, together with
Lemmas \ref{exchange} and \ref{reduceM}, we conclude that
\begin{align}\label{Q4}
{\rm J_1}
\lesssim\fint_{Q_i}\left\|A_{Q_i}^{-1}W^{\frac{1}{p}}(y)\right\|^p\,d\mu(y)
\sim\left\|A_{Q_i}^{-1}A_{Q_i}\right\|^p=1.
\end{align}
We then consider $\rm J_2$. For any $i\in \mathscr I$, suppose that $Q_i=Q_{\tau_i}^{k_i,t}$ for $k_i\in\mathbb Z$ and $\tau_i\in\mathcal A_{k_i}$.
Using  (i) and (ii) of Lemma \ref{dcs}(a), we conclude that there exists $\beta\in\mathcal A_{k_i-1}$ such that
$Q_i\subset Q_{\beta}^{k_i-1,t}\subset Q$.
From this, Lemma \ref{QQ}, and the maximality of $Q_i$,
it follows that, for any $i\in \mathscr I$,
\begin{align*}
{\rm J_2}
&\le \frac{ \mu(Q_{\beta}^{k_i-1,t})}{\mu\left(Q_i\right)}\fint_{Q_{\beta}^{k_i-1,t}}
\log_+\left\|W^{-\frac{1}{p}}(y)A_Q\right\|\,d\mu(y)\notag\\
&\sim\fint_{Q_{\beta}^{k_i-1,t}}
\log_+\left\|W^{-\frac{1}{p}}(y)A_Q\right\|\,d\mu(y)
\le 2M,
\end{align*}
which together with \eqref{Q4}, \eqref{Q3}, and \eqref{Q2},
further implies that there exist two positive constants
$C_1,C_2$, independent of $W$, such that,
for any $i\in \mathscr I$ and $x\in Q_i$,
$$
\log_+\left\|W^{-\frac{1}{p}}(x)A_Q\right\|
\le\log_+\left\|W^{-\frac{1}{p}}(x)A_{Q_i}\right\|+ C_1+C_2M,
$$
and hence
\begin{align}\label{Q5}
II(x)\leq \sum_{i\in \mathscr I}\mathbf{1}_{Q_i}(x)
\left(\log_+\left\|W^{-\frac{1}{p}}(x)A_{Q_i}\right\|+ C_1+C_2M\right).
\end{align}
From \eqref{QQi},   \eqref{Q1}, and \eqref{Q5},
it follows that
\begin{align}\label{iterate}
\mathbf{1}_{Q}\log_{+}\left\|W^{-\frac{1}{p}}A_{Q}\right\|
&\leq2M\mathbf{1}_{Q\setminus\bigcup_{i\in \mathscr I}Q_i}
+\sum_{i\in \mathscr  I}\mathbf{1}_{Q_i}
\left(\log_+\left\|W^{-\frac{1}{p}}A_{Q_i}\right\|
+C_1+C_2M\right)\notag\\
&\leq C_{\mathrm{max}}\mathbf{1}_Q+\sum_{i\in\mathscr  I}\mathbf{1}_{Q_i}\log_+\left\|W^{-\frac{1}{p}}A_{Q_i}\right\|,
\end{align}
where $C_{\mathrm{max}}:=\max\{2M,C_1+C_2M\}.$

Observe that the same argument for \eqref{iterate}  also holds for each $\log_+\|W^{-\frac{1}{p}}A_{Q_i}\|$. Iterating this process infinitely, and eventually, we obtain
$$\mathbf{1}_Q\log_+\left\|W^{-\frac{1}{p}}A_Q\right\|\leq C_{\mathrm{max}} \sum_{k=0}^\infty\mathbf{1}_{\Omega_k}
=C_{\mathrm{max}} \sum_{k=0}^\infty(k+1)\mathbf{1}_{\Omega_k\setminus\Omega_{k+1}}, $$
where $\Omega_0:=Q$ and $\Omega_1:=\bigcup_{i\in \mathscr I}Q_i$, and  $\Omega_k\subset\Omega_{k-1}$ is a union of adjacent dyadic cubes with $\mu(\Omega_k)\leq\frac12\mu(\Omega_{k-1})\leq\cdots\leq2^{-k}\mu(Q)$.
Consequently, for any $u\in(0,\frac{\log(6/5)}{C_{\mathrm{max}}})$, we conclude that
\begin{align*}
\fint_Q\left\|W^{-\frac1p}(x)A_Q\right\|^u\,d\mu(x)
&\leq\fint_Q\exp\left(u\log_+\left\|W^{-\frac1p}(x)A_Q\right\|\right)\,d\mu(x)\\
&\leq\frac{1}{\mu(Q)}\sum_{j=0}^{\infty}\int_{\Omega_j\setminus\Omega_{j+1}}
\exp\left(uC_{\mathrm{max}}\sum_{k=0}^{\infty}(k+1)\mathbf1_{\Omega_k\setminus\Omega_{k+1}}\right)\,d\mu(x)\\
&=\frac{1}{\mu(Q)}\sum_{j=0}^{\infty}\int_{\Omega_j\setminus\Omega_{j+1}}
e^{uC_{\mathrm{max}}(j+1)}\,d\mu(x) \\
&\leq\sum_{j=0}^\infty e^{uC_{\mathrm{max}}(j+1)}2^{-j}
= \frac{e^{uC_{\mathrm{max}}}}{1- \frac{e^{uC_{\mathrm{max}}}}{2}}
\leq 3.
\end{align*}
This completes the proof of (i).

Now, we prove (ii).
By H\"older's inequality, we obtain
$\bigcup_{q\in(u,\infty)}\mathcal A_{p,q}(\mathcal X,\mathbb{C}^{m})
\subset \mathcal A_{p,u}(\mathcal X,\mathbb{C}^{m})$.
It remains to show that
\begin{align*}
\mathcal A_{p,u}(\mathcal X,\mathbb{C}^{m})
\subset\bigcup_{q\in(u,\infty)}\mathcal A_{p,q}(\mathcal X,\mathbb{C}^{m}).
\end{align*}
Let $W\in \mathcal A_{p,u}(\mathcal X,\mathbb{C}^{m})$.
From Lemmas \ref{two Ainfty}
and  \ref{4.13}, it follows that,
for any ball $B\subset \mathcal X$,
$w_B:=\|W^{-\frac1p}A_B\|^u\in A_{\infty}(\mathcal X)$ and
$$
[w_B]_{A_{\infty}(\mathcal X)}^*
\lesssim [w_B]_{A_{\infty}(\mathcal X)}
\leq [W]_{\mathcal A_{p,u}(\mathcal X,\mathbb{C}^{m})}^{\frac up}.
$$
Thus, there exists a positive constant $C$,
independent of $B$ and $W$, such that
$$
[w_B]_{A_{\infty}(\mathcal X)}^*
\leq C [W]_{\mathcal A_{p,u}(\mathcal X,\mathbb{C}^{m})}^{\frac up},
$$
and hence
$$
r
:=1+\frac{1}{6[32 A_0^2 (4A_0^2+A_0)^2]^{\omega}C [W]_{\mathcal A_{p,u}(\mathcal X,\mathbb{C}^{m})}^{\frac up}}
\leq r(w_B),
$$
where $r(w_B)$ is as in Lemma \ref{scalarrh}.
Using this, Lemma \ref{exchange},  H\"older's inequality,
and Lemma \ref{scalarrh}, we obtain
\begin{align} \label{6.13(2)}
\left[ \fint_B \left\|A_B W^{-\frac1p}(y) \right\|^{ur}\,d\mu(y) \right]^{\frac1r}
&= \left[ \fint_B [w_B(y)]^r \,d\mu(y) \right]^{\frac1r} \notag \\
&\leq \left[\fint_B [w_B(y)]^{r(w_B)} \,d\mu(y) \right]^{\frac1{r(w_B)}}
\lesssim \fint_{2 A_0 B} w_B(y) \,d\mu(y).
\end{align}
By Lemmas \ref{reduceM} and \ref{exchange} and \eqref{doublelambda}, we conclude that
\begin{align*}
w_B(y)
&=\left\|A_B W^{-\frac1p}(y) \right\|^u
\sim \left[ \frac{\mu(2A_0B)}{\mu(B)}\frac1{\mu(2A_0B)}\int_{B} \left\|W^{\frac1p}(x) W^{-\frac1p}(y) \right\|^p \,d\mu(x) \right]^{\frac up} \\
&\lesssim \left[ \fint_{2A_0 B} \left\|W^{\frac1p}(x) W^{-\frac1p}(y) \right\|^p \,d\mu(x) \right]^{\frac up}
\sim \left\|A_{2A_0B} W^{-\frac1p}(y) \right\|^u.
\end{align*}
Substituting the above estimate for $w_B$ into \eqref{6.13(2)}
and applying Lemma \ref{Apu}, we obtain
\begin{align*}
[W]_{\mathcal A_{p,ur}(\mathcal X,\mathbb{C}^{m})}
&\sim [W]_{\mathcal A_{p,ur}^*(\mathcal X,\mathbb{C}^{m})}
\lesssim \sup_{{\mathrm {ball}}\,B\subset \mathcal X}
\left[\fint_{2 A_0 B} \left\|A_{2A_0B} W^{-\frac1p}(y) \right\|^u \,d\mu(y) \right]^{\frac pu} \\
&=[W]_{\mathcal A_{p,u}^*(\mathcal X,\mathbb{C}^{m})}
\sim [W]_{\mathcal A_{p,u}(\mathcal X,\mathbb{C}^{m})}<\infty.
\end{align*}
Therefore, $W\in \mathcal A_{p,ur}(\mathcal X,\mathbb{C}^{m})$.
This completes the proof of (ii).

Next, we prove (iii).
We first show ``$\Longrightarrow$''.
From Lemma \ref{intexchange},
it follows that, for any ball $B\subset\mathcal X$,
\begin{align} \label{729}
I(B)
&:=\fint_B \mathop{\operatorname{ess\,sup}}_{z\in B}
\left\|W^{\frac1p}(x) W^{-\frac1p}(z) \right\|^p\,d\mu(x) \notag \\
&\phantom{:}\sim \mathop{\operatorname{ess\,sup}}_{z\in B}
\fint_B \left\|W^{\frac1p}(x) W^{-\frac1p}(z) \right\|^p \,d\mu(x)
=: II(B).
\end{align}
By this and Lemmas \ref{Apu} and \ref{reduceM},
we find that, for any $W\in \mathcal A_p(\mathcal X,\mathbb C^m)$ and $u\in(0,\infty)$,
\begin{align*}
[W]_{\mathcal A_{p,u}(\mathcal X,\mathbb{C}^{m})}
\leq \sup_{\mathrm{ball}\,B\subset\mathcal X} I(B)
\sim \sup_{\mathrm{ball}\,B\subset\mathcal X} II(B)
= [W]_{\mathcal A_p(\mathcal{X},\mathbb C^m)}.
\end{align*}
This completes the proof of ``$\Longrightarrow$''.

We next prove ``$\Longleftarrow$'' by contradiction.
Assume that $W\notin\mathcal A_p(\mathcal{X},\mathbb C^m)$.
From the definition of $\mathcal A_p(\mathcal{X},\mathbb C^m)$
and \eqref{729}, it follows that,
for any $M\in(0,\infty)$, there exists a ball $B\subset\mathcal X$
such that $M<II(B)\sim I(B)$.
This, together with Lemma \ref{Apu}, further implies that
\begin{align*}
\sup_{u\in(0,\infty)} [W]_{\mathcal A_{p,u}(\mathcal X,\mathbb{C}^{m})}
&\geq \fint_{B} \lim_{u\to\infty} \left[
\fint_B \left\|W^{\frac1p}(x) W^{-\frac1p}(y) \right\|^{u}\,d\mu(y)\right]^{\frac pu} \,d\mu(x) \\
&= I(B) \gtrsim M.
\end{align*}
This contradicts the assumption that $\sup_{u\in(0,\infty)} [W]_{\mathcal A_{p,u}(\mathcal X,\mathbb{C}^{m})}<\infty$,
and hence $W\in\mathcal A_p(\mathcal{X},\mathbb C^m)$,
which completes the proof of (iii).

Finally, we prove (iv).
By the definitions of $\mathcal A_p(\mathcal X,\mathbb C^m)$
and $\mathcal A_{p,p'}(\mathcal X,\mathbb C^m)$, we obtain
$$
\mathcal A_p(\mathcal X,\mathbb C^m)
=\mathcal A_{p,p'}(\mathcal X,\mathbb C^m)
=\bigcup_{u\in(p',\infty)}\mathcal A_{p,u}(\mathcal X,\mathbb{C}^{m}),
$$
where the last equality used (ii).
This completes the proof of (iv) and hence Proposition \ref{improve}.
\end{proof}

From \cite{C87} (see also \cite[Theorem IX.2.1]{b97}), it follows that,
for any two nonnegative definite matrices $A,B\in M_m(\mathbb C)$
and any $s\in[0,1]$,
\begin{equation}\label{Cordes}
\|A^sB^s\|\leq \|AB\|^s.
\end{equation}
Repeating the proof of \cite[Proposition 4.2]{bf4} with
cube $Q$ replaced by ball $B$
and using \eqref{Cordes} instead of \cite[Lemma 2]{ipr21},
we obtain the following results; we omit the details.

\begin{proposition}\label{subset}
Let $p\in(0,\infty)$. Then the following two statements hold.
\begin{enumerate}[{\rm(i)}]
\item If $q\in(p,\infty)$, then $\mathcal A_{p,\infty}(\mathcal X,\mathbb C^m)
\subset \mathcal A_{q,\infty}(\mathcal X,\mathbb C^m)$.
Moreover, for any matrix weight $W$,
\begin{align*}
[W]_{\mathcal A_{q,\infty}(\mathcal X,\mathbb C^m)}
\leq [W]_{\mathcal A_{p,\infty}(\mathcal X,\mathbb C^m)}.
\end{align*}

\item $\mathcal A_p(\mathcal X,\mathbb{C}^m)
\subset
\mathcal A_{p,\infty}(\mathcal X,\mathbb{C}^m)
\subset
\bigcup_{q\in(0,\infty)}\mathcal A_q(\mathcal X,\mathbb{C}^m)$.
Moreover, there exists a positive constant $C$, depending only on $m$ and $p$, such that, for any matrix weight $W$,
$$[W]_{\mathcal A_{p,\infty}(\mathcal X,\mathbb{C}^m)}\leq C[W]_{\mathcal A_{p}(\mathcal X,\mathbb{C}^m)},$$
where $C=1$ when $p\in(0,1]$.
\end{enumerate}
\end{proposition}

\begin{remark}
If $\mathcal{X}=\mathbb R^n$ and $m\geq 2$,
then all inclusions in Proposition \ref{subset} are proper
(see (v) through (vii) of \cite[Proposition 2.26]{byyz25}).
\end{remark}

\subsection{Rescaled Maximal Operators}\label{rescaled}

The classical Hardy--Littlewood maximal operator $\mathcal M$
is bounded on $L^p$ for $p\in(1,\infty]$.
To extend this boundedness to $p\in(0,\infty)$,
a standard technique is to use the rescaled maximal operator
$[\mathcal M(|\cdot|^a)]^{\frac1a}$.
However, the situation is different for the matrix-weighted maximal operator
because the matrix and the vector cannot be separated.
Very recently, Yang et al. \cite{yyzong} and Bu et al. \cite{bchyy26}
introduced the rescaled matrix-weighted maximal operator
in the Euclidean setting and established its boundedness.
In this subsection, we systematically investigate the
rescaled matrix-weighted maximal operator on spaces of homogeneous type.

\begin{definition}
Let $p,v\in(0,\infty)$ and $W$ be a matrix weight on $\mathcal X$.
The \emph{rescaled maximal operator} $\mathcal M_{W,p}^{(v)}$ is defined by setting,
for any measurable matrix-valued function
$G:\ \mathcal X\to M_m(\mathbb C)$ and any $x\in\mathcal X$,
$$
\mathcal M_{W,p}^{(v)} (G)(x)
:=\sup_{\genfrac{}{}{0pt}{}{\mathrm{ball}\,B\subset\mathcal X}{B \ni x}}
\left[\fint_{B} \left\| W^{\frac1p}(x) W^{-\frac1p}(y) G(y) \right\|^v d\mu(y)\right]^{\frac1v}
$$
or for any measurable vector-valued function
$\vec f:\ \mathcal X\to \mathbb C^m$ and any $x\in\mathcal X$,
$$
\mathcal M_{W,p}^{(v)} \left(\vec f\right) (x)
:=\sup_{\genfrac{}{}{0pt}{}{\mathrm{ball}\,B\subset\mathcal X}{B \ni x}}
\left[\fint_{B} \left| W^{\frac1p}(x) W^{-\frac1p}(y) \vec f(y) \right|^v d\mu(y)\right]^{\frac1v}.
$$
Moreover, $\mathcal M_{W,p}^{(1)}$ is simply denoted by $\mathcal M_{W,p}$.
\end{definition}

For any $r\in(0,\infty)$, let $\mathcal{M}_{W,p}^{(v,r)}$ be
the restricted version of the rescaled maximal operator $\mathcal M_{W,p}^{(v)}$,
defined by taking the supremum over all balls containing $x$ with radius $r$.

\begin{definition}\label{general Lp}
Let $p \in (0, \infty)$. The \emph{Lebesgue space $L^p(\mathcal X, M_m(\mathbb C))$}
is defined to be the set of all matrix-valued measurable functions
$G$ on $\mathcal X$ such that
$$
\left\| G\right\|_{L^p(\mathcal X, M_m(\mathbb C))} := \left[ \int_{\mathcal{X}} \| G(x)\|^p \, d\mu(x) \right]^{\frac{1}{p}}<\infty.
$$
\end{definition}

The space $L^p(\mathcal X, \mathbb C^m)$
is defined analogously to Definition \ref{general Lp}
with $M_m(\mathbb C)$ replaced by $\mathbb C^m$.
The following three theorems are the main results of this subsection.
The first theorem provides an equivalent characterization of
the class $\mathcal A_{p,\infty}(\mathcal{X},\mathbb C^m)$
in terms of the boundedness of the rescaled maximal operator.

\begin{theorem} \label{bounded B1}
Let $p\in(0,\infty)$ and $W$ be a matrix weight on $\mathcal X$.
Then $W\in \mathcal A_{p,\infty}(\mathcal{X},\mathbb C^m)$
if and only if there exists $v\in(0,p)$ such that
$$
\left\| \mathcal M_{W,p}^{(v)} \right\|
_{L^p(\mathcal X,M_m(\mathbb C))\to L^p(\mathcal X)}
:=\sup_{\|G\|_{L^p(\mathcal X,M_m(\mathbb C))}=1}
\left\| \mathcal M_{W,p}^{(v)}(G) \right\|_{L^p(\mathcal X)}
<\infty.
$$
\end{theorem}

\begin{remark}
\begin{enumerate}[\rm(i)]
\item The implication ``$\Longleftarrow$''
is new even in the Euclidean setting.

\item The range $v\in(0,p)$ is sharp, that is,
if $0<p\leq v<\infty$ and $W$ is a matrix weight, then
$\| \mathcal M_{W,p}^{(v)} \|
_{L^p(\mathcal X,M_m(\mathbb C))\to L^p(\mathcal X)}=\infty$;
see Proposition \ref{bounded fail} below.

\item In the Euclidean setting, \cite[Theorem 2.22]{yyzong}
established the boundedness of $\mathcal M_{W,p}^{(v)}$
for variable matrix $\mathcal A_{p(\cdot),\infty}$ weights,
while \cite[Corollary 4.32]{bchyy26}
proved the boundedness of $\mathcal M_{W,p}^{(v)}$
for product matrix $\mathcal A_p$ weights.
\end{enumerate}
\end{remark}

The second theorem provides  a further refinement of Theorem \ref{bounded B1},
because  $\mathcal A_{p,\infty}(\mathcal{X},\mathbb C^m)
=\bigcup_{u\in(0,\infty)} \mathcal A_{p,u}(\mathcal{X},\mathbb C^m)$
[see Proposition \ref{improve}(i) above].

\begin{theorem} \label{bounded B2}
Let $0<v<p<\infty$ and $W$ be a matrix weight on $\mathcal X$.
Then the following statements are equivalent:
\begin{enumerate}[\rm(i)]
\item $\| \mathcal M_{W,p}^{(v)} \|
_{L^p(\mathcal X,M_m(\mathbb C))\to L^p(\mathcal X)}<\infty$;

\item $\| \mathcal M_{W,p}^{(v)} \|
_{L^p(\mathcal X,\mathbb C^m)\to L^p(\mathcal X)}<\infty$;

\item $\sup_{r\in(0,\infty)} \| \mathcal M_{W,p}^{(v,r)} \|
_{L^p(\mathcal X,M_m(\mathbb C))\to L^p(\mathcal X)}<\infty$;

\item $\sup_{r\in(0,\infty)} \| \mathcal M_{W,p}^{(v,r)} \|
_{L^p(\mathcal X,\mathbb C^m)\to L^p(\mathcal X)}<\infty$;

\item $W\in\mathcal A_{p,\frac{pv}{p-v}}(\mathcal{X},\mathbb C^m)$.
\end{enumerate}
\end{theorem}

The third theorem shows that if $W\in \mathcal A_{p,u}(\mathcal{X},\mathbb C^m)$ for some $u\in(0,\infty)$,
then the rescaled maximal operator satisfies a stronger boundedness property.

\begin{theorem} \label{bounded B3}
Let $p,u\in(0,\infty)$ and $W\in \mathcal A_{p,u}(\mathcal{X},\mathbb C^m)$. Then, for any
$$
v\in \left(0, \frac{pu}{p+u}\right)
\quad\text{and}\quad
q\in \left( \frac{uv}{u-v}, pr(W)\right],
$$
where $r(W)$ is as in \eqref{rW},
  $\| \mathcal M_{W,p}^{(v)} \|
_{L^q(\mathcal X,M_m(\mathbb C))\to L^q(\mathcal X)}
<\infty$.
\end{theorem}

%
%

To prove these theorems, we need to
the following dyadic rescaled maximal function.

\begin{definition}
Let $p,v\in(0,\infty)$, $t\in\{1,\ldots,K\}$, and $W$ be a matrix weight on $\mathcal X$. The \emph{dyadic rescaled maximal operator} $\mathcal M_{W,p}^{\mathbb D^t,(v)}$ is defined by setting,
for any measurable matrix-valued function
$G:\ \mathcal X\to M_m(\mathbb C)$ and any $x\in\mathcal X$,
$$
\mathcal M_{W,p}^{\mathbb D^t,(v)} (G)(x)
:=\sup_{\genfrac{}{}{0pt}{}{Q\in\mathbb D^t}{Q \ni x}}
\left[\fint_{Q}
\left\| W^{\frac1p}(x) W^{-\frac1p}(y) G(y) \right\|^v \,  d\mu(y)\right]^{\frac1v}.
$$
\end{definition}

The following lemma establishes the equivalence
between the rescaled and dyadic rescaled maximal functions.
Its proof is standard (see, for instance, \cite[Lemma A.34]{bf}), and we omit the details.

\begin{lemma}\label{sumk}
Let $p,v\in(0,\infty)$ and $W$ be a matrix weight on $\mathcal X$.
Then, for any measurable matrix-valued function
$G:\ \mathcal X\to M_m(\mathbb C)$ and any $x\in\mathcal X$,
\begin{align*}
\mathcal M_{W,p}^{(v)} (G)(x)
\sim \sum_{t=1}^K \mathcal M_{W,p}^{\mathbb D^t,(v)} (G)(x),
\end{align*}
where the positive equivalence constants depend only on $\mathcal X$
and the parameter $\delta$ from Lemma \ref{dcs}.
\end{lemma}

Repeating the proof of \cite[Proposition 3.8]{bf4} with cube $Q$ replaced by ball $B$, we have
the following result, which provides an equivalent characterization of
$\mathcal A_{p,\infty}(\mathcal X,\mathbb{C}^{m})$-matrix weights.

\begin{proposition}\label{chawap}
Let $p\in (0,\infty)$ and $W$ be a matrix weight.
Assume that, for any ball $B\subset\mathcal X$,
$$
\log_+\left(\int_B\left\|W^{\frac1p}(x)W^{-\frac1p}(\cdot)\right\|^p\,d\mu(x)
\right)\in L^1(B).
$$
Then
\begin{align*}
[W]_{\mathcal A_{p,\infty}(\mathcal X,\mathbb{C}^{m})}
&=\sup_{\mathrm{ball}\,B\subset\mathcal X}
\sup_{G\in\mathcal{G}_B}\left[
\fint_B\left\|W^{\frac1p}(y)G(y)\right\|^p\,d\mu(y)\right]^{-1}\notag\\
&\quad\times\exp\left(\fint_B\log\left(\fint_B\left\|W^{\frac1p}(x)G(y)\right\|^p\,d\mu(x)
\right)\,d\mu(y)\right),
\end{align*}
where
\begin{align*}
\mathcal{G}_{B}:=
&\left\{G:\:\mathcal X\to M_m(\mathbb{C})\:
measurable:\:\fint_B\left\|
W^{\frac1p}(y)G(y)\right\|^p\,d\mu(y)\neq0,\right. \\
&\quad\left. \log_+\left(\fint_B\left\|W^{\frac1p}(x)
G(\cdot)\right\|^p\,d\mu(x)\right)\in L^1(B)\right\}.
\end{align*}
\end{proposition}

The following result was first obtained by Volberg \cite[Lemma 3.1]{v97}
in the Euclidean setting. Later, Bu et al. \cite[Lemma 3.10]{bf4}
provided an alternative proof.
We extend this result to the spaces of homogeneous type.

\begin{lemma}\label{8 prepare}
Let $t\in\{1,\ldots,K\}$, $p\in(0,\infty)$,
$W\in \mathcal A_{p,\infty}(\mathcal{X},\mathbb C^m)$,
and $\{A_Q\}_{Q\in\mathbb D^t}$ be a family of
reducing operators of order $p$ for $W$.
Let $R\in\mathbb D^t$ and let
$\{Q_j\}_{j\in J}\subset \mathbb D^t$ be a pairwise disjoint family of
dyadic cubes contained in $R$.
Then there exists a positive constant $\mathscr C$,
depending on $[W]_{\mathcal A_{p,\infty}(\mathcal{X},\mathbb C^m)}$,
such that, for any $M\in(1,\infty)$,
if $\|A_{2A_0 B(R)} A_{Q_j}^{-1}\|^p\geq M$ for all $j\in J$, then
\begin{align*}
\sum_{j\in J} \mu(Q_j)
\leq \frac{\mathscr C}{\log M} \mu(R),
\end{align*}
where $B(R)$ is as in Lemma \ref{dcs}(a)(iii).
\end{lemma}

\begin{proof}
Let $B:=2A_0 B(R)$ and
$G:=\sum_{j\in J} \mathbf{1}_{Q_j} A_{Q_j}^{-1}
+\mathbf{1}_{B\setminus\bigcup_j Q_j} A_B^{-1}$. Then
\begin{align*}
&\int_B\log_+\left(\fint_B\left\|W^{\frac1p}(x)
G(y)\right\|^p\,d\mu(x)\right)\,d\mu(y) \\
&\quad\leq \int_B\log_+\left(\fint_B\left\|W^{\frac1p}(x)
W^{-\frac1p}(y)\right\|^p\,d\mu(x)\right)\,d\mu(y) \\
&\qquad+\int_B\log_+\left(\left\|W^{\frac1p}(y)
G(y)\right\|^p\right)\,d\mu(y) \\
&\quad=:\mathrm{J}_1+\mathrm{J}_2.
\end{align*}
By $W\in \mathcal A_{p,\infty}(\mathcal{X},\mathbb C^m)$,
we find that $\mathrm{J}_1<\infty$.
From the inequality $\log_+t<t$ for all $t\in(0,\infty)$,
it follows that
\begin{align*}
\mathrm{J}_2
&\leq \sum_{j\in J} \int_{Q_j} \left\|W^{\frac1p}(y) A_{Q_j}^{-1}\right\|^p \,d\mu(y)
+ \int_{B\setminus\bigcup_j Q_j} \left\|W^{\frac1p}(y) A_B^{-1}\right\|^p \,d\mu(y) \\
&\lesssim \sum_{j\in J} \mu(Q_j) \left\|A_{Q_j} A_{Q_j}^{-1}\right\|^p
+ \mu(B) \left\|A_B A_B^{-1}\right\|^p
\leq 2 \mu(B).
\end{align*}
Thus, $G\in \mathcal{G}_{B}$,
where $\mathcal{G}_{B}$ is as in Proposition \ref{chawap}.
Consequently, the same proposition shows that
\begin{align} \label{final}
\mathrm{I}
&:=\exp\left( \fint_B \log \left( \fint_B \left\|W^{\frac{1}{p}}(x)G(y)\right\|^p \,d\mu(x)\right) \,d\mu(y) \right) \notag \\
&\phantom{:}\leq [W]_{\mathcal A_{p,\infty}(\mathcal X,\mathbb{C}^{m})}
\fint_B \left\|W^{\frac{1}{p}}(y)G(y)\right\|^p \,d\mu(y)
=: [W]_{\mathcal A_{p,\infty}(\mathcal X,\mathbb{C}^{m})} \mathrm{II}.
\end{align}
By Lemma \ref{reduceM}, we find that
\begin{align*}
\mathrm{I}
&\sim\exp \left( \fint_B \log \|A_B G(y)\|^p \,d\mu(y) \right) \\
&=\exp\left( \frac{1}{\mu(B)} \left[\sum_{j\in J} \mu(Q_j) \log\left\|A_B A_{Q_j}^{-1}\right\|^p+ \mu\left( B\setminus\bigcup_{j\in J} Q_j\right)
\log\left\|A_B A_B^{-1}\right\|^p \right] \right)  \\
&=\exp\left( \frac{1}{\mu(B)}\sum_{j\in J} \mu(Q_j) \log\left\|A_B A_{Q_j}^{-1}\right\|^p \right)
\geq \exp \left( \frac{1}{\mu(B)}\sum_{j\in J} \mu(Q_j) \log M \right).
\end{align*}
On the other hand, by Lemmas \ref{dcs}(a)(iii) and \ref{ballcube}(i),
we find that $R\subset B$ and $\mu(R)\sim\mu(B)$.
From these and Lemma \ref{reduceM}, we infer that
\begin{align*}
\mathrm{II}
&=\frac{1}{\mu(B)} \left[ \sum_{j\in J} \int_{Q_j}\left\|W^{\frac{1}{p}}(y)A_{Q_j}^{-1}\right\|^p \,d\mu(y)
+\int_{B\setminus\bigcup_j Q_j}\left\|W^{\frac{1}{p}}(y) A_B^{-1}\right\|^p \,d\mu(y) \right] \\
&\leq\frac{1}{\mu(B)} \left[ \sum_{j\in J}\int_{Q_j}\left\|W^{\frac{1}{p}}(y)A_{Q_j}^{-1}\right\|^p \,d\mu(y)
+\int_B \left\|W^{\frac{1}{p}}(y)A_B^{-1}\right\|^p \,d\mu(y) \right] \\
&\sim\frac{1}{\mu(B)} \left[ \sum_{j\in J} \mu(Q_j) \left\|A_{Q_j}A_{Q_j}^{-1}\right\|^p
+\mu(B) \left\|A_B A_B^{-1}\right\|^p \right]
= \frac{1}{\mu(B)} \left[ \sum_{j\in J} \mu(Q_j) +\mu(B) \right]
\sim 1.
\end{align*}
Substituting the estimates of $\mathrm{I}$ and
$\mathrm{II}$ into \eqref{final}, we obtain
\begin{align*}
\exp\left( \frac{1}{\mu(B)}\sum_{j\in J} \mu(Q_j) \log M \right)
\lesssim [W]_{\mathcal A_{p,\infty}(\mathcal X,\mathbb{C}^{m})},
\end{align*}
which completes the proof of Lemma \ref{8 prepare}.
\end{proof}

The following result was already established for $W\in \mathcal A_p(\mathcal X,\mathbb C^m)$ in \cite[Lemma A.32]{bf};
indeed, it remains valid for the more general class $\mathcal A_{p,\infty}(\mathcal X,\mathbb C^m)$.

\begin{lemma} \label{lem1}
Let $t\in\{1,\ldots,K\}$, $p\in(0,\infty)$,
$W\in \mathcal A_{p,\infty}(\mathcal X,\mathbb C^m)$, and
$\{A_Q\}_{Q \in \mathbb{D}^t}$ be a sequence of reducing operators of order $p$ for $W$.
Then there exists a positive constant $C$,
depending on $[W]_{\mathcal A_{p,\infty}(\mathcal{X},\mathbb C^m)}$,
such that, for any $q\in(0,p r(W)]$,
$$
\sup_{R\in \mathbb{D}^t}
\left[\fint_R \sup_{\genfrac{}{}{0pt}{}{Q\in \mathbb{D}^t}{x\in Q\subset R}}
\left\| W^{\frac1p} (x) A_Q^{-1} \right\|^q \, d\mu(x) \right]^{\frac1q}
\leq C.
$$
\end{lemma}

\begin{proof}
By H\"older's inequality, to prove the present lemma it suffices to show
$$
\sup_{R\in \mathbb{D}^t}
\left[\fint_R \sup_{\genfrac{}{}{0pt}{}{Q\in \mathbb{D}^t}{x\in Q\subset R}}
\left\| W^{\frac1p} (x) A_Q^{-1} \right\|^{pr(W)} \, d\mu(x)\right]^{\frac1{pr(W)}}
<\infty.
$$
For any $l\in\mathbb Z$, $R\in\mathbb{D}^t$, and $x\in R$, let
$$
N_{R, t, l}(x) := \sup_{\{ Q_\alpha^{k, t} :\,
k \leq l,\, \alpha \in \mathcal{A}_k,\,
x \in Q_\alpha^{k, t} \subset R\}}
\left\| W^{\frac1p} (x) A_{Q_{\alpha}^{k, t}}^{-1} \right\|,
$$
with the convention that the supremum over an empty index set is $0$.
To prove the present lemma, it suffices to show that,
for any $l\in\mathbb Z$ and $R\in \mathbb{D}^t $,
\begin{equation} \label{1010}
\int_R [N_{R,t,l}(x)]^{pr(W)} \, d\mu(x) \lesssim \mu(R),
\end{equation}
where the implicit positive constant is independent of $l$ and $R$.

Now, we show \eqref{1010}. To this end, fix $R:=Q_\alpha^{k_0,t}\in \mathbb{D}^t$.
Let $l\in\{k_0,k_0+1,k_0+2,\ldots\}$ and
$M:= \exp(2\mathscr C)$,
where $\mathscr C$ is as in Lemma \ref{8 prepare}.
Let
$$
F_1:=\left\{ Q_\alpha^{k,t} \subset R :\ k \leq l,\ \alpha \in \mathcal{A}_k,\
\left\|A_{2A_0 B(R)} A_{Q_{\alpha}^{k, t}}^{-1} \right\|^p \geq M \right\},
$$
and
$$
D_1:=\{P\in F_1 :\ P \text{ is the maximal element of } F_1 \}.
$$
It is easy to show that $D_1$ is pairwise disjoint
and $\bigcup_{P\in D_1} P\subset R$.
From this, the definition of $D_1$, and Lemma \ref{8 prepare},
we deduce that
\begin{align} \label{haha5}
\sum_{P \in D_1} \mu(P)
\leq \frac12 \mu(R).
\end{align}
Note that, by the definition of $D_1$,
we have, for any $x\in R\setminus \bigcup_{P\in D_1} P$
and $ Q_{\alpha}^{k,t} $ satisfying
$ x \in Q_{\alpha}^{k,t} \subset R $ with $ k \leq l $ and $ \alpha \in \mathcal{A}_k $,
\begin{align*}
\left\| W^{\frac1p}(x) A_{Q_{\alpha}^{k, t}}^{-1} \right\|
\leq \left\| W^{\frac1p}(x) A_{2A_0 B(R)}^{-1} \right\|
\left\| A_{2A_0 B(R)} A_{Q_{\alpha}^{k, t}}^{-1} \right\|
\leq \left\| W^{\frac1p}(x) A_{2A_0 B(R)}^{-1} \right\| M^{\frac1p}.
\end{align*}
From this, we deduce that
\begin{align}\label{intmea}
\int_{R \setminus (\bigcup_{P \in D_1} P)} [ N_{R,t,l}(x)]^{pr(W)} \, d\mu(x)
\lesssim  \int_R \left\| W^{\frac1p}(x) A_{2A_0 B(R)}^{-1} \right\|^{pr(W)} \, d\mu(x).
\end{align}
By Lemma \ref{dcs}, Proposition \ref{matrixrhi}, and Lemma \ref{reduceM}, we have
\begin{align} \label{2.14}
\fint_R \left\| W^{\frac1p}(x) A_{2A_0 B(R)}^{-1} \right\|^{pr(W)} \, d\mu(x)
&\lesssim \fint_{B(R)} \left\| W^{\frac1p}(x) A_{2A_0 B(R)}^{-1} \right\|^{pr(W)} \, d\mu(x) \notag \\
&\lesssim \left[\fint_{2 A_0 B(R)} \left\| W^{\frac1p}(x) A_{2A_0 B(R)}^{-1} \right\|^p \, d\mu(x) \right]^{r(W)} \notag \\
&\sim \left\|A_{2 A_0 B(R)} A_{2A_0 B(R)}^{-1}\right\|^{p r(W)}
=1,
\end{align}
which, combined with \eqref{intmea}, further implies that
\begin{align} \label{haha3}
\int_{R \setminus (\bigcup_{P \in D_1} P)} [N_{R,t,l}(x)]^{pr(W)} \, d\mu(x)
\lesssim \mu(R).
\end{align}

If $D_1=\emptyset$, then \eqref{haha3}
finishes the proof of the present lemma in this case.
If $D_1\neq\emptyset$, then by \eqref{haha5},
we find that, for any $P := Q_\alpha^{k, t} \in D_1 $,
one has $P \subsetneqq R$, and hence
$
k \in [k_0+1, l].
$
For any $P \in D_1 $, let
\begin{equation*}
F_P := \{ x \in P :\ N_{R,t,l}(x) > N_{P,t,l}(x) \}.
\end{equation*}
From the definitions of $N_{R,t,l}$ and $F_P$ and Lemma \ref{dcs}(a)(ii),
we deduce that, for any $P\in D_1$ and $x \in F_P$,
\begin{equation} \label{3.5x}
N_{R, t, l}(x) = \max \left\{ \sup_{P \subsetneqq Q_{\alpha}^{k, t} \subset R}
\left\| W^{\frac1p} (x) A_{Q_{\alpha}^{k,t}}^{-1} \right\|, N_{P, t, l}(x) \right\}
= \sup_{P \subsetneqq Q_{\alpha}^{k, t} \subset R}
\left\| W^{\frac1p} (x) A_{Q_{\alpha}^{k,t}}^{-1} \right\|.
\end{equation}
By the definition of $ D_1 $,
we find that, for any $ Q_{\alpha}^{k,t} \in \mathbb{D}^t$
satisfying $P \subsetneqq Q_{\alpha}^{k,t} \subset R$,
$$
\left\|A_{2A_0 B(R)} A_{Q_{\alpha}^{k, t}}^{-1} \right\| \leq M^{\frac1p},
$$
which, together with \eqref{3.5x},
further implies that, for any $x\in F_P$,
\begin{align}  \label{haha8}
N_{R,t,l}(x)
\leq \left\| W^{\frac1p}(x) A_{2A_0 B(R)}^{-1} \right\|
\sup_{\{Q_{\alpha}^{k, t} :\, P \subsetneqq Q_{\alpha}^{k, t} \subset R\}}
\left\|A_{2A_0 B(R)} A_{Q_{\alpha}^{k, t}}^{-1} \right\|
\leq \left\| W^{\frac1p}(x) A_{2A_0 B(R)}^{-1} \right\| M^{\frac1p}.
\end{align}
Since $ D_1 $ is pairwise disjoint, it then follows that
$ \{F_P\}_{P \in D_1} $ is pairwise disjoint.
From this, the definitions of $F_P$ and $D_1$,
\eqref{haha8}, and \eqref{2.14}, we deduce that
\begin{align*}
\sum_{P \in D_1} \int_{F_P} [N_{R,t,l}(x)]^{pr(W)} \, d\mu(x)
\lesssim \int_R \left\| W^{\frac1p}(x) A_{2A_0 B(R)}^{-1} \right\|^{pr(W)} \, d\mu(x)
\lesssim \mu(R),
\end{align*}
which, together with the fact that $D_1$ is pairwise disjoint,
\eqref{haha3}, and the definition of $F_P$, further implies that
there exists a positive constant $\widetilde{C}$,
independent of $l$ and $R$, such that
\begin{align*}
\int_R [N_{R,t,l}(x)]^{pr(W)} \, d\mu(x)
&= \int_{R\setminus (\bigcup_{P \in D_1} P)}\cdots
+ \sum_{P\in D_1} \int_{F_P} \cdots
\quad+ \sum_{P\in D_1} \int_{P\setminus F_P} \cdots \\
&\leq \widetilde{C} \mu(R) + \sum_{P \in D_1} \int_P [N_{P,t,l}(x)]^{pr(W)} \, d\mu(x).
\end{align*}
Each term in the last sum has the same form as $\int_R [N_{R,t,l}(x)]^{pr(W)} \, d\mu(x)$.
Using stopping time argument as in the proof of \cite[Lemma A.32]{bf}
(see also \cite[Lemma 3.6]{IM19}),
we obtain \eqref{1010}.
This completes the proof of Lemma \ref{lem1}.
\end{proof}

\begin{definition}
Let $t\in\{1,\ldots,K\}$. The \emph{Hardy--Littlewood maximal operator}
$\mathcal M^{\mathbb D^t}$ is defined by setting, for any
$ f \in L^1_{\mathrm{loc}}(\mathcal X)$ and $x\in\mathcal X$,
$$\mathcal M^{\mathbb D^t}(f)(x):=\sup_{\genfrac{}{}{0pt}{}{Q\in\mathbb D^t}{Q \ni x}}
\fint_Q  |f(y)|\,d\mu(y).$$
\end{definition}

In Euclidean spaces, the level sets of the Hardy--Littlewood maximal operator are open,
and hence each of them can be decomposed into
a union of maximal dyadic cubes via the Whitney decomposition.
In spaces of homogeneous type, the level sets might be not open;
nevertheless, the following decomposition property remains valid.

\begin{lemma} \label{Whitney}
Let $t\in\{1,\ldots,K\}$, $\lambda\in(0,\infty)$, and $g\in L^1_{\mathrm{loc}}(\mathcal X)$.
If the set
$$
\Omega:=\left\{x\in\mathcal X:\ \mathcal M^{\mathbb D^t}(g) (x) > \lambda \right\}
$$
has finite measure, then it can be decomposed into a sequence of maximal dyadic cubes
$\{Q_i\}_{i\in\mathcal{I}}$ in $\mathbb D^t$, that is,
\begin{enumerate}[\rm(i)]
\item $\Omega=\bigcup_{i\in\mathcal{I}} Q_i$,
the cubes $\{Q_i\}_{i\in\mathcal{I}}$ are pairwise disjoint,

\item for any dyadic cube $Q\in \mathbb D^t$ with $Q\subset \Omega$,
there exists $i\in\mathcal{I}$ such that $Q\subset Q_i$.
\end{enumerate}
\end{lemma}

\begin{proof}
For any $x\in \Omega$, there exists $Q_x\in \mathbb D^t$ such that
$Q_x\ni x$ and
$
\fint_{Q_x} |g(y)| \, d\mu(y) > \lambda.
$
By this and the definition of $\Omega$, we conclude that
$Q_x\subset\Omega$. Thus,
$\Omega= \bigcup\{Q\in\mathbb D^t:\ Q\subset\Omega \}$.
Since $\mu(\Omega)<\infty$,
we can let $\{Q_i\}_{i\in\mathcal{I}}$ be the collection of
all maximal dyadic cubes in $\{Q\in\mathbb D^t:\ Q\subset\Omega \}$.
It is easy to show that $\{Q_i\}_{i\in\mathcal{I}}$ satisfies (i) and (ii).
This completes the proof of Lemma \ref{Whitney}.
\end{proof}

Now, we can prove Theorem \ref{bounded B3}.

\begin{proof}[Proof of Theorem \ref{bounded B3}]
Due to Lemma \ref{sumk}, to prove the present theorem, we need only to show that,
for any $t\in\{1,\ldots,K\}$ and $G\in L^q(\mathcal X,M_m(\mathbb C))$,
$$
\left\| \mathcal M_{W,p}^{\mathbb D^t,(v)} (G) \right\|_{L^q(\mathcal X)}
\lesssim \| G \|_{L^q(\mathcal X,M_m(\mathbb C))}.
$$

Since $\frac{uv}{u-v}$ is monotonically increasing with respect to $v$,
it follows that, for any $v\in(0, \frac{pu}{p+u}]$,
$\frac{uv}{u-v}\le p < p{r(W)}$.
Consequently, the interval $(\frac{uv}{u-v}, pr(W)]$ is well-defined.
Let $t\in\{1,\ldots,K\}$ and $\{A_Q\}_{Q\in\mathbb{D}^t}$ be
a sequence of reducing operators of order $p$ for $W$.
Note that $v\leq\frac{pu}{p+u}<u$. From this,  H\"older's inequality, and Lemma \ref{improve D},
it follows that, for any $x\in\mathcal X$,
\begin{align*}
\mathcal M_{W,p}^{\mathbb D^t,(v)} (G) (x)
&\leq \sup_{\genfrac{}{}{0pt}{}{Q\in\mathbb D^t}{Q \ni x}}
\left\| W^{\frac1p}(x) A_Q^{-1} \right\|
\left[\fint_{Q} \left\| A_Q W^{-\frac1p}(y) G(y) \right\|^v d\mu(y)\right]^{\frac1v}\\
&\leq \sup_{\genfrac{}{}{0pt}{}{Q\in\mathbb D^t}{Q \ni x}}
\left\| W^{\frac1p}(x) A_Q^{-1} \right\|
\left[\fint_{Q} \left\| A_Q W^{-\frac1p}(y)\right\|^{u} d\mu(y)\right]^{\frac 1{u}}
\left[\fint_{Q} \left\|G(y)\right\|^r d\mu(y)\right]^{\frac 1r} \\
&\lesssim \sup_{\genfrac{}{}{0pt}{}{Q\in\mathbb D^t}{Q \ni x}}
\left\| W^{\frac1p}(x) A_Q^{-1} \right\|
\left[\fint_{Q} \left\|G(y)\right\|^r d\mu(y)\right]^{\frac 1r},
\end{align*}
where $r:=\frac{uv}{u-v}$.
For every $\kappa\in\mathbb Z$, let
$$
\mathscr R_\kappa := \left\{Q\in\mathbb D^t:\
2^\kappa<\left[\fint_{Q} \left\|G(y)\right\|^r d\mu(y)\right]^{\frac 1r}
\leq  2^{\kappa+1} \right\}.
$$
Then
\begin{align} \label{2.11}
\mathcal M_{W,p}^{\mathbb D^t,(v)} (G)(x)
&\lesssim \sup_{\kappa\in\mathbb Z} \sup_{\genfrac{}{}{0pt}{}{Q\in \mathscr R_\kappa}{Q \ni x}}
\left\| W^{\frac1p}(x) A_Q^{-1} \right\|
\left[\fint_{Q} \left\|G(y)\right\|^r d\mu(y)\right]^{\frac 1r} \notag \\
&\sim \sup_{\kappa\in\mathbb Z} 2^\kappa \sup_{\genfrac{}{}{0pt}{}{Q\in \mathscr R_\kappa}{Q \ni x}}
\left\| W^{\frac1p}(x) A_Q^{-1} \right\|
\leq \sup_{\kappa\in\mathbb Z} 2^\kappa
\sup_{\genfrac{}{}{0pt}{}{Q\in \mathbb D^t}{x\in Q\subset \Omega_\kappa}}
\left\| W^{\frac1p}(x) A_Q^{-1} \right\|,
\end{align}
where
$$
\Omega_\kappa:=\left\{ y\in\mathcal X:\
\left[ \mathcal M^{\mathbb D^t}( \|G\|^r ) (y) \right]^{\frac 1r} >2^\kappa \right\}.
$$
Note that $\frac{q}{r}\in(1,\infty)$. Applying the boundedness of
$\mathcal{M}^{\mathbb D^t}$ on $L^{\frac{q}{r}}(\mathcal X)$
(see, for instance, \cite{CG} or \cite[(3.6)]{CG77}), we conclude that
\begin{align*}
\mu(\Omega_\kappa)
\leq \int_{\Omega_\kappa} 2^{-\kappa q}
\left[ \mathcal M^{\mathbb D^t}( \|G\|^r )(x) \right]^{\frac q r}\,d\mu(x)
\lesssim 2^{-\kappa q} \int_{\mathcal X} \|G(x)\|^q\,d\mu(x)  <\infty,
\end{align*}
which, together with Lemma \ref{Whitney}, further implies that
$\Omega_\kappa$ can be decomposed into a sequence of maximal dyadic cubes
$\{Q_i\}_{i\in\mathcal{I}}$ in $\mathbb D^t$.
By this, \eqref{2.11}, the assumption
$q\leq pr(W)$, and Lemma \ref{lem1}, we obtain
\begin{align}\label{MDt}
\left\| \mathcal M_{W,p}^{\mathbb D^t,(v)} (G) \right\|_{L^q(\mathcal X)}^q
&\lesssim \sum_{\kappa\in\mathbb Z} 2^{\kappa q}
\int_{\mathcal X} \sup_{\genfrac{}{}{0pt}{}{Q\in \mathbb{D}^t}{x\in Q\subset \Omega_\kappa}}
\left\| W^{\frac1p}(x) A_Q^{-1} \right\|^q d\mu(x)\notag\\
&= \sum_{\kappa\in\mathbb Z} 2^{\kappa q}
\sum_{i\in\mathcal{I}}
\mu(Q_i) \fint_{Q_i} \sup_{\genfrac{}{}{0pt}{}{Q\in \mathbb{D}^t}{x\in Q\subset Q_i}}
\left\| W^{\frac1p}(x) A_Q^{-1} \right\|^q d\mu(x)\notag\\
&\lesssim \sum_{\kappa\in\mathbb Z} 2^{\kappa q}
\sum_{i\in\mathcal{I}} \mu(Q_i)
= \sum_{\kappa\in\mathbb Z} 2^{\kappa q} \mu(\Omega_\kappa)
\sim \left\| \mathcal M^{\mathbb D^t} (\|G\|^r) \right\|_{L^{\frac{q}{r}}(\mathcal X)}^{\frac{q}{r}},
\end{align}
where, in the last step, we used Cavalieri's principle
(see, for instance, \cite[Proposition 1.1.4]{hsdyk01}).
From $\frac qr\in(1,\infty)$ and the boundedness of $\mathcal{M}^{\mathbb D^t}$ on $L^{\frac qr}(\mathcal X)$, we infer that
\begin{align*}
\left\| \mathcal M^{\mathbb D^t} (\|G\|^r) \right\|_{L^{\frac qr}(\mathcal X)}^{\frac qr}
\lesssim \left\| \|G\|^r \right\|_{L^{\frac qr}(\mathcal X)}^{{\frac{q}{r}}}
= \| G \|_{L^q(\mathcal X,M_m(\mathbb C))}^q,
\end{align*}
which, combined with \eqref{MDt}, completes the proof of Theorem \ref{bounded B3}.
\end{proof}

Next, we show Theorem \ref{bounded B2}.

\begin{proof}[Proof of Theorem \ref{bounded B2}]
We first prove (i) $\Longleftrightarrow$ (ii).
Note that
\begin{align*}
\mathcal M_{W,p}^{(v)} \left( \vec f \right)
=\mathcal M_{W,p}^{(v)} \left( \left[\vec f, \vec{\mathbf 0}, \ldots,  \vec{\mathbf 0} \right] \right),
\end{align*}
and hence (i) $\Longrightarrow$ (ii).
On the other hand, let $ \{\vec{e}_i\}_{i=1}^m$ be any orthonormal basis of $\mathbb C^m$.
Using the equivalence $\|A\|\sim\sum_{i=1}^m|Ae_i|$ for all $A\in M_m(\mathbb C)$
(see, for instance, \cite[Lemma 3.2]{ro03}),
we find that, for any $x\in\mathcal X$,
\begin{align*}
\mathcal M_{W,p}^{(v)} (G)(x)
&\sim\sup_{\genfrac{}{}{0pt}{}{\mathrm{ball}\,B\subset\mathcal X}{B \ni x}}
\left[\fint_{B} \sum_{i=1}^m
\left| W^{\frac1p}(x) W^{-\frac1p}(y) G(y) \vec e_i \right|^v\, d\mu(y)\right]^{\frac1v} \\
&\sim \sum_{i=1}^m \mathcal M_{W,p}^{(v)} (G \vec e_i)(x),
\end{align*}
and hence (ii) $\Longrightarrow$ (i).
This completes the proof of (i) $\Longleftrightarrow$ (ii).

Repeating the above proof with
$\mathcal M_{W,p}^{(v)}$ replaced by $\mathcal M_{W,p}^{(v,r)}$,
we obtain (iii) $\Longleftrightarrow$ (iv).
Therefore, it remains to prove
(i) $\Longleftrightarrow$ (iii) $\Longleftrightarrow$ (v).

Note that, for any $r\in(0,\infty)$,
$$
\mathcal M_{W,p}^{(v,r)} (G) (x)
\leq \mathcal M_{W,p}^{(v)} (G) (x),
$$
and hence (i) $\Longrightarrow$ (iii).

Next, we show (iii) $\Longrightarrow$ (v).
Let $u:=\frac{pv}{p-v}$ and $q:=\frac{p}{v}\in(1,\infty)$.
Then, from a dual argument and Lemmas \ref{reduceM} and \ref{intexchange}, we infer that,
for any $r\in(0,\infty)$ and any ball $B\subset \mathcal X$ with radius $r$,
\begin{align*}
\mathrm{I}
&:=\left[\fint_B\left\|A_BW^{-\frac1p}(y)\right\|^u\,d\mu(y) \right]^{\frac1{q'}}
=\left[\fint_B\left\|A_BW^{-\frac1p}(y)\right\|^{vq'}\,d\mu(y) \right]^{\frac1{q'}}\\
&\phantom{:}=\sup_{\fint_B|h(y)|^q\,d\mu(y)=1}
\fint_B\left\|A_BW^{-\frac1p}(y)\right\|^{v}|h(y)|\,d\mu(y) \\
&\phantom{:}\sim\sup_{\fint_B|h(y)|^q\,d\mu(y)=1}
\fint_B\left[\fint_B \left\|W^{\frac1p}(x)W^{-\frac1p}(y)|h(y)|^{\frac1v}
\right\|^{p}\,d\mu(x)\right]^{\frac vp}\,d\mu(y).
\end{align*}
Applying this, Lemma \ref{intexchange}, and (iii), we obtain
\begin{align*}
\mathrm{I}
&\sim\sup_{\fint_B|h(y)|^q\,d\mu(y)=1}
\left\{\fint_B\left[\fint_B\left\|W^{\frac1p}(x)W^{-\frac1p}(y) |h(y)|^{\frac1v}
\right\|^{v} \,d\mu(y)\right]^{\frac pv}\,d\mu(x) \right\}^{\frac vp}\\
&\leq \sup_{\fint_B|h(y)|^q\,d\mu(y)=1}
\left\{\fint_B\left[ \mathcal M_{W,p}^{(v,r)}\left( \mathbf1_B|h|^{\frac1v}I_m\right) \right]^{p}\,d\mu(x) \right\}^{\frac vp}\\
&\lesssim \sup_{\fint_B|h(y)|^q\,d\mu(y)=1}
\left[\fint_B|h(y)|^{\frac pv}\,d\mu(y) \right]^{\frac vp}=1.
\end{align*}
Thus, $\sup_{\mathrm{ball}\, B\subset \mathcal{X}}
\fint_B\|A_BW^{-\frac1p}(y)\|^u\,d\mu(y)<\infty$.
From this and Lemma \ref{Apu}, we infer that
$W\in \mathcal A_{p,u}(\mathcal X,\mathbb C^m)$,
which completes the proof of (iii) $\Longrightarrow$ (v).

Finally, we prove (v) $\Longrightarrow$ (i).
From (v) [$W\in \mathcal A_{p,u}(\mathcal X,\mathbb C^m)$]
and the self-improving property of $\mathcal A_{p,u}(\mathcal X,\mathbb C^m)$
[see Proposition \ref{improve}(ii)], we deduce that
there exists $r\in(1,\infty)$ such that
$W\in \mathcal A_{p,ur}(\mathcal X,\mathbb C^m)$.
By the assumption $v<p$, we conclude that
$$
\frac{pur}{p+ur}
=\frac{pr}{p+v(r-1)}v
>\frac{pr}{p+p(r-1)}v
=v
$$
and
$$
\frac{urv}{ur-v}
=\frac{vr}{p(r-1)+v} p
<\frac{vr}{v(r-1)+v} p
=p.
$$
Using these and Theorem \ref{bounded B3} with $q$ replaced by $p$, we obtain (i).
This completes the proof of (v) $\Longrightarrow$ (i)
and hence Theorem \ref{bounded B2}.
\end{proof}

Now, we prove Theorem \ref{bounded B1}.

\begin{proof}[Proof of Theorem \ref{bounded B1}]
If $\| \mathcal M_{W,p}^{(v)} \|
_{L^p(\mathcal X,M_m(\mathbb C))\to L^p(\mathcal X)}<\infty$
for some $v\in(0,p)$,
then Theorem \ref{bounded B2} shows that
$W\in\mathcal A_{p,\frac{pv}{p-v}}(\mathcal{X},\mathbb C^m)$.
This, together with Proposition \ref{improve}(i),
further implies that $W\in \mathcal A_{p,\infty}(\mathcal{X},\mathbb C^m)$.

Conversely, if $W\in \mathcal A_{p,\infty}(\mathcal{X},\mathbb C^m)$,
then Proposition \ref{improve}(i) shows that
$W\in\mathcal A_{p,u}(\mathcal{X},\mathbb C^m)$
for some $u\in(0,\infty)$.
Let $v:=\frac{pu}{p+u}$ (so that $u=\frac{pv}{p-v}$).
Applying Theorem \ref{bounded B2}, we obtain
$$
\| \mathcal M_{W,p}^{(v)} \|
_{L^p(\mathcal X,M_m(\mathbb C))\to L^p(\mathcal X)}<\infty,
$$
which completes the proof of Theorem \ref{bounded B1}.
\end{proof}

To show that $v\in (0,p)$ is sharp, we also need the following lemma.

\begin{lemma} \label{small}
Let $t\in\{1,\ldots,K\}$, $l\in\mathbb Z$,
$\beta\in\mathcal A_l$, and $Q:=Q_{\beta}^{l,t}$.
For any $\varepsilon\in(0,\infty)$, there exists
$k\in\mathbb Z\cap[l,\infty)$ such that,
for any $\tau\in\mathcal A_k$ with $Q_{\tau}^{k,t}\subset Q$,
$\mu(Q_{\tau}^{k,t})<\varepsilon$.
\end{lemma}

\begin{proof}
Arguing by contradiction, suppose that the conclusion of the present lemma does not hold.
Then there exists $\varepsilon\in(0,\infty)$ and
a sequence $\{Q_k\}_{k=l}^\infty\in\mathbb D^t$ such that,
for any $k\in\{l,l+1,\ldots\}$,
$Q_k=Q_{\tau}^{k,t}$ for some $\tau\in \mathcal A_k$,
$Q_{k+1} \subset Q_k\subset Q$, and $\mu(Q_k)\geq \varepsilon$.
Then
\begin{equation}\label{zero}
\mu\left(\bigcap_{k=l}^\infty Q_k\right)
=\lim_{k\to\infty} \mu(Q_k)
\geq\varepsilon.
\end{equation}
Choose $x\in\bigcap_{k=l}^\infty Q_k$.
From Lemma \ref{dcs}(a)(iii), it follows that,
for any $k\in\{l,l+1,\ldots\}$,
$x\in Q_k\subset B_d(z_{\tau}^{k,t},4A_0^{2}\delta^{k})$.
For any $y\in B_d(z_{\tau}^{k,t},4A_0^{2}\delta^{k})$,
$$
d(y,x)
\leq A_0[d(y,z_{\tau}^{k,t})+d(z_{\tau}^{k,t},x)]
< 8A_0^{3}\delta^{k},
$$
and hence $B_d(z_{\tau}^{k,t},4A_0^{2}\delta^{k})
\subset B_d(x,8A_0^{3}\delta^{k})$.
Therefore,
$$
\mu\left(\bigcap_{k=l}^\infty Q_k\right)
\leq \mu\left(\bigcap_{k=l}^\infty B_d(x,8A_0^{3}\delta^{k})\right)
= \mu(\{x\}) =0,
$$
which yields a contradiction to \eqref{zero}.
This completes the proof of Lemma \ref{small}.
\end{proof}

\begin{proposition} \label{bounded fail}
If $0<p\leq v<\infty$ and $W$ is a matrix weight, then
$\| \mathcal M_{W,p}^{(v)} \|
_{L^p(\mathcal X,M_m(\mathbb C))\to L^p(\mathcal X)}=\infty$.
\end{proposition}

\begin{proof}
Arguing by contradiction, suppose that
$
\| \mathcal M_{W,p}^{(v)} \|
_{L^p(\mathcal X,M_m(\mathbb C))\to L^p(\mathcal X)}<\infty
$
for some $0<p\leq v<\infty$.
This, together with Lemma \ref{sumk} and H\"older's inequality, further implies that,
for any $t\in\{1,\ldots,K\}$ and $G\in L^p(\mathcal X,M_m(\mathbb C))$,
\begin{equation}\label{bounded B mat2y}
\left\| \mathcal M_{W,p}^{\mathbb D^t,(p)} (G) \right\|_{L^p(\mathcal X)}
\lesssim \left\| \mathcal M_{W,p}^{(p)} (G) \right\|_{L^p(\mathcal X)}
\leq\left\| \mathcal M_{W,p}^{(v)} (G) \right\|_{L^p(\mathcal X)}
\lesssim \| G \|_{L^p(\mathcal X,M_m(\mathbb C))}.
\end{equation}

Let $t\in\{1,\ldots,K\}$, $\beta\in\mathcal A_0$, and $Q:=Q_{\beta}^{0,t}$.
Due to Lemma \ref{small}, there exists a sequence $\{k_l\}_{l\in\mathbb N}$ such that,
for any $l\in\mathbb N$, $k_{l+1}> k_l\geq 0$ and
\begin{align} \label{onetwo}
\max_{{\tau\in\mathcal A_{k_{l+1}}},{Q_{\tau}^{k_{l+1},t}\subset Q}}
\mu\left( Q_{\tau}^{k_{l+1},t} \right)
\leq \frac12 \min_{{\tau\in\mathcal A_{k_l}},{Q_{\tau}^{k_l,t}\subset Q}}
\mu\left( Q_{\tau}^{k_l,t} \right).
\end{align}
Let $L\in\mathbb N$ and $w:=\|W^{\frac1p}\|^p$.
From \eqref{bounded B mat2y}, it follows that,
for any $\tau\in\mathcal{A}_{k_L}$,
\begin{align*}
\int_{\mathcal{X}} \mathcal M^{\mathbb D^t}
\left( \mathbf 1_{Q_{\tau}^{k_L,t}} \right) (x) w(x) \, d\mu(x)
&=\int_{\mathcal{X}} \sup_{\genfrac{}{}{0pt}{}{Q\in\mathbb D^t}{Q \ni x}}
\fint_{Q} \left\| W^{\frac1p}(x) \mathbf 1_{Q_{\tau}^{k_L,t}}(y) \right\|^p d\mu(y)  \,d\mu(x)\\
&=\left\| \mathcal M_{W,p}^{\mathbb D^t,(p)} \left( W^{\frac1p} \mathbf 1_{Q_{\tau}^{k_L,t}} \right) \right\|_{L^p(\mathcal X)}^p
\lesssim \int_{Q_{\tau}^{k_L,t}} w(x) \,d\mu(x).
\end{align*}
Summing over all $Q_{\tau}^{k_L,t}\subset Q$ on both sides of the above inequality, we obtain
\begin{align} \label{3.17}
\int_{\mathcal{X}} \sum_{\tau\in\mathcal{A}_{k_L}, Q_{\tau}^{k_L,t} \subset Q}
\mathcal M^{\mathbb D^t} \left( \mathbf 1_{Q_{\tau}^{k_L,t}} \right) (x) w(x) \, d\mu(x)
\lesssim \int_{Q} w(x) \,d\mu(x).
\end{align}
For any $x\in Q$, there exists a sequence $\{Q_l\}_{l=1}^L$ such that,
for any $l\in\{1,\ldots,L\}$, $x\in Q_l$ and
$Q_l = Q_{\tau}^{k_l,t}$ for some $\tau\in \mathcal{A}_{k_l}$.
Using \eqref{onetwo}, we obtain, for any $l \in \{0, \ldots, L-1\}$,
$\mu(Q_l) \ge 2 \mu(Q_{l+1})$, where $Q_0:=Q$.
This further implies that
\begin{align*}
\sum_{\tau\in\mathcal{A}_{k_L}, Q_{\tau}^{k_L,t} \subset Q_0}
\mathcal M^{\mathbb D^t} \left( \mathbf 1_{Q_{\tau}^{k_L,t}} \right) (x)
&> \sum_{l=0}^{L-1} \sum_{\tau\in\mathcal{A}_{k_L}, Q_{\tau}^{k_L,t} \subset Q_l\setminus Q_{l+1}}
\mathcal M^{\mathbb D^t} \left( \mathbf 1_{Q_{\tau}^{k_L,t}} \right) (x) \\
&\geq \sum_{l=0}^{L-1} \sum_{\tau\in\mathcal{A}_{k_L}, Q_{\tau}^{k_L,t} \subset Q_l\setminus Q_{l+1}}
\frac{\mu(Q_{\tau}^{k_L,t})}{\mu(Q_l)}   \\
&= \sum_{l=0}^{L-1} \frac{\mu(Q_l\setminus Q_{l+1})}{\mu(Q_l)}
\geq \frac L2.
\end{align*}
Substituting this back into \eqref{3.17}, we obtain
\begin{align*}
\frac L2 \int_Q w(x) \, d\mu(x)
\lesssim \int_Q w(x) \,d\mu(x).
\end{align*}
Since $L\in\mathbb N$ is arbitrary, the above inequality fails,
yielding a contradiction. This completes the proof of Proposition \ref{bounded fail}.
\end{proof}

The following result is a direct consequence of Proposition \ref{bounded fail}; we omit the details.

\begin{corollary} \label{3.20x}
If $p\in(0,1]$ and $w$ is a scalar weight,
then $\| \mathcal M \|_{L^p(w)\to L^p(w)}=\infty$.
\end{corollary}

\begin{remark}
In the case $\mathcal X=\mathbb R^n$ and $w\equiv1$,
Corollary \ref{3.20x} coincides with the well-known result that,
for any $p\in(0,1]$, $\| \mathcal M \|_{L^p(\mathbb R^n)\to L^p(\mathbb R^n)}=\infty.
$
\end{remark}

When restricted to the class $\mathcal A_p(\mathcal X,\mathbb C^m)$,
Theorem \ref{bounded B3} yields the following noteworthy result.

\begin{proposition}\label{Apbd}
Let $p\in(0,\infty)$, $W\in\mathcal A_p(\mathcal X,\mathbb C^m)$,
and $r(W)$ be as in \eqref{rW}.
\begin{enumerate}[\rm(i)]
\item If $p\in(0,1]$, then, for any $v\in(0,p)$ and $q\in(v,pr(W)]$,
$\| \mathcal M_{W,p}^{(v)} \|_{L^q(\mathcal X,M_m (\mathbb C))\to L^q(\mathcal X)} <\infty$.

\item If $p\in(1,\infty)$, then there exists $u\in(p',\infty)$ such that,
for any  $v\in(0,\frac{pu}{p+u})$ and
$q\in(\frac{uv}{u-v},pr(W)]$,
$
\| \mathcal M_{W,p}^{(v)} \|_{L^q(\mathcal X,M_m (\mathbb C))\to L^q(\mathcal X)} <\infty.
$
\end{enumerate}
\end{proposition}

\begin{proof}
By Proposition \ref{improve}(iii), we conclude that
$W\in \mathcal A_{p,u}(\mathcal X,\mathbb C^m)$ for all $u\in(0,\infty)$.
Note that $\lim_{u\to\infty} \frac{pu}{p+u}=p$
and $\lim_{u\to\infty} \frac{uv}{u-v}=v$.
Then there exists $u\in(0,\infty)$ such that
$$
v\in \left(0, \frac{pu}{p+u}\right)
\quad\text{and}\quad
q\in \left( \frac{uv}{u-v}, pr(W)\right].
$$
These, together with Theorem \ref{bounded B3},
finish the proof of (i).

From Proposition \ref{improve}(iv), it follows that
$W\in \mathcal A_{p,u}(\mathcal X,\mathbb C^m)$ for some $u\in(p',\infty)$.
This, together with Theorem \ref{bounded B3}, further implies (ii),
which completes the proof of Proposition \ref{Apbd}.
\end{proof}

\begin{remark}\label{Apbd rem}
In Proposition \ref{Apbd}, the range of $q$ contains $p$.
Indeed, this statement is obvious for $p\in(0,1]$.
For $p\in(1,\infty)$, this assertion follows from the estimate
$$
\frac{uv}{u-v}
= \frac{u^2}{u-v} - u
< \frac{u^2}{u-\frac{pu}{p+u}} - u
=p.
$$
In particular, in Proposition \ref{Apbd}(ii), the range of $v$ contains $1$,
which follows from the estimate
$$
\frac{pu}{p+u}
= p- \frac{p^2}{p+u}
> p- \frac{p^2}{p+p'}
= \frac{pp'}{p+p'}
= 1.
$$
In the case where $\mathcal X=\mathbb R^n$ and $v=1$,
Proposition \ref{Apbd}(ii) coincides with \cite[Theorem 3.2]{g03}.
\end{remark}

\subsection{Critical Rescaling Index of Matrix Weights}\label{critical}

In this subsection, we introduce a new concept
to characterize the self-improving property of matrix weights.
For any scalar weight $w\in A_\infty(\mathcal X)$,
the \emph{critical self-improvement index} of $w$
(see, for instance, \cite[p.\,11]{st89}) is defined by
$$
q_w:= \inf \{q\in[1,\infty):\ w\in A_q(\mathcal X)\}.
$$
However, matrix weights do not possess such kind of self-improving property
(see \cite[Remark 5.4]{b01} or \cite[Proposition 2.25(vii)]{byyz25}), that is,
for $m\geq2$ and $p\in(1,\infty)$,
there exists $W\in \mathcal A_p(\mathbb R^n,\mathbb C^m)$
such that $W\notin \bigcup_{q\in(0,p)} \mathcal A_q(\mathbb R^n,\mathbb C^m)$.
Consequently, we need to characterize the self-improving property of
matrix weights from a different perspective.

Let $p\in(0,\infty)$. Theorem \ref{bounded B1} shows that
$W\in \mathcal A_{p,\infty}(\mathcal X,\mathbb C^m)$
if and only if there exists $v\in(0,p)$ such that
$\| \mathcal M_{W,p}^{(v)}\|_{L^p(\mathcal X,M_m(\mathbb C)) \to L^p(\mathcal X)} <\infty.$
Motivated by this, for any
$W\in \mathcal A_{p,\infty}(\mathcal X,\mathbb C^m)$, we call
\begin{align}\label{spW}
v_{W,p}:= \sup \left\{v\in(0,p):\ \left\| \mathcal M_{W,p}^{(v)}\right\|_{L^p(\mathcal X,M_m(\mathbb C)) \to L^p(\mathcal X)} <\infty\right\}
\end{align}
the \emph{critical rescaling index} of $W$.
It follows from H\"older's inequality
that the set defining the supremum in \eqref{spW} is an interval.

\begin{proposition} \label{prop of sp(W)}
The following statements hold.
\begin{enumerate}[\rm(i)]
\item If $p\in(0,\infty)$ and $W\in \mathcal A_{p,\infty}(\mathcal X,\mathbb C^m)$, then $v_{W,p}\in (0,p]$.

\item If $p\in(0,1]$ and $W\in \mathcal A_p(\mathcal X,\mathbb C^m)$, then $v_{W,p}=p$.

\item If $p\in(1,\infty)$ and $W\in \mathcal A_p(\mathcal X,\mathbb C^m)$, then $v_{W,p}\in(1,p]$.
\end{enumerate}
\end{proposition}

\begin{proof}
Statement (i) follows from Theorem \ref{bounded B1}.
Statements (ii) and (iii) follow from
Proposition \ref{Apbd} and Remark \ref{Apbd rem}.
This completes the proof of Proposition \ref{prop of sp(W)}.
\end{proof}


To further understand $v_{W,p}$,
we investigate its properties in the scalar case ($m=1$).
In this case, the \emph{critical rescaling index} of $w\in A_\infty(\mathcal X)$
admits a simpler equivalent definition, namely
\begin{align} \label{sc scaling}
v_{w,p}
&:= \sup \left\{v\in(0,\infty):\
\left\| \left[\mathcal M (|f|^v)\right]^{\frac1v} \right\|_{L^p(w)}
\lesssim \| f \|_{L^p(w)}
\text{ for all } f\in L^p(w) \right\} \notag \\
&\phantom{:}= \sup \left\{v\in(0,\infty):\
\| \mathcal M (g) \|_{L^{\frac pv}(w)}
\lesssim \| g \|_{L^{\frac pv}(w)}
\text{ for all } g\in L^{\frac pv}(w) \right\}.
\end{align}

\begin{proposition} \label{prop of sp(w)}
The following statements hold.
\begin{enumerate}[\rm(i)]
\item If $p\in(0,\infty)$ and $w\in A_{\infty}(\mathcal X)$, then $v_{w,p}\in (0,p]$.

\item If $p\in(0,1]$ and $w\in A_1(\mathcal X)$, then $v_{w,p}=p$.

\item If $p\in(1,\infty)$ and $w\in A_p(\mathcal X)$, then $v_{w,p}\in(1,p]$.

\item If $p\in(0,\infty)$ and $w\in A_{\infty}(\mathcal X)$, then $v_{w,p} q_w=p$.
\end{enumerate}
\end{proposition}

\begin{proof}
Statements (i) through (iii) follow directly from Proposition \ref{prop of sp(W)} with $m=1$.

Now, we prove (iv).
For any $\varepsilon\in(0,\infty)$, we have
$w\in A_{q_\varepsilon}(\mathcal X)$, where $q_\varepsilon:=q_w+\varepsilon$.
Recall that, for any $q\in(1,\infty)$,
\begin{align} \label{iff}
w\in A_q(\mathcal X)
\Longleftrightarrow
\| \mathcal M (g) \|_{L^q(w)} \lesssim \| g \|_{L^q(w)}
\text{ for all } g\in L^q(w);
\end{align}
see, for instance, \cite[Theorem 9 of Page 5]{st89}. Therefore,
$$
\| \mathcal M (g) \|_{L^{q_\varepsilon}(w)}
\lesssim \| g \|_{L^{q_\varepsilon}(w)}
\text{ for all } g\in L^{q_\varepsilon}(w).
$$
This, together with \eqref{sc scaling}, further implies that
$\frac{p}{q_\varepsilon}\leq v_{w,p}$.
Letting $\varepsilon\to 0^+$, we obtain $p\leq v_{w,p}q_w$.

On the other hand, by \eqref{sc scaling}, we find that,
for any $\varepsilon\in(0,v_{w,p})$,
$$
\| \mathcal M (g) \|_{L^{\frac p{v_\varepsilon}}(w)}
\lesssim \| g \|_{L^{\frac p{v_\varepsilon}}(w)}
\text{ for all } g\in L^{\frac p{v_\varepsilon}}(w),
$$
where $v_\varepsilon:=v_{w,p}-\varepsilon\in(0,p)$.
From this and \eqref{iff}, we deduce that
$w\in A_{\frac p{v_\varepsilon}}(\mathcal X)$.
Thus, $q_w\leq\frac p{v_\varepsilon}$.
Letting $\varepsilon\to 0^+$, we obtain $v_{w,p}q_w\leq p$.
This completes the proof of (iv) and hence Proposition \ref{prop of sp(w)}.
\end{proof}

Proposition \ref{prop of sp(w)}(iv) shows that
the critical rescaling index $v_{W,p}$
is equivalent to the   critical self-improvement index $q_w$ when $m=1$.
From their definitions, we also easily deduce the  following
relationship between $v_{W,p}$ and $v_{w,p}$.

\begin{lemma}\label{scamtr1}
Let $p\in(0,\infty)$, let $w$ be a scalar weight, and set $W:=wI_m$. Then
\begin{enumerate}[\rm(i)]
\item $W\in \mathcal A_{p}(\mathcal X,\mathbb C^m)$
if and only if $w\in A_{p\vee1}(\mathcal X)$;

\item $v_{W,p}=v_{w,p}$.
\end{enumerate}
\end{lemma}

The following result extends the classical result
(see, for instance, \cite[Lemma 2.40]{bf3}) to the anisotropic setting.
Although the Euclidean metric is replaced by a quasi-metric $\rho$,
the proof requires only standard modifications; we omit the details.

\begin{lemma}\label{ballsim}
Let $a\in(-1,\infty)$.
Then,
\begin{enumerate}[\rm(i)]
\item for any $x_0\in\mathbb{R}^n$ and $r\in(0,\infty)$,
\begin{align*}
\fint_{B_{\rho}(x_0,r)}\left[\rho(x)\right]^a\,dx\sim\left[\rho(x_0)+r\right]^a;
\end{align*}

\item for any $j\in\mathbb Z$ and $Q\in\mathcal D_j$,
\begin{align*}
\fint_{Q}\left[\rho(x)\right]^a\,dx\sim\left[\rho(x_Q)+b^{-j}\right]^a.
\end{align*}
\end{enumerate}
Here all the positive equivalence constants depend only on $a$, $\rho$, and $A$.
\end{lemma}

The following conclusion is well known in the Euclidean setting
(see, for instance, \cite[Example 7.1.7]{hsdyk01})
and can be extended to the anisotropic setting
via a standard argument based on Lemma \ref{ballsim}; we omit the details.

\begin{lemma}\label{example}
Let $a\in(-1,\infty)$ and  $w(\cdot):=[\rho(\cdot)]^a$. Then
\begin{enumerate}[\rm(i)]
\item $w\in A_1(\mathbb{R}^n)$ if and only if $a\in(-1,0]$;

\item if $p\in(1,\infty)$, then $w\in A_p(\mathbb{R}^n)$ if and only if $a\in(-1,p-1)$.
\end{enumerate}
\end{lemma}

Now, we can compute the critical rescaling index of the matrix weight $W:= w I_m$.

\begin{lemma}\label{value}
Let $a\in(-1,\infty)$, $w(\cdot):=[\rho(\cdot)]^a$, and $W(\cdot):=w(\cdot)I_m$.
Then, for any $p\in(0,\infty)$,
$$
v_{W,p}=v_{w,p}=\frac{p}{1+a_+}.
$$
\end{lemma}

\begin{proof}
From Lemma \ref{example}, we infer that $q_w=1+a_+$.
This, together with Proposition \ref{prop of sp(w)}(iv)
and Lemma \ref{scamtr1}, then finishes the proof of Lemma \ref{value}.
\end{proof}

\begin{remark} \label{6.16}
The critical rescaling index $v_{W,p}$
can range over the entire interval $(0,p]$.
Indeed, for any $t\in(0,p]$, let $W:=[\rho(\cdot)]^{\frac pt-1} I_m$.
It then follows from Lemma \ref{value} that $v_{W,p}=t$.
\end{remark}

\subsection{A Key Lemma}\label{keylemma}


Let $W$ be a matrix weight, $Q,R$ two cubes in $\mathbb R^n$, and
$A_Q,A_R$ the reducing operators of order $p$ for $W$.
The estimate of $\|A_QA_R^{-1}\|$ was first established by Frazier and Roudenko in \cite{fr21},
and was subsequently refined by Bu et al.
(see \cite[Lemma 2.28]{bf3} and \cite[Proposition 6.6]{bf4}).
This estimate plays a key role in obtaining the $\varphi$-transform characterization of $\dot B_{p,q}^{\alpha}(W)$
and establishing boundedness of almost diagonal operators on $\dot b_{p,q}^{\alpha}(W)$.
In this subsection, we establish the sharp estimate of $\|A_QA_R^{-1}\|$ in the anisotropic setting.

Let us begin with some properties of $\mathcal A_{p,\infty}$-matrix weights.
Repeating the argument of \cite[Lemma 3.5 and Propositions 6.1]{bf4} with
cube $Q$ replaced by ball $B$,
we have the following two conclusions; we omit the details.

\begin{lemma}\label{logabm}
Let $p\in(0,\infty)$, $W\in \mathcal A_{p,\infty}$, and
$M\in M_m(\mathbb{C})$ be nonzero. Then, for any ball $B\in\mathcal{B}$,
$
\log(\|W^{-\frac{1}{p}}(\cdot)M\|^p)\in L^1(B)
$
and, for any ball $B\in\mathcal{B}$,
\begin{align*}
\left\|A_B^{-1}M\right\|^p\sim\exp\left(\fint_B\log
\left(\left\|W^{-\frac1p}(x)M\right\|^p\right)\,dx\right),
\end{align*}
where the positive equivalence constants depend
only on $m$, $p$, and $[W]_{\mathcal A_{p,\infty}}$.
\end{lemma}

\begin{lemma}\label{2sim}
Let $p\in(0,\infty)$, $W\in \mathcal A_{p,\infty}$,
and $\{A_{B}\}_{B\in\mathcal{B}}$ be a family
of reducing operators of order $p$ for $W$.
Then, for all balls $B, \widetilde{B} \in\mathcal{B}$,
\begin{align*}
\left\|A_{B}A_{\widetilde{B}}^{-1}\right\|^{p}
&\sim\fint_{B}\exp\left(\fint_{\widetilde{B}}\log\left(\left\|W^{{\frac{1}{p}}}(x)
W^{{-\frac{1}{p}}}(y)\right\|^{p}\right)\,dy\right)\,dx\\
&\sim\exp\left(\fint_{\widetilde{B}}\log\left(\fint_{B}\left\|
W^{{\frac{1}{p}}}(x)W^{{-\frac{1}{p}}}(y)\right\|^{p}\,dx\right)\,dy\right),
\end{align*}
where the positive equivalence constants depend only on $m$, $p$,
and $[W]_{{\mathcal A_{p,\infty}}}$.
\end{lemma}

We recall the following concept of upper and lower dimensions
of $\mathcal A_{p,\infty}$-matrix weights (see, for instance, \cite[Definition 6.2]{bf4}),
which provide a finer classification of $\mathcal A_{p,\infty}$-matrix weights
and play a key role in establishing the sharp
estimate of $\|A_QA_R^{-1}\|$ (see Lemma \ref{fuhe} below).

\begin{definition} \label{dim}
Let  $p\in(0,\infty)$ and $d \in \mathbb{R}$.
A matrix weight $ W $ is said to have \emph{$\mathcal A_{p, \infty}$-lower dimension $d$},
denoted by $W\in\mathbb{D}_{p,\infty,d}^{\mathrm{lower}}
(\mathbb{R}^n,\mathbb{C}^m,{A})$,
if there exists a positive constant $C$ such that,
for any  $\lambda\in[1,\infty)$ and any $B\in \mathcal B$,
$$\exp \left(\fint_{\lambda B} \log \left(\fint_{B}\left\|W^{\frac{1}{p}}(x)
W^{-\frac{1}{p}}(y)\right\|^{p} \,dx\right) \,dy\right)
\leq C \lambda^d.$$
A matrix weight $W$ is said to have \emph{$\mathcal A_{p,\infty}$-upper dimension $d$},
denoted by $W\in\mathbb{D}_{p,\infty,d}^\mathrm{upper}
{(\mathbb{R}^n,\mathbb{C}^m,A)}$, if there exists a positive constant $C$
such that, for any  $\lambda \in[1, \infty) $ and any $B\in{\mathcal B}$,
$$\exp \left(\fint_{B} \log \left(\fint_{\lambda B}\left\|W^{\frac{1}{p}}(x)
W^{-\frac{1}{p}}(y)\right\|^{p} \,dx\right) \,dy\right) \leq C \lambda^{d}.$$
\end{definition}

In what follows, we denote  $\mathbb{D}_{p,\infty,d}^{\mathrm{lower}}(\mathbb{R}^n,\mathbb{C}^m,{A})$
[resp. $\mathbb{D}_{p,\infty,d}^\mathrm{upper}{(\mathbb{R}^n,\mathbb{C}^m,A)}$]
simply by $\mathbb{D}_{p,\infty,d}^{\mathrm{lower}}$
(resp. $\mathbb{D}_{p,\infty,d}^\mathrm{upper}$).
Repeating the proofs of \cite[Propositions 6.4 and 6.5]{bf4}
with Lemmas 2.3, 2.9, and 3.5, Propositions 3.8 and 5.6,
and Proposition 6.1 therein replaced, respectively,  by
Lemmas \ref{exchange}, \ref{reduceM}, and \ref{logabm},
Propositions \ref{chawap} and \ref{matrixrhi}, and Lemma \ref{2sim},
we obtain the following conclusions; we omit the details.

\begin{proposition}
Let $p\in(0,\infty)$. Then the following statements hold.
\begin{enumerate}[\rm(i)]
\item  For any $d\in(-\infty,0)$,
one has $\mathbb{D}_{p, \infty, d}^{\mathrm {upper }}=\mathbb{D}_{p,\infty,d}^\mathrm{lower}=\emptyset$.

\item $\bigcup_{d \in[0, \infty)} \mathbb{D}_{p, \infty, d}^{\mathrm {upper }}
=\mathcal A_{p, \infty}$.

\item $\bigcup_{d\in[0,1)}\mathbb{D}_{p,\infty,d}^\mathrm{lower}=\mathcal A_{p,\infty}$.

\item For any $d\in[1,\infty)$,
one has $\mathbb{D}_{p,\infty,d}^\mathrm{lower}=\mathcal A_{p,\infty}$.
\end{enumerate}
\end{proposition}

\begin{remark}\label{dimint}
Let $p\in ( 0, \infty )$ and $W$ be a matrix weight.
Then the following statements hold.
\begin{enumerate}[(i)]
\item The possible forms of the set \{$d\in\mathbb{R}:$ $W$ has
$\mathcal A_{p,\infty}$-lower dimension $d$\} are the empty set and
all the intervals of the form $(a,\infty)$ and $[a,\infty )$, where $a\in [0, 1).$
\item The possible forms of the set \{$d\in\mathbb{R}:W$
has $\mathcal A_{p,\infty}$-upper dimension $d$\} are the empty set and
all the intervals of the form $(b,\infty)$ and $[b,\infty)$, where $b\in[0,\infty).$
\item In both (i) and (ii), the empty set corresponds to $W\notin \mathcal A_{p,\infty}.$
\end{enumerate}
\end{remark}

\begin{definition}\label{lower}
Let $p\in(0,\infty)$.
For any matrix weight  $W\in \mathcal A_{p,\infty}$, let
\begin{align*}
d_{p,\infty}^{\mathrm{lower}}(W)& :=\inf\{d\in[0,\infty):
W\text{ has }\mathcal A_{p,\infty}\text{-lower dimension }d\},  \\
[\![d_{p,\infty}^{\mathrm{lower}}(W),\infty)& :=
\begin{cases}
[d_{p,\infty}^{\mathrm{lower}}(W),\infty)&\text{if }W\text{ has }
\mathcal A_{p,\infty}\text{-lower dimension }d_{p,\infty}^{\mathrm{lower}}(W),\\
(d_{p,\infty}^{\mathrm{lower}}(W),\infty)&\text{otherwise.}
\end{cases}
\end{align*}
Similarly, we can define $d_{p,\infty}^{\mathrm{upper}}(W)$ and
$[\![d_{p,\infty}^\mathrm{upper}(W),\infty)$
with lower dimension replaced by upper dimension.
\end{definition}

\begin{remark}
It follows from Remark \ref{dimint}(i) that  $d_{p,\infty}^{\mathrm{lower}}(W) \in [0,1)$.
\end{remark}

Applying the concepts of upper and lower dimensions of $\mathcal A_{p,\infty}$,
we obtain the following estimate for the composition of reducing operators.
This result recovers the corresponding sharp estimate in the Euclidean setting \cite[Proposition 6.6]{bf4}.

\begin{lemma}\label{fuhe}
Let $p\in (0,\infty)$, $W\in \mathcal A_{p,\infty}$, and
$\{A_{B}\}_{B\in\mathcal{B}}$ be a family of
reducing operators of order $p$ for $W$.
Let $d_1\in[\![d_{p,\infty}^\mathrm{lower}(W),\infty)$
and $d_2\in[\![{d}_{p,\infty}^\mathrm{upper}(W),\infty)$.
Then there exists a positive constant $C$ such that,
\begin{enumerate}[{\rm(i)}]
\item  for any balls $B,
\widetilde{B}\in\mathcal{B}$ with $B\cap\widetilde B\neq\emptyset$,
$$\left\|A_BA_{\widetilde{B}}^{-1}\right\|^p
\leq C \max\left\{\left(\frac{r_{\widetilde{B}}}{r_B}\right)^{d_1},
\left(\frac{r_{B}}{r_{\widetilde{B}}}\right)^{d_2}\right\};
$$

\item
for any balls $B,\widetilde{B}\in\mathcal{B}$,
$$
\left\|A_BA_{\widetilde{B}}^{-1}\right\|^p
\leq C \max\left\{\left(\frac{r_{\widetilde{B}}}{r_B}\right)^{d_1},
\left(\frac{r_{B}}{r_{\widetilde{B}}}\right)^{d_2}\right\}
\left[1+\frac{\rho(c_{B}-c_{\widetilde{B}})}
{r_{B} \vee r_{\widetilde{B}}}\right]^{d_1+d_2}.
$$
\end{enumerate}
\end{lemma}

\begin{proof}
We first prove (i) by considering  two cases for $r_B$ and $r_{\widetilde B}$.

\emph{Case (1)}
$r_B\le r_{\widetilde{B}}$. Let $\lambda=3H^2r_{\widetilde B}/r_{B}$ and
$y\in B\cap\widetilde B$. Then,
for any $x\in \widetilde B$,
\begin{align*}
\rho(x-c_{B})\le H^2\left[\rho(x-c_{\widetilde{B}})+\rho(c_{\widetilde{B}}-y)
+\rho(y-c_{B})\right]
\le 3H^2r_{\widetilde B}=\lambda r_B,
\end{align*}
and hence $\widetilde B\subset \lambda B$. Write
\begin{align}\label{necessity}
\left\|A_BA_{\widetilde B}^{-1} \right\|^p
\le\left\|A_BA_{\lambda B}^{-1}\right\|^p\left\|A_{\lambda B}A_{\widetilde B}^{-1}\right\|^p=:\rm{I_1I_2}.
\end{align}
From Lemma \ref{2sim} and  $d_1\in[\![d_{p,\infty}^\mathrm{lower}(W),\infty)$, we infer that
\begin{align*}
{\rm{I_1}}\sim\exp\left(\fint_{\lambda B}\log\left(\fint_B\left\| W^{\frac{1}{p}}(x) W^{-\frac{1}{p}}(y)\right\|^p \,dx\right)\,dy\right)\lesssim \lambda^{d_1}\sim \left(\frac{r_{\widetilde B}}{r_{B}}\right)^{d_1}.
\end{align*}
Using Lemmas \ref{exchange} and \ref{logabm},
Jensen's inequality, and Proposition \ref{improve}(i),
we conclude that
\begin{align*}
\rm{I_2}
&=\left\|A_{\widetilde B}^{-1}A_{\lambda B}\right\|^p
\sim\left[\exp\left(\fint_{\widetilde B}\log\left(\left\|W^{-\frac{1}{p}}(x)A_{\lambda B}\right\|\right)\,dx\right)\right]^p\\
&=\left[\exp\left(\fint_{\widetilde B}\log\left(\left\|W^{-\frac{1}{p}}(x)A_{\lambda B}\right\|^{u}\right)\,dx\right)\right]^\frac{p}{u}
\leq\left[\fint_{\widetilde B}\left\|W^{-\frac{1}{p}}(x)A_{\lambda B}\right\|^{u}\,dx\right]^{\frac{p}{u}}\\
&\lesssim\left[\fint_{\lambda B}\left\|W^{-\frac{1}{p}}(x)A_{\lambda B}\right\|^{u}dx\right]^{\frac{p}{u}}
\lesssim1,
\end{align*}
where $u\in(0,\infty)$ is as in Proposition \ref{improve}.
Substituting the above estimates of $\rm{I_1}$
and $\rm{I_2}$ into \eqref{necessity}, we obtain
\begin{align*}
\left\|A_BA_{\widetilde{B}}^{-1}\right\|^p
\lesssim\left(\frac{r_{\widetilde{B}}}{r_B}\right)^{d_1}.
\end{align*}

\emph{Case (2)} $r_{\widetilde B}<r_{B}$.
In this case, by symmetry, we have
$B\subset\lambda\widetilde{B}$, where  $\lambda=3H^2r_{B}/r_{\widetilde{B}}$. This, together with Lemma \ref{2sim} and Definition \ref{dim}, further implies that
\begin{align*}
\left\|A_BA_{\widetilde{B}}^{-1}\right\|^p
&\sim\exp\left(\fint_{\widetilde{B}}\log\left(\fint_B\left\|W^{\frac1p}(x)
W^{-\frac1p}(y)\right\|^p\,dx\right)\,dy\right)\notag\\
&\lesssim\exp\left(\fint_{\widetilde{B}}\log\left(\fint_{\lambda\widetilde{B}}
\left\|W^{\frac1p}(x)W^{-\frac1p}(y)\right\|^p\,dx\right)\,dy\right)
\lesssim\left(\frac{r_B}{r_{\widetilde{B}}}\right)^{d_2},
\end{align*}
which completes  the proof of (i).

We  next prove (ii). For any $B,\widetilde B\in\mathcal B$,
we can choose a ball $B^*\in\mathcal B$ such that
$B\cup\widetilde{B}\subset B^*$. Indeed, for any $x\in\widetilde{B}$,
\begin{align*}
\rho(x-c_B)
\leq H\left[\rho\left(x-c_{\widetilde{B}}\right)
+\rho\left(c_{\widetilde{B}}-c_B\right)\right]
\leq H\left[r_{\widetilde{B}}
+\rho\left(c_B-c_{\widetilde{B}}\right)\right],
\end{align*}
and hence
\begin{align*}
B\cup\widetilde{B}
\subset \max\left\{1,\frac{H[r_{\widetilde{B}}
+\rho(c_B-c_{\widetilde{B}})]}{r_B}\right\} B=:B^*,
\end{align*}
where
$r_{B^*}=\max\{r_B,H[r_{\widetilde{B}}
+\rho(c_B-c_{\widetilde{B}})]\}\sim r_B+r_{\widetilde{B}}+\rho(c_B-c_{\widetilde{B}})$.
From these and the just proven (i), we deduce that
\begin{align*}
\left\|A_BA_{\widetilde{B}}^{-1}\right\|^p
&\leq\left\|A_BA_{B^*}^{-1}\right\|^p
\,\left\|A_{B^*}A_{\widetilde{B}}^{-1}\right\|^p
\lesssim\left(\frac{r_{B^*}}{r_B}\right)^{{d_1}}
\left(\frac{r_{B^*}}{r_{\widetilde{B}}}\right)^{{d_2}}\\
&=\left(\frac{r_B\vee r_{\widetilde{B}}}{r_B}\right)^{{d_1}}
\left(\frac{r_{B^*}}{r_B\vee r_{\widetilde{B}}}\right)^{d_1+d_2}
\left(\frac{r_B\vee r_{\widetilde{B}}}{r_{\widetilde{B}}}\right)^{{d_2}}\\
&\sim\max\left\{\left(
\frac{r_{\widetilde{B}}}{r_B}\right)^{d_1},
\left(\frac{r_B}{r_{\widetilde{B}}}\right)^{d_2}\right\}
\left[1+\frac{\rho(c_B-c_{\widetilde{B}})}
{r_B\vee r_{\widetilde{B}}}\right]^{d_1+d_2}.
\end{align*}
This completes the proof of (ii) and hence Lemma \ref{fuhe}.
\end{proof}

The following lemma establishes the relationship between the families
$\mathcal B$ and $\mathcal D$.
Using this result, we can extend Lemma \ref{fuhe} to anisotropic cubes.

\begin{lemma}\label{yl1201}
For any $j\in\mathbb{Z}$ and $Q\in{\mathcal D}_{j}$, there exist  constants
$C_1, C_2\in(0,\infty)$, depending only on $\rho$ and $A$, such that
\begin{align}\label{BQB}
B_{\rho}(c_Q,C_1b^{-j})
\subset Q\subset B_{\rho}(x_Q,C_2b^{-j}).
\end{align}
Moreover,
\begin{align}\label{BsimQ}
|B_{\rho}(c_Q,C_1b^{-j})|\sim|B_{\rho}(x_Q,C_2b^{-j})|\sim |Q|,
\end{align}
where the positive equivalence constants depend only on $\rho$ and $A$.
\end{lemma}
\begin{proof}
By Lemma \ref{byl2d2},   there exists $k_1\in\mathbb Z$ such that
$[0,1)^n\supset (\frac{1}{2},\ldots,\frac{1}{2})+A^{-k_1}\Delta$.
Then, for any $j\in\mathbb{Z}$ and $Q\in\mathcal{D}_j$,
\begin{align*}
Q=x_Q+A^{-j}[0,1)^n\supset c_Q+A^{-j-k_1}\Delta
=B_{\rho}\left(c_Q,b^{-k_1}b^{-j}\right)=:B_{\rho}(c_Q,C_1b^{-j}).
\end{align*}
Similarly,  there exists $k_2\in\mathbb Z_+$ such that
$[0,1)^n\subset A^{k_2}\Delta$.
Then, for any $j\in\mathbb{Z}$ and $Q\in\mathcal{D}_j$,
\begin{align*}
Q=x_Q+A^{-j}[0,1)^n\subset x_Q+A^{-j+k_2}\Delta
=B_{\rho}(x_Q,b^{k_2}b^{-j})=:B_{\rho}(x_Q,C_2b^{-j}),
\end{align*}
and hence \eqref{BQB} holds. Moreover,
from  \eqref{ballmea}, it follows that \eqref{BsimQ} holds.
This completes the proof of Lemma \ref{yl1201}.
\end{proof}

The following two basic lemmas are widely used in this article.

\begin{lemma} \label{yl1101}
Let $i\in\mathbb Z$. Then the following statements hold.
\begin{enumerate}[\rm(i)]
\item There exists a positive constant $C_{0}$,
depending only on $\rho$ and $A$,
such that,	for any $x\in(-1,1)^n$,
$\rho(x)\leq C_{0}.$

\item For any $i\in\mathbb{Z}$ and $x,y\in Q\in\mathcal D_i$,
$\rho(x-y)\leq C_0 b^{-i}=C_0|Q|$.
\end{enumerate}
\end{lemma}

\begin{proof}
Statement (ii) follows from (i), so it suffices to prove the latter.
From Lemma \ref{byl2d2}, we infer that
there exists $k_0\in\mathbb N$ such that $(-1,1)^n\subset B_{k_0+1}$.
Then, by the definition of $\rho$, we find that, for any $x\in(-1,1)^n$, $\rho(x)\le b^{k_0}=:C_0$.
This completes the proof of Lemma \ref{yl1101}.
\end{proof}

%

\begin{lemma}\label{yl111503}
Let $i, j\in\mathbb{Z}$, $P\in \mathcal{D}_i $, and
$ Q \in \mathcal{D}_j $. Then the following statements hold.
\begin{enumerate}[\rm(i)]
\item  For any $ x $, $x'\in P $, and $y$, $y'\in Q $,
one has $1 + b^{i\wedge j}\rho(x-y)
\sim 1 + b^{i\wedge j}\rho(x'-y')$.

\item For any $x\in P$,
one has
$1+b^{i\wedge j}\rho(x-x_Q)
\sim 1+ b^{i\wedge j}\rho(x_P-x_Q)$.

\item For any $x\in P$ and $y\in \mathbb R^n$,
one has $1+b^{i}\rho(x-y)
\sim 1+ b^{i}\rho(x_P-y)$.

\item If $i\ge0$, then, for any $x,x'\in P$ and $y\in \mathbb R^n$, one has $1+\rho(x-y)
\sim 1+\rho(x'-y)$.
\end{enumerate}
Here all the positive equivalence constants depend only on $\rho$ and $A$.
\end{lemma}

\begin{proof}
We first prove (i). Without loss of generality, we  may assume that
$j\le i$.
In this case, $i\wedge j=j$.
By Definition \ref{quasi-norm}(iii) and  Lemma \ref{yl1101}(ii),
we find that,
for any $ P\in \mathcal{D}_i $, $ Q \in \mathcal{D}_j $,
$ x, x' \in P $, and $ y, y' \in Q $,
\begin{align*}
1+b^j\rho(x-y)
&\lesssim 1+ b^j\left[\rho(x-x')
+\rho(x'-y')+\rho(y'-y)\right]\\
&\lesssim 1+b^j \left[ b^{-i}
+  \rho(x'-y')
+b^{-j}
\right]
\sim 1+b^j\rho(x'-y').
\end{align*}
By symmetry, we obtain the reverse inequality.
This completes the proof of (i).

From the just proven (i), it is simple to see that (ii) holds.

We then prove (iii).
For any $y\in\mathbb R^n$, there exists $j\in\mathbb Z$ satisfying $j\ge i$ and $Q\in\mathcal D_j $
such that $y\in Q$. This, combined with the just proven (i),
finishes the proof of (iii).

Finally, we prove (iv).
For any $y\in\mathbb R^n$, there exists $k\in\mathbb Z^n$ such that $y\in Q_{0,k}$.
Using this and the just proven (i) with $j=0$,
we conclude that
$1+\rho(x-y)\sim1+\rho(x'-y)$. This completes the proof of (iv) and hence Lemma \ref{yl111503}.
\end{proof}

%

We can now establish an analogue of Lemma \ref{fuhe} for
anisotropic dyadic cubes.

\begin{lemma}\label{tl012001}
Let $p\in(0,\infty), W\in \mathcal A_{p,\infty}$, and
$\{A_Q\}_{Q\in\mathcal{D}}$ be a family of reducing operators of order $p$ for $W$.
Let $d_{1}\in[\![d_{p,\infty}^{\mathrm{lower}}(W),\infty)$
and $d_{2}\in[\![d_{p,\infty}^{\mathrm{upper}}(W),\infty)$.
Then there exists a positive constant $C$ such that,
\begin{enumerate}[(\rm i)]
\item for any $Q,R\in \mathcal{D}$ with $Q\bigcap R\neq\emptyset$,
$$
\left\|A_QA_R^{-1}\right\|^p\leq C\max\left\{\left(\frac{|R|}{|Q|}\right)^{d_1},
\left(\frac{|Q|}{|R|}\right)^{d_2}\right\};
$$

\item for any $Q,R\in \mathcal{D}$,
$$
\left\|A_QA_R^{-1}\right\|^p\leq C\max\left\{\left(\frac{|R|}{|Q|}\right)^{d_1},
\left(\frac{|Q|}{|R|}\right)^{d_2}\right\}\left[1+\frac{\rho(x_Q-x_R)}
{\max\{|Q|,|R|\}}\right]^{d_1+d_2}.
$$
\end{enumerate}
\end{lemma}

\begin{proof}
Statement (i) follows from (ii) and Lemma \ref{yl1101}(ii),
so it suffices to prove (ii).
By Lemma \ref{yl1201}, we find that, for any $Q,R\in\mathcal D$,
there exist $B,\widetilde B\in\mathcal B$ such that,
$Q\subset B$ with $|Q|\sim |B|$ and
$\widetilde B\subset R$ with $|R|\sim |\widetilde B|$.
From these and Lemmas \ref{reduceM} and \ref{fuhe}(ii), we infer that
\begin{align} \label{key2}
\left\|A_QA_R^{-1}\right\|^p
&= \sup_{\vec z\in\mathbb C^m\setminus\{\vec{\mathbf 0}\}}
\frac{|A_QA_R^{-1}\vec z|^p}{|\vec z|^p}
= \sup_{\vec z\in\mathbb C^m\setminus\{\vec{\mathbf 0}\}}
\frac{|A_Q\vec z|^p}{|A_R\vec z|^p}
\sim \sup_{\vec z\in\mathbb C^m\setminus\{\vec{\mathbf 0}\}}
\frac{\fint_Q|W^{\frac1p}(x)\vec z|^p\,dx}
{\fint_R|W^{\frac1p}(x)\vec z|^p\,dx} \notag\\
&\lesssim \sup_{\vec z\in\mathbb C^m\setminus\{\vec{\mathbf 0}\}}
\frac{\fint_B|W^{\frac1p}(x)\vec z|^p\,dx}
{\fint_{\widetilde B}|W^{\frac1p}(x)\vec z|^p\,dx}
\sim \sup_{\vec z\in\mathbb C^m\setminus\{\vec{\mathbf 0}\}}
\frac{|A_B\vec z|^p}{|A_{\widetilde B}\vec z|^p}
=\left\|A_B A_{\widetilde B}^{-1}\right\|^p \notag\\
&\lesssim \max\left\{\left(\frac{r_{\widetilde{B}}}{r_B}\right)^{d_1},
\left(\frac{r_{B}}{r_{\widetilde{B}}}\right)^{d_2}\right\}
\left[1+\frac{\rho(c_{B}-c_{\widetilde{B}})}
{r_{B} \vee r_{\widetilde{B}}}\right]^{d_1+d_2}.
\end{align}
Note that $r_{B}\sim |B|\sim |Q|$
and $r_{\widetilde B}\sim |\widetilde B|\sim |R|$.
Using this, $x_Q\in Q\subset B$, $c_{\widetilde{B}}\in \widetilde{B}\subset R$,
and Lemma \ref{yl1101}(ii), we obtain
\begin{align*}
\rho(c_{B}-c_{\widetilde{B}})
&\lesssim \rho(c_{B}-x_Q)+\rho(x_Q-x_R)+\rho(x_R-c_{\widetilde{B}}) \\
&\leq r_B+\rho(x_Q-x_R)+|R|
\sim (|Q|\vee |R|) +\rho(x_Q-x_R).
\end{align*}
Substituting the above estimates of $r_{B}$, $r_{\widetilde B}$,
and $\rho(c_{B}-c_{\widetilde{B}})$
into \eqref{key2} yields the desired result  and then completes the proof of Lemma \ref{tl012001}.
\end{proof}

Let $\{A_Q\}_{Q\in\mathcal D}$ be
a sequence of positive definite matrices.
For any $j\in\mathbb Z$, let
\begin{align}\label{Aj}
A_j:=\sum_{Q\in\mathcal D_j}\mathbf 1_QA_Q.
\end{align}

\begin{definition}
Let $p\in(0,\infty)$ and $d_1,d_2\in[0,\infty)$.
A sequence of positive definite matrices
$\{A_Q\}_{Q\in\mathcal{D}}$ is called \emph{doubling of order
$(d_1,d_2; p)$} if it satisfies Lemma \ref{tl012001}(ii).
\end{definition}

If $\{A_Q\}_{Q\in\mathcal{D}}$ is doubling of order $(d_1,d_2; p)$,
then, by Lemma \ref{tl012001},
we find that,  for any $ j \in \mathbb{Z} $ and $ Q, R\in \mathcal{D}_j $,
\begin{align*}
\left\|A_QA_R^{-1}\right\|^p\lesssim\left[1+b^j\rho(x_Q-x_R)\right]^{d_1+d_2}.
\end{align*}
From this and Lemma \ref{yl111503}, we further deduce that,
for any $ j \in \mathbb{Z} $, $ Q,R \in \mathcal{D}_j $,
$ x \in Q $, and $ y \in R $,
\begin{align*}
\left\|A_QA_R^{-1}\right\|^p\lesssim\left[1+b^j\rho(x-y)\right]^{d_1+d_2}.
\end{align*}
Using the definition of $A_j$, we can rewrite the above estimate as,
for any $j\in\mathbb Z$ and $x,y\in\mathbb R^n$,
\begin{align}\label{AiAj}
\left\|A_j(x)[A_j(y)]^{-1}\right\|^p
\lesssim\left[1+b^j\rho(x-y)\right]^{d_1+d_2}.
\end{align}

\section{Proof of the $\varphi$-Transform Characterization}
\label{phi}

In this section, we prove the $\varphi$-transform characterization of $\dot{B}^{\alpha}_{p,q}(W,\varphi)$ (Theorem \ref{dl1103}).
We first introduce the averaging matrix-weighted anisotropic Besov spaces
$\dot B_{p,q}^{\alpha}(\mathbb A,\varphi)$ (see Definition \ref{AMW})
and prove their equivalence to $\dot B_{p,q}^{\alpha}(W,\varphi)$
in Subsection \ref{BB}.
Next, we  establish the equivalence relation for
the corresponding sequence space in Subsection \ref{equibb}.
Using these equivalences, we finally prove the $\varphi$-transform characterization of $\dot B_{p,q}^{\alpha}(W,\varphi)$ in Subsection \ref{PTC}.

\subsection{Relations between $\dot{B}^{\alpha}_{p,q}(W,\varphi)$
and $\dot B^{\alpha}_{p,q}(\mathbb{A},\varphi)$}\label{BB}

We introduce the averaging matrix-weighted anisotropic Besov spaces.

\begin{definition}\label{AMW}
Let $\alpha\in\mathbb{R}$, $p\in(0,\infty)$,
$q\in(0,\infty]$, and $\varphi\in{\mathcal S}$
satisfy \eqref{hs2}. Assume that
${\mathbb A}:=\{A_Q\}_{Q\in\mathcal{D}}$ is a sequence of
positive definite matrices. The \emph{averaging matrix-weighted anisotropic Besov space
$\dot{B}^{\alpha}_{p,q}(A,{\mathbb A},\varphi)$} is defined by setting
\begin{align*}
\dot{B}^{\alpha}_{p,q}(A,{\mathbb A},\varphi)
:=\left\{\vec f\in\left(\mathcal{S}'_\infty\right)^m
:\  \left\|\vec{f}\right\|_{\dot{B}^{\alpha}_{p,q}(A,{\mathbb A},\varphi)}
<\infty\right\},
\end{align*}
where, for any $\vec f\in(\mathcal{S}'_\infty)^m$,
\begin{align*}
\left\|\vec{f}\right
\|_{\dot{B}^{\alpha}_{p,q}(A,{\mathbb A},\varphi)}
:= \left\|\left\{b^{j\alpha}
\left|A_j\left(\varphi_j\ast\vec{f}\right)\right|\right\}_{j\in\mathbb Z}\right\|_{\ell^qL^p}
\end{align*}
with $A_j$ and $\|\cdot\|_{\ell^qL^p}$ as, respectively, in \eqref{Aj} and \eqref{lqLp}.
In what follows, if there is no confusion,
we denote $\dot{B}^{\alpha}_{p,q}(A,{\mathbb A},\varphi)$
simply by $\dot{B}^{\alpha}_{p,q}(\mathbb A,\varphi)$.
\end{definition}

For any $ \vec f \in ({\mathcal S}'_\infty)^m $ and
$ \varphi \in {\mathcal S}_\infty$,
let $ \sup_{\mathbb{A}, \varphi} (\vec f)
:= \{\sup_{\mathbb{A}, \varphi, Q}(\vec f)\}_{Q \in \mathcal{D}},$
where, for any $ Q \in \mathcal{D}$,
\begin{align} \label{sup}
\sup_{\mathbb{A}, \varphi, Q} \left( \vec f \right)
:= |Q|^{\frac{1}{2}} \sup_{y \in Q}
\left| A_Q \left( \varphi_{j_Q} * \vec f \right) (y) \right|.
\end{align}
The following theorem is the main result of this subsection.

\begin{theorem}\label{dl111301}
Let $\alpha\in\mathbb{R}$, $p\in(0,\infty)$,
$q\in(0,\,\infty]$, and $\varphi\in{\mathcal S}$ satisfy \eqref{hs2}.
Assume that $W\in \mathcal A_{p,\infty}$ and ${\mathbb A}:=\{A_Q\}_{Q\in\mathcal{D}}$ is
a sequence of reducing operators of order $ p $ for $ W $.
Then,  for any $\vec f\in ({\mathcal S}'_\infty)^m$,
$$
\left\| \vec f \right\|_{\dot B^{\alpha}_{p,q}(\mathbb{A},\varphi)}
\sim\left\|\sup_{\mathbb{A},\varphi}
\left(\vec{f}\right)\right\|_{\dot b^{\alpha}_{p,q}}
\sim
\left\|\vec f\right\|_{{\dot{B}^{\alpha}_{p,q}(W,\varphi)}},
$$
where the positive equivalence constants are independent of $\vec f$.
\end{theorem}

To prove Theorem \ref{dl111301}, we need several technical lemmas.
The proof of the following result is standard; we omit the details.

\begin{lemma} \label{yl111505}
Let $a \in \mathbb R$ and $j\in\mathbb Z$. Then the following statements hold.
\begin{enumerate}[\rm(i)]
\item If $a\in(1,\infty)$, then, for any $ y \in \mathbb{R}^n $,
$$
\int_{B_\rho(y,b^{-j})} \frac{b^{j}}{[1 + b^j\rho( x - y)]^a} \, dx
\sim 1
\sim \int_{\mathbb{R}^n} \frac{b^{j}}{[1 + b^j\rho( x - y)]^a} \, dx.
$$

\item If $a\in(1,\infty)$ and $j \leq 0$, then,
for any $y \in \mathbb{R}^n $,
$$
\sum_{k \in \mathbb{Z}^n\cap B_\rho(y, C_0 b^{-j})} \frac{b^{j}}{ [1 + b^j\rho( k - y)]^a}
\sim 1
\sim \sum_{k \in \mathbb{Z}^n} \frac{b^{j}}{ [1 + b^j\rho( k - y)]^a},
$$
where $C_0$ is as in Lemma \ref{yl1101}.

\item If
$\int_{\mathbb{R}^n} \frac{b^{j}}{[1 + b^j\rho( x - y)]^a} \, dx
<\infty$
for some $y\in\mathbb R^n$, then $a\in(1,\infty)$.

\item If $ j \leq 0 $ and
$\sum_{k \in \mathbb{Z}^n} \frac{b^{j}}{ [1 + b^j\rho( k - y)]^a}
< \infty$
for some $y\in\mathbb R^n$, then $a\in(1,\infty)$.
\end{enumerate}
Here all the positive equivalence constants depend only on  $a$, $\rho$, and $ A $.
\end{lemma}

Inspired by \cite[Appendix B.1]{g14}, we obtain the following conclusion.

\begin{lemma}\label{intmn}
Let $M, N\in(1, \infty)$.
Then there exists a positive constant $C$, depending only on $M, N$, and $\rho$,
such that, for any  $y, z\in\mathbb{R}^n$, $j, k\in\mathbb{Z}$,
\begin{align}\label{jkmn}
\int_{\mathbb{R}^n}\frac{b^{j}}{[1+b^j\rho(x-y)]^M}
\frac{b^k}{[1+b^k\rho(x-z)]^N}\,dx
\leq C\frac{b^{j\wedge k}}{[1+b^{j\wedge k}\rho(y-z)]^{M\wedge N}}.
\end{align}
\end{lemma}
\begin{proof}
By symmetry, without loss  of generality, we  may assume that $k\le j$.
We prove the present lemma
by considering  two cases for $b^k\rho(y-z)$.
For brevity, for any measurable set $E\subset\mathbb R^n$, let
$$I(E):=\int_{E}\frac{b^{j}}{[1+b^j\rho(x-y)]^M}
\frac{b^k}{[1+b^k\rho(x-z)]^N}\,dx.$$

\emph{Case (1)} ${b^k}\rho(y-z)\le{2H}$. In this case,
we have
$$\frac{b^k}{[1+b^k\rho(x-z)]^{N}}
\le b^{k}\le\frac{b^{k}(1+2H)^{M\wedge N}}{[1+b^k\rho(y-z)]^{M\wedge N}}.
$$
Using this	and  Lemma \ref{yl111505}(i), we obtain
\begin{align*}
{I}(\mathbb R^n)
\leq\frac{b^{k}(1+2H)^{M\wedge N}}{[1+b^k\rho(y-z)]^{M\wedge N}}
\int_{\mathbb{R}^n}\frac{b^{j}}{[1+b^j\rho(x-y)]^M}
\,dx
\lesssim\frac{b^{k}}{[1+b^k\rho(y-z)]^{M\wedge N}}.
\end{align*}
This completes the proof of \eqref{jkmn} in this case.

\emph{Case (2)} ${b^k}\rho(y-z)>{2H}$.
For any $x\in\mathbb{R}^n$,
$$
\rho(y-z)\le H[\rho(x-y)+\rho(x-z)]
\le 2H\max\{\rho(x-y),\rho(x-z)\},
$$
and hence $x\in E_y\bigcup E_z$,
where, for any $a\in\mathbb R^n$,
$$ E_a:=\left\{x\in\mathbb R^n:\,\rho(y-z)\le2H\rho(x-a)\right\}.$$
Therefore, ${I}(\mathbb R^n)\le{I}(E_y)+{I}(E_z)$.
By the definition of ${E}_z$ and Lemma \ref{yl111505}(i), we obtain
\begin{align*}
{I}(E_z)
&\lesssim\frac{b^{k}}{[1+{b^k}\rho(y-z)]^N}
\int_{E_{z}}\frac{b^j}{[1+b^j\rho(x-y)]^M}\,dx\notag\\
&\lesssim\frac{b^{k}}{[1+{b^k}\rho(y-z)]^N}
\le\frac{b^{k}}{[1+{b^k}\rho(y-z)]^{M\wedge N}}.
\end{align*}
From the definition of ${E}_y$ and Lemma \ref{yl111505}(i), it follows that
\begin{align}\label{jz}
{I}(E_y)
&\lesssim\frac{b^j}{[1+{b^{j}}\rho(y-z)]^M}\int_{E_{y}}
\frac{b^k}{[1+b^k\rho(x-z)]^N}\,dx\notag\\
&\lesssim\frac{b^j}{[1+{b^{j}}\rho(y-z)]^M}
\le \frac{b^j}{[b^j\rho(y-z)]^M}.
\end{align}
Moreover, by $M\in(1,\infty), k\le j$, and the assumption that
${b^k}\rho(y-z)>{2H}$, we obtain
\begin{align*}
\frac{b^j}{[b^j\rho(y-z)]^M}
\le\frac{b^{k(1-M)}}{[\rho(y-z)]^M}
\sim\frac{b^k}{[1+b^k\rho(y-z)]^M}
\le\frac{b^{k}}{[1+{b^k}\rho(y-z)]^{M\wedge N}},
\end{align*}
which, together with \eqref{jz}, further implies that
$$ {I}(E_y)\lesssim \frac{b^{k}}{[1+{b^k}\rho(y-z)]^{M\wedge N}}.$$
Combining the above estimates, we obtain \eqref{jkmn} in this case.
This completes the proof of Lemma \ref{intmn}.
\end{proof}

Applying Lemma \ref{intmn},
we obtain the following conclusion; we omit the details.

\begin{lemma}\label{summn}
Let $M, N\in( 1, \infty)$. Then there exists a positive constant
$C$, depending only on $M, N$, and  $\rho$, such that, for any
$y, z\in \mathbb{R} ^n$, $j,k\in \mathbb Z\cap(-\infty, 0]$,
$$\sum_{l\in\mathbb{Z}^n}\frac{b^j}{[1+b^j\rho(l-y)]^M}
\frac{b^k}{[1+b^k\rho(l-z)]^N}
\leq C\frac{b^{j\wedge k}}{[1+b^{j\wedge k}\rho(y-z)]^{M\wedge N}}.$$
\end{lemma}

By \cite[Theorem 2.3.21]{hsdyk01}, we find that if
$g\in\mathcal S'$ and $\widehat g$ has compact support, then
$g$ coincides with a $C^{\infty}$ function.
We now recall the following reproducing formula.

\begin{lemma}\label{repro}
The following statements hold.
\begin{enumerate}[\rm(i)]
\item If $g\in\mathcal{S}'$,
$h\in\mathcal{S}$ satisfy
$\operatorname{supp}\widehat{g}$, $\operatorname{supp}\widehat{h}\subset(A^*)^j[-\pi,\pi]^n$
for some $j\in\mathbb Z$, then
\begin{align}\label{reproduce2}
(g\ast h)(x)
=\sum_{k\in\mathbb{Z}^n}
b^{-j}g(A^{-j}k)h(x-A^{-j}k)
\end{align}
converges pointwise and  in $\mathcal{S}'$.

\item If $\varphi,\psi\in\mathcal{S}$ satisfy \eqref{hs3} and
$\operatorname{supp}\widehat{\varphi},\operatorname{supp}\widehat{\psi}
\subset[-\pi,\pi]^n\setminus\{\mathbf0\}$,
then, for any $f\in L^2$ (resp. $\mathcal S_{\infty}$ or $\mathcal S_{\infty}'$),
\begin{align}\label{reproduce}
f=\sum_{Q\in\mathcal{D}}\langle f,\varphi_Q\rangle\psi_Q,
\end{align}
where the series converges in $L^2$ (resp. $\mathcal S_{\infty}$ or $\mathcal S_{\infty}'$).
\end{enumerate}
\end{lemma}

Lemma \ref{repro}(i) and \eqref{reproduce} for $f\in\mathcal S_{\infty}'$
are from \cite[Lemma 2.8]{gjabownik06}.
Indeed, although the pointwise convergence in \eqref{reproduce2}
is not stated in \cite[Lemma 2.8]{gjabownik06},
it is clearly demonstrated in its proof.
The fact that \eqref{reproduce} holds for $f\in L^2$
is precisely \cite[(2.8)]{gjabownik05}.
Moreover, its validity for $f\in\mathcal S_{\infty}$
was proved in \cite[p.\,63]{bb10}.

The following lemma can be found in the proof of \cite[Lemma 3.1]{gjabownik05}
with some modifications. For the convenience of the reader, we give the details of its proof.

\begin{lemma}\label{dl111410}
Let $\gamma\in\mathcal{S}$ satisfy $\operatorname{supp}\widehat{\gamma}\subset [-5,5]^n$ and
$\widehat{\gamma}(\xi)=1$ for all $\xi \in [-\pi,\pi]^n$.
Assume that $j\in\mathbb{Z}$ and $f\in\mathcal{S}'$
with $ \operatorname{supp} \widehat{f} \subset (A^*)^j[-\pi, \pi]^n $.
Then $f\in C^{\infty}$ and,
for any $x,y\in\mathbb{R}^n$, the following pointwise identity holds
\begin{align*}
f(x)=\sum_{R\in\mathcal{D}_{j+j_0}}b^{-(j+j_0)}f(x_R+y)\gamma_j(x-x_R-y),
\end{align*}
where $j_0$ is a positive integer depending only on $A$.
\end{lemma}

\begin{proof}
Let $j_0\in\mathbb N$ satisfy
$[-5,5]^n\subset(A^*)^{j_0}[-\pi,\pi]^n$.
Then $$\operatorname{supp}\widehat{\gamma_{j}}\subset (A^*)^{j}[-5,5]^n
\subset(A^*)^{j+j_0}[-\pi,\pi]^n.$$
For any $y\in\mathbb R^n$, let $g(\cdot)=f(\cdot+y)$.
Then $\operatorname{supp}\widehat g
=\operatorname{supp}\widehat f\subset (A^*)^{j}[-\pi,\pi]^n$.
Note that, for any $\xi\in(A^*)^{j}[-\pi,\pi]^n$,
$\widehat{\gamma_{j}}(\xi)
=\widehat{\gamma}([(A^{-1})^*]^j\xi)=1$.
These further imply that $g=g\ast\gamma_{j}$.
From this and Lemma \ref{repro}(i) with $j$ replaced by $j+j_0$,
we deduce that, for any $x\in\mathbb R^n$,
\begin{align*}
f(x+y)
=g(x)
=(g*\gamma_{j})(x)
=\sum_{k\in\mathbb Z^n} b^{-{(j+j_0)}} g\left(A^{-{(j+j_0)}}k\right)
\gamma_{j}\left(x-A^{-{(j+j_0)}}k\right),
\end{align*}
and hence
$$
f(x)=\sum_{R\in\mathcal{D}_{j+j_0}} b^{-(j+j_0)} f(x_R+y) \gamma_j(x-x_R-y).
$$
This completes the proof of Lemma \ref{dl111410}.
\end{proof}

\begin{remark}
Let $ f \in {\mathcal S}'_\infty$
and $\varphi\in\mathcal S$ satisfy \eqref{hs2}.
Note that, for any $ j \in \mathbb{Z} $,
\begin{align*}
\operatorname{supp} \widehat{\varphi_j * f}
\subset	\operatorname{supp} \widehat{\varphi_j}
\subset (A^*)^j[-\pi,\pi]^n.
\end{align*}
Applying Lemma \ref{dl111410} with $f$ replaced by $\varphi_j*f$,
we conclude that, for any $j\in\mathbb{Z}$ and $x,y\in\mathbb{R}^n$,
\begin{align} \label{Reproduce}
(\varphi_j * f) (x)
= \sum_{R\in\mathcal{D}_{j+j_0}} b^{-(j+j_0)} (\varphi_j * f) (x_R + y)
\gamma_j (x - x_R - y),
\end{align}
where $\gamma$ and $j_0$
are  as in Lemma \ref{dl111410}.
\end{remark}

The following lemma is well known; we omit the details.

\begin{lemma}\label{equivalent}
Let $ a \in (0, 1] $. Then, for any
$ \{z_i\}_{i \in \mathbb Z}$ in $\mathbb{C}$,
$
( \sum_{i \in \mathbb Z} |z_i| )^{a}
\leq \sum_{i \in \mathbb Z} |z_i|^{a}.
$
\end{lemma}

\begin{definition}
The \emph{anisotropic Hardy--Littlewood maximal operator $\mathcal M$ }
is defined by setting, for any
$f\in L_{\mathrm{loc}}^1$ and $x\in\mathbb{R}^n$,
\begin{align*}
\mathcal M(f)(x):=
\sup_{\genfrac{}{}{0pt}{}{B\in\mathcal B}{B\ni x}} \fint_{B} |f(y)|\,dy.
\end{align*}
\end{definition}

\begin{lemma} \label{summary B}
Let $\alpha\in\mathbb R$, $p\in(0,\infty)$,
$q\in(0,\,\infty]$,  $ M \in (1, \infty) $, and $r\in(0,p]$.
Suppose two sequences $ \{g_j\}_{j \in \mathbb{Z}} $
and $ \{h_j\}_{j \in \mathbb{Z}} $
of measurable functions on $\mathbb{R}^n $
satisfy that there exists a positive constant $C$ such that,
for any $j\in\mathbb{Z}$ and $x\in\mathbb{R}^n$,
\begin{align} \label{sz111511}
\left|g_j(x)\right|^r\leq C^r b^{j}
\int_{\mathbb{R}^n}
\frac{1}{[1+b^{j}\rho(x-z)]^M}
\left|h_j(z)\right|^r\,dz.
\end{align}
Then there exists a positive constant $\widetilde{C}$,
depending only on $p$, $q$, $r$, $M$, $\rho$, and $A$, such that
\begin{align}\label{intgj}
\left\|\left\{b^{j\alpha} g_j\right\}_{j\in\mathbb{Z}} \right\|_{\ell^q L^p}
\leq \widetilde{C} C
\left\|\left\{b^{j\alpha} h_j\right\}_{j\in\mathbb{Z}} \right\|_{\ell^q L^p},
\end{align}
where $\|\cdot\|_{\ell^q L^p}$ is as in \eqref{lqLp}.
\end{lemma}

\begin{proof}
Without loss of generality, we may assume that $\alpha=0$.
We prove \eqref{intgj} by considering the two cases for $r$.

\emph{Case (1)} $r=p$. In this case,
by \eqref{sz111511} with $r=p$,  Tonelli's theorem, Lemma \ref{yl111505}(i),
and $M\in(1,\infty)$,
we find that, for any $j\in{\mathbb Z}$,
\begin{align*}
\left\| g_j\right\|_{L^p}^p
&\leq C^p \int_{\mathbb{R}^n}b^{j}\int_{\mathbb{R}^n}
\frac{1}{[1 + b^{j} \rho(x - z)]^M}
\left| h_j(z) \right|^p \, dz\,dx\\
&= C^p \int_{\mathbb{R}^n}\int_{\mathbb{R}^n}
\frac{b^{j}}{[1 + b^{j} \rho(x - z)]^M} \, dx
\left| h_j(z) \right|^p \, dz\\
&\sim C^p \int_{\mathbb{R}^n} \left|h_j(z)\right|^p\,dz
= C^p \left\| h_j\right\|_{L^p}^p.
\end{align*}
Taking the $\frac1p$-th power on both sides and then applying the $\ell^q$ norm
yields \eqref{intgj}. This completes the proof of \eqref{intgj} in this case.

\emph{Case (2)} $r\in(0,p)$. In this case,
applying \eqref{sz111511}, H\"older's inequality,
and  Lemma \ref{yl111505}(i),
we obtain, for any $j\in{\mathbb Z}$ and $x\in\mathbb R^n$,
\begin{align} \label{sz111512}
\left|g_j(x)\right|^r
&\leq C^r \left\{b^{j}\int_{\mathbb{R}^n}\frac{1}{[1+b^{j}\rho(x-z)]^M}
\,dz\right\}^{1-\frac{r}{p}}
\left\{b^{j}\int_{\mathbb{R}^n}\frac{|h_j(z)|^p}{[1+b^{j}\rho(x-z)]^M}
\,dz\right\}^{\frac{r}{p}}\notag\\
&\sim C^r \left\{b^{j}\int_{\mathbb{R}^n}\frac{|h_j(z)|^p}{[1+b^{j}\rho(x-z)]^M}
\,dz\right\}^{\frac{r}{p}}.
\end{align}
Taking the $\frac{p}{r}$-th power on both sides of \eqref{sz111512}, we return to Case (1).
This completes the proof of \eqref{intgj} in this case and hence
Lemma \ref{summary B}.
\end{proof}


We can now prove the first equivalence of Theorem \ref{dl111301}.

\begin{lemma} \label{yl111601}
Let $ \alpha \in \mathbb{R} $, $p\in(0,\infty)$,
$q\in(0,\,\infty]$, and
$\varphi \in \mathcal{S}$ satisfy \eqref{hs2}. Assume that
$ \mathbb{A} := \{ A_Q \}_{Q \in \mathcal{D}} $ is
doubling of order $(d_1,d_2; p)$ for some $d_1,d_2\in[0,\infty)$.
Then, for any
$ \vec f \in ({\mathcal S}'_\infty)^m $,
\begin{align*}
\left\|\vec{f}\right\|_{\dot B^{\alpha}_{p,q}(\mathbb{A},\varphi)}
\sim\left\|\sup_{\mathbb{A},\varphi}
\left(\vec{f}\right)\right\|_{\dot b^{\alpha}_{p,q}},
\end{align*}
where the positive equivalence constants are independent of $\vec f$.
\end{lemma}

\begin{proof}
By the definition of $\sup_{\mathbb{A}, \varphi} (\vec f) $,
we find that, for any $ \vec f \in ({\mathcal S}'_\infty)^m $,
\begin{align*}
\left\| \vec f \right\|_{\dot B^{\alpha}_{p,q}(\mathbb{A}, \varphi)}
\leq \left\| \sup_{\mathbb{A}, \varphi} \left( \vec f \right)
\right\|_{\dot b^{\alpha}_{p,q}}.
\end{align*}

To prove the reverse inequality, fix $r\in(0,\frac{1}{1\vee p})$ and $M\in (\frac1{rp}+\frac{d_1+d_2}{p}, \infty)$.
Let $\gamma$ and $j_0$ be as in Lemma \ref{dl111410}.
Applying \eqref{Reproduce}, Lemma \ref{equivalent},
$\gamma \in \mathcal{S}$,
and Lemma \ref{yl111503}(iii), we conclude that,
for any $j \in \mathbb{Z} $, $ Q \in \mathcal{D}_j $,
$ x \in Q $, and $ y \in \mathbb{R}^n $,
\begin{align*}
\left| A_Q \left( \varphi_j * \vec f \right) (x) \right|^{pr}
&=\left|\sum_{R \in \mathcal{D}_{j+j_0}} b^{-(j+j_0)}
A_Q \left( \varphi_j * \vec f \right) (x_R + y)
\gamma_j (x - x_R - y)\right|^{pr}\\
&\leq \sum_{R \in \mathcal{D}_{j+j_0}}\left|b^{-(j+j_0)}\gamma_j(x-x_R-y)\right|^{pr}
\left| A_Q \left( \varphi_j * \vec f \right) (x_R + y) \right|^{pr}\\
&\lesssim \sum_{R \in \mathcal{D}_{j+j_0}}
\frac{1}{[1 + b^j\rho(x - x_R-y)]^{Mpr}}
\left| A_Q \left( \varphi_j * \vec f \right) (x_R + y) \right|^{pr}\\
&\sim \sum_{R \in \mathcal{D}_{j+j_0}} \frac{1}{[1 + b^j\rho(x_Q-x_R-y)]^{Mpr}}
\left| A_Q \left( \varphi_j * \vec f \right) (x_R + y) \right|^{pr}.
\end{align*}
This, together with  Lemma \ref{yl111503}(iii) again, further implies that
\begin{align*}
\left[ |Q|^{-\frac{1}{2}} \sup_{\mathbb{A}, \varphi, Q}
\left( \vec f \right) \right]^{pr}
\lesssim \sum_{R \in \mathcal{D}_{j+j_0}} \frac{1}{[1 + b^j\rho(x - x_R-y)]^{Mpr}}
\left| A_Q \left( \varphi_j * \vec f \right) (x_R + y) \right|^{pr}.
\end{align*}
Using this, Tonelli's theorem, and \eqref{AiAj},
we conclude that, for any $ j \in \mathbb{Z} $,
$ Q \in \mathcal{D}_j $, and $ x \in \mathbb{R}^n $,
\begin{align}\label{kappa}
\left[\widetilde{\mathbf{1}}_Q(x) \sup_{\mathbb{A}, \varphi, Q} \left( \vec f \right)
\right]^{pr}
&\lesssim \sum_{R \in \mathcal{D}_{j+j_0}}
\fint_{Q_{{j+j_0},\mathbf{0}}}\frac{1}{[1 + b^j\rho(x - x_R-y)]^{Mpr}}
\left| A_Q \left( \varphi_j * \vec f \right) (x_R + y) \right|^{pr}\,dy\notag\\
&\sim b^j \sum_{R \in \mathcal{D}_{j+j_0}} \int_R
\frac{1}{[1 +b^j\rho(x - z)]^{Mpr}}
\left| A_j(x) \left( \varphi_j * \vec f \right) (z) \right|^{pr}\, dz\notag\\
&\leq b^j \sum_{R \in \mathcal{D}_{j+j_0}} \int_R
\frac{\| A_j(x) [A_j(z)]^{-1} \|^{pr}}{[1 + b^{j} \rho(x - z)]^{Mpr}}
\left| A_j(z) \left( \varphi_j * \vec f \right) (z) \right|^{pr} \, dz\notag\\
&\lesssim b^j \sum_{R \in \mathcal{D}_{j+j_0}} \int_R
\frac{1}{[1 + b^{j} \rho(x - z)]^{[Mp-(d_1+d_2)]r}}
\left|A_j(z) \left( \varphi_j * \vec f \right) (z) \right|^{pr} \, dz\notag\\
&= b^j \int_{\mathbb{R}^n}
\frac{1}{[1 + b^{j} \rho(x - z)]^{[Mp-(d_1+d_2)]r}}
\left|A_j(z)
\left( \varphi_j * \vec{f} \right) (z) \right|^{pr} \, dz,
\end{align}
where $A_j$ is as in \eqref{Aj}.
For any $ j \in \mathbb{Z} $, let
\begin{align} \label{gj}
g_j := \sum_{Q\in\mathcal{D}_j}\widetilde{\mathbf{1}}_Q
\sup_{\mathbb{A}, \varphi, Q} \left(\vec f \right)
\ \ \text{and} \ \
h_j :=A_{j}\left( \varphi_j * \vec{f} \right).
\end{align}
Then, from \eqref{kappa} and Lemma \ref{summary B} with $r$ and $ M $
replaced, respectively, by $pr$ and $ [Mp-(d_1+d_2)]r$,
it follows that
$$
\left\| \sup_{\mathbb{A}, \varphi} \left( \vec f \right)
\right\|_{\dot{b}^\alpha_{p,q}}
= \left(\sum_{j\in\mathbb{Z}}
\left\|b^{j\alpha} g_j\right\|^{q}_{L^{p}}\right)^{\frac{1}{q}}
\lesssim \left( \sum_{j\in\mathbb{Z}}
\left\|b^{j\alpha} h_j\right\|^{q}_{L^{p}}
\right)^{\frac{1}{q}}
= \left\| \vec f \right\|_{\dot B^{\alpha}_{p,q}(\mathbb{A}, \varphi)}.
$$
This completes the proof of Lemma \ref{yl111601}.
\end{proof}

To prove the second equivalence of Theorem \ref{dl111301},
we need the following lemma.

\begin{lemma}\label{fsj}
Let $p,u\in(0,\infty)$ and $M\in(1,\infty)$.
Assume that $W\in \mathcal A_{p,u}$
and ${\mathbb A}:=\{A_Q\}_{Q\in\mathcal{D}}$ is
a sequence of reducing operators of order $ p $ for $ W $.
Then there exists a positive constant $C$,
depending only on $[W]_{\mathcal A_{p,u}}$, $M$, $\rho$, and $A$, such that,
for any $j\in\mathbb Z$ and $x\in\mathbb R^n$,
$$
b^j \int_{\mathbb{R}^n}
\frac{\|A_{j}(z)W^{-\frac1p}(z)\|^u}{[1 + b^{j} \rho(x - z)]^{M}} \,dz
\leq C,
$$
where $A_j$ is as in \eqref{Aj}.
\end{lemma}

\begin{proof}
Applying Lemmas \ref{yl111503}(iii), \ref{improve D},
and \ref{yl111505}(ii), we obtain
\begin{align*}
b^j \int_{\mathbb{R}^n}
\frac{\|A_{j}(z)W^{-\frac1p}(z)\|^u}{[1 + b^{j} \rho(x - z)]^{M}} \,dz
&\sim b^j\sum_{R\in\mathcal D_{j}} \int_{R}
\frac{\|A_RW^{-\frac1p}(z)\|^{u}}{[1 + b^{j} \rho(x -x_R)]^{M}} \,dz \\
&\lesssim b^j\sum_{R\in\mathcal D_{j}}
\frac{|R|}{[1 + \rho(A^jx-A^{j}x_R)]^{M}}\\
&=\sum_{k\in\mathbb Z^n}
\frac{1}{[1 +\rho(A^jx -k)]^{M}} \sim 1.
\end{align*}
This completes the proof of Lemma \ref{fsj}.
\end{proof}

We now prove the second equivalence of Theorem \ref{dl111301}.

\begin{lemma}\label{WsimA}
Let $\alpha\in\mathbb{R}$, $p\in(0,\infty)$,
$q\in(0,\,\infty]$,
and	$\varphi\in{\mathcal S}$
satisfy \eqref{hs2}. Assume that $W\in \mathcal A_{p,\infty}$
and ${\mathbb A}:=\{A_Q\}_{Q\in\mathcal{D}}$ is
a sequence of reducing operators of order $ p $ for $ W $.
Then,  for any $\vec f\in ({\mathcal S}'_\infty)^m$,
$$
\left\|\vec f\right\|_{{\dot{B}^{\alpha}_{p,q}(W,\varphi)}}
\sim \left\| \sup_{\mathbb{A}, \varphi} \left( \vec f\right)
\right\|_{\dot{b}^\alpha_{p,q}},
$$
where the positive equivalence constants are independent of $\vec f$.
\end{lemma}
\begin{proof}
We first prove $\|\vec f\|_{{\dot{B}^{\alpha}_{p,q}(W,\varphi)}}
\lesssim\| \sup_{\mathbb{A}, \varphi} (\vec f)
\|_{\dot{b}^\alpha_{p,q}}$.
For any $j\in\mathbb{Z}$, let $g_j$  be as in \eqref{gj} and
$
l_j :=| W^{\frac{1}{p}} ( \varphi_j * \vec f)|.
$
By Lemma \ref{reduceM},
we find that, for any $j\in\mathbb Z$ and $x\in\mathbb{R}^n$,
\begin{align*}
\int_{\mathbb{R}^n} \left| l_j(x)\right|^p\, dx
&\leq \sum_{Q \in \mathcal{D}_j}\int_{Q}
\left\| W^{\frac{1}{p}}(x) A_Q^{-1} \right\|^p
\left| A_Q \left( \varphi_j * \vec f \right)(x) \right|^p \,dx \\
&\leq \sum_{Q \in \mathcal{D}_j}|Q|\sup_{y\in Q}\left|
A_Q \left( \varphi_j * \vec f \right)(y) \right|^p\fint_{Q}
\left\| W^{\frac{1}{p}}(x) A_Q^{-1} \right\|^p\,dx\\
&\sim\sum_{Q \in \mathcal{D}_j} |Q|
\sup_{y\in Q}\left| A_Q \left( \varphi_j * \vec f \right)(y) \right|^p \\
&=\int_{\mathbb{R}^n} \left|\sum_{Q\in \mathcal{D}_j}\widetilde{\mathbf{1}}_Q
\sup_{\mathbb{A}, \varphi, Q}\left(\vec f \right)
\right|^p\,dx
=\int_{\mathbb{R}^n}\left|g_j(x)\right|^p\,dx.
\end{align*}
Therefore,
\begin{align*}
\left\| \vec f \right\|_{\dot B^{\alpha}_{p,q}(W,\varphi)}
=\left[\sum_{j\in\mathbb{Z}}\left\|b^{j\alpha} l_j
\right\|^{q}_{L^{p}}\right]^{\frac{1}{q}}
\lesssim\left[\sum_{j\in\mathbb{Z}}
\left\|b^{j\alpha} g_j\right\|^{q}_{L^{p}}
\right]^{\frac{1}{q}}
= \left\| \sup_{\mathbb{A}, \varphi}
\left( \vec f \right) \right\|_{\dot b^{\alpha}_{p,q}}.
\end{align*}

We next prove  the reverse inequality.
Applying Proposition \ref{improve}(i), we conclude that
$W\in \mathcal A_{p,u}$ for some $u\in(0,\infty)$.
By  Lemma \ref{tl012001}, we obtain $\mathbb{A}$
is doubling of order $(d_1,d_2; p)$ for some $d_1,d_2\in[0,\infty)$.
Let $r\in(0,[\max\{p,1+\frac{p}{u}\}]^{-1})$
and $M \in (\frac{1}{pr}+\frac{d_1+d_2}{p}, \infty) $.
From \eqref{kappa}, H\"older's inequality, and Lemma \ref{fsj} with $u$ and $M$ replaced by
$\frac{pr}{1-r}$ and $[Mp-(d_1+d_2)]r$ respectively,
we infer that,
for any $j\in\mathbb Z$ and $x\in\mathbb R^n$,
\begin{align*}
\left[\sum_{Q\in \mathcal{D}_j} \widetilde{\mathbf{1}}_Q(x)
\sup_{\mathbb{A}, \varphi, Q} \left( \vec f \right)
\right]^{pr}
&\lesssim b^j \int_{\mathbb{R}^n}
\frac{\|A_{j}(z)W^{-\frac1p}(z)\|^{pr}}{[1 + b^{j} \rho(x - z)]^{[Mp-(d_1+d_2)]r}}
\left| l_j (z) \right|^{pr} \, dz\notag\\
&\le \left[ b^j \int_{\mathbb{R}^n}
\frac{\|A_{j}(z)W^{-\frac1p}(z)\|^{\frac{pr}{1-r}}}{[1 + b^{j} \rho(x - z)]^{[Mp-(d_1+d_2)]r}} \,dz\right]^{1-r}\notag\\
&\quad\times
\left[b^j\int_{\mathbb{R}^n}\frac{1}{[1 + b^{j} \rho(x - z)]^{[Mp-(d_1+d_2)]r}}\left| l_j (z) \right|^{p} \, dz\right]^{r}\notag\\
&\lesssim \left[b^j\int_{\mathbb{R}^n}\frac{1}{[1 + b^{j} \rho(x - z)]^{[Mp-(d_1+d_2)]r}}\left| l_j (z) \right|^{p} \, dz\right]^{r}.
\end{align*}
Using this and Lemma \ref{summary B} with $M$
replaced by $[Mp-(d_1+d_2)]r$,
we obtain
$$
\left\| \sup_{\mathbb{A}, \varphi} \left( \vec f \right)
\right\|_{\dot{b}^\alpha_{p,q}}
=\left[\sum_{j\in\mathbb{Z}}
\left\|b^{j\alpha} g_j\right\|^{q}_{L^{p}}\right]^{\frac{1}{q}}
\lesssim\left[\sum_{j\in\mathbb{Z}}
\left\|b^{j\alpha} l_j\right\|^{q}_{L^{p}}
\right]^{\frac{1}{q}}
=\left\|\vec f\right\|_{{\dot{B}^{\alpha}_{p,q}(W,\varphi)}}.
$$
This completes the proof of Lemma \ref{WsimA}.
\end{proof}

Finally, the desired result of Theorem \ref{dl111301} follows immediately from Lemmas \ref{yl111601} and \ref{WsimA}.

\subsection{Relations Between $\dot{b}^\alpha_{p,q}(W)$
and $\dot{b}^\alpha_{p,q}(\mathbb{A})$}\label{equibb}

In this subsection, we introduce the
averaging matrix-weighted anisotropic Besov sequence space $\dot{b}^{\alpha}_{p,q}({\mathbb A})$
and establish its equivalence to $\dot b_{p,q}^{\alpha}(W)$.

\begin{definition}
Let $\alpha\in\mathbb{R}$, $p\in(0,\infty)$, and $q\in(0,\,\infty]$.
Assume that ${\mathbb A}:=\{A_Q\}_{Q\in\mathcal{D}}$ is
a sequence of positive definite matrices.
The \emph{averaging matrix-weighted anisotropic Besov sequence space}
$\dot{b}^{\alpha}_{p,q}(A,{\mathbb A})$ is defined
to be the set of all sequences
$\vec{s}:=\{\vec{s}_Q\}_{Q\in\mathcal{D}}$ in ${\mathbb C}^m$ such that
\begin{align*}
\left\|\vec{s}\right\|_{\dot{b}^{\alpha}_{p,q}(A,{\mathbb A})}
:=\left\|\left\{b^{j\alpha}
\left|A_j\vec{s}_j \right|\right\}_{j\in\mathbb Z}
\right\|_{\ell^q L^p}
<\infty,
\end{align*}
where $A_j$, $\vec s_j$, and $\|\cdot\|_{\ell^q L^p}$ are, respectively,
 as in \eqref{Aj}, \eqref{sj}, and \eqref{lqLp}.
In what follows, if there is no confusion,
we denote $\dot{b}^{\alpha}_{p,q}(A,{\mathbb A})$ simply by $\dot{b}^{\alpha}_{p,q}(\mathbb A)$.
\end{definition}

\begin{theorem}\label{dj}
Let $\alpha\in\mathbb{R}$, $p\in(0,\infty)$, $q\in(0,\infty]$.
Assume that $W\in \mathcal A_{p,\infty}$ and
${\mathbb A}:=\{A_Q\}_{Q\in\mathcal{D}}$
is a sequence of reducing operators of order $p$ for $W$.
Then $\dot{b}^{\alpha}_{p,q}(W)=\dot{b}^{\alpha}_{p,q}(\mathbb A)$.
Moreover, for any $\vec{s}\in\dot{b}^{\alpha}_{p,q}(W)$,
\begin{align*}
\left\|\vec{s}\right\|_{\dot{b}^{\alpha}_{p,q}(W)}
\sim\left\|\vec{s}\right\|_{\dot{b}^{\alpha}_{p,q}(\mathbb A)},
\end{align*}
where the positive equivalence constants are independent of $\vec{s}$.
\end{theorem}

\begin{proof}
By Lemma \ref{reduceM}, we find that, for any $j\in\mathbb Z$,
\begin{align*}
\left\|\,\left| W^{\frac1p}\vec s_j\right|\,\right\|_{L^p}^p
=\sum_{Q\in\mathcal{D}_j}|Q|^{1-\frac{p}{2}}\fint_Q
\left|W^{\frac{1}{p}}(x)\vec{s}_Q
\right|^pdx
\sim
\sum_{Q\in\mathcal{D}_j}|Q|^{1-\frac{p}{2}}
\left|A_Q\vec{s}_Q\right|^p
=\left\|\,\left| A_j\vec s_j\right|\,\right\|_{L^p}^p.
\end{align*}
Taking the $\frac1p$-th power on both sides and then applying the $\ell^q$ norm,
we obtain the desired result.
This completes the proof of Theorem \ref{dj}.
\end{proof}

\subsection{Proof of the $\varphi$-Transform Characterization}\label{PTC}

In this subsection, we prove the $\varphi$-transform characterization
of $\dot B^\alpha_{p,q}(W)$ (Theorem \ref{dl1103}).
To this end, we first show that $T_\psi$ is well defined.

\begin{lemma}\label{hhhk}
Let $\alpha\in{\mathbb R}$, $p\in(0,\infty)$,
$q\in(0,\,\infty]$, and $W\in \mathcal A_{p,\infty}$.
Then, for any $\vec{s}:=\{\vec s_Q\}_{Q\in\mathcal{D}}
\in\dot{b}^\alpha_{p,q}(W)$
and $\psi\in{\mathcal S}_\infty$,
$T_\psi \vec{s}=\sum_{Q\in\mathcal{D}} \vec{s}_Q\psi_Q$
converges in $(\mathcal{S}'_\infty)^m$.
Moreover, there exist $N\in(0,\infty)$ and a
positive constant $C$ such that,
for any $\vec s \in \dot{b}^\alpha_{p,q}(W) $
and $ \psi,\phi \in \mathcal{S}_\infty$,
\begin{align*}
\sum_{Q \in \mathcal{D}} \left|\vec{s}_Q\right| |\langle \psi_Q, \phi \rangle|
\leq C \left\| \vec s \right\|_{\dot{b}^\alpha_{p,q}(W)}\|\psi\|_{N}\|\phi\|_{N},
\end{align*}
where
\begin{align}\label{SM}
\|\phi\|_N:=\sup_{x\in\mathbb{R}^n}\sup_{|\gamma|\leq N}
\left(1 + |x|\right)^N|\partial^\gamma\phi(x)|.
\end{align}
\end{lemma}

\begin{proof}
Let $\mathbb A:=\{A_Q\}_{Q\in\mathcal{D}}$
be a sequence of reducing operators of order $p$ for $W$.
Assume that $W$ has $\mathcal A_{p,\infty}$-lower dimension $d_1\in[0,1)$
and $\mathcal A_{p,\infty}$-upper dimension $d_2\in[0,\infty)$.
By Lemma \ref{tl012001}, we find that, for any $Q\in\mathcal{D}$,
\begin{align}\label{bb}
\left\|A_Q^{-1}\right\|
\leq\left\|A_{Q_{0,\mathbf{0}}}^{-1}\right\|
\left\|A_{Q_{0,\mathbf{0}}}A_Q^{-1}\right\|
\lesssim
\max\left\{|Q|^{\frac {d_1}{p}},|Q|^{-\frac{d_2}{p}}\right\}
\left[1+\frac{\rho(x_Q)}{\max\{1,|Q|\}}\right]^{\frac{d_1+d_2}{p}}.
\end{align}
Using  Theorem \ref{dj}, we conclude that, for any $Q\in\mathcal D$,
\begin{align}\label{AQSQ}
\left|A_{Q}\vec{s}_{Q}\right|
\le |Q|^{\frac{1}{2}+\alpha-\frac{1}{p}}\left\|\vec{s}
\right\|_{\dot{b}^{\alpha}_{p,q}(\mathbb A)}
\sim |Q|^{\frac{1}{2}+\alpha-\frac{1}{p}}
\left\|\vec{s}\right\|_{\dot{b}^{\alpha}_{p,q}(W)}.
\end{align}
Let
\begin{align*}
L>\max\left\{1+\frac{d_1+d_2}{p}, \frac{1+d_2}{p}+\frac{1}{2}-
\alpha, \frac{d_1-1}{p}+\frac{3}{2}+\alpha
\right\}.
\end{align*}
Applying \cite[(3.18)]{gjabownik07}, we obtain,
for any $Q\in\mathcal{D}$,
\begin{align*}
|\langle\psi_Q,\phi\rangle|
=|\langle\psi_Q,\phi_{Q_{0,\mathbf{0}}}\rangle|
\lesssim\|\psi\|_N\|\phi\|_N
\left[1+\frac{\rho(x_Q)}{1 \vee |Q|}\right]^{-L}
\min\left\{|Q|,|Q|^{-1}\right\}^L.
\end{align*}
From this, \eqref{bb},  \eqref{AQSQ},
and Lemma \ref{yl111505}(ii),
we deduce that
\begin{align*}
\sum_{Q\in\mathcal{D}}
\left|\vec{s}_Q\right||\langle\psi_Q,\phi\rangle|
&\le\sum_{Q\in\mathcal{D}}
\left\|A_{Q}^{-1}\right\|\left|A_{Q}\vec{s}_{Q}\right| |\langle\psi_Q,\phi\rangle|\\
&\lesssim\sum_{Q\in\mathcal{D}}
\max\left\{|Q|^{\frac {d_1}{p}},|Q|^{-\frac{d_2}{p}}\right\}
|Q|^{\frac{1}{2}+\alpha-\frac{1}{p}}
\left\|\vec{s}\right\|_{{\dot{b}^{\alpha}_{p,q}(W)}}
\|\psi\|_N\|\phi\|_N\\
&\quad\times\left[1+\frac{\rho(x_Q)}{1 \vee |Q|}\right]^{\frac{d_1+d_2}{p}-L}
\min\left\{|Q|^{L},|Q|^{-L}\right\}\\
&\lesssim\left\|\vec{s}\right\|_{{\dot{b}^{\alpha}_{p,q}(W)}}
\|\psi\|_N\|\phi\|_N
\left\{\sum_{j=0}^{\infty}b^{-j(\frac12+\alpha+L-\frac{1+d_2}{p})}
\sum_{k\in \mathbb Z^n}\left[1+\rho\left(A^{-j}k \right)
\right]^{\frac{d_1+d_2}{p}-L}\right.\\
&\quad\left.+\sum_{j=-\infty}^{-1}b^{-j(\frac12+\alpha-L+\frac{d_1-1}{p})}
\sum_{k\in \mathbb Z^n}
\left[1+\rho\left(k\right)
\right]^{\frac{d_1+d_2}{p}-L}\right\}\\
&\lesssim\left\|\vec{s}\right\|_{{\dot{b}^{\alpha}_{p,q}(W)}}
\|\psi\|_N\|\phi\|_N
\left[\sum_{j=0}^{\infty}b^{-j(\frac12+\alpha+L-\frac{1+d_2}{p}-1)}
+\sum_{j=-\infty}^{-1}b^{-j(\frac12+\alpha-L+\frac{d_1-1}{p}+1)}
\right]\\
&\sim\left\|\vec{s}\right\|_{{\dot{b}^{\alpha}_{p,q}(W)}}
\|\psi\|_N\|\phi\|_N.
\end{align*}
This completes the proof of Lemma \ref{hhhk}.
\end{proof}

For any sequence
$s:=\{s_P\}_{P\in\mathcal{D}}$ in ${\mathbb C}$
and $\lambda\in(0,\infty)$,
let $s^*_{\lambda}:=\{(s^*_{\lambda})_Q\}_{Q\in\mathcal{D}}$,
where, for any $Q\in\mathcal{D}$,
\begin{align*}
\left(s^*_{\lambda}\right)_Q:=
\sum_{\genfrac{}{}{0pt}{}{P\in\mathcal{D}}{|P|=|Q|}}
\frac{|s_P|}{[1+|P|^{-1}
\rho(x_P-x_Q)]^\lambda}.
\end{align*}



\begin{lemma}\label{dj3}
Let $\alpha\in{\mathbb R}$, $p\in(0,\infty)$, $ q\in(0,\infty]$,
and $\lambda\in(\frac1{1\wedge p},\infty)$. Then, for any
$s:=\{s_Q\}_{Q\in\mathcal{D}}$ in $\mathbb{C}$,
$
\|s\|_{\dot{b}^{\alpha}_{p,q}}
\sim \|s^*_{\lambda}\|_{\dot{b}^{\alpha}_{p,q}},
$
where the positive equivalence constants are independent of $s$.
\end{lemma}

\begin{proof}
By the definition of $s^*_{\lambda}$, we obtain,
for any $Q\in\mathcal{D}$,
$|s_Q|\leq(s^*_\lambda)_Q$, and hence
$\|s\|_{\dot{b}^{\alpha}_{p,q}}
\leq\|s^*_\lambda\|_{\dot{b}^{\alpha}_{p,q}}.$
The reverse inequality follows from \cite[Lemma 3.4]{gjabownik05}.
Indeed, although Bownik only stated the existence of
$\lambda\in(0,\infty)$ in \cite[Lemma 3.4]{gjabownik05},
his proof actually shows that the inequality holds for all $\lambda\in(\frac1{1\wedge p},\infty)$.
This completes the proof of Lemma \ref{dj3}.
\end{proof}

\begin{lemma}\label{ddjj}
Let $\alpha\in{\mathbb R}$,
$p\in(0,\infty)$,
$q\in(0,\infty]$, and
$\lambda\in(\frac1{1\wedge p},\infty)$.
Assume that $W\in \mathcal A_{p,\infty}$ and
$\{A_Q\}_{Q\in\mathcal{D}}$ is a sequence
of reducing operators of order $p$ for $W$.
Then, for any $\vec{s}\in{\dot{b}^{\alpha}_{p,q}(W)}$,
\begin{align*}
\left\|\left(\left\{\left|A_Q\vec{s}_Q\right|\right\}_{Q\in\mathcal{D}}
\right)^*_{\lambda}\right\|_{\dot{b}^{\alpha}_{p,q}}
\sim\left\|\vec{s}\right\|_{\dot{b}^{\alpha}_{p,q}(W)},
\end{align*}
where the positive equivalence constants are independent of $\vec{s}$.
\end{lemma}

\begin{proof}
By Lemma \ref{dj3} and the definitions of
$\dot{b}^{\alpha}_{p,q}$ and $\dot{b}^{\alpha}_{p,q}(W)$, we find that
\begin{align*}
\left\|\left(\left\{\left|
A_Q\vec{s}_Q\right|\right\}_{Q\in\mathcal{D}}\right)^*_{\lambda}
\right\|_{\dot{b}^{\alpha}_{p,q}}
\sim
\left\|\left\{\left|A_Q\vec{s}_Q\right|\right\}_{Q\in\mathcal D}\right\|_{\dot{b}^{\alpha}_{p,q}}
=
\left\|\vec{s}\right\|_{\dot{b}^{\alpha}_{p,q}({\mathbb A})}
\sim
\left\|\vec{s}\right\|_{\dot{b}^{\alpha}_{p,q}(W)},
\end{align*}
where the last equivalence follows from Theorem \ref{dj}.
This completes the proof of Lemma \ref{ddjj}.
\end{proof}

We now prove Theorem \ref{dl1103}.

\begin{proof}[Proof of Theorem \ref{dl1103}]
We first prove (i).
Let $\mathbb A:=\{A_Q\}_{Q\in\mathcal{D}}$
be a sequence of reducing operators of order $p$ for $W$.
Assume that $\vec{f}\in\dot{B}^{\alpha}_{p,q}(W,\widetilde{\varphi})$.
From the definition of $S_{\varphi}$, \cite[Theorem 2.3.20]{hsdyk01}, and \eqref{sup},
we infer that, for any $Q\in\mathcal{D}$,
\begin{align*}
\left|A_Q\left(S_\varphi\vec{f}\right)_Q\right|
=\left|A_Q\left\langle \vec f,\varphi_Q \right\rangle\right|
=|Q|^{\frac{1}{2}}
\left|A_Q\left(\widetilde{\varphi}_{j_Q}
\ast\vec{f}\right)(x_Q)\right|
\leq\sup_{\mathbb{A},\widetilde{\varphi},Q}\left(\vec f \right),
\end{align*}
which, together with Theorems \ref{dj} and
\ref{dl111301}, further implies that
\begin{align*}
\left\| S_{\varphi} \vec f \right\|_{\dot{b}^\alpha_{p,q}(W)}
\sim \left\| S_{\varphi}
\vec f \right\|_{\dot{b}^\alpha_{p,q}(\mathbb{A})}
\leq \left\| \sup_{\mathbb{A},\widetilde{\varphi}}
\left( \vec f \right) \right\|_{\dot{b}^\alpha_{p,q}}
\sim \left\| \vec f \right\|_{\dot{B}^{\alpha}_{p,q}
(W,\widetilde{\varphi})}.
\end{align*}
This completes the proof of the boundedness of $S_{\varphi}$.

Next, we prove the boundedness of
$T_\psi$.
Let $\vec s := \{\vec s_Q\}_{Q \in \mathcal{D}}
\in \dot b^{\alpha}_{p,q}(W)  $.
By \eqref{hs2}, we find that there exists $l\in\mathbb N$ such that,
for any $i,j\in\mathbb Z$ with $|i-j|>l$,
$\operatorname{supp} \widehat{\varphi}_j\cap\operatorname{supp} \widehat{\psi}_i=\emptyset$.
Using this and Lemma \ref{hhhk}, we conclude that,
for any $j\in\mathbb{Z}$, $Q\in\mathcal{D}_j$, and $x\in Q$,
\begin{align} \label{71}
\left|A_Q\left[\varphi_j *\left(T_\psi \vec s \right)\right](x)\right|
&=\left|A_Q\left[\varphi_j *\left(\sum_{R \in \mathcal{D}}
\vec s_R\psi_R\right)\right](x)\right|\notag\\
&\le \sum_{i \in \mathbb{Z}} \sum_{R \in \mathcal{D}_i}
\left|A_Q \vec s_R\right|\left|\left(\varphi_j *\psi_R\right)(x) \right| \notag\\
&= \sum_{i = j - l}^{j + l} \sum_{R \in \mathcal{D}_i}
\left|A_Q \vec s_R\right| \left|\left(\varphi_j *\psi_R\right)(x)\right|\notag\\
&\leq \sum_{i=j-l}^{j+l}\sum_{R \in \mathcal{D}_i}
\left\|A_Q A_R^{-1}\right\|
\left| A_R \vec s_R \right|
\left| \left( \varphi_j * \psi_R \right) (x) \right|.
\end{align}
Let $d_{1}\in[\![d_{p,\infty}^{\mathrm{lower}}(W),\infty)$
and $d_{2}\in[\![d_{p,\infty}^{\mathrm{upper}}(W),\infty)$.
Applying Lemma \ref{tl012001}, we obtain,
for any $i,j\in\mathbb{Z}$ with $|i-j|\le l$,
$Q\in\mathcal{D}_j $, and  $R\in\mathcal{D}_i$,
\begin{align*}
\left\|A_Q A_R^{-1} \right\|
&\lesssim \max \left\{\left(\frac{|R|}{|Q|}
\right)^\frac{d_1}{p},\left(\frac{|Q|}{|R|}
\right)^\frac{d_2}{p} \right\}
\left[1+\frac{\rho(x_Q-x_R)}{\max\{|R|
, |Q|\}}\right]^{\frac{d_1+d_2}{p}}\notag\\
&\sim \left[ 1+|R|^{-1}\rho(x_Q-x_R)\right]^{\frac{d_1+d_2}{p}}.
\end{align*}
Let $\lambda \in(\frac{d_1+d_2}{p}+\max\{1,\frac1p\},\infty)$.
It follows from the proof of \cite[Theorem 3.5]{gjabownik06} that,
for any $i,j\in \mathbb{Z} $ with $|i-j|\le l$,
any $R\in \mathcal{D}_i $, and $ x \in \mathbb{R}^n$,
\begin{align*}
\left| \left( \varphi_j * \psi_R \right) (x) \right|
\lesssim |R|^{-\frac{1}{2}}
\frac{1}{[1+b^{i}\rho(x - x_R)]^{\lambda}}
\sim |R|^{-\frac{1}{2}} \frac{1}{[1+|R|^{-1}\rho(x-x_R)]^{\lambda}}.
\end{align*}

Let $ u := \{u_R\}_{R \in \mathcal{D}} $,
where $ u_R := |A_R \vec s_R| $.
Substituting the estimates of $\|A_Q A_R^{-1} \|$
and $| ( \varphi_j * \psi_R ) (x) |$ into \eqref{71}
and applying  Lemma \ref{yl111503}(ii), we find that,
for any $ j \in \mathbb{Z} $, $ Q \in \mathcal{D}_j $, and $ x \in Q $,
\begin{align*}
\left| A_Q \left[ \varphi_j * \left( T_\psi \vec s \right) \right] (x) \right|
&\lesssim\sum_{i=j-l}^{j + l}\sum_{R \in \mathcal{D}_i}
u_R|R|^{-\frac{1}{2}}
\frac{[1+|R|^{-1}\rho(x_Q-x_R)]^{\frac{d_1+d_2}{p}}}
{[1+|R|^{-1}\rho(x - x_R)]^{\lambda}} \\
&\sim |Q|^{-\frac{1}{2}} \sum_{i=j-l}^{j + l}
\sum_{R \in \mathcal{D}_i}
\frac{u_R}{[1+|R|^{-1}\rho(x-x_R)]^{\widetilde\lambda}} \\
&= \sum_{i=j-l}^{j + l}
\left( u_{\widetilde{\lambda}}^* \right)_i(x)
=\sum_{i=-l}^{l}
\left( u_{\widetilde{\lambda}}^* \right)_{j+i}(x),
\end{align*}
where $\widetilde{\lambda} := \lambda-\frac{d_1+d_2}{p}$ and
$$
\left( u_{\widetilde{\lambda}}^*\right)_{i}
:=\sum_{Q\in\mathcal{D}_i}\widetilde{\mathbf{1}}_Q \left( u_{ \widetilde{\lambda}}^*\right)_{Q}.
$$
Therefore,
\begin{align*}
\left\| T_\psi \vec s \right\|_{\dot B^{\alpha}_{p,q}(\mathbb{A},\varphi)}
&\lesssim\left[\sum_{j\in\mathbb{Z}}b^{j\alpha q}
\left\|\sum_{i=-l}^{l} \left( u_{\widetilde{\lambda}}^*\right)_{j+i}\right\|^{q}_{L^{p}}
\right]^{\frac{1}{q}} \\
&\sim \sum_{i=-l}^{l} b^{-i\alpha}
\left[\sum_{j\in\mathbb{Z}}b^{(j+i)\alpha q}
\left\| \left( u_{\widetilde{\lambda}}^*\right)_{j+i}\right\|^{q}_{L^{p}}
\right]^{\frac{1}{q}}
\sim\left\| u_{\widetilde{\lambda}}^*\right\|_{\dot b^{\alpha}_{p,q}}.
\end{align*}
From this, Theorem \ref{dl111301}, and Lemma \ref{ddjj},
it follows that
$$
\left\|T_\psi \vec s \right\|_{\dot B^{\alpha}_{p,q}(W,\varphi)}
\sim \left\|T_\psi \vec s \right\|_{
\dot B^{\alpha}_{p,q}(\mathbb{A},\varphi)}
\lesssim\left\|u_{\widetilde{\lambda}}^*
\right\|_{\dot b^{\alpha}_{p,q}}
\sim \left\|\vec s \right\|_{\dot b^{\alpha}_{p,q}(W)}.
$$
This completes the proof of the boundedness of $ T_\psi $
and hence (i).

Statement (ii) follows from \eqref{reproduce}.

Finally, we prove (iii).
Let $ \varphi^{(1)}, \varphi^{(2)}, \psi^{(2)}\in\mathcal{S}$
satisfy \eqref{hs2}.
Assume that $\varphi^{(2)},\psi^{(2)}$ satisfy \eqref{hs3}.
Then, from  Lemma \ref{repro}(ii) and the just proven (i), we deduce that,
for any $ \vec{f} \in \dot B^{\alpha}_{p,q}(W,\varphi^{(2)}) $,
\begin{align*}
\left\| \vec f \right\|_{\dot B^{\alpha}_{p,q}(W,\varphi^{(1)})}
&= \left\| \left( T_{\widetilde{\psi^{(2)}}}
\circ S_{\widetilde{\varphi^{(2)}}} \right)
\left(\vec f \right) \right\|_{\dot B^{\alpha}_{p,q}(W,\varphi^{(1)})} \\
&\lesssim \left\| S_{\widetilde{\varphi^{(2)}}}
\vec f \right\|_{\dot b^{\alpha}_{p,q}(W)}
\lesssim \left\| \vec f \right\|_{\dot B^{\alpha}_{p,q}(W, \varphi^{(2)})}.
\end{align*}
By symmetry, we obtain the reverse inequality.
This completes the proof of Theorem \ref{dl1103}.
\end{proof}

Based on Theorem \ref{dl1103}(iii), in what follows,
we denote $\dot B^{\alpha}_{p,q}(W,\varphi)$
simply by $\dot B^{\alpha}_{p,q}(W)$.

The following  proposition shows that $\dot{B}_{p,q}^{\alpha}(W)$
is continuously embedded into  $(\mathcal S'_{\infty})^m$.

\begin{proposition} \label{174}
Let $ \alpha \in \mathbb{R} $,
$ p\in(0, \infty) $, $q\in(0, \infty]$,
and $ W\in\mathcal A_{p,\infty}$.
Then  there exist $N\in(0,\infty)$ and
a positive constant $C$ such that,
for any $ \vec f \in\dot B^{\alpha}_{p,q}(W)$ and
$ \phi \in \mathcal{S}_\infty$,
$
|\langle \vec f, \phi \rangle|
\leq C \| \vec f \|_{\dot B^{\alpha}_{p,q}(W)}
\| \phi \|_{N},
$
where $\|\cdot\|_{N}$ is as in \eqref{SM}.
\end{proposition}

\begin{proof}
Let $ \varphi, \psi \in \mathcal{S}$
satisfy \eqref{hs2} and \eqref{hs3}.
By  Lemma \ref{hhhk} and Theorem \ref{dl1103}, we find that,
for any $ \vec f \in \dot B^{\alpha}_{p,q}(W) $
and $ \phi \in \mathcal{S}_\infty$,
\begin{align*}
\left|\left\langle \vec f, \phi \right\rangle\right|
&= \left|\left\langle \left( T_\psi \circ S_\varphi \right)
\vec f, \phi \right\rangle\right|
\leq \sum_{Q \in \mathcal{D}} \left|
\left(S_\varphi \vec f \right)_Q \right|
|\langle \psi_Q, \phi \rangle| \\
&\lesssim \left\| S_\varphi \vec f \right\|_{\dot b^{\alpha}_{p,q}(W)}
\left\| \phi \right\|_{N}
\lesssim \left\| \vec f \right\|_{\dot B^{\alpha}_{p,q}(W)}
\left\| \phi \right\|_{N}.
\end{align*}
This completes the proof of Proposition \ref{174}.
\end{proof}

Applying Proposition \ref{174} with
an argument similar to that used in the proof of
\cite[Proposition 2.3.1]{g14},
we obtain the following conclusion; we omit the details.

\begin{proposition}
Let $\alpha \in \mathbb{R}$,
$p \in (0, \infty) $, $ q \in (0, \infty]$,
and $W \in \mathcal A_{p,\infty}$.
Then $\dot B^{\alpha}_{p,q}(W) $ is a complete quasi-normed space.
\end{proposition}

\section{Almost Diagonal Operators}
\label{jh dj}

In this section, we focus on the boundedness  and sharpness of almost diagonal operators on
$\dot b_{p,q}^{\alpha}(W)$.
Let $ B:=\{ b_{Q,R}\}_{Q,R\in\mathcal{D}}$ in $\mathbb{C}$.
For any sequence $ \vec s := \{ \vec s_R \}_{R \in \mathcal{D}}$
in $ \mathbb{C}^m $,  define $ B\vec s:= \{(B\vec s)_Q\}_{Q\in\mathcal{D}}$ by setting,
for any $Q\in\mathcal{D}$,
$\left(B\vec s\right)_Q :=\sum_{R\in \mathcal{D}}b_{Q,R}\vec s_R$
if this series converges absolutely.

We now recall the concept of almost diagonal operators.
Unlike the traditional notation in \cite{FJ90,gjabownik06},
we use a new equivalent definition from \cite[Definition 4.1]{bf6}.

\begin{definition}\label{almdia}
Let $D,E,F\in\mathbb{R}$. The special infinite matrix
$B^{D,E,F}:=\{b_{Q,R}^{D,E,F}\}_{Q,R\in\mathcal{D}}$
is defined by setting, for any $Q,R\in\mathcal{D}$,
\begin{align}\label{BDEF}
b_{Q,R}^{D,E,F}:=\left[1+\frac{\rho(x_Q-x_R)}{|Q|\vee|R|}\right]^{-D}
\begin{cases}\displaystyle
\left(\frac{|Q|}{|R|}\right)^E & \text{if }|Q|\leq|R|, \\
\displaystyle
\left(\frac{|R|}{|Q|}\right)^F & \text{if }|R|<|Q|.
\end{cases}
\end{align}
An infinite matrix $ B:=\{b_{Q, R}\}_{Q,R\in\mathcal{D}}$ in $\mathbb{C} $
is said to be \emph{$(D,E,F)$-almost diagonal}
if there exists a positive constant $ C $ such that, for any $Q,R\in\mathcal D$, $|b_{Q,R}|\le C b_{Q,R}^{D,E,F}$.
\end{definition}

The remainder of this section is organized as follows.
In Subsection \ref{BADO}, we prove the boundedness of
almost diagonal operators on $\dot b_{p,q}^{\alpha}(W)$.
Moreover, we also prove that the composition of two almost diagonal operators
remains an almost diagonal operator.
In Subsection \ref{SADO}, we prove the sharpness of almost diagonal conditions.
In Subsection \ref{compare}, we compare the results obtained in
Subsection \ref{BADO} with existing ones.

\subsection{Boundedness of Almost Diagonal Operators}
\label{BADO}

In this subsection, we establish the boundedness of almost diagonal operators.
The following theorem is the main result of this subsection.

\begin{theorem}\label{ad FJ-YY}
Let $\alpha\in\mathbb{R}$,
$p\in(0,\infty)$, $q\in(0,\infty]$, and
$W\in\mathcal A_{p,\infty}$.  Assume that
$B:=\{b_{Q, R}\}_{Q, R\in\mathcal{D}}$ in $\mathbb{C}$
is $(D,E,F)$-almost diagonal with parameters
\begin{align}\label{ad FJ}
D>J,\
E>\frac{1}{2}+\alpha, \text{ and }
F>J-\frac{1}{2}-\alpha,
\end{align}
where
\begin{equation}\label{J}
J
:= \frac{1}{1\wedge p}+\mathrm{AT}(p,W)
:= \frac{1}{1\wedge p} + \min\left\{\frac{d_{p,\infty}^{\mathrm{upper}}(W)}{p},
\left( \frac1{v_{W,p}}-\frac1{1\wedge p} \right)_+\right\}
\end{equation}
with $d_{p,\infty}^{\mathrm{upper}}(W)$ and $v_{W,p}$
as, respectively, in Definition \ref{lower} and \eqref{spW}.
Then $B\vec s$ is well defined for all $\vec s \in\dot{b}_{p,q}^{\alpha}(W)$,
and $B$ is bounded on $\dot b^\alpha_{p,q}(W)$.
\end{theorem}

\begin{remark}
\begin{enumerate}[\rm(i)]
\item When $W$ is a power weight, condition \eqref{ad FJ} is sharp
(see Theorem \ref{ad Besov sharp}).

\item When $p\in(0,1]$, we have
$\mathrm{AT}(p,W)=\frac{d_{p,\infty}^{\mathrm{upper}}(W)}{p}$,
whereas, when $p\in(1,\infty)$, the term $\mathrm{AT}(p,W)$
cannot be simplified further, since the minimum in its definition
may equal either of its two arguments;
see Remark \ref{rem1} for further details.
\end{enumerate}
\end{remark}

To prove Theorem \ref{ad FJ-YY}, we need several technical lemmas.
The following lemma allows us to restrict consideration to $\alpha=0$
when studying the boundedness of almost diagonal operators.

\begin{lemma}\label{alpha0}
Let $\alpha\in\mathbb{R}$,
$p\in(0,\infty)$, $q\in(0,\infty]$, and
$W\in\mathcal A_{p,\infty}$.
Let $D,E,F\in\mathbb R$,
$B:=\{b_{Q, R}\}_{Q,R\in\mathcal{D}}$ in $\mathbb{C}$,
and $\widetilde B:=\{\widetilde b_{Q,R}\}_{Q,R\in\mathcal{D}}$,
where, for any $Q$, $R\in\mathcal{D}$,
$ \widetilde{b}_{Q, R}:=(|R|/|Q|)^\alpha b_{Q, R}$.
Then
\begin{enumerate}[\rm(i)]
\item $B$ is $(D,E,F)$-almost diagonal if and only if
$\widetilde B$ is $(D,E-\alpha,F+\alpha)$-almost diagonal;

\item $B$ is bounded on $\dot b^{\alpha}_{p,q}(W)$ if and only if
$\widetilde B$ is bounded on $\dot b^{0}_{p,q}(W)$.
\end{enumerate}
\end{lemma}

\begin{proof}
From Definition \ref{almdia}, it is simple to see that (i) holds.
To prove (ii), for any $R\in\mathcal{D}$, define $(J_\alpha\vec{s})_R:=|R|^{-\alpha}\vec{s}_R$.
It is simple to check that $J_\alpha:\dot b^{\alpha}_{p,q}(W)\to\dot b^{0}_{p,q}(W)$
is an isometric isomorphism. Moreover,
a direct computation gives that,
for any $\vec s\in \dot b^{\alpha}_{p,q}(W)$,
$B\vec{s}=J_{\alpha}^{-1}\widetilde{B}J_\alpha\vec{s}$.
From these observations, we infer that (ii) holds.
This completes the proof of Lemma \ref{alpha0}.
\end{proof}

\begin{lemma}\label{yl022301}
Let $i,l\in\mathbb{Z}$ satisfy $l+i\ge0$ and $a_1,a_2\in(0,\infty)$.
Then there exists a positive constant $C$,
depending only on $\rho$, $A$, $a_1$, and $a_2$, such that,
for any $\vec s:=\{\vec s_Q\}_{Q\in\mathcal D}$ in $\mathbb C^m$, $x\in\mathbb{R}^n$,
and $U\in M_m(\mathbb{C})$,
\begin{align*}
\left[\fint_{B_{\rho}(x,b^{l})}\left|U\vec s_i(y)\right|^{a_1}\,dy\right]^{\frac1{a_1}}
\leq C b^{(l+i)(\frac{1}{a_2}-\frac{1}{a_1})_+}
\left[\fint_{B_{\rho}(x,C_3b^{l})}
\left|U\vec s_i(y)\right|^{a_2}\,dy
\right]^{\frac{1}{a_2}},
\end{align*}
where $\vec s_i$ is as in \eqref{sj} and $C_3:=H(1+C_0)$ with $C_0$ as in Lemma \ref{yl1101}.
\end{lemma}

\begin{proof}
If $a_1\leq a_2$, then from  H\"older's inequality, it follows that
\begin{align*}
\left[\fint_{B_{\rho}(x,b^{l})}\left|U\vec s_i(y)\right|^{a_1}\,dy\right]^{\frac1{a_1}}
\leq \left[\fint_{B_{\rho}(x,b^{l})}\left|U\vec s_i(y)\right|^{a_2}\,dy\right]^{\frac1{a_2}}
\lesssim \left[\fint_{B_{\rho}(x,C_3b^{l})}
\left|U\vec s_i(y)\right|^{a_2}\,dy
\right]^{\frac{1}{a_2}}.
\end{align*}
This completes the proof of Lemma \ref{yl022301} in this case.

Now, we consider the case where $a_2 < a_1$.
For any $x\in\mathbb R^n$,
we first claim that
\begin{align}\label{4.5.1}
B_{\rho}\left( x,b^{l}\right) \subset
\bigcup_{\genfrac{}{}{0pt}{}{Q\in\mathcal{D}_i}{Q\cap B_\rho(x,b^{l})\neq\emptyset}}Q
\subset  B_{\rho}\left( x, C_3b^{l}\right).
\end{align}
The first inclusion is obvious.
To prove the second inclusion, it suffices to show that, for any
$Q\in\mathcal D_i$ with $Q\cap B_{\rho}(x,b^{l})\neq\emptyset$, $Q\subset B_{\rho}(x,C_3b^{l})$.
Fix $z\in Q\cap B_{\rho}(x,b^{l})$.
Applying Lemma \ref{yl1101}(ii) and $l+i\geq0$,
we obtain, for any $y\in Q$,
\begin{align*}
\rho(x-y)\le H[\rho(x-z)+\rho(z-y)]\le
H\left(b^{l}+C_0b^{-i}\right)
\le C_3 b^{l},
\end{align*}
and hence $Q\subset B_{\rho}(x,C_3b^{l})$. This proves \eqref{4.5.1}.
Using this and Lemma \ref{equivalent}, we conclude that
\begin{align*}
\left[\fint_{B_{\rho}(x,b^{l})}\left|U\vec s_i(y)\right|^{a_1}\,dy\right]^{\frac1{a_1}}
&\lesssim b^{\frac i2-\frac{i+l}{a_1}} \left[
\sum_{\genfrac{}{}{0pt}{}{Q\in\mathcal{D}_i}{Q\cap B_{\rho}(x,b^{l})\neq\emptyset}}
\left|U\vec s_Q\right|^{a_1} \right]^{\frac1{a_1}}
\leq b^{\frac i2-\frac{i+l}{a_1}}
\left[\sum_{\genfrac{}{}{0pt}{}{Q\in\mathcal{D}_i}{Q\cap B_{\rho}(x,b^{l})\neq\emptyset}}
\left|U\vec s_Q\right|^{a_2}
\right]^\frac{1}{a_2}
\\
&\lesssim b^{-\frac{i+l}{a_1}} \left[ b^{i+l} \fint_{B_{\rho}(x,C_3b^{l})}
\left|U\vec s_i(y)\right|^{a_2}\,dy \right]^\frac{1}{a_2} \\
&= b^{(l+i)(\frac 1{a_2}-\frac 1{a_1})}
\left[ \fint_{B_{\rho}(x,C_3b^{l})}
\left|U\vec s_i(y)\right|^{a_2} \,dy \right]^\frac{1}{a_2}.
\end{align*}
This completes the proof of Lemma \ref{yl022301}.
\end{proof}

The following lemma generalizes \cite[Lemma 4.8]{bf2} to the anisotropic setting.

\begin{lemma}\label{ad prelim}
Let $p\in(0,\infty)$, $q\in(0,\infty]$, $a\in (0,1]$,
and $B$ be $(D,E,F)$-almost diagonal for some $D,E,F\in\mathbb R$.
Then there exists a positive constant $ C $ such that,
for any $\vec{s}:=\{\vec{s}_R\}_{R\in\mathcal D}$ in $\mathbb C^m$
satisfying that $B\vec s$ is well defined
and, for any sequence $\{H_j :\ \mathbb{R}^n
\to M_m(\mathbb{C})\}_{j \in \mathbb{Z}}$
of locally integrable matrix-valued functions, one has
\begin{align}\label{ABt}
\left\|\left\{H_j(B\vec s)_j\right\}_{j\in\mathbb Z} \right\|_{\ell^qL^p}^r
&\le C \sum_{k\in\mathbb{Z}}
\sum_{l=0}^\infty \left[ b^{-(E-\frac{1}{2}) k_-}
b^{-k_+(F+\frac{1}{2}-\frac{1}{a})}
b^{-(D-\frac{1}{a})l}\right]^r\notag \\
&\quad\times\left\|\left\{\left[\fint_{B_{\rho}(\cdot, b^{l+k_+-i})}
\left| H_{i-k}(\cdot)
\vec{s}_i(y) \right|^a\,dy
\right]^{\frac{1}{a}} \right\}_{i\in\mathbb Z} \right\|_{\ell^qL^p}^r,
\end{align}
where $r:=\min\{1,p,q\}$ and $\|\cdot\|_{\ell^q L^p}$ and $\vec s_i$
are as, respectively, in \eqref{lqLp} and \eqref{sj}.
\end{lemma}

\begin{proof}
Let $j\in\mathbb{Z}$ and $x\in Q\in\mathcal{D}_j$.
Then, from Lemma \ref{yl111503}(ii), we infer that
\begin{align}\label{ad step 1}
\left| H_j(x)(B\vec s )_j(x) \right|
&= b^{\frac j2} \left| H_j(x)( B\vec{s} )_Q\right|
\le b^{\frac j2} \sum_{R\in\mathcal{D}} |b_{Q,R}| \left|H_j(x)\vec{s}_R\right|\notag\\ &\lesssim\sum_{i\in\mathbb{Z}}
b^{-(j-i)_+E} b^{-(i-j)_+ F} b^{\frac j2}
\sum_{R\in\mathcal{D}_i}
\left[ 1+b^{i\wedge j}\rho(x-x_R) \right]^{-D} \left| H_j(x)\vec{s}_R\right|.
\end{align}
Using Lemma \ref{yl111503}(ii) again, we conclude that
\begin{align*}
&\sum_{R\in\mathcal{D}_i}
\left[ 1+b^{i\wedge j}\rho(x-x_R) \right]^{-D} \left| H_j(x)\vec{s}_R\right|\\
&\quad\sim\sum_{R\in\mathcal {D}_i} b^{\frac{i}{2}}\int_R
\left[1+b^{i\wedge j}\rho(x-y)
\right] ^{-D} |R|^{-\frac{1}{2}}\left| H_j(x)\vec{s}_R\right| \, dy \\
&\quad= b^{\frac{i}{2}}\int_{\mathbb{R}^n}
\left[1+b^{i\wedge j}\rho(x-y)
\right] ^{-D} \left| H_j(x)\vec s_i(y)\right|\, dy\\
&\quad\lesssim b^{\frac{i}{2}}\int_{B_{\rho}(x,b^{-(i\wedge j)})}
\left| H_j(x)\vec s_i(y)\right|\,dy \\
&\qquad+b^{\frac{i}{2}}\sum_{l=1}^\infty\int_{B_{\rho}(x,b^{l-(i\wedge j)})
\setminus B_{\rho}(x,b^{l-1-(i\wedge j)})}
b^{-Dl} \left| H_j(x)\vec s_i(y)\right|\,dy \\
&\quad\lesssim b^{\frac{i}{2}}\sum_{l=0}^\infty b^{-Dl}
b^{l-(i\wedge j)}
\fint_{B_{\rho}(x,b^{l-(i\wedge j)})} \left| H_j(x)\vec s_i(y)\right|\, dy,
\end{align*}
which, combined with \eqref{ad step 1}, further implies that
\begin{align}\label{Bsj}
\left|H_{j}(x)(B\vec s)_j(x)\right|
&\lesssim\sum_{i\in\mathbb{Z}}b^{-(j-i)_+E} b^{-(i-j)_+ F}
b^{\frac{i+j}{2}} b^{-(i\wedge j)}\sum_{l=0}^\infty b^{-(D-1)l}
\fint_{B_{\rho}(x,b^{l-(i\wedge j)})}
\left|H_j(x)\vec s_i(y)\right|\,dy\notag\\
&=\sum_{i\in\mathbb{Z}}
b^{-(j-i)_+(E-\frac{1}{2})}b^{-(i-j)_+(F-\frac{1}{2})}
\sum_{l=0}^\infty b^{-(D-1)l}\fint_{B_{\rho}(x,b^{l-(i\wedge j)})}
\left|H_j(x)\vec s_i(y)\right|\,dy.
\end{align}
By Lemma \ref{yl022301} with $U$ and $l$ replaced by $H_j(x)$ and $l-(i\wedge j)$  respectively,
we find that
\begin{align*}
\fint_{B_{\rho}(x,b^{l-(i\wedge j)})}
\left| H_j(x)\vec s_i(y)\right|\,dy
&\lesssim b^{[l-(i\wedge j)+i](\frac{1}{a}-1)}\left[\fint_{B_{\rho}(x,C_3b^{l-(i\wedge j)})}
\left|H_j(x)\vec s_i(y)\right|^a\,dy
\right]^{\frac{1}{a}}\\
&=b^{[l+(i-j)_+](\frac{1}{a}-1)}\left[\fint_{B_{\rho}(x,C_3b^{l-(i\wedge j)})}
\left|H_j(x)\vec s_i(y)\right|^a\,dy
\right]^{\frac{1}{a}}.
\end{align*}
Applying this and \eqref{Bsj}, we obtain
\begin{align*}
\left| H_j(x)\left(B\vec s\right)_j(x)\right|
&\lesssim\sum_{i\in\mathbb{Z}}
b^{-(j-i)_+(E-\frac{1}{2})}b^{-(i-j)_+(F+\frac{1}{2}-\frac{1}{a})}\\
&\quad\times\sum_{l=0}^\infty b^{-(D-\frac{1}{a})l}
\left[\fint_{B_{\rho}(x, C_3b^{l-(i\wedge j)})}
\left| H_j(x)\vec s_i(y)\right|^a\,dy\right]^{\frac{1}{a}}\\
&\lesssim\sum_{i\in\mathbb{Z}}
b^{-(j-i)_+(E-\frac{1}{2})}b^{-(i-j)_+(F+\frac{1}{2}-\frac{1}{a})}\\
&\quad\times\sum_{l=0}^\infty b^{-(D-\frac{1}{a})(l+\lceil\log_b{ C_3\rceil})}
\left[\fint_{B_{\rho}(x, b^{l+\lceil\log_b{ C_3\rceil}-(i\wedge j)})}
\left| H_j(x)\vec s_i(y)\right|^a\,dy\right]^{\frac{1}{a}}\\
&
\le\sum_{k\in\mathbb{Z}}b^{-(E-\frac{1}{2})k_-}
b^{-k_+(F+\frac{1}{2}-\frac{1}{a})}\\
&\quad\times\sum_{l=0}^\infty b^{-(D-\frac{1}{a})l}
\left[ \fint_{B_{\rho}(x, b^{l-[(j+k)\wedge j]})}
\left| H_j(x)\vec s_{j+k}(y)\right|^a\,dy\right]^{\frac{1}{a}}.
\end{align*}
Since  $\|\cdot\|_{\ell^qL^p}^r $  satisfies the triangle inequality, we  obtain,
\begin{align*}
\left\|\left\{H_j\left(B\vec s\right)_j\right\}_{j\in\mathbb Z} \right\|_{\ell^qL^p}^r
&\lesssim \sum_{k\in\mathbb{Z}}
\sum_{l=0}^\infty \left[b^{-(E-\frac{1}{2}) k_-}
b^{-k_+(F+\frac{1}{2}-\frac{1}{a})}b^{-(D-\frac{1}{a})l}\right]^r\\
&\quad\times
\left\|\left\{\left[ \fint_{B_{\rho}(\cdot, b^{l-[(j+k)\wedge j]})}
\left| H_j(\cdot)\vec s_{j+k}(y)\right|^a \,dy
\right]^{\frac{1}{a}}\right\}_{j\in\mathbb Z}\right\|_{\ell^qL^p}^r\\
&=\sum_{k\in\mathbb{Z}}\sum_{l=0}^\infty \left[ b^{-(E-\frac{1}{2})k_-}
b^{-k_+(F+\frac{1}{2}-\frac{1}{a})}b^{-(D-\frac{1}{a})l}\right]^r\\
&\quad\times
\left\|\left\{\left[ \fint_{B_{\rho}(\cdot, b^{l-[i\wedge(i-k)]})}
\left| H_{i-k}(\cdot)\vec s_{i}(y)\right|^a \,dy
\right]^{\frac{1}{a}}\right\}_{i\in\mathbb Z}\right\|_{\ell^qL^p}^r,
\end{align*}
where $i\wedge(i-k)=i-k_+$.
This completes the proof of Lemma \ref{ad prelim}.
\end{proof}

\begin{lemma}\label{ad conv}
Let all the symbols be the same as in Lemma \ref{ad prelim}.
Assume that, for any $j\in\mathbb Z$
and almost every $x\in \mathbb R^n$,
the matrix $H_j(x)$ is invertible.
If the right-hand side of \eqref{ABt} is finite, then $B\vec s$ is well defined.
\end{lemma}

\begin{proof}
By the proof of Lemma \ref{ad prelim}, we find that the left-hand side of \eqref{ABt} can be replaced by
\begin{align*}
\left\|\left\{\sum_{Q\in\mathcal{D}_j}
\widetilde{\mathbf{1}}_Q(\cdot)\sum_{R\in\mathcal{D}}
\left|H_j(\cdot)b_{Q,R}\vec{s}_R \right| \right\}_{j\in\mathbb Z}  \right\|_{\ell^qL^p}^r.
\end{align*}
In this case, if the right-hand side of \eqref{ABt} is finite, then, for any $j\in\mathbb Z$, $Q\in\mathcal{D}_j$, and almost every $x\in Q$, we have
$\sum_{R\in\mathcal{D}}
| H_j(x)b_{Q,R}\vec{s}_R | < \infty.$
Moreover,  by the invertibility of $H_j(x)$,  we obtain
\begin{align*}
\left| b_{Q,R}\vec{s}_R \right|
\leq \left\| \left[H_j(x)\right]^{-1} \right\|
\left| H_j(x)b_{Q,R}\vec{s}_R \right|,
\end{align*}
which further implies that
$\sum_{R\in\mathcal{D}} |b_{Q,R}\vec{s}_R|<\infty$.
This completes the proof of Lemma \ref{ad conv}.
\end{proof}

\begin{lemma}\label{geshu}
For any $j\in\mathbb Z$,  $\lambda\in(0,\infty)$, and $x\in\mathbb R^n$,
$$\sum_{Q\in\mathcal D_j}\mathbf 1_{B_\rho(x_Q,\lambda b^{-j})}(x)\le \#\left\{k\in\mathbb Z^n:\ \rho(k)\le H(C_0+\lambda)\right\},$$
where $C_0$ is as in Lemma \ref{yl1101}.
\end{lemma}

\begin{proof}
For any $x\in\mathbb R^n$, let $\mathcal D_j^{(x)}:=\{Q\in\mathcal D_j:\ \rho(x-x_Q)<\lambda b^{-j} \}$. Then
$$
\sum_{Q\in\mathcal D_j} \mathbf 1_{B_\rho(x_Q,\lambda b^{-j})}(x) =\#\mathcal D_j^{(x)}.
$$
For any $x\in\mathbb R^n$, there exists $R\in \mathcal D_j$ such that $x\in R$.
By Definition \ref{quasi-norm} and  Lemma \ref{yl1101}(ii), we obtain, for any $Q\in \mathcal D_j^{(x)} $,
$$\rho(x_R-x_Q)\le H\left[\rho(x_R-x)+\rho(x-x_Q)\right]\le H(C_0+\lambda)b^{-j},$$
and hence
\begin{align*}
\mathcal D_j^{(x)}
&\subset\left\{Q\in\mathcal D_j:\ \rho(x_R-x_Q)\le H(C_0+\lambda )b^{-j} \right\} \\
&=\left\{Q\in\mathcal D_j:\ \rho(k_R-k_Q)\le H(C_0+\lambda) \right\}.
\end{align*}
This further implies that, for any $x\in\mathbb R^n$,
$$
\#\mathcal D_j^{(x)}
\le \#\left\{Q\in\mathcal D_j:\ \rho(k_R-k_Q)\le H(C_0+\lambda) \right\}
= \#\left\{k\in\mathbb Z^n:\ \rho(k)\le H(C_0+\lambda)\right\},
$$
which completes the proof of Lemma \ref{geshu}.
\end{proof}

Now, we estimate the quantities on the right-hand side of Lemma \ref{ad prelim}.

\begin{lemma}\label{bq0}
Let $p\in(0,\infty)$, $q\in(0,\infty]$, and
$W\in\mathcal A_{p,\infty}$ have $\mathcal A_{p,\infty}$-upper dimension $d_2\in[0,\infty)$.
Assume that, for some $v\in(0,p)$,
the operator $\mathcal M_{W,p}^{(v)}:\ L^p(\mathbb R^n,\mathbb C^m)\to L^p$ is bounded.
Then there exists a positive constant $C$ such that,
for any $k\in \mathbb{Z}$, $l\in \mathbb{Z}_+$, and $\vec{s} \in\dot b_{p,q}^{0}(W)$,
$$\left\|\left\{\left[\fint_{B_{\rho}(\cdot,b^{l+k_{+}-i})}
\left|W^{\frac{1}{p}}(\cdot)
\vec s_i(y)\right|^{1\wedge p} \,dy\right]^{\frac1{1\wedge p}} \right\}_{i\in\mathbb Z} \right\|_{\ell^qL^p}
\leq Cb^{(l+k_{+})\min\{\frac{d_{2}}{p}, (\frac1v-\frac1{1\wedge p})_+\}}
\|\vec{s}\|_{\dot b_{p,q}^{0}(W)},
$$
where $\|\cdot\|_{\ell^qL^p}$ and $\vec s_i$ are as, respectively, in \eqref{lqLp} and \eqref{sj}.
\end{lemma}

\begin{proof}
By Lemma \ref{yl022301} and Theorem \ref{bounded B1},
we find that, for any $i\in\mathbb Z$,
\begin{align*}
I_i
&:=\left\| \left[\fint_{B_{\rho}(\cdot,b^{l+k_{+}-i})}
\left|W^{\frac{1}{p}}(\cdot) \vec s_i(y)\right|^{1\wedge p}
\,dy\right]^{\frac1{1\wedge p}} \right\|_{L^p} \\
&\phantom{:}\lesssim b^{(l+k_{+})(\frac1v-\frac1{1\wedge p})_+}
\left\|\left[\fint_{B_{\rho}(\cdot,C_3b^{l+k_{+}-i})}
\left|W^{\frac{1}{p}}(\cdot)
\vec s_i(y)\right|^v \,dy\right]^{\frac1v} \right\|_{L^p} \\
&\phantom{:}\leq b^{(l+k_{+})(\frac1v-\frac1{1\wedge p})_+}
\left\| \mathcal M_{W,p}^{(v)} \left( W^{\frac1p} \vec s_i \right) \right\|_{L^p}
\lesssim b^{(l+k_{+})(\frac1v-\frac1{1\wedge p})_+}
\left\| \, \left|W^{\frac1p} \vec s_i \right| \, \right\|_{L^p},
\end{align*}
which is one desired estimate.

Let $k,i\in\mathbb{Z}$, $l\in\mathbb{Z}_{+}$,
and $\vec{s}\in\dot{b}_{p,q}^{0}(W)$ be fixed.
Let $j:=i-l-k_{+}$ and $C_3:=H(1+C_0)$, where $C_0$ is as in Lemma \ref{yl1101}.
For any $Q\in\mathcal D_j$,  suppose that
$B_Q:=B_{\rho}(x_Q,C_3b^{-j})$.
Then,  by Lemma \ref{yl1201}, we obtain $Q\subset B_Q$.
Moreover,  using Definition \ref{quasi-norm}(iii) and  Lemma \ref{yl1101}(ii), we conclude that,
for any $x\in Q$ and  $y\in B_{\rho}(x,b^{-j})$,
$$
\rho(y-x_Q)
\le H \left[\rho(y-x)+\rho(x-x_Q)\right]
\le H \left(b^{-j}+C_0b^{-j}\right)
\leq C_3 b^{-j},
$$
and hence  $B_{\rho}(x,b^{-j})\subset B_Q$.
From these,  H\"older's inequality, Tonelli's theorem,
and \eqref{equ_reduce}, we infer that
\begin{align}\label{es0}
I_i^p
&\leq\sum_{Q\in\mathcal{D}_j}\int_Q
\fint_{B_{\rho}(x,b^{-j})}\left|W^{\frac{1}{p}}(x)\vec s_i(y)\right|^p\,dy\,dx
\lesssim\sum_{Q\in\mathcal{D}_j}\int_Q\fint_{B_Q}
\left|W^{\frac{1}{p}}(x)\vec s_i(y)\right|^p\,dy\,dx\notag\\
&\leq\sum_{Q\in\mathcal{D}_j}\int_{B_Q}
\fint_{B_Q}\left|W^{\frac{1}{p}}(x)\vec s_i(y)\right|^p\,dx\,dy
\sim\sum_{Q\in\mathcal{D}_j}\int_{B_Q}\left|A_{B_Q}\vec s_i(y)\right|^p\,dy.
\end{align}
By Lemma \ref{yl1201}, we find that, for any $R\in\mathcal D_i$,
there exists $\widetilde B_R\in \mathcal B$ such that
$r_{\widetilde B_R}\sim b^{-i}$ and $\widetilde B_R\subset R$. Then
\begin{align}\label{516}
\int_{B_Q}\left|A_{B_Q}\vec s_i(y)\right|^p\,dy
&\lesssim\sum_{\genfrac{}{}{0pt}{}{R\in\mathcal{D}_i}{R\cap B_Q\neq\emptyset}}
\left\|A_{B_Q}A_R^{-1}\right\|^p
\int_R\left|A_{R}\vec s_i(y)\right|^p\,dy\notag\\
&\lesssim\sum_{\genfrac{}{}{0pt}{}{R\in\mathcal{D}_i}{R\cap B_Q\neq\emptyset}}
\left\|A_{B_Q}A_{\widetilde B_R}^{-1}\right\|^p
\left\|A_{\widetilde B_R} A_{R}^{-1}\right\|^p
\int_R\left|A_{R}\vec s_i(y)\right|^p\,dy.
\end{align}
Repeating the proof of \eqref{key2}, we obtain
$$
\left\|A_{\widetilde B_R} A_{R}^{-1}\right\|^p
\lesssim \left\|A_{\widetilde B_R} A_{\widetilde B_R}^{-1}\right\|^p
=1.
$$
Let $W$ have $\mathcal A_{p,\infty}$-lower dimension $d_1\in[0,n)$.
From Lemma \ref{fuhe}(ii) and  Lemma \ref{yl1101}(ii), it follows that,
for any $R\in\mathcal D_i$ with $R\cap B_Q\neq\emptyset$,
\begin{align*}
\left\|A_{B_Q}A_{\widetilde B_R}^{-1}\right\|^p
&\lesssim \max\left\{\left(\frac{r_{\widetilde B_R}}{r_{B_Q}}\right)^{d_1},
\left(\frac{r_{B_Q}}{r_{\widetilde B_R}}\right)^{d_2}\right\}
\left[1+\frac{\rho(x_Q-c_{\widetilde B_R})}
{r_{B_Q} \vee r_{\widetilde B_R}}\right]^{d_1+d_2} \\
&\lesssim b^{(l+k_+)d_2}
\left[1+\frac{b^{-i+l+k_+}+b^{-i}}{b^{-i+l+k_+}} \right]^{d_1+d_2}
\sim b^{(l+k_+)d_2}.
\end{align*}
Substituting the estimates of $\|A_{\widetilde B_R} A_{R}^{-1}\|^p$
and $\|A_{B_Q}A_{\widetilde B_R}^{-1}\|^p$ back into \eqref{516}, we obtain
\begin{align}\label{IntBQ}
\int_{B_Q}\left|A_{B_Q}\vec s_i(y)\right|^p\,dy
\lesssim b^{(l+k_+)d_2}\sum_{\genfrac{}{}{0pt}{}{R\in\mathcal{D}_i}{R\cap B_Q\neq\emptyset}}
\int_R\left|A_{R}\vec s_i(y)\right|^p\,dy.
\end{align}
For any $R\in\mathcal D_i$
with $R\bigcap B_Q\neq \emptyset$, let $y\in R\bigcap B_Q$ be fixed.
Then Lemma \ref{yl1101}(ii) implies that,
for any $x\in R$,
\begin{align*}
\rho(x-x_Q)
\le H[\rho(x-y)+\rho(y-x_Q) ]\le H(C_0b^{-i}+C_3b^{-j})
\le H(C_0+C_3)b^{-j}
=: \lambda b^{-j}.
\end{align*}
Therefore,
$$
\bigcup_{\genfrac{}{}{0pt}{}{R\in\mathcal{D}_i}{R\cap B_Q\neq\emptyset}}R
\subset B_{\rho}\left(x_Q, \lambda b^{-j}\right).
$$
This, together with \eqref{IntBQ}, further implies that
\begin{align*}
\int_{B_Q}\left|A_{B_Q}\vec s_i(y)\right|^p\,dy
\lesssim b^{(l+k_+)d_2} \int_{B_{\rho}(x_Q, \lambda b^{-j})}
\left|A_i(y)\vec {s_i}(y)\right|^p\,dy,
\end{align*}
where $A_i$ is as in \eqref{Aj}.
Substituting this estimate back into \eqref{es0},
and using Lemma \ref{geshu}, we conclude that
\begin{align*}
I_i^p
&\lesssim b^{(l+k_+)d_2}\sum_{Q\in\mathcal{D}_j}
\int_{B_{\rho}(x_Q, \lambda b^{-j})}\left|A_i(y)\vec {s_i}(y)\right|^p\,dy \\
&=b^{(l+k_+)d_2}
\int_{\mathbb R^n} \sum_{Q\in\mathcal{D}_j}\mathbf1_{B_{\rho}(x_Q, \lambda b^{-j})}(y)
\left|A_i(y)\vec {s_i}(y)\right|^p\,dy\\
&\lesssim  b^{(l+k_+)d_2}\int_{\mathbb{R}^n}\left|
A_i(y)\vec {s_i}(y)\right|^p\,dy.
\end{align*}
From this and Theorem \ref{dj}, it follows that
\begin{align*}
\|\{I_i \}_{i\in\mathbb Z} \|_{\ell^q}
\lesssim b^{(l+k_+)\frac{d_2}{p}}\left\|\vec{s}\right\|_{\dot b^0_{p,q}(\mathbb{A})}
\sim b^{(l+k_+)\frac{d_2}{p}}\left\|\vec{s}\right\|_{\dot b^0_{p,q}(W)},
\end{align*}
which completes the proof of Lemma \ref{bq0}.
\end{proof}

Now, we  prove Theorem \ref{ad FJ-YY}.
\begin{proof}[Proof of Theorem \ref{ad FJ-YY}]
We first note that it is sufficient to prove the theorem for
$\alpha=0$. To see this, suppose the result holds in that case and let $\widetilde{B}$  be as in Lemma \ref{alpha0}.
Clearly, since $B$ satisfies \eqref{ad FJ} for  $\alpha\in\mathbb R$,
by Lemma \ref{alpha0}(i), $\widetilde{B}$
satisfies \eqref{ad FJ} with $\alpha=0$.
By the assumption,
  $\widetilde{B}$ is bounded on $\dot b^0_{p,q}(W)$,
which, combined with Lemma \ref{alpha0}(ii),
further implies that  $B$ is bounded on $\dot b^\alpha_{p,q}(W)$.
Thus, the proof reduces to the case  $\alpha=0$.

From the definitions of $d_{p,\infty}^\text{upper}(W)$ and $v_{W,p}$,
we infer that there exists $d_2\in[\![d_{p,\infty}^\text{upper}(W),\infty)$ and
$v\in(0,v_{W,p})$
such that $W$ has $\mathcal A_{p,\infty}$-upper dimension $d_2$ and
\begin{equation}\label{DEFcoff}
D> \frac1{1\wedge p} +\varepsilon,\quad E>\frac 12,
\quad\text{and} \quad F> \frac1{1\wedge p} -\frac 12+\varepsilon,
\end{equation}
where $\varepsilon:=\min\{\frac{d_2}{p},
(\frac1{v}-\frac1{1\wedge p})_+\}$.
Applying Lemmas \ref{ad prelim} and \ref{bq0}, we conclude that
\begin{align} \label{163}
\left\|\left\{W^{\frac1p}(B\vec s)_j\right\}_{j\in\mathbb Z} \right\|_{\ell^qL^p}^r
&\lesssim\sum_{k\in\mathbb{Z}}
\sum_{l=0}^\infty \left[b^{-(E-\frac{1}{2}) k_-}
b^{-k_+(F+\frac{1}{2}-\frac{1}{1\wedge p})}
b^{-(D-\frac{1}{1\wedge p})l}\right]^r\notag \\
&\quad\times
\left\|\left\{\left[\fint_{B_\rho(\cdot,b^{l+k_{+}-i})}
\left|W^{\frac{1}{p}}(\cdot)
\vec s_i(y)\right|^{1\wedge p} \,dy\right]^{\frac1{1\wedge p}} \right\}_{i\in\mathbb Z} \right\|_{\ell^qL^p}^r\notag\\
&\lesssim\sum_{k\in\mathbb{Z}}
\sum_{l=0}^\infty \left[b^{-(E-\frac{1}{2}) k_-}
b^{-k_+(F+\frac{1}{2}-\frac{1}{1\wedge p})}
b^{-(D-\frac{1}{1\wedge p})l} b^{(l+k_{+})\varepsilon}\right]^r
\left\|\vec{s}\right\|_{\dot b_{p,q}^{0}(W)}^r,
\end{align}
where $r:=\min\{1,p,q\}$.
Note that the coefficient in \eqref{163} takes the form
\begin{align*}
b^{-(E-\frac{1}{2})k_-}
b^{-[F-(\frac{1}{1\wedge p}-\frac{1}{2}+\varepsilon)]k_+}
b^{-[D-(\frac{1}{1\wedge p}+\varepsilon)]l},
\end{align*}
which, together with \eqref{DEFcoff}, further implies that the series in \eqref{163} converges.
By this and  Lemma \ref{ad conv} with $H_j(x)$ replaced by $W^{\frac1p}(x)$, we find that,
for any $\vec s\in \dot b_{p,q}^0(W)$,
$B\vec s$ is well defined and
$$\left\|B\vec s \right\|_{\dot b_{p,q}^{0}(W)}
=\left\|\left\{W^{\frac1p}(B\vec s)_j\right\}_{j\in\mathbb Z} \right\|_{\ell^qL^p}
\lesssim\left\|\vec{s}\right\|_{\dot b_{p,q}^{0}(W)}.$$
This completes the proof of Theorem \ref{ad FJ-YY}.
\end{proof}

\begin{definition}\label{almodiag}
Let $p\in(0,\infty)$, $q\in(0,\infty]$, $\alpha\in\mathbb{R}$,
and $W\in \mathcal A_{p,\infty}.$ An infinite matrix
$B:=\{b_{Q,R}\}_{Q,R\in\mathcal{D}}$ in $\mathbb{C}$ is said to be
\emph{$\dot{b}_{p,q}^{\alpha}(W)$-almost diagonal}
if it is $(D,E,F)$-almost diagonal
for some $D,E,F\in\mathbb R$ satisfying \eqref{ad FJ}.
\end{definition}

As a corollary of Theorem \ref{ad FJ-YY},
we readily obtain the boundedness of almost diagonal operators
for $W\in\mathcal A_{p}$.

\begin{corollary}\label{ad FJ-YY2}
Let $\alpha\in\mathbb{R} $,
$p\in(0,\infty)$, $q\in(0,\infty]$, and
$W\in\mathcal A_{p}$.  Assume that
$B:=\{b_{Q, R}\}_{Q, R\in\mathcal{D}}$ in $\mathbb{C}$
is $(D,E,F)$-almost diagonal with parameters
\begin{align*}
D>\frac{1}{1\wedge p},\quad
E>\frac{1}{2}+\alpha,
\quad\text{and}\quad
F>\frac{1}{1\wedge p}-\frac{1}{2}-\alpha.
\end{align*}
Then $B\vec s$ is well defined for all $\vec s \in\dot{b}_{p,q}^{\alpha}(W)$,
and $B$ is bounded on $\dot b^\alpha_{p,q}(W)$.
\end{corollary}

\begin{proof}
From (ii) and (iii) of Proposition \ref{prop of sp(W)},
we infer that $(\frac{1}{v_{W,p}}-\frac{1}{p\wedge1})_+=0$,
which further implies that ${\rm{AT}}(p,W)=0$.
This, together with Theorem \ref{ad FJ-YY},
completes the proof of Corollary \ref{ad FJ-YY2}.
\end{proof}

\begin{remark}
In the case where $A=2I_n$ and $\rho(\cdot):=|\cdot|^n$,
Corollary \ref{ad FJ-YY2} coincides with the
sharp boundedness of almost diagonal operators with $\mathcal A_p$-matrix weights
obtained by Bu et al. \cite[Theorem 6.2]{bf6}
(see Remark \ref{sharp1} below).
\end{remark}

Finally, we focus on the composition of two almost diagonal operators.

\begin{lemma}\label{fuheyinli}
Let $D,E,F,D',E',F'\in\mathbb R$ satisfy
$$D,D'>1,\ E+F'>\min\{D,D'\},\ F+E'>1,\ E\neq E', \ \text{and}\ F\neq F'.$$
If $A: =\{a_{P,Q}\}_{P, Q\in \mathcal{D}}$ is
$(D,E,F)$-almost diagonal and
$B:=\{b_{Q,R}\}_{Q,R\in\mathcal{D}}$ is
$(D',E',F')$-almost diagonal, then
$$A\circ B:=\left\{\sum\limits_{Q\in\mathcal{D}}
a_{P,Q}b_{Q,R}\right\}_{P,R\in\mathcal{D}}$$
is $(D\wedge D', E\wedge E', F\wedge F')$-almost diagonal.
\end{lemma}

\begin{proof}
By Definition \ref{almdia}, we find that, for any $P,R\in\mathcal D$,
\begin{align*}
|(A\circ B)_{P,R}|
&\lesssim\sum_{Q\in\mathcal D} b_{P,Q}^{D,E,F} b_{Q,R}^{D',E',F'} \\
&=\sum_{j=-\infty}^{j_P\wedge j_R}\sum_{Q\in\mathcal D_j}\cdots
+\sum_{j=(j_P \wedge j_R)+1}^{j_P\vee j_R}\sum_{Q\in\mathcal D_j}\cdots
+\sum_{j=(j_P\vee j_R)+1}^{\infty}\sum_{Q\in\mathcal D_j}\cdots\\
&=: \mathrm{J}_1+\mathrm{J}_2+\mathrm{J}_3.
\end{align*}
Using Lemma \ref{summn}, we conclude that
\begin{align*}
\mathrm{J}_1
&=\sum_{j=-\infty}^{j_P\wedge j_R}
\sum_{k\in\mathbb Z^n}\left[1+\rho(A^{j}x_P-k) \right]^{-D}
\left[1+\rho(k-A^{j}x_R) \right]^{-D'}b^{(j-j_P)E}b^{(j-j_R)F'}\\
&\lesssim \sum_{j=-\infty}^{j_P\wedge j_R}
\left[1+\rho(A^j(x_P-x_R))\right]^{-(D\wedge D')}
b^{(j-j_P)E}b^{(j-j_R)F'}\\
&= \sum_{j=-\infty}^{j_P\wedge j_R}
\left[1+b^j\rho(x_P-x_R)\right]^{-(D\wedge D')}
b^{(j-j_P)E}b^{(j-j_R)F'}.
\end{align*}
This, together with the fact that $b^j(|P|\vee|R|)\le 1$
for all $j\in\mathbb Z\cap (-\infty, j_P\wedge j_R]$, further implies that
\begin{align*}
\mathrm{J}_1
&\lesssim \sum_{j=-\infty}^{j_P\wedge j_R}
\left[1+\frac{\rho(x_P-x_R)}{|P|\vee|R|}\right]^{-(D\wedge D')}
\left[b^j(|P|\vee|R|)\right]^{-(D\wedge D')}
b^{(j-j_P)E}b^{(j-j_R)F'}\\
&\sim \left[1+\frac{\rho(x_P-x_R)}{|P|\vee|R|}\right]^{-(D\wedge D')}
b^{(j_P\wedge j_R)(E+F')-j_PE-j_RF'} \\
&=b_{P,R}^{D\wedge D',E,F'}
\leq b_{P,R}^{D\wedge D', E\wedge E', F\wedge F'}.
\end{align*}
As in the estimation of $\mathrm{J}_1$,
using Lemma \ref{summn} again, we obtain
the same bound for both $\mathrm{J}_2$ and $\mathrm{J}_3$.
This completes the proof of Lemma \ref{fuheyinli}.
\end{proof}

The following proposition shows that the composition of
two $\dot b_{p, q}^{\alpha}(W)$-almost diagonal operators
is also a $\dot b_{p, q}^{\alpha}(W)$-almost diagonal operator.

\begin{proposition}\label{compose}
Let $p\in(0,\infty)$, $q\in(0,\infty ]$,  $\alpha\in \mathbb{R}$,
and $W\in\mathcal A_{p, \infty }$.
If, for any $i\in\{1,2\}$, $A^{(i)}:=\{a_{P,Q}^{(i)}\}_{P, Q\in \mathcal{D}}$
is $\dot b_{p, q}^{\alpha}(W)$-almost diagonal,
then
$A^{(1)}\circ A^{(2)}$
is also $\dot b_{p, q}^{\alpha}(W)$-almost diagonal.
\end{proposition}

\begin{proof}
From Definition \ref{almodiag}, we infer that, for any $i\in\{1,2\}$,
there exist $D^{(i)},E^{(i)},F^{(i)}\in\mathbb R$ such that
$A^{(i)}$ is $(D^{(i)},E^{(i)},F^{(i)})$-almost diagonal with
\begin{align}\label{DEF}
D^{(i)}>J,\quad E^{(i)}>\frac{1}{2}+\alpha
,\quad F^{(i)}>J-\frac12-\alpha,
\end{align}
where $J$ is as in \eqref{J}. Observe that
\begin{align}\label{E+F}
\left(\frac12+\alpha \right)+\left(J-\frac12-\alpha \right)
= J
\ge 1.
\end{align}
Since $D^{(i)},E^{(i)},F^{(i)}$ can be made as close as desired to
their lower bounds specified in \eqref{DEF},
without loss of generality, we may assume that
$$E^{(1)}+F^{(2)}>\min_{i\in\{1,2\}}D^{(i)}, \quad E^{(1)}\neq E^{(2)},\quad {\rm{and}}\quad F^{(1)}\neq F^{(2)}.$$
Moreover, applying \eqref{DEF} and \eqref{E+F}, we also obtain $ D^{(1)}, D^{(2)}>1$ and
$ F^{(1)}+E^{(2)}>1$.
By these and Lemma \ref{fuheyinli}, we find that $A^{(1)}\circ A^{(2)}$ is
$(D,E,F)$-almost diagonal, where
$$
D=\min_{i\in\{1,2\}} D^{(i)}> J,\quad
E= \min_{i\in\{1,2\}} E^{(i)}>\frac{1}{2}+\alpha,
$$
and
$$
F
=\min_{i\in\{1,2\}}F^{(i)}
>J-\frac12-\alpha.
$$
Thus, $A^{(1)}\circ A^{(2)}$ is also $\dot{b}_{p,q}^{\alpha}(W)$-almost diagonal.
This completes the proof of Proposition \ref{compose}.
\end{proof}

\subsection{Sharpness of Almost Diagonal Conditions}\label{SADO}

In \cite[Theorem 7.1]{bf6}, Bu et al. characterized the boundedness of
$(D,E,F)$-almost diagonal operators on
unweighted Besov spaces in the Euclidean setting.
In this subsection, we extend this result not only to
the anisotropic setting but also to the power-weighted case.

\begin{theorem}\label{ad Besov sharp}
Let $\alpha\in\mathbb R$, $p\in(0,\infty)$, $q\in(0,\infty]$,
and $D,E,F\in \mathbb{R}$.
Assume that $\gamma\in(-1,\infty)$, $W(\cdot):=[\rho(\cdot)]^\gamma I_m$, and
\begin{align*}
J_{p,\gamma}:=\max\left\{1,\frac{1}{p},\frac{1+\gamma}{p}\right\}.
\end{align*}
Then every $(D,E,F)$-almost diagonal operator
$B$ is bounded on $\dot b^\alpha_{p,q}(W)$ if and only if
$D>J_{p,\gamma}$, $E>\frac12+\alpha$,
and $F>J_{p,\gamma}-\frac{1}{2}-\alpha$.
\end{theorem}

To prove Theorem \ref{ad Besov sharp},
corresponding to Definition \ref{dim}, we introduce the
concept of upper dimensions of scalar weights.

\begin{definition}
Let $d\in[0,\infty).$
A scalar weight $w$ is said to have \emph{$A_\infty$-upper dimension $d$},
denoted by $w\in\mathbb{D}_{\infty,d}^\text{upper}$,
if there exists a positive constant $C$ such that,
for any $\lambda\in[1,\infty)$ and any ball $B\in\mathcal B$,
$$\fint_{\lambda B}w(x)\,dx\exp\left(\fint_B
\log\left([w(x)]^{-1}\right)\,dx\right)\leq C\lambda^d.$$
For any scalar weight $w$, let
$
d_\infty^{\mathrm{upper}}(w):=\inf\{d\in[0,\infty):\:w
\text { has } A_\infty\text{-upper dimension }d\}.
$
\end{definition}

The following lemma gives the relation between scalar and matrix weights,
which follows immediately from their definitions;
we omit the details.

\begin{lemma}\label{scamtr}
Let $p\in(0,\infty)$ and $d\in [0, \infty)$.
Assume that $w$ is a scalar weight and $W:=wI_m$.
Then the following statements hold.
\begin{enumerate}[\rm(i)]
\item $W\in \mathcal A_{p,\infty}$
if and only if $w\in A_{\infty}$.

\item $W\in \mathbb{D}_{p,\infty,d}^\mathrm{upper}$ if and only if
$w\in\mathbb{D}_{\infty, d}^\mathrm{upper}$.
\end{enumerate}
\end{lemma}

Using Lemma \ref{ballsim}(i), we obtain the critical upper dimension of power weights.

\begin{lemma}\label{rhoinf}
Let $a\in(-1,\infty)$ and $w(\cdot):=[\rho(\cdot)]^a$.
Then $w\in A_{\infty}$ and $d_{\infty}^{\mathrm{upper}}(w)=a_+$.
\end{lemma}

\begin{proof}
Lemma \ref{example} implies that $w\in A_{\infty}$.
Applying Jensen's inequality and $w\in A_{\infty}$,
we conclude that, for any $B\in\mathcal B$,
\begin{align}\label{expsim}
\exp{\left(\fint_{B}\log\left(\left[w(x)\right]^{-1}\right)\,dx\right)}
\sim\left[\fint_{B}w(x)\,dx\right]^{-1}.
\end{align}
Next, we prove that $d_{\infty}^{\text{upper}}(w)=a_+$
by  considering two cases for $a$.

\emph{Case (1)} $a\in(0,\infty)$. In this case, $a_+=a$.
Using \eqref{expsim} and Lemma \ref{ballsim}(i), we conclude that,
for any $\lambda\in[1,\infty)$ and any $B\in\mathcal B$,
\begin{align*}
I(\lambda, B)
&:= \fint_{\lambda B}w(x)\,dx \exp \left(\fint_{B}\log\left(\left[w(x)\right]^{-1}\right)\,dx\right) \\
&\phantom{:}\sim\fint_{\lambda B}w(x)\,dx
\left[\fint_{B}w(x)\,dx\right]^{-1}
\sim\frac{[\rho(c_B)+\lambda r_B]^a}{[\rho(c_B)+r_B]^a}
\le\frac{[\lambda\rho(c_B)+\lambda r_B]^a}{[\rho(c_B)+r_B]^a}
=\lambda^a,
\end{align*}
which further implies that $w$ has $A_{\infty}$-upper dimension $a$.
On the other hand, if $w$ has $A_{\infty}$-upper dimension $d$,
then, taking $c_B=\mathbf{0}$, by Lemma \ref{ballsim}(i) and \eqref{expsim}, we find that
\begin{align*}
\lambda^a\sim\fint_{\lambda B}w(x)\,dx
\left[\fint_{B}w(x)\,dx\right]^{-1}
\sim I(\lambda, B)
\lesssim \lambda^d,
\end{align*}
and hence $a\le d$. Thus, $d_\infty^{\text{upper}}(w)=a=a_+$ in the case.

\emph{Case (2)} $a\in(-1,0]$. In this case, $a_+=0$.
From \eqref{expsim} and Lemma \ref{ballsim}(i), we infer that,
for any $\lambda\in[1,\infty)$ and any $B\in\mathcal B$,
\begin{align*}
I(\lambda,B)
\sim\fint_{\lambda B}w(x)\,dx\left[\fint_{B}w(x)\,dx\right]^{-1}
\sim\frac{[\rho(c_B)+\lambda r_B]^a}{[\rho(c_B)+r_B]^a}
\le\frac{[\rho(c_B)+\lambda r_B]^a}{[\rho(c_B)+\lambda r_B]^a}=1,
\end{align*}
which indicates that $w$ has $A_{\infty}$-upper dimension $0$,
and hence $d_{\infty}^{\text{upper}}(w)=0=a_+$.
This completes the proof of Lemma \ref{rhoinf}.
\end{proof}

\begin{lemma} \label{far}
Let $N\in\mathbb Z$ and $\lambda\in(0,\infty)$.
Then, for any $x\in \mathbb R^n \setminus B_\rho(\mathbf 0, 2H \lambda b^{N})$
and $y\in B_\rho(x,\lambda b^N)$,
\begin{align*}
\frac1{2H}\rho(y)< \rho(x) < 2H\rho(y).
\end{align*}
\end{lemma}

\begin{proof}
By the quasi-triangle inequality of $\rho$, we obtain
$\rho(y)\geq \frac1H\rho(x)-\rho(x-y)>\lambda b^N$, and hence
\begin{align*}
\rho(x)
\leq  H [\rho(x-y) + \rho(y)]
< H [\lambda b^N + \rho(y)]
< 2H\rho(y).
\end{align*}
On the other hand,
\begin{align*}
\rho(y)
\leq  H [\rho(y-x) + \rho(x)]
< H [\lambda b^N + \rho(x)]
< 2H\rho(x).
\end{align*}
This completes the proof of Lemma \ref{far}.
\end{proof}

We now prove Theorem \ref{ad Besov sharp}.

\begin{proof}[Proof of Theorem \ref{ad Besov sharp}]
We first prove the sufficiency.
By Lemmas \ref{scamtr}, \ref{rhoinf}, and \ref{value}, we have
\begin{align*}
W\in\mathcal A_{p,\infty}, \quad
d_{p,\infty}^{\mathrm{upper}}(W)=\gamma_+, \quad\text{and}\quad
v_{W,p}=\frac{p}{1+\gamma_+}.
\end{align*}
Therefore the exponent $J$ appearing in Theorem \ref{ad FJ-YY} satisfies
\begin{align*}
J
&=\frac{1}{1\wedge p}
+\min\left\{\frac{\gamma_+}{p},
\left(\frac{1+\gamma_+}{p}-\frac{1}{1\wedge p}\right)_+\right\} \\
&=\max\left\{1,\frac{1+\gamma_+}{p}\right\}
=\max\left\{1,\frac{1}{p},\frac{1+\gamma}{p}\right\}
=J_{p,\gamma},
\end{align*}
where the second equality follows from the definition of $\gamma_+$.
Thus Theorem \ref{ad FJ-YY} directly gives the sufficiency.

We now prove the necessity.
Let $\mathbb A:=\{A_Q\}_{Q\in\mathcal D}$
be a sequence of reducing operators of order $p$ for $W$.
By the definition of reducing operators
and Lemma \ref{ballsim}(ii), we conclude that,
for any $Q:=Q_{j,k}\in\mathcal D$ and $\vec z\in\mathbb C^m$,
\begin{align*}
\left|A_Q\vec z\right|
\sim \left\{\fint_Q[\rho(x)]^\gamma\,dx\right\}^{\frac1p}|\vec z|
\sim b^{-j\frac\gamma p}[1+\rho(k)]^{\frac\gamma p}|\vec z|.
\end{align*}
Thus, we can set
$A_Q:=b^{-j\frac\gamma p}[1+\rho(k)]^{\frac\gamma p}I_m$.
Combining Theorem \ref{dj} with the assumption that
every $(D,E,F)$-almost diagonal operator
$B$ is bounded on $\dot b_{p,q}^{\alpha}(W)$, we conclude that
the specific operator $B^{D,E,F}$ (see Definition \ref{almdia})
is also bounded on $\dot b_{p,q}^{\alpha}(\mathbb A)$.

We begin by estimating $D$.
We first prove $D>\frac{1+\gamma}p$.
Fix $\vec e\in\mathbb C^m$ with $|\vec e|=1$.
Let $\vec s$ be nonzero only at $Q_{0,\mathbf0}$, with value $\vec e$.
Then $\|\vec s\|_{\dot b_{p,q}^{\alpha}(\mathbb A)}<\infty$, and hence
\begin{align*}
\infty
&>\left\|B^{D,E,F}\vec s\right\|_{\dot b_{p,q}^{\alpha}(\mathbb A)}^p
> \sum_{k\in\mathbb Z^n}|Q_{0,k}|^{1-\frac p2}
\left|A_{Q_{0,k}}(B^{D,E,F}\vec s)_{Q_{0,k}}\right|^p \notag\\
&=\sum_{k\in\mathbb Z^n}
[1+\rho(k)]^\gamma [1+\rho(k)]^{-Dp}
=\sum_{k\in\mathbb Z^n}[1+\rho(k)]^{-(Dp-\gamma)}.
\end{align*}
This, together with Lemma \ref{yl111505}(iv),
further implies $D>\frac{1+\gamma}{p}$.

We now prove the remaining restrictions on $D$.
For $N\in\mathbb N$, choose $k_N\in\mathbb Z^n$ such that
$\rho(k_N)> 2H b^N.$
For any $x\in\mathbb R^n$ and $r\in(0,\infty)$,
$Z_\rho(x,r):=\{k\in \mathbb Z^n:\ \rho(k-x) \leq r \}$.
Let $C_0$ be as in Lemma \ref{yl1101}(ii).
By this lemma, we conclude that,
for any $r\in [C_0,\infty)$, we have
\begin{equation} \label{noempty}
\# Z_\rho(x,r) \ge 1.
\end{equation}
Furthermore, the lattice-counting estimate implies that, for any $r\in [C_0,\infty)$,
\begin{equation}\label{number}
\# Z_\rho(x,r)\sim |B_\rho(x,r)| \sim r.
\end{equation}
By Lemma \ref{far} with $\lambda$ replaced by $1$,
we conclude that, for any $k\in Z_\rho(k_N, b^N)$,
\begin{align}\label{far-grid-comparable}
1+\rho(k)\sim1+\rho(k_N).
\end{align}
For any nonnegative family
$\{a_h\}_{h\in Z_\rho(k_N, b^N)}$, define
$\vec s_{Q_{0,h}}
:=a_hA_{Q_{0,h}}^{-1}\vec e$
for all $h\in Z_\rho(k_N, b^N)$,
and let all other coefficients vanish.
From the boundedness of $B^{D,E,F}$
and \eqref{far-grid-comparable}, we deduce that
\begin{align} \label{D-convolution-test}
\sum_{h\in Z_\rho(k_N, b^N)}a_h^p
&=\|\vec s\|_{\dot b_{p,q}^{\alpha}(\mathbb A)}^p
\gtrsim \left\|B^{D,E,F}\vec s\right\|_{\dot b_{p,q}^{\alpha}(\mathbb A)}^p \notag \\
&> \sum_{k\in Z_\rho(k_N, b^N)}|Q_{0,k}|^{1-\frac p2}
\left|A_{Q_{0,k}}(B^{D,E,F}\vec s)_{Q_{0,k}}\right|^p \notag \\
&= \sum_{k\in Z_\rho(k_N, b^N)}
\left\{ \sum_{h\in Z_\rho(k_N, b^N)}
\left[\frac{1+\rho(k)}{1+\rho(h)}\right]^{\frac\gamma p}
\frac{a_h}{[1+\rho(k-h)]^D} \right\}^p \notag \\
&\sim \sum_{k\in Z_\rho(k_N, b^N)}
\left\{ \sum_{h\in Z_\rho(k_N, b^N)}
\frac{a_h}{[1+\rho(k-h)]^D} \right\}^p.
\end{align}

Taking $a_{k_N}:=1$ and all other $a_h:=0$, and then letting
$N\to\infty$, we obtain
\begin{align*}
\sum_{h\in\mathbb Z^n}[1+\rho(h)]^{-Dp}<\infty.
\end{align*}
This, together with Lemma \ref{yl111505}(iv), further implies $D>\frac1p$.

Next take $a_h:=1$ for every $h\in Z_\rho(k_N, b^N)$. Then
\begin{equation} \label{7.10}
\sum_{h\in Z_\rho(k_N, b^N)} a_h^p =\# Z_\rho(k_N, b^N) \sim b^N.
\end{equation}
On the other hand,
\begin{align*}
\mathrm{RHS\ of}\ \eqref{D-convolution-test}
&= \sum_{k\in Z_\rho(\mathbf 0, b^N)}
\left\{ \sum_{h\in Z_\rho(\mathbf 0, b^N)}
\frac{1}{[1+\rho(k-h)]^D} \right\}^p \\
&= \sum_{k\in Z_\rho(\mathbf 0, b^N)}
\left\{ \sum_{h\in Z_\rho(k, b^N)}
\frac{1}{[1+\rho(h)]^D} \right\}^p.
\end{align*}
For any $k,h\in Z_\rho(\mathbf 0, (2H)^{-1}b^N)$,
$\rho(k-h) \le H[\rho(k)+\rho(h)] <b^N$,
and hence $h\in Z_\rho(k, b^N)$. Thus,
\begin{align*}
\mathrm{RHS\ of}\ \eqref{D-convolution-test}
&> \sum_{k\in Z_\rho(\mathbf 0, (2H)^{-1}b^N)}
\left\{ \sum_{h\in Z_\rho(\mathbf 0, (2H)^{-1}b^N)}
\frac{1}{[1+\rho(h)]^D} \right\}^p \\
&= \# Z_\rho(\mathbf 0, (2H)^{-1}b^N)
\left\{ \sum_{h\in Z_\rho(\mathbf 0, (2H)^{-1}b^N)}
\frac{1}{[1+\rho(h)]^D} \right\}^p \\
&\sim b^N
\left\{ \sum_{h\in Z_\rho(\mathbf 0, (2H)^{-1}b^N)}
\frac{1}{[1+\rho(h)]^D} \right\}^p.
\end{align*}
Substituting this estimate and \eqref{7.10} back into \eqref{D-convolution-test},
we obtain, for any $N\in\mathbb N$,
$$
\sum_{h\in Z_\rho(\mathbf 0, (2H)^{-1}b^N)}
\frac{1}{[1+\rho(h)]^D}\lesssim 1.
$$
This, together with Lemma \ref{yl111505}(iv), further implies $D>1$,
which completes the estimation of $D$.

Now, we estimate $E$ and $F$.
For any sequence $\vec s\in \dot b_{p,q}^{\alpha}(\mathbb A)$, write
\begin{align*}
\|\vec s\|_{\dot b_{p,q}^{\alpha}(\mathbb A)}
&= \left\| \left\{ \left( \sum_{Q\in\mathcal D_j}
b^{j\alpha p} |Q|^{1-\frac p2} |A_Q\vec s_Q|^p
\right)^{\frac1p} \right\}_{j\in\mathbb Z} \right\|_{\ell^q} \\
&= \left\| \left\{ \left( \sum_{k\in\mathbb Z^n}
b^{jp(\alpha +\frac 12- \frac{1+\gamma}{p})}
[1+\rho(k)]^{\gamma} |\vec s_{Q_{j,k}}|^p
\right)^{\frac1p} \right\}_{j\in\mathbb Z} \right\|_{\ell^q}
=:\|\{\mathcal N_j(\vec s)\}_{j\in\mathbb Z}\|_{\ell^q}.
\end{align*}
For any $N\in\mathbb N$ and $Q:=Q_{j,k}\in \mathcal D$, define
\begin{align*}
\vec s_Q^{(N)} :=
\begin{cases}
b^{-j(\alpha+\frac12)} \vec e
&\text{if } 0\le j\le N \text{ and } \rho(k)< 2C_0 H b^j,\\
\mathbf 0 &\text{otherwise}.
\end{cases}
\end{align*}
From \eqref{number}, it follows that,
for any $j\in\mathbb Z$ with $0\le j\le N$,
\begin{align} \label{7292}
\left[\mathcal N_j\left(\vec s^{(N)}\right)\right]^p
&=\sum_{k\in Z_\rho(\mathbf 0, 2C_0 H b^j)}
b^{-j(1+\gamma)} [1+\rho(k)]^{\gamma} \notag \\
&=b^{-j(1+\gamma)}
\left\{ \sum_{k\in Z_\rho(\mathbf 0, 2C_0 H)}
[1+\rho(k)]^{\gamma}
+ \sum_{i=1}^j \sum_{k\in Z_\rho(\mathbf 0, 2C_0 H b^i)
\setminus Z_\rho(\mathbf 0, 2C_0 H b^{i-1})}
[1+\rho(k)]^{\gamma} \right\} \notag \\
&\sim b^{-j(1+\gamma)}
\left[ 1 + \sum_{i=1}^j b^{i(1+\gamma)} \right]
\sim 1,
\end{align}
and hence
\begin{align}\label{block-input}
\left\|\vec s^{(N)}\right\|_{\dot b_{p,q}^{\alpha}(\mathbb A)}
\sim N^{\frac 1q}.
\end{align}
For any $j\in\mathbb Z$ with $0\le j\le N-1$
and any $k\in Z_\rho(\mathbf 0, C_0 b^j)$,
\begin{align*}
\left|\left( B^{D,E,F}\vec s^{(N)} \right)_{Q_{j,k}}\right|
&= \sum_{i=0}^j \sum_{h\in Z_\rho(\mathbf 0, 2C_0 H b^i)}
b^{(i-j)E} [1+b^i\rho(A^{-j}k-A^{-i}h)]^{-D} b^{-i(\alpha+\frac12)} \\
&\quad+ \sum_{i=j+1}^N \sum_{h\in Z_\rho(\mathbf 0, 2C_0 H b^i)}
b^{(j-i)F} [1+b^j\rho(A^{-j}k-A^{-i}h)]^{-D} b^{-i(\alpha+\frac12)} \\
&= \sum_{i=0}^j b^{(i-j)(E-\alpha-\frac12)} b^{-j(\alpha+\frac12)}
\sum_{h\in Z_\rho(\mathbf 0, 2C_0 H b^i)}
[1+\rho(A^{i-j}k-h)]^{-D} \\
&\quad+ \sum_{i=j+1}^N b^{(j-i)(F+\alpha+\frac12)} b^{-j(\alpha+\frac12)}
\sum_{h\in Z_\rho(\mathbf 0, 2C_0 H b^i)}
[1+b^{j-i}\rho(A^{i-j}k-h)]^{-D}.
\end{align*}
From Lemma \ref{yl111505}(ii), we deduce that, for any $i\in\{0,\ldots,j\}$,
$$
\sum_{h\in Z_\rho(\mathbf 0, 2C_0 H b^i)}
[1+\rho(A^{i-j}k-h)]^{-D}
\geq \sum_{h\in Z_\rho(A^{i-j}k, C_0)} \cdots
\sim 1
$$
and, for any $i\in\{j+1,\ldots,N\}$,
$$
\sum_{h\in Z_\rho(\mathbf 0, 2C_0 H b^i)}
[1+b^{j-i}\rho(A^{i-j}k-h)]^{-D}
\geq \sum_{h\in Z_\rho(A^{i-j}k, C_0 b^{i-j})} \cdots
\sim b^{i-j}.
$$
Thus,
\begin{align*}
\left|\left( B^{D,E,F}\vec s^{(N)} \right)_{Q_{j,k}}\right|
\gtrsim \sum_{i=0}^j b^{(i-j)(E-\alpha-\frac12)} b^{-j(\alpha+\frac12)}
+ \sum_{i=j+1}^N b^{(j-i)(F+\alpha+\frac12)} b^{-j(\alpha+\frac12)} b^{i-j},
\end{align*}
and hence
\begin{align*}
\left[ \mathcal N_j \left( B^{D,E,F}\vec s^{(N)} \right) \right]^p
&\gtrsim \left[ \sum_{i=0}^j b^{(i-j)(E-\alpha-\frac12)}
+ \sum_{i=j+1}^N b^{(j-i)(F+\alpha-\frac12)} \right]^p\\
&\quad\times
\sum_{k\in Z_\rho(\mathbf 0, C_0 b^j)} b^{-j (1+\gamma)}
[1+\rho(k)]^{\gamma}  \\
&\sim \left[ \sum_{i=0}^j b^{(i-j)(E-\alpha-\frac12)}
+ \sum_{i=j+1}^N b^{(j-i)(F+\alpha-\frac12)} \right]^p,
\end{align*}
where the last equivalence follows from \eqref{7292}.
If $E-\alpha-\frac12\leq 0$ or $F+\alpha-\frac12\leq 0$, then
$$
\left\|B^{D,E,F}\vec s^{(N)}\right\|_{\dot b_{p,q}^{\alpha}(\mathbb A)}
\gtrsim \left\|\{j+1\}_{j=0}^{N-1}\right\|_{\ell^q}
\sim N^{1+\frac1q}.
$$
Combining this with \eqref{block-input} and the boundedness of $B^{D,E,F}$, we obtain
$N^{1+\frac1q} \lesssim N^{\frac1q}$, which yields a contradiction.
Therefore,
\begin{align*}
E>\alpha+\frac12
\quad\text{and}\quad
F>\frac12-\alpha.
\end{align*}

It remains to prove the other two lower bounds of $F$.
Let $\{k_j\}_{j=0}^N$ be a sequence in $\mathbb Z^n$ to be determined.
For any $N\in\mathbb N$ and $Q:=Q_{j,k}\in \mathcal D$, define
\begin{align*}
\vec s_Q^{(N)} :=
\begin{cases}
b^{-j(\alpha+\frac12-\frac{1+\gamma}p)} [1+\rho(k)]^{-\frac\gamma p} \vec e
&\text{if } 0\le j\le N \text{ and } k=k_j,\\
\mathbf 0 &\text{otherwise}.
\end{cases}
\end{align*}
Then, for any $j\in\{0,\ldots,N\}$,
$\mathcal N_j(\vec s^{(N)})=1$.
Consequently,
\begin{align}\label{7.13}
\left\|\vec s^{(N)}\right\|_{\dot b_{p,q}^{\alpha}(\mathbb A)}
= (N+1)^{\frac1q}.
\end{align}
For any $j\in\{0,\ldots,N-1\}$,
\begin{align*}
\left|\left(B^{D,E,F}\vec s^{(N)}\right)_{Q_{j,k_j}}\right|
> \sum_{i=j+1}^N b^{(j-i)F}
[1+ b^j\rho(A^{-j}k_j-A^{-i}k_i)]^{-D}
b^{-i(\alpha+\frac12-\frac{1+\gamma}p)}
[1+\rho(k_i)]^{-\frac\gamma p},
\end{align*}
and hence
\begin{align} \label{7.13(2)}
\mathcal N_j\left(B^{D,E,F}\vec s^{(N)}\right)
> \sum_{i=j+1}^N b^{(j-i)
(F+\alpha+\frac12-\frac{1+\gamma}p)}
\left[ \frac{1+\rho(k_j)}{1+\rho(k_i)} \right]^{\frac{\gamma}{p}}
[1+ \rho(k_j-A^{j-i}k_i)]^{-D}.
\end{align}

Take $k_j:=\mathbf 0$ for all $j\in\{0,\ldots,N\}$.
If $F+\alpha+\frac12-\frac{1+\gamma}{p}\leq 0$,
then \eqref{7.13(2)} implies that
$$
\left\|B^{D,E,F}\vec s^{(N)}\right\|_{\dot b_{p,q}^{\alpha}(\mathbb A)}
> \left\|\{N-j\}_{j=0}^{N-1}\right\|_{\ell^q}
\sim N^{1+\frac1q}.
$$
Combining this with \eqref{7.13}
and the boundedness of $B^{D,E,F}$, we obtain
$N^{1+\frac1q} \lesssim (N+1)^{\frac1q}$, which yields a contradiction.
Therefore,
$$
F>\frac{1+\gamma}{p}-\frac12-\alpha.
$$

Let $x_0\in\mathbb R^n\setminus B_\rho(\mathbf 0,\max\{2HC_0,1\})$.
For any $j\in\{0,\ldots,N\}$, \eqref{noempty} ensures that
we can choose $k_j\in Z_\rho (A^j x_0, C_0)$.
Then Lemma \ref{far} implies that
$$
\frac{1+\rho(k_j)}{1+\rho(k_i)}
\sim \frac{1+\rho(A^j x_0)}{1+\rho(A^i x_0)}
= \frac{1+b^j\rho(x_0)}{1+b^i\rho(x_0)}
\sim b^{j-i}.
$$
By the fact that $i\ge j$ and the definition of $k_j$, we obtain
\begin{align*}
\rho(k_j-A^{j-i}k_i)
&\leq H[\rho(k_j-A^jx_0)+ \rho(A^jx_0-A^{j-i}k_i)] \\
&= H[\rho(k_j-A^jx_0)+ b^{j-i}\rho(A^ix_0-k_i)]
< 2 HC_0.
\end{align*}
Substituting these estimates into \eqref{7.13(2)}, we obtain
\begin{align*}
\mathcal N_j\left(B^{D,E,F}\vec s^{(N)}\right)
\gtrsim \sum_{i=j+1}^N b^{(j-i)(F+\alpha+\frac12-\frac1p)}.
\end{align*}
If $F+\alpha+\frac12-\frac{1}{p}\leq 0$, then
$$
\left\|B^{D,E,F}\vec s^{(N)}\right\|_{\dot b_{p,q}^{\alpha}(\mathbb A)}
> \left\|\{N-j\}_{j=0}^{N-1}\right\|_{\ell^q}
\sim N^{1+\frac1q}.
$$
Combining this with \eqref{7.13}
and the boundedness of $B^{D,E,F}$, we obtain
$N^{1+\frac1q} \lesssim (N+1)^{\frac1q}$, which yields a contradiction.
Therefore,
$$
F>\frac{1}{p}-\frac12-\alpha.
$$
This completes the proof of Theorem \ref{ad Besov sharp}.
\end{proof}

\subsection{Comparison with Existing Results}\label{compare}

The boundedness of almost diagonal operators on
Besov sequence spaces has been extensively studied
(see \cite[Theorem 6.2]{bf6}, \cite[Theorem 4.12]{bf2},
\cite[Theorem 4.5]{yyzong}, and \cite[Theorem 4.2]{gjabownik05}).
In this subsection, we show that the results obtained in
Subsection \ref{BADO} not only coincide with the existing sharp results,
but also improve the remaining ones.

We first compare the results in Euclidean spaces.
In this case, we assume that $A:=2I_n$ and $\rho(x):=|x|^n$.
The boundedness of almost diagonal operators on Besov sequence spaces with matrix $\mathcal A_p$ weights
was first studied by Frazier and Roudenko (see \cite{fr04,ro03}).
Recently, their results were improved by Bu et al. \cite[Theorem 6.2]{bf6},
which can be equivalently stated as follows.

\begin{myenv}\label{ad FJ-YY3}
Let $p\in (0,\infty)$, $q\in(0, \infty]$, $\alpha\in\mathbb{R}$,
and $W\in \mathcal A_{p}(\mathbb R^n,\mathbb C^m,2I_n)$.
Let $B$ be $(D, E, F)$-almost diagonal with parameters
\begin{equation} \label{sharp con1}
D>\frac{1}{1\wedge p},
\quad E>\frac{1}{2} + \alpha,
\quad \text{and}\quad
F>\frac{1}{1\wedge p}-\frac{1}{2}-\alpha.
\end{equation}
Then $B$ is bounded on $\dot b^{\alpha}_{p,q}(2I_n,W)$.
\end{myenv}


\begin{remark} \label{sharp1}
In the case $A:=2I_n$ and $\rho(x):=|x|^n$,
Corollary \ref{ad FJ-YY2} coincides with Theorem \ref{ad FJ-YY3}.
Moreover,
the condition \eqref{sharp con1} in Theorem \ref{ad FJ-YY3}
was shown to be sharp (see \cite[Theorem 7.1]{bf6}).
\end{remark}

Later, Bu et al. \cite[Theorem 4.12]{bf2} established the boundedness of
almost diagonal operators for a more general class of matrix weights.
As in Theorem \ref{ad FJ-YY3},
their result can be equivalently stated as follows; we omit the details.

\begin{myenv}\label{bualmost2}
Let $p\in (0,\infty)$, $q\in(0, \infty]$, $\alpha\in\mathbb{R}$,
and $W\in \mathcal A_{p,\infty}(\mathbb R^n, \mathbb C^m,2I_n)$.
Let $B$ be $(D, E, F)$-almost diagonal with parameters
\begin{equation}\label{sharp con2}
\begin{split}
&D>\frac{1}{1\wedge p}+\frac{d_{p,\infty}^{\mathrm{upper}}(W)}{p},
\quad E>\frac{1}{2} + \alpha, \\
&\text{and}\quad
F>\frac{1}{1\wedge p}-\frac{1}{2}-\alpha+\frac{d_{p,\infty}^{\mathrm{upper}}(W)}{p},
\end{split}
\end{equation}
Then $B$ is bounded on $\dot b^{\alpha}_{p,q}(2I_n,W)$.
\end{myenv}

The condition in Theorem \ref{ad FJ-YY} is
$$
D>\frac{1}{1\wedge p}+\mathrm{AT}(p,W),\quad
E>\frac{1}{2}+\alpha,
\quad\text{and}\quad
F>\frac{1}{1\wedge p}-\frac{1}{2}-\alpha+\mathrm{AT}(p,W),
$$
where $$\mathrm{AT}(p,W):=\min\left\{\frac{d_{p,\infty}^{\mathrm{upper}}(W)}{p},
\left(\frac1{v_{W,p}}-\frac1{1\wedge p}\right)_+\right\}
\leq\frac{d_{p,\infty}^{\mathrm{upper}}(W)}{p}.$$
To perform a finer comparison between Theorems \ref{ad FJ-YY} and \ref{bualmost2},
we need the following lemma.

\begin{lemma}\label{comparison}
If $W\in \mathcal A_{p,\infty}$,
then $d_{p,\infty}^{\mathrm{upper}}(W)\leq \frac{p}{v_{W,p}}-1$.
\end{lemma}

\begin{proof}
Let $v\in(0,p)$ satisfy
$\| \mathcal M_{W,p}^{(v)}\|_{L^p(\mathcal X,M_m(\mathbb C)) \to L^p(\mathcal X)} <\infty$.
It then follows from Theorem \ref{bounded B2} that
$W\in\mathcal A_{p,u}$, where $u:=\frac{pv}{p-v}$.
This, together with Lemma \ref{Apu}, further implies that
\begin{align*}
\sup_{B\in\mathcal B}
\fint_B \left\|A_B W^{-\frac1p}(y) \right\|^{u}\,dy<\infty.
\end{align*}
Using this, Lemma \ref{reduceM}, and Jensen's inequality, we obtain,
for any $\lambda \in[1, \infty)$ and any $B\in\mathcal B$,
\begin{align*}
&\exp \left(\fint_{B} \log \left(\fint_{\lambda B}\left\|W^{\frac{1}{p}}(x)
W^{-\frac{1}{p}}(y)\right\|^{p} \,dx\right) \,dy\right) \\
&\quad\sim \exp \left(\fint_{B} \log
\left\|A_{\lambda B} W^{-\frac{1}{p}}(y)\right\|^p \,dy\right)
= \left[\exp \left(\fint_{B} \log
\left\|A_{\lambda B} W^{-\frac{1}{p}}(y)\right\|^u \,dy\right) \right]^{\frac pu} \\
&\quad\leq \left[\fint_{B}
\left\|A_{\lambda B} W^{-\frac{1}{p}}(y)\right\|^u \,dy \right]^{\frac pu}
\lesssim \lambda^{\frac pu} \left[\fint_{\lambda B}
\left\|A_{\lambda B} W^{-\frac{1}{p}}(y)\right\|^u \,dy \right]^{\frac pu}
\lesssim \lambda^{\frac pu}.
\end{align*}
Therefore, $W$ has $\mathcal A_{p,\infty}$-upper dimension
$\frac pu=\frac pv-1$.
From the definitions of $d_{p,\infty}^{\mathrm{upper}}(W)$ and $v_{W,p}$,
we deduce that $d_{p,\infty}^{\mathrm{upper}}(W)\leq \frac{p}{v_{W,p}}-1$.
This completes the proof of Lemma \ref{comparison}.
\end{proof}

\begin{remark} \label{rem1}
If $p\in(0,1]$, then Theorem \ref{ad FJ-YY} coincides with Theorem \ref{bualmost2}.
Indeed, Lemma \ref{comparison} implies that
$$
\left( \frac1{v_{W,p}}-\frac1{1\wedge p}\right)_+
= \frac1{v_{W,p}}-\frac1{p}
\geq \frac{d_{p,\infty}^{\mathrm{upper}}(W)}{p},
$$
and hence $\mathrm{AT}(p,W)=\frac{d_{p,\infty}^{\mathrm{upper}}(W)}{p}$.
Thus, the conditions in Theorems \ref{ad FJ-YY} and \ref{bualmost2} coincide.
Moreover, for the boundedness of almost diagonal operators
with $\mathcal A_{p,\infty}$-matrix weights and $p\in(0,1]$,
the condition \eqref{sharp con2} in Theorem \ref{bualmost2}
was shown to be sharp (see \cite[Lemma 4.13]{bf2}).

If $p\in(1,\infty)$, then Theorem \ref{ad FJ-YY} improves Theorem \ref{bualmost2}.
Let $a\in(0,\infty)$ and $W(\cdot):=[\rho(\cdot)]^{a} I_m$.
From Lemmas \ref{scamtr} and \ref{rhoinf}, it follows that
$W\in \mathcal A_{p,\infty}(\mathbb R^n,\mathbb C^m,2I_n)$
satisfies $d_{p,\infty}^{\mathrm{upper}}(W)=a$.
By Lemma \ref{value}, we conclude that $v_{W,p}= \frac p{1+a}$.
These further imply that
$$
\left( \frac1{v_{W,p}}-\frac1{1\wedge p}\right)_+
= \left( \frac{1+a}p - 1 \right)_+
< \frac{a}p=\frac{d_{p,\infty}^{\mathrm{upper}}(W)}{p}.
$$
Thus, the condition in Theorem \ref{ad FJ-YY}
is strictly weaker than that in Theorem \ref{bualmost2}.

When $p\in(1,\infty)$, $(\frac1{v_{W,p}}-\frac1{1\wedge p})_+
< \frac{d_{p,\infty}^{\mathrm{upper}}(W)}{p}$
for all power weights;
however, for some special weights this strict inequality may be reversed.
For example, let $n\geq 2$, $a\in((p-1)\frac n{n-1},\infty)$, and, for any $x:=(x_1,\ldots,x_n)$,
define $W(x):=|x_1|^{na} I_m$.
From \cite[Lemma 4.30]{bf2} and Proposition \ref{prop of sp(w)}(iv),
we deduce that $W\in \mathcal A_{p,\infty}(\mathbb R^n,\mathbb C^m,2I_n)$
satisfies $v_{W,p}=\frac p{1+a}$ and
$d_{p,\infty}^{\mathrm{upper}}(W)=\frac an$. Thus, by $a>(p-1)\frac n{n-1}$, we have
\begin{align*}
\left( \frac1{v_{W,p}}-\frac1{1\wedge p}\right) _+
=  \frac{1+a}p -1
> \frac a{np}
= \frac{d_{p,\infty}^{\mathrm{upper}}(W)}{p}.
\end{align*}
\end{remark}

Very recently, Yang et al. \cite[Theorem 4.5]{yyzong} established the boundedness of
almost diagonal operators for variable matrix weights.
For the convenience of comparison, we recall their result in the constant exponent case.
Let $p\in(0,\infty)$, $W$ be a matrix weight, and
\begin{align*}
\gamma_W := \sup\left\{v \in (0,1]:\
\sup_{j\in\mathbb Z} \left\| \eta^{(v)}_{j,u,W} \right\|_{L^p(\mathbb R^n, \mathbb C^m)\to L^p(\mathbb R^n)} <\infty
\text{ for some } u\in (0,\infty) \right\},
\end{align*}
where, for any $x\in\mathbb R^n$,
\begin{align*}
\eta_{j,u,W}^{(v)}\left(\vec{f}\right)(x) :=
\left[\int_{\mathbb{R}^n} \frac{2^{jn}|W^{\frac1p}(x)W^{-\frac1p}(y)\vec{f}(y)|^{v}}{(1 + 2^j|x-y|)^{v u}}\,dy\right]^\frac{1}{v}.
\end{align*}
Although \cite[Theorem 4.5]{yyzong} established the boundedness of
almost diagonal operators on inhomogeneous Besov sequence spaces,
in the case $p(\cdot)\equiv p$, $q(\cdot)\equiv q$,
and $\alpha(\cdot)\equiv \alpha$,
we can readily adapt the argument used in \cite[Theorem 4.5]{yyzong}
to obtain the boundedness on homogeneous Besov sequence spaces.
As in Theorem \ref{ad FJ-YY3},
this result can be equivalently stated as follows;
we omit the details.

\begin{myenv}\label{zzzalmost}
Let $\alpha\in\mathbb{R}$, $p\in (0,\infty)$, $q\in(0, \infty]$,
and $W\in \mathcal A_{p,\infty}(\mathbb R^n, \mathbb C^m,2I_n)$.
Assume that $B$ is $(D, E, F)$-almost diagonal with parameters
$$
D>\frac{1}{\gamma_W},
\quad E>\frac{1}{2} + \alpha,
\quad \text{and}\quad
F>\frac{1}{\gamma_W}-\frac12-\alpha.
$$
Then $B$ is bounded on $\dot b^{\alpha}_{p,q}(2I_n,W)$.
\end{myenv}


To compare Theorem \ref{zzzalmost} with the results obtained in
Subsection \ref{BADO},
further properties of $\gamma_W$ are required and stated as follows.

\begin{lemma} \label{sim}
Let $W$ be a matrix weight, $p\in(0,\infty)$, $j\in\mathbb Z$,
and $u,v\in(0,\infty)$ satisfy $uv>n$.
Assume that $\rho(\cdot):=|\cdot|^n$ and $A:=2I_n$.
For any $\vec{f} \in L^1_{\rm loc}(\mathbb R^n, \mathbb C^m)$ and $x\in\mathbb R^n$,
$$
\mathcal{M}_{W,p}^{(v,2^{-jn})} \left(\vec{f}\right) (x)
\lesssim \eta_{j,u,W}^{(v)}\left(\vec{f}\right)(x)
\lesssim \mathcal{M}_{W,p}^{(v)} \left(\vec{f}\right) (x),
$$
where the implicit positive constants depend only on $u$ and $v$.
\end{lemma}

\begin{proof}
By \eqref{B_k}, we find that, for any $x\in\mathbb R^n$ and any ball $B\in\mathcal B$ with $x\in B$ and $r_B=2^{-jn}$,
\begin{align*}
\fint_B \left|W^{\frac1p}(x)W^{-\frac1p}(y)\vec{f}(y)\right|^{v} \,dy
\sim\int_B \frac{2^{jn}|W^{\frac1p}(x)W^{-\frac1p}(y)\vec{f}(y)|^{v}}{(1 + 2^j |x-y|)^{v u}}\,dy
\leq \left[\eta_{j,u,W}^{(v)}\left(\vec{f}\right)(x)\right]^v,
\end{align*}
and hence $\mathcal{M}_{W,p}^{(v,2^{-jn})} (\vec{f}) (x)
\lesssim \eta_{j,u,W}^{(v)}(\vec{f})(x)$.

For any $x\in\mathbb R^n$,
\begin{align*}
\left[\eta_{j,u,W}^{(v)}\left(\vec{f}\right)(x)\right]^v
&=\int_{B_{\rho}(x,2^{-jn})} \frac{2^{jn}|W^{\frac1p}(x)W^{-\frac1p}(y)\vec{f}(y)|^{v}}{(1 + 2^j |x-y|)^{v u}}\,dy
+\sum_{l=1}^\infty \int_{B_{\rho}(x,2^{(-j+l)n})\setminus B_{\rho}(x,2^{(-j+l-1)n})} \cdots \\
&\lesssim \fint_{B_{\rho}(x,2^{-jn})} \left|W^{\frac1p}(x)W^{-\frac1p}(y)\vec{f}(y)\right|^{v} \,dy
+\sum_{l=1}^\infty 2^{l(n-vu)} \fint_{B_{\rho}(x,2^{(-j+l)n})} \cdots \\
&\leq \left[\mathcal{M}_{W,p}^{(v)} \left(\vec{f}\right) (x) \right]^v \sum_{l=0}^\infty 2^{l(n-vu)}
\sim \left[\mathcal{M}_{W,p}^{(v)} \left(\vec{f}\right) (x) \right]^v.
\end{align*}
This completes the proof of Lemma \ref{sim}.
\end{proof}

\begin{lemma} \label{relation}
If $p\in(0,\infty)$ and $W\in \mathcal A_{p,\infty}$,
then $\gamma_W= 1\wedge v_{W,p}$.
\end{lemma}

\begin{proof}
By Lemma \ref{sim}, we conclude that, for any $v\in (0,1\wedge v_{W,p})$,
if $u\in (\frac nv,\infty)$, then
$$
\sup_{j\in\mathbb Z} \left\| \eta^{(v)}_{j,u,W} \right\|_{L^p(\mathbb R^n, \mathbb C^m)\to L^p(\mathbb R^n)}
\lesssim \left\| \mathcal{M}_{W,p}^{(v)} \right\|_{L^p(\mathbb R^n, \mathbb C^m)\to L^p(\mathbb R^n)} < \infty,
$$
and hence $\gamma_W\geq v$.
Since $v\in (0,1\wedge v_{W,p})$ is arbitrary,
we obtain $\gamma_W\geq 1\wedge v_{W,p}$.

From Lemma \ref{sim} again, we deduce that, for any $v\in (0,\gamma_W)$,
\begin{align*}
\sup_{r\in(0,\infty)} \left\| \mathcal{M}_{W,p}^{(v,r)} \right\|_{L^p(\mathbb R^n, \mathbb C^m)\to L^p(\mathbb R^n)}
&\sim \sup_{j\in\mathbb Z} \left\| \mathcal{M}_{W,p}^{(v,2^{-jn})} \right\|_{L^p(\mathbb R^n, \mathbb C^m)\to L^p(\mathbb R^n)} \\
&\lesssim \sup_{j\in\mathbb Z} \left\| \eta^{(v)}_{j,u,W} \right\|_{L^p(\mathbb R^n, \mathbb C^m)\to L^p(\mathbb R^n)}
<\infty,
\end{align*}
which, together with Theorem \ref{bounded B2}, further implies that
$\| \mathcal{M}_{W,p}^{(v)} \|_{L^p(\mathbb R^n, \mathbb C^m)\to L^p(\mathbb R^n)}<\infty$,
and hence $v\leq v_{W,p}$.
Since $v\in (0,\gamma_W)$ is arbitrary and $\gamma_W\in(0,1]$,
we obtain $\gamma_W\leq 1\wedge v_{W,p}$.
This completes the proof of Lemma \ref{relation}.
\end{proof}

\begin{remark}
Theorem \ref{ad FJ-YY} improves Theorem \ref{zzzalmost}.
Indeed, using Lemma \ref{relation}, the condition in Theorem \ref{zzzalmost} can be equivalently rewritten as
$$ D>\frac{1}{1\wedge v_{W,p}},
\quad E>\frac{1}{2} + \alpha,
\quad \text{and}\quad
F>\frac{1}{1\wedge v_{W,p}}-\frac12-\alpha.
$$
By comparison, the condition in Theorem \ref{ad FJ-YY} is
$$
D>\frac{1}{1\wedge p}+\mathrm{AT}(p,W),\quad
E>\frac{1}{2}+\alpha,
\quad\text{and}\quad
F>\frac{1}{1\wedge p}-\frac{1}{2}-\alpha+\mathrm{AT}(p,W),
$$
where $\mathrm{AT}(p,W):=\min\{\frac{d_{p,\infty}^{\mathrm{upper}}(W)}{p},
(\frac1{v_{W,p}}-\frac1{1\wedge p})_+\}$.
Note that
\begin{equation*}
\frac{1}{1\wedge p}+\mathrm{AT}(p,W)
\leq \frac{1}{1\wedge p}+\left( \frac1{1\wedge v_{W,p}}-\frac1{1\wedge p}\right)
=\frac1{1\wedge v_{W,p}}.
\end{equation*}
Moreover, the example in Remark \ref{rem1} shows that this inequality is strict for some $W$.
Thus, the condition in Theorem \ref{ad FJ-YY}
is strictly weaker than that in Theorem \ref{zzzalmost}.
\end{remark}

Finally, we present a comparison between the results obtained in
Subsection \ref{BADO} and those of Bownik \cite[Theorem 4.2]{gjabownik05}
on their common intersection of scalar-weighted anisotropic Besov spaces.
Now $A$ and $\rho$ are no longer required to take special values.
For any $w\in A_\infty$, let
\begin{align} \label{betaw}
\beta_w:=\sup_{B\in\mathcal B} \log_b \frac{w(bB)}{w(B)}.
\end{align}
Then, for any $B\in\mathcal B$,
\begin{align} \label{betaw2}
w(bB)\leq b^{\beta_w} w(B),
\end{align}
which, together with \eqref{B_k}, further implies that,
for any $k\in\mathbb Z$ and $x\in\mathbb R^n$,
$$
\fint_{x+B_{k+1}} w(y)\,dy
\leq b^{\beta_w-1} \fint_{x+B_k} w(y)\,dy.
$$
Combining this with Lebesgue's differentiation theorem
(see, for instance, \cite[Lemma 2.4]{acm15}),
we obtain $1\leq b^{\beta_w-1}$, and hence $\beta_w\in[1,\infty)$.
The following theorem is a special case of \cite[Theorem 4.2]{gjabownik05} with $d\mu(x)=w(x)dx$.

\begin{myenv}\label{bownikalmost}
Let $\alpha\in\mathbb{R}$, $p\in(0,\infty)$, $q\in(0,\infty]$, and $w\in A_{\infty}$.
Assume that $B$ is $(D, E, F)$-almost diagonal with parameters
$$
D> \widetilde{J},
\quad E>\frac{1}{2} + \alpha,
\quad \text{and}\quad
F> \widetilde{J}-\frac{1}{2}-\alpha,
$$
where
\begin{equation}\label{tildeJ}
\widetilde{J}:= \frac{\beta_w}{p}+\left(1-\frac1p\right)_+
\end{equation}
with $\beta_w$ as in \eqref{betaw}.
Then $B$ is bounded on $\dot b^{\alpha}_{p,q}(w)$.
\end{myenv}

When $m=1$, the term $d_{p,\infty}^{\mathrm{upper}}(W)$ is independent of $p$
and is simply denoted by $d_{\infty}^{\mathrm{upper}}(W)$.

\begin{lemma}\label{dle}
If $w\in A_{\infty}$, then $d_{\infty}^{\mathrm {upper}}(w)\le\beta_w-1$,
where $\beta_w$ is as in \eqref{betaw}.
\end{lemma}

\begin{proof}
For any $\lambda\in[1,\infty)$, there exists $k\in\mathbb Z_+$ such that
$b^k\leq \lambda <b^{k+1}$.
From \eqref{betaw2}, it follows that, for any $B\in \mathcal B$,
\begin{align*}
w(\lambda B)
\leq w(b^{k+1} B)
\leq b^{\beta_w(k+1)} w(B)
\leq \lambda^{\beta_w} b^{\beta_w} w(B).
\end{align*}
Using this, \eqref{ballmea}, and $w\in A_\infty$, we conclude that
\begin{align*}
\fint_{\lambda B}w(x)\,dx
&\lesssim \lambda^{\beta_w} \frac{|B|}{|\lambda B|} \fint_{B}w(x)\,dx
\sim \lambda^{\beta_w-1} \fint_{B} w(x)\,dx\\
&\lesssim \lambda^{\beta_w-1}
\exp\left(\fint_{B}\log w(x)\,dx\right).
\end{align*}
Therefore,
\begin{align*}
\fint_{\lambda B}w(x)\,dx \exp\left(\fint_B
\log\left([w(x)]^{-1}\right)\,dx\right)
\lesssim \lambda^{\beta_w-1},
\end{align*}
and hence $w$ has $A_{\infty}$-upper dimension $\beta_w-1$.
This, together with the definition of $d_{\infty}^{\mathrm {upper}}(w)$,
then completes the proof of Lemma \ref{dle}.
\end{proof}

\begin{remark} \label{comparebow}
Theorem \ref{ad FJ-YY} improves Theorem \ref{bownikalmost}.
Indeed, from Lemma \ref{dle}, we infer that
\begin{align*}
J
\leq \frac{1}{1\wedge p}+\frac{d_{\infty}^{\mathrm{upper}}(w)}{p}
\leq \frac{1}{1\wedge p}+\frac{\beta_w-1}{p}
= \widetilde{J}.
\end{align*}
Moreover, the example in Remark \ref{rem1} shows that this inequality is strict for some scalar weights.
Thus, the condition in Theorem \ref{ad FJ-YY}
is strictly weaker than that in Theorem \ref{bownikalmost}.
\end{remark}

\section{Molecular Characterization}
\label{f z}

In this section, we establish the molecular characterization of
$\dot B^{\alpha}_{p,q}(W)$.
We begin by recalling the concept of smooth molecule.
The study of smooth molecules with the usual dyadic dilation
was initiated in \cite{FJ90,ck110301,ck110302},
and later extended to the anisotropic setting in \cite{gjabownik05,gjabownik06}.
Here we adopt the concept introduced in \cite[Definition 3.4]{bf5}.

\begin{definition}
Let $K, M\in [0,\infty)$ and $L,N\in\mathbb{R}$. For any $j\in\mathbb Z$ and
$Q\in\mathcal D_j$,
a function $m_{Q}$ is called a \emph{$(K,L,M,N)$-molecule on a cube $Q$} if,
for any $x\in\mathbb{R}^n$ and $\gamma\in\mathbb{Z}_+^{n}$,
\begin{align*}
|m_Q(x)|\leq b^{\frac{j}{2}}\left[1+b^j\rho(x-x_Q)\right]^{-K},
\end{align*}
\begin{align*}
\int_{\mathbb{R}^n}x^\gamma m_Q(x)\,dx=0\quad\mathrm{if}\quad|\gamma|\leq L,
\end{align*}
and
\begin{align}\label{last}
|\partial^\gamma m_Q(A^{-j}\cdot)(x)|
\leq b^{\frac{j}{2}}\left[1+\rho(x-A^jx_Q)\right]^{-M}
\quad\mathrm{if}\quad|\gamma|\leq N.
\end{align}
\end{definition}

\begin{remark}
Any condition on  $\gamma$ with a negative length is void.
Moreover, to avoid any ambiguity, \eqref{last}
requires $m_Q$ to have continuous partial derivatives of order $N$.
\end{remark}

Before establishing the molecular characterization of
$\dot B^{\alpha}_{p,q}(W)$, we need several technical lemmas.
The following lemma is precisely \cite[Lemma 3.2]{Bownik}.

\begin{lemma}\label{edbj}
Let $\lambda_{-}$ and $\lambda_{+}$ be two  constants such that
\begin{align*}
1<\lambda_{-}<\min\{|\lambda|:\ \lambda\in\sigma(A)\}
\leq\max\{|\lambda|:\ \lambda\in\sigma(A)\}<\lambda_{+}.
\end{align*}
Then there exists a positive constant $C$ such that,
for any $x\in\mathbb{R}^n$ with $\rho(x)>1$,
\begin{align*}
C^{-1}[\rho(x)]^{\zeta_-}
&\leq |x|\leq C[\rho(x)]^{\zeta_+}
\end{align*}
and, for any $x\in\mathbb{R}^n$ with $\rho(x)\leq1$,
\begin{align*}
C^{-1}[\rho(x)]^{\zeta_+}
&\leq |x|\leq C[\rho(x)]^{\zeta_-},
\end{align*}
where
$\zeta_-:=\frac{\log\lambda_-}{\log b}$ and
$\zeta_+:=\frac{\log\lambda_+}{\log b}$.
\end{lemma}

For any $k\in\mathbb N$, let $C^k$ be the set of all
$k$-times continuously differentiable functions on $\mathbb R^n$;
for any $k\in\mathbb Z\setminus\mathbb N$, let $C^k:=L^1$.
The following two results are precisely \cite[Lemmas 6.3 and 6.4]{gjabownik06}.

\begin{lemma}\label{fj B1}
Assume that $E\in\mathbb Z_+$, $D>1$, $F>1+E\zeta_+$,
$j,k\in\mathbb Z$ with $k\geq j$,
and $x_0 \in \mathbb{R}^n $.
Suppose that $g\in C^{E+1}$ and $h\in L^1$ satisfy,
for any $ x\in \mathbb{R}^n $,
\begin{align*}
\left| \partial^\gamma g\left(A^{-j}\cdot\right)(x) \right|
\leq b^{\frac{j}{2}}
\left[ 1+\rho(x) \right] ^{-D}\quad\text{if}\quad |\gamma|\leq E+1,
\end{align*}
\begin{align*}
|h(x)|\leq b^{\frac{k}{2}} \left[ 1+b^k\rho(x-x_0) \right]^{-(D\vee F)},
\end{align*}
and
\begin{align*}
\int_{\mathbb R^n}x^\gamma h(x) \, dx=0\quad\text{if}\quad|\gamma|\leq E.
\end{align*}
Then, for any $\theta\in(0,1]$ satisfying
$\left(E+\theta\right)\zeta_-+1<F$,
there exists a positive constant $ C $,
independent of $g$, $h$, $j$, $k$, and $x_0$,
such that, for any $ x \in \mathbb{R}^n $,
\begin{align*}
|g*h(x)| \leq C b^{-(k-j)[(E+\theta)\zeta_-+\frac{1}{2}]}
\left[ 1+b^j\rho(x-x_0) \right]^{-D}.
\end{align*}
\end{lemma}

\begin{lemma}\label{fj B2}
Let  $D\in(1,\infty)$, $j,k\in\mathbb Z$ with $k\geq j$, and
$x_0 \in \mathbb{R}^n $. Suppose that $g,h\in L^1$ satisfy,
for any $ x \in \mathbb{R}^n $,
\begin{align*}
|g(x)|\leq b^{\frac{j}{2}} \left[1+b^j\rho(x) \right]^{-D}
\quad\text{and}\quad
|h(x)|\leq b^{\frac{k}{2}} \left[ 1+b^k\rho(x-x_0) \right]^{-D}.
\end{align*}
Then there exists a positive constant $ C $,
independent of $g$, $h$, $j$, $k$, and $x_0$,
such that, for any $ x \in \mathbb{R}^n $,
\begin{align*}
|g*h(x)| \leq C b^{-\frac{k-j}{2}} \left[1+b^j\rho(x-x_0) \right]^{-D}.
\end{align*}
\end{lemma}

For any $r\in\mathbb R$, let
$\lfloor r\rfloor:=\max\{k\in\mathbb Z:\ k\leq r\}$ and
$\lceil r\rceil:=\min\{k\in\mathbb Z:\ k\geq r\}$.

Motivated by \cite[Lemma 3.7]{bf5}, we have the following lemma.

\begin{lemma}\label{mole1}
Let  $\zeta_-$ and $\zeta_+$ be as in Lemma \ref{edbj}.
Suppose that $m_{Q}$ is a $(K_{m},L_{m},M_{m},N_{m})$-molecule on a cube $Q$
and  that $b_P$ is a $(K_b,L_b,M_b,N_b)$-molecule on a cube $P$,
where $K_m,M_m,K_b,M_b\in(1,\infty)$ and $L_m,N_m,L_b,N_b\in\mathbb R$.
Then, for any $\varepsilon\in(0,\infty)$, there exists a positive constant $C$,
independent of $Q$ and $P$, such that
\begin{equation} \label{aim}
|\langle m_{Q},b_{P}\rangle|\leq Cb_{Q,P}^{M,G,L},
\end{equation}
where $b_{Q,P}^{M,G,L}$ is as in \eqref{BDEF} with
$M:=K_m\wedge M_m\wedge K_b\wedge M_b\in(1,\infty)$,
$$
G := \frac{1}{2} + \zeta_- \left[\min\left\{N_b-1, L_{m},
\frac{K_{m}-1-\varepsilon}{\zeta_+}\right\}\right]_+,
$$
and
$$
L := \frac{1}{2} + \zeta_-\left[\min\left\{N_m-1, L_{b},
\frac{K_{b}-1-\varepsilon}{\zeta_+}\right\}\right]_+.
$$
\end{lemma}

\begin{proof}
By symmetry, we only need to consider the case where
$|Q|=b^{-j}\geq b^{-k}=|P|$.
Let $g(\cdot):=m_Q(x_Q-\cdot)$ and $h(\cdot):=\overline{b_P(x_P+\cdot)}$. Then
\begin{align}\label{convo}
\langle m_{Q},b_{P}\rangle
&=\int_{\mathbb{R}^{n}}g(x_{Q}-y)h(y-x_{P})dy\notag\\
&=\int_{\mathbb{R}^{n}}g(x)h(x_{Q}-x_{P}-x)dx=(g*h)(x_{Q}-x_{P}).
\end{align}
Observe that $g$ and $h$ satisfy the assumptions of
Lemma \ref{fj B2} with $x_0=\mathbf0$ and $D=K_m\wedge K_b\in(1,\infty)$.
Let $\varepsilon\in(0,\infty)$.
By \eqref{convo} and Lemma \ref{fj B2}, we have
\begin{align*}
|\langle m_{Q},b_{P}\rangle|
&\lesssim b^{-\frac{k-j}{2}}\left[1+b^j\rho(x_Q-x_P) \right]^{-(K_m\wedge K_b)}\notag\\
&\leq\left(\frac{|P|}{|Q|}\right)^{\frac{1}{2}}
\left[1+\frac{\rho(x_{Q}-x_{P})}{|Q|}\right]
^{-(K_{m}\wedge M_{m}\wedge K_{b}\wedge M_{b})}.
\end{align*}
This completes the proof of \eqref{aim} in the case
where $N_m\le 1$ or $L_b\leq 0$ or $K_b\leq 1+\varepsilon$.

Next, we consider the case where $N_m>1$, $L_b>0$, and $K_b>1+\varepsilon$.
By the definitions of $g$ and $h$, we find that,
for any $\gamma\in\mathbb Z_+^n$ and $x\in\mathbb R^n$,
\begin{align*}
\left| \partial^\gamma g\left(A^{-j}\cdot\right)(x) \right|
\leq b^{\frac{j}{2}}
\left[ 1+\rho(x) \right]^{-(K_m\wedge M_m)}
\quad\text{if}\quad |\gamma|\leq N_m,
\end{align*}
\begin{align*}
|h(x)|\leq b^{\frac{k}{2}} \left[ 1+b^k\rho(x) \right]^{-K_b},
\end{align*}
and
\begin{align*}
\int_{\mathbb R^n}x^\gamma h(x) \, dx=0\quad\text{if}\quad|\gamma|\leq L_b.
\end{align*}
Thus, for any $D,E,F\in\mathbb R$ with
$D\in (1, K_m\wedge M_m\wedge K_b]$,
$E\in\mathbb Z_+\cap(-\infty, (N_m-1)\wedge L_b]$,
and $F\in(1+E\zeta_+, K_b]$,
both $g$ and $h$ satisfy all conditions of Lemma \ref{fj B1},
and this lemma implies that,
for any $\theta\in(0,1]$ with
$\left(E+\theta\right)\zeta_-+1<F$
and any $ x \in \mathbb{R}^n $,
\begin{align} \label{6.21}
|g*h(x)| \lesssim b^{-(k-j)[(E+\theta)\zeta_-+\frac{1}{2}]}
\left[ 1+b^j\rho(x) \right]^{-D}.
\end{align}
Let $D:=K_{m}\wedge M_{m}\wedge K_{b}\wedge M_{b}$,
$\varepsilon\in(0,\infty)$,
$$
s:=\min\left\{N_m-1, L_b,
\frac{K_b-1-\varepsilon}{\zeta_+}\right\}, \quad
E:=\lceil s\rceil -1,
\quad
\theta:= s - E,
$$
and $F:=K_b$. Then the interval $(1+E\zeta_+, K_b]$ is well defined,
and $F$ belongs to this interval.
Substituting $D$, $E$, and $\theta$ into \eqref{6.21},
and combining with \eqref{convo}, we obtain \eqref{aim}.
This completes the proof of Lemma \ref{mole1}.
\end{proof}

\begin{theorem}\label{molecule}
Let $\alpha\in \mathbb{R}$, $p\in ( 0, \infty )$, $q\in ( 0, \infty ]$, and
$W\in \mathcal A_{p,\infty}$.
Let $J$ be as in \eqref{J} and  $\zeta_-$ and $\zeta_+$  as in Lemma \ref{edbj}.
Suppose that
\begin{equation}\label{mm}
\begin{split}
&K_m> \max\left\{J, \frac{\alpha\zeta_+}{\zeta_-} + 1\right\},\quad
L_m>\frac{\alpha}{\zeta_-},\\
&M_m>J,\quad
\begin{cases}
N_m>\frac{J-1-\alpha}{\zeta_-}+1 &\text{if } \alpha\le J-1,\\
N_m=-1 &\text{if } \alpha> J-1
\end{cases}
\end{split}
\end{equation}
and
\begin{equation}\label{mb}
\begin{split}
&K_b>\max\left\{J,\frac{(J-1-\alpha)\zeta_+}{\zeta_-} +1\right\},\quad
L_b>\frac{J-1-\alpha}{\zeta_-},\\
&M_b>J,\quad
\begin{cases}
N_b=-1 &\text{if } \alpha<0,\\
N_b>\frac{\alpha}{\zeta_-}+1 &\text{if } \alpha\ge 0.
\end{cases}
\end{split}
\end{equation}
If $\{m_Q\}_{Q\in\mathcal{D}}$
is a family of $(K_{m}, L_{m}, M_{m}, N_{m})$-molecules
and $\left\{b_{P}\right\}_{P\in \mathcal{D}}$ is a family of
$( K_{b}, L_{b}, M_{b}, N_{b})$-molecules,
then the infinite matrix
$\{\langle m_Q, b_P\rangle\}_{Q, P\in \mathcal{D}}$
is $\dot b_{p, q}^{\alpha}(W)$-almost diagonal.
\end{theorem}

\begin{proof}
By considering the ranges of $\alpha$, i.e., $\alpha<0$, $0\le \alpha\le J-1$, and $\alpha>J-1$,
we find that there exists $\varepsilon\in(0,\infty)$ such that
\begin{align*}
K_{b}\wedge M_{b}\wedge K_{m}\wedge M_{m}
> J,
\end{align*}
\begin{align*}
\zeta_- \left[\min\left\{N_b-1, L_{m},
\frac{K_{m}-1-\varepsilon}{\zeta_+}\right\}\right]_+
>\alpha,
\end{align*}
and
\begin{align*}
\zeta_-\left[\min\left\{N_m-1, L_{b},
\frac{K_{b}-1-\varepsilon}{\zeta_+}\right\}\right]_+
> J-1-\alpha.
\end{align*}
This, together with Lemma \ref{mole1}, further implies that
$\{\langle m_Q, b_P\rangle\}_{Q, P\in \mathcal{D}}$
is $\dot b_{p, q}^{\alpha}(W)$-almost diagonal,
which completes the proof of Theorem \ref{molecule}.
\end{proof}

\begin{definition}\label{moleculedef}
Let $\alpha\in\mathbb{R}, p\in(0,\infty), q\in(0,\infty]$,
and $W\in \mathcal A_{p,\infty}$. Then
\begin{enumerate}[(i)]
\item a $(K_m,L_m,M_m,N_m)$-molecule on a cube $Q$ is called
a \emph{$\dot B_{p,q}^{\alpha}(W)$-analysis molecule on $Q$}
if $K_m,L_m,M_m,N_m$ satisfy \eqref{mm};

\item a $(K_b,L_b,M_b,N_b)$-molecule on a cube $Q$
is called a \emph{$\dot B_{p,q}^{\alpha}(W)$-synthesis molecule on $Q$}
if $K_b,L_b,M_b,N_b$ satisfy \eqref{mb}.
\end{enumerate}

A family of functions $\{m_Q\}_{Q\in\mathcal{D}}$ is called
\emph {a family of $\dot B_{p,q}^{\alpha}(W)$-analysis molecules
(resp. synthesis molecules)} if
there exist $K,L,M,N\in\mathbb{R}$ satisfying \eqref{mm}
[resp. \eqref{mb}] such that, for any $Q\in\mathcal{D}$,
$m_Q$ is a $(K,L,M,N)$-molecule on the respective cube $Q$.
\end{definition}

\begin{corollary}\label{commolecule}
Let $\alpha\in \mathbb{R}$, $p\in(0,\infty)$, $q\in(0,\infty]$,
and $W\in$ $\mathcal A_{p,\infty }$. Assume that $\varphi, \psi\in \mathcal{S}$
satisfy \eqref{hs2}. Suppose that for both $i\in \{1, 2\}$,
$\{m_{Q}^{(i)}\}_{Q\in \mathcal{D}}$ are families of
$\dot B_{p,q}^{\alpha}(W)$-analysis molecules and
$\{b_Q^{(i) }\}_{Q\in \mathcal D}$ are families of
$\dot B_{p, q}^{\alpha}(W)$-synthesis molecules. Then
\begin{enumerate}[\rm(i)]
\item the matrices $\{\langle m_{P}^{(1)},
b_{Q}^{(1)}\rangle\}_{P,Q\in\mathcal{D}}$,
$\{\langle m_{P}^{(1)},\psi_{Q}\rangle\}_{P,Q\in\mathcal{D}}$,
and $\{\langle\varphi_{P},b_{Q}^{(1)}\rangle\}_{P,Q\in\mathcal{D}}$
are all $\dot b_{p,q}^{\alpha}(W)$-almost diagonal;

\item for any $\vec{s}:=\{\vec{s}_{R}\}_{R\in \mathcal{D}}\in
\dot b_{p,q}^{\alpha}(W)$ and $P\in \mathcal{D}$,
$$
\vec{t}_P:=\sum\limits_{Q,R\in\mathcal{D}}
\left\langle m_P^{(1)},b_Q^{(1)}\right\rangle\left\langle
m_Q^{(2)},b_R^{(2)}\right\rangle\vec{s}_R
$$
converges absolutely and there exists a positive constant $C$,
independent of $\vec{s}$, $\{m_Q^{(i)}\}_{Q\in\mathcal{D}}$, and $\{b_Q^{(i)}\}_{Q\in\mathcal{D}}$,
such that $\|\vec{t}\|_{\dot b_{p,q}^{\alpha}(W)}
\leq C\|\vec{s}\|_{\dot b_{p,q}^{\alpha}(W)}$.
\end{enumerate}
\end{corollary}

\begin{proof}
We first prove $\rm(i)$.
From Definition \ref{moleculedef}
and Theorem \ref{molecule}, we deduce that
$\{\langle m_{P}^{(1)},b_{Q}^{(1)}\rangle\}_{P,Q\in\mathcal{D}}$
is $\dot b_{p,q}^{\alpha}(W)$-almost diagonal.
Note that   $\varphi, \psi\in \mathcal{S}_\infty$.
Thus, for any $Q\in\mathcal D$, each of $\varphi_Q$ and $\psi_Q$
is a harmless constant multiple of both
a $\dot B_{p, q}^{\alpha}(W)$-analysis molecule
and a $\dot B_{p, q}^{\alpha}(W)$-synthesis molecule on $Q$.
By this and Theorem \ref{molecule}, we conclude that $\{\langle m_{P}^{(1)},
\psi_{Q}\rangle\}_{P,Q\in\mathcal{D}}$
and $\{\langle\varphi_{P},b_{Q}^{(1)}\rangle\}_{P,Q\in\mathcal{D}}$
are $\dot b_{p,q}^{\alpha}(W)$-almost diagonal.
This completes the proof of (i).

Next, we prove (ii). From the just proven (i) and Proposition \ref{compose},
we infer that
$$
B:=\{b_{P,R}\}_{P,R\in\mathcal D}
:=\sum_{Q\in\mathcal{D}}\left|\left\langle
m_P^{(1)},b_Q^{(1)}\right\rangle\right|\left|\left\langle
m_Q^{(2)},b_R^{(2)}\right\rangle\right|
$$
is also $\dot b_{p,q}^{\alpha}(W)$-almost diagonal.
Using this and Theorem \ref{ad FJ-YY}, we conclude that
$$\sum_{Q,R\in\mathcal{D}}\left|\left\langle
m_P^{(1)},b_Q^{(1)}\right\rangle\right|\left|\left\langle
m_Q^{(2)},b_R^{(2)}\right\rangle\right|\left|\vec{s}_R\right|
=\sum_{R\in\mathcal{D}}b_{P,R}\left|\vec{s}_R\right|<\infty$$
and $\|\vec{t}\|_{\dot b_{p,q}^{\alpha}(W)}
\lesssim \|\vec{s}\|_{\dot b_{p,q}^{\alpha}(W)}$.
This completes the proof of (ii) and hence Corollary \ref{commolecule}.
\end{proof}

Elements   in  $\dot{B}_{p,q}^{\alpha}(W)$ are distributions,
whereas the molecules $\Phi$ do not necessarily belong to the Schwartz class;
hence, it is necessary to ensure that the pairing
$\langle \vec{f}, \Phi \rangle$ is well-defined.

\begin{lemma}\label{heli}
Let $\alpha\in\mathbb{R}$,  $p\in(0,\infty)$,
$q\in(0,\infty]$,  and $W\in \mathcal A_{p,\infty}$.
Assume that $\vec{f} \in \dot{B}_{p, q}^{\alpha}(W)$ and
$\Phi$ is a $\dot{B}_{p,q}^{\alpha}(W)$-analysis molecule.
Then, for any $\varphi,\psi\in \mathcal{S}$ satisfying
\eqref{hs2} and \eqref{hs3}, the pairing
\begin{align}\label{welldefi}
\left\langle\vec{f},\Phi\right\rangle:=\sum_{R\in\mathcal{D}}
\left\langle\vec{f},\varphi_{R}\right\rangle\langle\psi_{R},\Phi\rangle
\end{align}
is well defined and its value is independent of the choices of $\varphi$ and $\psi$.
\end{lemma}

\begin{proof}
For each $i\in\{1,2\}$, let $\varphi^{(i)},\psi^{(i)}\in \mathcal{S}$
be a pair of functions satisfying \eqref{hs2} and \eqref{hs3}.
Applying $\vec{f} \in \dot{B}_{p, q}^{\alpha}(W)$
and Theorem \ref{dl1103}(i), we obtain
$\{\langle\vec{f},\varphi_{R}^{(1)}\rangle\}_{R\in\mathcal D}
\in \dot{b}_{p, q}^{\alpha}(W)$.
By this and Corollary \ref{commolecule}(ii) with
$m_P^{(1)}$, $b_Q^{(1)}$, $m_Q^{(2)}$, $b_R^{(2)}$, and $\vec{s}_R$
replaced, respectively, by $\overline{\Phi}$,
$\overline{\psi_{Q}^{(2)}}$, $\overline{\varphi_{Q}^{(2)}}$,
$\overline{\psi_{R}^{(1)}}$, and $\langle\vec{f},\varphi_{R}^{(1)}\rangle$, we find that
\begin{align*}
\sum_{Q,R\in\mathcal{D}}
\left\langle\vec{f},\varphi_{R}^{(1)}\right\rangle
\left\langle\psi_{R}^{(1)},\varphi_{Q}^{(2)}\right\rangle
\left\langle\psi_{Q}^{(2)},\Phi\right\rangle
=\sum_{Q,R\in\mathcal{D}}
\left\langle\overline{\Phi},\overline{\psi_{Q}^{(2)}}\right\rangle
\left\langle\overline{\varphi_{Q}^{(2)}}, \overline{\psi_{R}^{(1)}}\right\rangle
\left\langle\vec{f},\varphi_{R}^{(1)}\right\rangle
\end{align*}
converges absolutely. From this and Lemma \ref{repro}(ii), we deduce that
\begin{align*}
\sum_{R\in\mathcal{D}}\left\langle\vec{f},\varphi_{R}^{(1)}\right\rangle
\left\langle\psi_{R}^{(1)},\Phi\right\rangle
&=\sum_{R\in\mathcal{D}}
\left\langle\vec{f},\varphi_{R}^{(1)}\right\rangle
\left(\sum_{Q\in\mathcal{D}}\left\langle
\psi_{R}^{(1)},\varphi_{Q}^{(2)}\right\rangle
\left\langle\psi_{Q}^{(2)},\Phi\right\rangle\right)\\
&=\sum_{Q\in\mathcal{D}}\left(\sum_{R\in\mathcal{D}}
\left\langle\vec{f},\varphi_{R}^{(1)}\right\rangle
\left\langle\psi_{R}^{(1)},\varphi_{Q}^{(2)}\right\rangle\right)
\left\langle\psi_{Q}^{(2)},\Phi\right\rangle \\
&=\sum_{Q\in\mathcal{D}}\left\langle\vec{f},\varphi_{Q}^{(2)}\right\rangle
\left\langle\psi_{Q}^{(2)},\Phi\right\rangle.
\end{align*}
Thus, the right-hand side of \eqref{welldefi}
is absolutely convergent and independent of $\varphi$ and $\psi$,
which completes the proof of Lemma \ref{heli}.
\end{proof}

We now establish the molecular characterization of $\dot B_{p,q}^{\alpha}(W)$.

\begin{theorem}\label{molecha}
Let $\alpha\in\mathbb{R}$, $p\in(0,\infty)$, $q\in(0,\infty]$,
and $W\in \mathcal A_{p,\infty}$.
\begin{enumerate}[\rm(i)]
\item If $\left\{m_{Q}\right\}_{Q\in \mathcal{D}}$ is a family of
$\dot B_{p, q}^{\alpha}(W)$-analysis molecules,
then there exists a positive constant $C$ such that,
for any $\vec{f} \in \dot B_{p, q}^{\alpha}(W)$,
$$\left\|\left\{\left\langle\vec{f},m_Q\right\rangle\right\}
_{Q\in\mathcal{D}}\right\|_{\dot b_{p,q}^{\alpha}(W)}
\leq C\left\|\vec{f}\right\|_{\dot B_{p,q}^{\alpha}(W)}.$$

\item
If $\left\{b_{Q}\right\}_{Q\in\mathcal{D}}$ is a family of
$\dot B_{p,q}^{\alpha}(W)$-synthesis molecules,  then, for any
$\vec{s}:=\{\vec{s}_{Q}\}_{Q\in\mathcal{D}}\in \dot b_{p, q}^{\alpha}(W)$,
there exists $\vec{f}\in \dot B_{p, q}^{\alpha}(W)$ such that
$\vec{f}=\sum_{Q\in\mathcal{D}}b_Q\vec{s}_Q$
in $(\mathcal{S_{\infty}'})^m$ and there exists a positive constant $C$,
independent of $\{\vec{s}_Q\}_{Q\in\mathcal{D}}$ and $\{b_Q\}_{Q\in\mathcal{D}}$,
such that
$$\left\|\vec{f}\right\|_{\dot B_{p,q}^{\alpha}(W)}
\leq C\left\|\vec{s}\right\|_{\dot b_{p,q}^{\alpha}(W)}.$$
\end{enumerate}

\end{theorem}
\begin{proof}
We first show $\rm(i)$. Let $\varphi,\psi\in\mathcal{S}$
satisfy  \eqref{hs2} and \eqref{hs3}.
By Lemma \ref{heli}, we find that, for any $Q\in{\mathcal{D}}$,
\begin{align}\label{usedefi}
\left\langle\vec{f},m_{Q}\right\rangle
=\sum_{R\in\mathcal{D}}\left\langle\vec{f},\varphi_{R}\right\rangle
\left\langle\psi_{R},m_{Q}\right\rangle
=\sum_{R\in\mathcal{D}}
\left\langle\overline{m_{Q}},\overline{\psi_{R}}\right\rangle
\left(S_{\varphi}\vec{f}\right)_{R}.
\end{align}
Using Corollary \ref{commolecule}(i), we conclude that
$B:=\{\langle\overline{m_{Q}},\overline{\psi_{R}}\rangle\}_{Q,R\in\mathcal{D}}$
is a $\dot{b}_{p,q}^{\alpha}(W)$-almost diagonal
operator. From this, \eqref{usedefi}, and Theorems \ref{ad FJ-YY}
and  \ref{dl1103}, it follows that
$$\left\|\left\{\left\langle\vec{f},m_Q
\right\rangle\right\}_{Q\in\mathcal{D}}\right\|_{\dot{b}_{p,q}^{\alpha}(W)}
=\left\|B\left(S_{\varphi}\vec{f}\right)\right\|_{\dot{b}_{p,q}^{\alpha}(W)}
\lesssim\left\|S_{\varphi}\vec{f}\right\|_{\dot{b}_{p,q}^{\alpha}(W)}
\lesssim\left\|\vec{f}\right\|_{\dot{B}_{p,q}^{\alpha}(W)}.$$
This completes the proof of (i).

We now prove (ii).
By Corollary \ref{commolecule}(ii) and Lemma \ref{repro}(ii),
we find that, for any $\phi\in \mathcal S_\infty$,
\begin{equation}\label{keykey}
\left\langle\vec{f},\phi\right\rangle
:=\sum_{R\in\mathcal{D}}\langle b_R,\phi\rangle \vec{s}_R
=\sum_{R\in\mathcal{D}}\langle \overline{\phi}, \overline{b_R}\rangle \vec{s}_R
=\sum_{R\in\mathcal{D}}\sum_{Q\in\mathcal{D}}
\langle \overline{\phi},\varphi_Q\rangle
\langle\psi_Q,\overline{b_R}\rangle \vec{s}_R
\end{equation}
with both series converging absolutely.
Therefore, $\vec{f}=\sum_{Q\in\mathcal{D}}b_Q\vec{s}_Q$
in $(\mathcal S_\infty')^m$.
Taking $\phi := \varphi_P$ in \eqref{keykey} and
using Corollary \ref{commolecule}(ii) again, we obtain
$$
\left\| S_\varphi \vec f \right\|_{\dot{b}_{p,q}^{\alpha}(W)}
= \left\| \left\{\left\langle\vec{f},\varphi_P\right\rangle\right\}_{P\in \mathcal{D}} \right\|_{\dot{b}_{p,q}^{\alpha}(W)}
\lesssim \| \vec s \|_{\dot{b}_{p,q}^{\alpha}(W)}.
$$
This, together with the $\varphi$-transform characterization of
$\dot{B}^\alpha_{p,q}(W)$ (Theorem \ref{dl1103}), further implies that
$$
\left\|\vec f\right\|_{\dot{B}_{p,q}^{\alpha}(W)}
\sim \left\| S_\varphi \vec f \right\|_{\dot{b}_{p,q}^{\alpha}(W)}
\lesssim\left\|\vec{s}\right\|_{\dot b_{p,q}^{\alpha}(W)},
$$
which completes the proof of (ii) and hence Theorem \ref{molecha}.
\end{proof}

\begin{remark} \label{Bownik}
Bownik \cite[Theorems 5.5 and 5.7]{gjabownik05} established
the molecular characterization of anisotropic Besov spaces with doubling measures.
In the setting of scalar-weighted anisotropic Besov spaces,
which is the intersection of Bownik's result and Theorem \ref{molecha},
the latter improves upon the former.
Indeed, let $w\in A_\infty$.
Bownik \cite[Theorem 5.7]{gjabownik05} showed that
if $\left\{m_{Q}\right\}_{Q\in \mathcal{D}}$ is a family of
$(\widetilde K,\widetilde L,\widetilde M,\widetilde N)$-molecules satisfying
\begin{equation}\label{bow}
\widetilde K>\max\left\{\widetilde{J},
\frac{\alpha\zeta_+}{\zeta_-}+1\right\},\quad
\widetilde L= \left\lfloor\frac{\alpha}{\zeta_-}\right\rfloor ,\quad
\widetilde M> \widetilde{J},\quad
\widetilde N=\left(\left\lfloor\frac{\widetilde{J}-1-\alpha}{\zeta_-}\right\rfloor +1\right)_+
\end{equation}
with $\widetilde{J}$ as in \eqref{tildeJ},
then, for any $f \in \dot B_{p, q}^{\alpha}(w)$,
\begin{equation}\label{bounded}
\left\|\left\{\left\langle f,m_Q\right\rangle\right\}
_{Q\in\mathcal{D}}\right\|_{\dot b_{p,q}^{\alpha}(w)}
\lesssim \|f\|_{\dot B_{p,q}^{\alpha}(w)}.
\end{equation}
Meanwhile, Theorem \ref{molecha}(i) implies that
if $\left\{m_{Q}\right\}_{Q\in \mathcal{D}}$ is a family of
$(K,L,M,N)$-molecules satisfying
\begin{equation}\label{our}
\begin{split}
&K> \max\left\{J, \frac{\alpha\zeta_+}{\zeta_-} + 1\right\}, \quad
L> \frac{\alpha}{\zeta_-}, \\
&M>J,\quad
\begin{cases}
N>\frac{J-1-\alpha}{\zeta_-}+1 &\text{if } \alpha\le J-1,\\
N=-1 &\text{if } \alpha> J-1
\end{cases}
\end{split}
\end{equation}
with $J$ as in \eqref{J}, then \eqref{bounded} holds.
From Remark \ref{comparebow}, it follows that
$J\leq \widetilde{J}$ and this inequality is strict for some scalar weights.
Moreover, for any $\gamma\in\mathbb Z_+^n$ and $a\in[0,\infty)$,
$|\gamma|\leq \lfloor a \rfloor$
if and only if $|\gamma|\leq a$.
Therefore, condition \eqref{our} is strictly weaker than condition \eqref{bow},
and hence Theorem \ref{molecha}(i) improves \cite[Theorem 5.7]{gjabownik05} in this case.
The proof that Theorem \ref{molecha}(ii) improves \cite[Theorem 5.5]{gjabownik05}
parallels the preceding one; we omit the details.
\end{remark}

As an application of Theorem \ref{molecha},
we obtain the following conclusion.

\begin{proposition}\label{distri}
Let $\alpha\in\mathbb{R}$,  $p\in(0,\infty)$, $q\in (0,\infty]$,
and $W \in\mathcal A_{p,\infty}$.
Then $(\mathcal S_{\infty})^m \subset \dot{B}_{p,q}^{\alpha}(W)$.
Moreover, there exist $M \in\mathbb{Z}_+$ and a positive constant $C$ such that,
for any $\vec{f}\in (\mathcal S_{\infty})^m$,
$$
\left\|\vec{f}\right\|_{\dot{B}_{p,q}^{\alpha}(W)}
\leq C \left\|\vec{f}\right\|_{S_M}
:= C \sup_{\genfrac{}{}{0pt}{}{\gamma\in\mathbb{Z}_+^n}{|\gamma|\leq M}}
\sup_{x\in\mathbb{R}^n}
\left|\partial^\gamma \vec f(x)\right| [1+\rho(x)]^{1+M+|\gamma|}.
$$
\end{proposition}

\begin{proof}
Let $\vec{f}:=(f_1,\ldots,f_m)^{T}\in(\mathcal{S}_\infty)^m$
and, for any $i\in\{1,\ldots,m\}$, let
$$\vec{e}_i:=(0,\ldots,0,1,0,\ldots,0)^{T}\in\mathbb{C}^m,$$
where only the $i$-th component is 1.
Then $\vec f=\sum_{i=1}^{m} f_i \vec e_i.$
Note that, for any $i\in\{1,\ldots,m\}$, $f_i\in\mathcal{S}_\infty$.
Consequently, there exists $M\in\mathbb Z_+$,
independent of $\vec f$, such that
$\|\vec f\|_{S_M}^{-1} f_i$ is a
$\dot{B}_{p,q}^{\alpha}(W)$-synthesis molecule on $Q_{0,\mathbf{0}}$.
From this and Theorem \ref{molecha}(ii), we infer that
$$
\left\|\vec{f}\right\|_{\dot{B}_{p,q}^{\alpha}(W)}
\sim\sum_{i=1}^m\left\|f_i \vec{e}_i\right\|_{\dot{B}_{p,q}^{\alpha}(W)}
\lesssim \left\|\vec{f}\right\|_{S_M},
$$
which completes the proof of Proposition \ref{distri}.
\end{proof}

\section{Pseudo-Differential Operators} 

Driven by their important role in the paradifferential calculus of Bony \cite{b81},
pseudo-differential operators with symbols in the class $S^m_{1,1}$
have garnered considerable attention
(see, for example, \cite{b88,h88,h89}).
Recently, several articles have investigated
the boundedness of pseudo-differential operators
on Besov--Triebel--Lizorkin spaces
(see, for example, \cite{gt99,p20,fsyy,syy10,syy}).
In this section, we establish the boundedness of
these operators on $\dot B_{p,q}^{\alpha}(W)$ in Subsection \ref{7.1},
and prove the sharpness of this boundedness result in Subsection \ref{7.2}.

\subsection{Boundedness of Pseudo-Differential Operators} \label{7.1}

Let $\rho_{A^*}$ be the step homogeneous quasi-norm associated with $A^*$,
where $A^*$ is the conjugate transpose of $A$.
We now recall the homogeneous anisotropic class $\dot S_{1,1}^{u}$
(see \cite[Definition 1.1]{bb10}).

\begin{definition}
Let $u\in\mathbb R$.
The homogeneous anisotropic class $\dot S_{1,1}^u$
is the set of all functions
$\sigma\in C^{\infty}(\mathbb R^n\times(\mathbb R^n\setminus\{\mathbf 0\}))$
such that, for any multi-indices $\beta,\gamma\in{\mathbb Z}_+^n$,
\begin{align*}
\sup_{(x,\xi)\in\mathbb R^n\times(\mathbb R^n\setminus\{\mathbf 0\})}
[\rho_{A^*}(\xi)]^{-u}
\left|\partial_x^\gamma\partial_\xi^\beta
\widetilde{\sigma} \left( A^{k}x,(A^*)^{-k}\xi\right) \right|
<\infty,
\end{align*}
where $\widetilde{\sigma}(\cdot,\cdot)
:= \sigma(A^{-k}\cdot, (A^*)^{k}\cdot)$
and $k\in\mathbb{Z}$ is determined by the relation $\rho_{A^{*}}(\xi)= b^{k}$.
\end{definition}

\begin{definition}
Let $u\in\mathbb R$ and $\sigma\in \dot{S}_{1,1}^u$.
Define the pseudo-differential operator $\sigma(x, D) $
with symbol $\sigma$ by setting, for any $ f\in\mathcal{S}_\infty$
and $x\in\mathbb{R}^n$,
$$
\sigma(x,D)(f):= \int_{\mathbb{R}^n}
\sigma(x, \xi) \widehat{f}(\xi) e^{i x \cdot \xi} \, d\xi.
$$
\end{definition}

The following lemma is exactly \cite[Lemma 4.9]{bb10} with $\delta=\gamma=1$.

\begin{lemma}\label{continue}
Let $u\in\mathbb R$ and $\sigma\in \dot{S}_{1,1}^u$.
Then $\sigma(x,D)$ maps $\mathcal S_{\infty}$ continuously
into $\mathcal S$. In particular, its formal
adjoint $\sigma(x, D)^*$ is a continuous linear
operator from $\mathcal S'$ to $\mathcal S_{\infty}'$,
where, for any $f\in\mathcal S'$ and
$\phi\in\mathcal S_{\infty}$,
$$\left\langle \sigma(x, D)^*f, \phi \right\rangle
:=\left\langle f, \sigma(x,D)\phi\right\rangle.$$
\end{lemma}

The following theorem is the main result of this section.

\begin{theorem} \label{pseudo}
Let $\alpha\in\mathbb R$, $p\in(0,\infty)$, $q\in(0,\infty]$,
and $ W \in \mathcal A_{p,\infty}$.
Assume that $u\in\mathbb R$, $\sigma\in\dot{S}_{1,1}^u$,
and $\sigma(x,D)$ is a pseudo-differential operator with symbol $\sigma$.
Let $J$ be as in \eqref{J}
and $\zeta_-$ as in Lemma \ref{edbj}.
Suppose that $\sigma(x, D)^*$ satisfies,
for any $\gamma\in\mathbb{Z}_+^n$
with $|\gamma|\leq\lfloor\frac{J-1-\alpha}{\zeta_-}\rfloor$,
\begin{equation}\label{irremovable}
\sigma(x, D)^* ( x^\gamma)=0\in \mathcal{S}_\infty'.
\end{equation}
Then $\sigma(x,D)$ can be extended to a continuous linear mapping
from $\dot B_{p,q}^{\alpha+u}(W) $ to $\dot B^{\alpha}_{p,q}(W)$.
\end{theorem}

\begin{proof}
To simplify the presentation of the present proof,
we denote $\sigma(x, D)$ simply by $T$.
Let $ \varphi \in \mathcal{S}_\infty$
satisfy \eqref{hs2} and
\begin{align*}
\sum_{j\in\mathbb{Z}}
\overline{\widehat{\varphi}((A^*)^j\xi)}
\widehat{\varphi}((A^*)^j\xi)=1.
\end{align*}
By this and Lemma \ref{repro}(ii), we find that,
for any $\vec f \in (\mathcal{S}_\infty)^m$,
$
\vec f=\sum_{Q\in\mathcal{D}}
\langle \vec f, \varphi_Q \rangle \varphi_Q
$
in $(\mathcal{S}_\infty)^m$.
Using this and Lemma \ref{continue}, we conclude that,
for any $\vec f \in (\mathcal{S}_\infty)^m$,
\begin{align}\label{Texpress}
T \left(\vec f\right)
= \sum_{Q\in\mathcal{D}}
\left\langle \vec f, \varphi_Q
\right\rangle T \left( \varphi_Q \right)
= \sum_{Q\in\mathcal{D}}
\left[ |Q|^{-u} \left(S_\varphi \vec f\right)_Q \right]
\left[|Q|^u T \left(\varphi_Q\right)\right]
\end{align}
in $\mathcal{S}^m$.
For any $\vec f \in \dot B_{p, q}^{\alpha + u}(W)$,
we can still define $T (\vec f)$ as in \eqref{Texpress}.
Next, we prove that $T (\vec f)$ is well defined.
In the proof of \cite[Theorem 4.8]{bb10},
B\'enyi and Bownik showed that, for any $K,M\in[0,\infty)$ and $N\in\mathbb R$,
$|Q|^u T(\varphi_Q)$ is a harmless constant multiple of
a $(K,\lfloor\frac{J-1-\alpha}{\zeta_-}\rfloor,M,N)$-molecule on $Q$,
and is consequently a
$\dot B^{\alpha}_{p,q}(W)$-synthesis molecule on $Q$.
From this and Theorem \ref{molecha}(ii), it follows that
$T (\vec f)$ is well defined and
\begin{align*}
\left\|T\left( \vec f\right) \right\|_{\dot B_{p,q}^\alpha(W)}
\lesssim\left\| \left\{ |Q|^{-u} \left( S_\varphi \vec f\right)_Q
\right\}_{Q\in\mathcal D}
\right\|_{\dot b^{\alpha}_{p,q}(W)}.
\end{align*}
This, together with Theorem \ref{dl1103}(i), further implies that
\begin{align*}
\left\|T\left( \vec f\right) \right\|_{\dot B_{p,q}^\alpha(W)}
\lesssim \left\| S_\varphi \vec f
\right\|_{\dot b_{p, q}^{\alpha + u}(W)}
\lesssim \left\| \vec f
\right\|_{\dot B_{p, q}^{\alpha + u}(W)},
\end{align*}
which completes the proof of Theorem \ref{pseudo}.
\end{proof}

\begin{remark}\label{bijiao}
In the case of scalar weighted anisotropic Besov spaces,
the condition corresponding to \eqref{irremovable} in \cite[Theorem 4.8]{bb10} is that
$\sigma(x, D)^* (x^\gamma)=0$
for all $\gamma\in\mathbb{Z}_+^n$
with $|\gamma|\leq\lfloor\frac{\widetilde J-1-\alpha}{\zeta_-}\rfloor$, where $\widetilde J$ is as in \eqref{tildeJ}.
From Remark \ref{comparebow}, it follows that $J\leq \widetilde J$ and
this inequality is strict for some scalar weights.
Therefore, Theorem \ref{pseudo} improves \cite[Theorem 4.8]{bb10} in this case.

In Theorem \ref{pseudo}, if $J-1-\alpha<0$, then \eqref{irremovable} is void.
Conversely, if $J-1-\alpha\geq 0$, then \eqref{irremovable} is
an additional assumption on the symbol $\sigma$ and is sharp
(see Subsection \ref{7.2}).
\end{remark}

\subsection{Sharpness of the Formal Adjoint Condition} \label{7.2}

To simplify the presentation, in this subsection
we assume that $n=1$, $m=1$, $A=2$, and $\rho(\cdot)=|\cdot|$.
In this case, $\sigma\in\dot S_{1,1}^u$ if and only if, for any
$\beta,\gamma\in{\mathbb Z}_+$,
\begin{align*}
\sup_{(x,\xi)\in\mathbb R\times(\mathbb R\setminus\{0\})}
|\xi|^{-u-\gamma+\beta}
\left|\partial_x^\gamma\partial_\xi^\beta \sigma (x,\xi) \right|
<\infty.
\end{align*}
The following theorem establishes
the sharp range of $\alpha$ ensuring the boundedness of pseudo-differential operators
with symbol $\sigma\in\dot S^u_{1,1}$.

\begin{theorem}\label{sharppse2}
Let $\lambda\in(-1,\infty)$ and $w(\cdot):=|\cdot|^\lambda$.
Assume that $\alpha,u\in\mathbb R$, $p\in(0,\infty)$, and $q\in(0,\infty]$.
Then the following statements are equivalent.
\begin{enumerate}[\rm(i)]
\item For any symbol $\sigma\in\dot S^u_{1,1}$, the associated operator
$\sigma(x,D)$ can be extended to a continuous linear mapping
from $\dot B_{p,q}^{\alpha+u}(w) $ to $\dot B^{\alpha}_{p,q}(w)$.

\item
\begin{align*}
\alpha>\theta_{\lambda,p}
:=\max\left\{0,\frac{1}{p}-1,\frac{1+\lambda}{p}-1\right\}.
\end{align*}
\end{enumerate}
\end{theorem}

Recall that a unweighted inhomogeneous counterpart of Theorem \ref{sharppse2} was obtained by Park
\cite[Theorem 2.6]{p20}.

For $\alpha\in(-\infty,\theta_{\lambda,p}]$,
the following theorem establishes the sharp order
of the formal-adjoint cancellation condition.

\begin{theorem}\label{sharppse1}
Let $\lambda\in(-1,\infty)$ and $w(\cdot):=|\cdot|^\lambda$.
Assume that $p\in(0,\infty)$,
$\alpha\in(-\infty,\theta_{\lambda,p}]$, $q\in(0,\infty]$,
$u\in\mathbb R$, and $L\in\mathbb Z_+$.
Then the following statements are equivalent.
\begin{enumerate}[\rm(i)]
\item For any symbol $\sigma\in\dot S^u_{1,1}$ satisfying
\begin{align*}
\sigma(x, D)^*(x^r)=0 \in \mathcal{S}_\infty'
\quad\text{for all } r\in\mathbb Z\cap [0,L],
\end{align*}
$\sigma(x,D)$ can be extended to a continuous linear mapping
from $\dot B_{p,q}^{\alpha+u}(w) $ to $\dot B^{\alpha}_{p,q}(w)$.

\item $L\ge\lfloor\theta_{\lambda,p}-\alpha\rfloor$.
\end{enumerate}
\end{theorem}

To prove Theorems \ref{sharppse2} and \ref{sharppse1},
we need several technical lemmas.
We first construct a class of symbols in $\dot S^u_{1,1}$.

\begin{lemma}\label{possym}
Let $u\in\mathbb R$ and $v\in\mathbb Z_+$.
Assume that $\chi\in C_{\rm c}^\infty$ satisfies
$\operatorname{supp} \chi \subset B(\frac54,2^{-2})$ and
$\chi=1$ on $B(\frac54,2^{-3})$.
For any $x\in\mathbb R$ and $\xi\in\mathbb R\setminus\{0\}$, define
\begin{align*}
\sigma(x,\xi):=
\sum_{k\in\mathbb Z}2^{k(u-v)}e^{-i \frac54 2^kx}
(\xi- \tfrac54 2^k)^v \chi( 2^{-k}\xi).
\end{align*}
Then $\sigma\in\dot S^u_{1,1}$ and
$\sigma(x, D)^*(x^r)=0\in \mathcal{S}_\infty'$
for all $r\in\mathbb Z\cap [0,v)$.
\end{lemma}

\begin{proof}
For any $k\in\mathbb Z$, $\operatorname{supp} \chi (2^{-k}\cdot)
\subset 2^k B(\frac54,2^{-2})$.
Thus, for any $\xi\in \operatorname{supp} \chi (2^{-k}\cdot)$,
\begin{align}\label{xi}
2^k<|\xi| < \tfrac32 2^k
\quad\text{and}\quad
|\xi- \tfrac54 2^k| < \tfrac14 2^k,
\end{align}
which further implies that
the summands defining $\sigma$ have pairwise disjoint supports for $\xi$.
By \eqref{xi} and Leibniz's rule, we obtain,
for any $\beta,\gamma\in\mathbb Z_+$,
$x\in\mathbb R$, and $\xi\in \operatorname{supp} \chi (2^{-k}\cdot)$,
\begin{align*}
&\left|\partial_x^\gamma\partial_\xi^\beta
\left[2^{k(u-v)}e^{-i \frac54 2^k x}
( \xi- \tfrac54 2^k )^v\chi(2^{-k}\xi)\right]\right|  \\
&\quad\lesssim 2^{k(u-v)} (\tfrac54 2^k)^{\gamma}
\sum_{t=0}^{\beta \wedge v} |\xi-\tfrac54 2^k|^{v-t} 2^{-k(\beta-t)}
\lesssim 2^{k(u+\gamma-\beta)}
\sim |\xi|^{u+\gamma-\beta},
\end{align*}
which further implies that $\sigma\in\dot S^u_{1,1}$.

We next prove the cancellation of the formal adjoint.
Let $f\in\mathcal S_\infty$.
For any $r\in\mathbb Z_+$,
\begin{align*}
\left\langle \sigma(x, D)^*(x^r),f\right\rangle
=\left\langle x^r,\sigma(x,D)f\right\rangle.
\end{align*}
For any $x\in\mathbb R$ and $\xi\in\mathbb R\setminus\{0\}$, define
\begin{equation*}
\widetilde\sigma(x,\xi):=
\sum_{k\in\mathbb Z} 2^{k(u-v)} e^{-i \frac54 2^k x}
\chi\left(2^{-k}\xi\right).
\end{equation*}
Repeating the argument used in the proof of $\sigma\in\dot S^u_{1,1}$,
we obtain $\widetilde\sigma\in\dot S^{u-v}_{1,1}$.
This, together with Lemma \ref{continue},
further implies that $\widetilde\sigma (x, D) f \in \mathcal S$.
From the fact that $\chi,f\in\mathcal S$, it follows that
\begin{align*}
\partial_x^v \widetilde\sigma(x,D)f
=\sum_{k\in\mathbb Z} 2^{k(u-v)}
\int_{\mathbb R} \partial_x^v e^{ix(\xi-\frac54 2^k)}
\chi(2^{-k}\xi)\widehat f(\xi)\,d\xi
=i^v \sigma(x,D)f.
\end{align*}
By these and integration by parts, we obtain,
for any $r\in\mathbb Z$ with $0\le r<v$,
\begin{align*}
\left\langle \sigma(x, D)^*(x^r),f\right\rangle
&= \left\langle x^r,\sigma(x,D)f\right\rangle
= \left\langle x^r, (-i)^v \partial_x^v \widetilde\sigma(x,D)f\right\rangle \\
&= (-i)^v \int_{\mathbb R}  \partial_x^v(x^r) \overline{\widetilde\sigma(x,D)f} \, dx
=0,
\end{align*}
and hence $\sigma(x, D)^*(x^r)=0\in \mathcal{S}_\infty'$.
This completes the proof of Lemma \ref{possym}.
\end{proof}

The following lemma establishes that applying the operator $\sigma(x,D)$
to a certain function is equivalent to taking a scaled derivative.

\begin{lemma} \label{Tf}
Let $u\in\mathbb R$, $v\in\mathbb Z_+$
and $\sigma$ be as in Lemma \ref{possym}.
Assume that $k\in\mathbb Z$, $x_0\in\mathbb R$, and $f\in\mathcal S$ satisfy
$$
\operatorname{supp} \left[ e^{i \frac54 2^k(\cdot -x_0)} f(\cdot-x_0) \right]^\wedge
\subset 2^k B(\tfrac54,2^{-3}).
$$
Then, for any $x\in\mathbb R$,
\begin{align*}
\sigma(x,D) \left( e^{i \frac54 2^k(\cdot -x_0)} f(\cdot -x_0) \right)
= e^{-i \frac54 2^k x_0}
2^{k(u-v)}(2\pi) (-i\partial_x)^v f(x-x_0).
\end{align*}
\end{lemma}

\begin{proof}
By the fact that the summands defining $\sigma$ have pairwise disjoint supports
(as shown in the proof of Lemma \ref{possym})
and that $\chi(2^{-k}\cdot)=1$ on $2^k B(\frac54,2^{-3})$, we obtain
\begin{align*}
\sigma(x,D) \left( e^{i \frac54 2^k(\cdot -x_0)} f(\cdot -x_0) \right)
&= 2^{k(u-v)} \int_{\mathbb R}
( \xi- \tfrac54 2^k)^v
e^{-ix_0 \xi} \widehat{f}( \xi- \tfrac54 2^k)
e^{ix(\xi-\frac54 2^k)}\,d\xi \\
&= e^{-i \frac54 2^k x_0} 2^{k(u-v)}
\int_{\mathbb R}\eta^v\widehat{f}(\eta)e^{i(x-x_0)\eta}\,d\eta \\
&= e^{-i \frac54 2^k x_0} 2^{k(u-v)} (-i\partial_x)^v
\left( \int_{\mathbb R}\widehat{f}(\eta)
e^{i(x-x_0)\eta}\,d\eta \right) \notag\\
&= e^{-i \frac54 2^k x_0}
2^{k(u-v)}(-i\partial_x)^v [2\pi (\widehat{f})^{\vee}(x-x_0)] \\
&= e^{-i \frac54 2^k x_0}
2^{k(u-v)}(2\pi) (-i\partial_x)^v f(x-x_0).
\end{align*}
This completes the proof of Lemma \ref{Tf}.
\end{proof}

%
%

The following lemma provides an estimate for the power-weighted integral of
Schwartz functions when the singularity of the weight lies away from the origin.

\begin{lemma} \label{721}
Let $p\in(0,\infty)$ and $\lambda\in(-1,\infty)$.
Assume that $h\in\mathcal S$ satisfies $\int_{B(0,\frac12)} |h|^p>0$.
Then, for any $y_0\in\mathbb R\setminus(-1,1)$,
\begin{align*}
\int_{\mathbb R} |h(y)|^p |y-y_0|^{\lambda} \, dy
\sim |y_0|^\lambda,
\end{align*}
where the positive equivalence constants are independent of $y_0$.
\end{lemma}

\begin{proof}
Without loss of generality, we may assume that $y_0\in[1,\infty)$. Then
\begin{align*}
\int_{\mathbb R} |h(y)|^p |y-y_0|^{\lambda} \, dy
\gtrsim y_0^\lambda \int_{B(0,\frac12)} |h(y)|^p \, dy.
\end{align*}
From $h\in\mathcal S$, it follows that
\begin{align*}
\int_{\mathbb R} |h(y)|^p |y-y_0|^{\lambda} \, dy
&\lesssim \int_{[-\frac32 y_0, \frac12 y_0]} |h(y)|^p y_0^\lambda \, dy
+ \int_{\mathbb R\setminus [-\frac32 y_0, \frac12 y_0]} \frac{|y-y_0|^{\lambda}}{|y|^{2+\lambda}} \, dy \\
&\lesssim \|h\|_{L^p}^p y_0^\lambda
+ \int_{(\frac12 y_0, \frac32 y_0]} \frac{|y-y_0|^{\lambda}}{|y_0|^{2+\lambda}} \, dy
+ \int_{|y|>\frac32 y_0} \frac{|y|^{\lambda}}{|y|^{2+\lambda}} \, dy \\
&\sim y_0^\lambda + y_0^{-1} + y_0^{-1}
\sim y_0^\lambda,
\end{align*}
where the last equivalence follows from $\lambda\in(-1,\infty)$ and $y_0\in[1,\infty)$.
This completes the proof of Lemma \ref{721}.
\end{proof}

The following lemma shows that if the cancellation order is insufficient,
the pseudo-differential operator    fails  to be bounded.

\begin{lemma}\label{sharpcounterexample}
Let $\lambda\in(-1,\infty)$ and $w(\cdot):=|\cdot|^\lambda$.
Assume that $p\in(0,\infty)$,
$\alpha\in(-\infty,\theta_{\lambda,p}]$,
$q\in(0,\infty]$, and  $u\in\mathbb R$.
Then there exists a symbol $\sigma\in\dot S^u_{1,1}$ such that
\begin{equation}\label{countercancellation}
\sigma(x, D)^*(x^r)=0
\quad\text{for all }
r\in\mathbb Z\cap [0,\lfloor\theta_{\lambda,p}-\alpha\rfloor),
\end{equation}
but $\sigma(x,D)$ is not bounded from
$\dot B_{p,q}^{\alpha+u}(w)$ to $\dot B_{p,q}^{\alpha}(w)$.
\end{lemma}

\begin{proof}
By Theorem \ref{dl1103}(iii), the Besov quasi-norms can be computed
using a specifically chosen Littlewood--Paley function.
Let $\varphi$ be as in the proof of Theorem \ref{boundedfail_p}.
Then $\varphi\in\mathcal S$ satisfies \eqref{hs2} and
\begin{enumerate}[{\rm(a)}]
\item $\operatorname{supp} \widehat\varphi
\subset [-2, -\frac34] \cup [\frac34, 2]$,

\item for any $\xi\in J:= [-\frac32, -1] \cup [1, \frac32]$,
we have $\widehat{\varphi}(\xi) = 1$
and $\widehat{\varphi}(2^v \xi) = 0$ for all $v\in\mathbb Z\setminus\{0\}$,

\item for any $\xi\in [-\frac74, -\frac78] \cup [\frac78, \frac74]$,
we have $\widehat{\varphi}(\xi) \geq \frac12$.
\end{enumerate}
Let $h\in\mathcal S$ satisfy
\begin{equation} \label{h}
\operatorname{supp} \widehat{h}
\subset B( \tfrac54,2^{-5})
\subset J
\quad\text{and}\quad
\widehat{h}=1
\text{ on } B( \tfrac54,2^{-6}).
\end{equation}
For any $k\in\mathbb Z$ and $x\in\mathbb R$,
let $h_k(\cdot):=2^kh(2^k\cdot)$
and $g_k(x):=e^{-i \frac54 2^k x}h_k(x)$.
Then
\begin{equation} \label{g}
\widehat{g_k}(\xi)=\widehat{h_k}(\xi+\tfrac54 2^k).
\end{equation}
Let $v:=\lfloor \theta_{\lambda,p}-\alpha \rfloor$
and $\sigma$ be as in Lemma \ref{possym}.
Then $\sigma$ satisfies \eqref{countercancellation}.
To prove the present lemma, we consider
the following three cases for $p$ and $\lambda$.

\emph{Case (1)} $p\in(0,1]$ and $\lambda\in(-1,0)$.
In this case, $\theta_{\lambda,p}=\frac{1}{p}-1$
and $v=\lfloor \frac{1}{p}-1-\alpha \rfloor$.
For any $N\in\mathbb N$, let
\begin{equation} \label{fN}
f_N := N^{-\frac1q} \sum_{k=N}^{2N-1} 2^{-k(\alpha+u)}
\frac{h_k(\cdot-x_N)}{\|h_k(\cdot-x_N)\|_{L^p(w)}},
\end{equation}
where $x_N\in\mathbb R$ will be determined later.
Note that, for any $j\in\mathbb Z$,
\begin{align*}
[h_j(\cdot-x_N)]^\wedge(\xi)
= e^{-ix_N\xi} \widehat{h_j}(\xi).
\end{align*}
This, together with (b) and \eqref{h},
further implies that, for any $j,k\in\mathbb Z$,
\begin{align*}
\varphi_j*[h_k(\cdot-x_N)]=
\begin{cases}
h_k(\cdot-x_N) &\text{if } j=k, \\
0 & \text{if } j\neq k.
\end{cases}
\end{align*}
Thus, for any $j\in\mathbb Z$,
\begin{align*}
\varphi_j* f_N
= \begin{cases}
N^{-\frac1q} 2^{-j(\alpha+u)}
\frac{h_j(\cdot-x_N)}{\|h_j(\cdot-x_N)\|_{L^p(w)}}
&\text{if } N\leq j \leq 2N-1, \\
0 & \text{otherwise}.
\end{cases}
\end{align*}
Then
\begin{align} \label{FN2}
\|f_N\|_{\dot B_{p,q}^{\alpha+u}(w)}
=\left\| \left\{ N^{-\frac1q} \right\}_{j=N}^{2N-1} \right\|_{\ell^q}
=1.
\end{align}

Next, we estimate $\|\sigma(x,D)f_N\|_{\dot B_{p,q}^{\alpha}(w)}$.
Note that, for any $k\in\mathbb Z$,
$$
\operatorname{supp} \left[e^{i \frac54 2^k (\cdot -x_N)} g_k(\cdot-x_N)\right]^\wedge
= \operatorname{supp} [h_k(\cdot-x_N)]^\wedge
= \operatorname{supp} \widehat{h_k}
\subset 2^k B(\tfrac54,2^{-3}).
$$
This, together with Lemma \ref{Tf}, further implies that
\begin{align*}
\sigma(x,D) h_k (\cdot-x_N)
&=  \sigma(x,D) \left( e^{i \frac54 2^k (\cdot -x_N)} g_k(\cdot-x_N) \right) \\
&= e^{-i \frac54 2^k x_N}
2^{k(u-v)}(2\pi) (-i\partial_x)^v g_k(x-x_N).
\end{align*}
Thus,
\begin{align} \label{GN}
G_N(x)
:=\sigma(x,D) f_N
= 2\pi  N^{-\frac1q}
\sum_{k=N}^{2N-1} e^{-i \frac54 2^k x_N} 2^{-k(\alpha+v)}
\frac{(-i\partial_x)^v g_k(x-x_N)}{\|h_k(\cdot-x_N)\|_{L^p(w)}}.
\end{align}
By \eqref{g} and \eqref{h}, we conclude that,
for any $k\in\mathbb Z$ and $\xi\in B(0,2^{k-6})$,
\begin{align*}
[(-i\partial)^v g_k(\cdot-x_N)]^\wedge(\xi)
&= \xi^v [g_k(\cdot-x_N)]^\wedge(\xi)
= \xi^v e^{-ix_N\xi} \widehat{g_k}(\xi) \\
&= \xi^v e^{-ix_N\xi} \widehat{h_k} (\xi + \tfrac54 2^k)
= \xi^v e^{-ix_N\xi}.
\end{align*}
This, together with (a), further implies that,
for any $j,k\in\mathbb Z$ with $j\leq k-7$,
\begin{align} \label{722}
\varphi_j * [(-i\partial)^v g_k(\cdot-x_N)]
= (\widehat{\varphi_j} \xi^v e^{-ix_N\xi})^\vee
=(-i\partial)^v \varphi_j(\cdot-x_N).
\end{align}
Let $x_N:=\frac85\pi$.
Then $e^{-i\frac{5}{4} 2^k x_N} = e^{-i2\pi 2^k} = 1$.
By \eqref{h} and \cite[p.\,225]{fs97}, we conclude that
$\{x\in\mathbb R:\ h(x)=0\}$ is a countable discrete set.
From this and Lemma \ref{721}, we deduce that, for any $k\in\mathbb Z_+$,
\begin{align*}
\|h_k(\cdot-\tfrac85\pi)\|_{L^p(w)}^p
&= \int_{\mathbb R} |h_k(x-\tfrac85\pi)|^p |x|^{\lambda} \, dx \\
&= 2^{k(p-1-\lambda)} \int_{\mathbb R} |h(y)|^p
|y+2^k \tfrac85\pi|^{\lambda} \, dy \\
&\sim 2^{k(p-1-\lambda)} (2^k\tfrac85\pi)^{\lambda}
\sim 2^{k(p-1)}.
\end{align*}
Thus, for any $N\in\mathbb N$ with $N\geq 7$,
\begin{align} \label{G_N}
\|G_N\|_{\dot B_{p,q}^{\alpha}(w)}
\gtrsim \sum_{k=N}^{2N-1} 2^{k\delta}
\left\| \left\{ N^{-\frac1q} 2^{j\alpha}
\|(-i\partial)^v \varphi_j(\cdot-\tfrac85\pi)\|_{L^p(w)}
\right\}_{j=0}^{N-7} \right\|_{\ell^q},
\end{align}
where $\delta:= \frac1p-1-\alpha - v
= \frac1p-1-\alpha - \lfloor\frac1p-1-\alpha\rfloor \in[0,1)$.
By the fact that
$$
\operatorname{supp} (\partial^v\varphi)^\wedge
=\operatorname{supp} \widehat{\varphi}
\subset [-2, 2]
$$
and \cite[p.\,225]{fs97}, we conclude that
$\{x\in\mathbb R:\ \partial^v\varphi(x)=0\}$ is a countable set.
From this and Lemma \ref{721}, we deduce that, for any $j\in\mathbb Z_+$,
\begin{align*}
\|(-i\partial)^v \varphi_j(\cdot-\tfrac85\pi)\|_{L^p(w)}^p
&= \int_{\mathbb R} |\partial_x^v \varphi_j(x-\tfrac85\pi)|^p |x|^{\lambda} \, dx \\
&= 2^{j(p+pv-1-\lambda)} \int_{\mathbb R} |\partial^v \varphi(y)|^p |y+2^j \tfrac85\pi|^{\lambda} \, dy \\
&\sim 2^{j(p+pv-1-\lambda)} (2^j \tfrac85\pi)^{\lambda}
\sim 2^{j(p+pv-1)}.
\end{align*}
Substituting this estimate into \eqref{G_N}, we obtain,
for any $N\in\mathbb N$ with $N\geq 7$,
\begin{align*}
\|G_N\|_{\dot B_{p,q}^{\alpha}(w)}
\gtrsim \sum_{k=N}^{2N-1} 2^{k\delta}
\left\| \left\{ N^{-\frac1q} 2^{-j\delta}
\right\}_{j=0}^{N-7} \right\|_{\ell^q}.
\end{align*}
If $\delta\in(0,1)$, retaining only the term with $k=N$ and $j=0$ gives
\begin{align*}
\|G_N\|_{\dot B_{p,q}^{\alpha}(w)}
\gtrsim 2^{N\delta} N^{-\frac1q} \to \infty
\end{align*}
as $N\to\infty$. If $\delta=0$, then
\begin{align*}
\|G_N\|_{\dot B_{p,q}^{\alpha}(w)}
\gtrsim N \left\| \left\{ N^{-\frac1q}
\right\}_{j=0}^{N-7} \right\|_{\ell^q}
\sim N\to \infty
\end{align*}
as $N\to\infty$. However, $\|f_N\|_{\dot B_{p,q}^{\alpha+u}(w)}=1$.
Consequently, $\sigma(x,D)$ is not bounded from
$\dot B_{p,q}^{\alpha+u}(w)$ to $\dot B_{p,q}^{\alpha}(w)$.

\emph{Case (2)} Either $p\in(0,1]$ and $\lambda\in[0,\infty)$,
or $ p\in(1,\infty)$ and $\lambda\in[p-1,\infty)$.
In this case, $\theta_{\lambda,p}=\frac{1+\lambda}{p}-1$
and $v=\lfloor \frac{1+\lambda}{p}-1-\alpha \rfloor$.
For any $N\in\mathbb N$, let
$f_N$ be as in \eqref{fN} with $x_N:=0$.
Then \eqref{FN2} and \eqref{GN} imply that
$\|f_N\|_{\dot B_{p,q}^{\alpha+u}(w)}=1$ and
$$
G_N (x)
:=\sigma(x,D) f_N
= 2\pi N^{-\frac1q} \sum_{k=N}^{2N-1} 2^{-k(\alpha+v)}
\frac{(-i\partial)^v g_k (x)}{\|h_k\|_{L^p(w)}}.
$$
By \eqref{722}, we conclude that, for any $j,k\in\mathbb Z$ with $j\leq k-7$,
$
\varphi_j * [(-i\partial)^v g_k]
=(-i\partial)^v \varphi_j.
$
For any $k\in\mathbb Z$,
$\|h_k\|_{L^p(w)}=2^{k(1-\frac{1+\lambda}{p})} \|h\|_{L^p(w)}$.
Thus, for any $N\in\mathbb N$,
\begin{align} \label{GN2}
\|G_N\|_{\dot B_{p,q}^{\alpha}(w)}
\gtrsim \sum_{k=N}^{2N-1} 2^{k\delta}
\left\| \left\{ N^{-\frac1q} 2^{j\alpha}
\|(-i\partial)^v \varphi_j\|_{L^p(w)}
\right\}_{j=-\infty}^{N-7} \right\|_{\ell^q},
\end{align}
where $\delta:= \frac{1+\lambda}{p}-1-\alpha - v
= \frac{1+\lambda}{p}-1-\alpha - \lfloor\frac{1+\lambda}{p}-1-\alpha\rfloor \in[0,1)$.
For any $j\in\mathbb Z$,
\begin{align*}
\|(-i\partial)^v \varphi_j\|_{L^p(w)}^p
=\int_{\mathbb R} |\partial^v \varphi_j(x)|^p |x|^{\lambda} \, dx
= 2^{jp(1+v-\frac{1+\lambda}{p})} \|\partial^v \varphi\|_{L^p(w)}^p.
\end{align*}
Substituting this estimate into \eqref{GN2}, we obtain,
for any $N\in\mathbb N$ with $N\geq 7$,
$$
\|G_N\|_{\dot B_{p,q}^{\alpha}(w)}
\gtrsim \sum_{k=N}^{2N-1} 2^{k\delta}
\left\| \left\{ N^{-\frac1q} 2^{-j\delta}
\right\}_{j=0}^{N-7} \right\|_{\ell^q}.
$$
We have just proven in Case (1) that the right-hand side of
the above inequality tends to $\infty$ as $N\to\infty$.
However, $\|f_N\|_{\dot B_{p,q}^{\alpha+u}(w)}=1$.
Consequently, $\sigma(x,D)$ is not bounded from
$\dot B_{p,q}^{\alpha+u}(w)$ to $\dot B_{p,q}^{\alpha}(w)$.

\emph{Case (3)} $p\in(1,\infty)$ and $\lambda\in(-1,p-1)$.
In this case, $\theta_{\lambda,p}=0$,
$v=\lfloor -\alpha \rfloor$, and $w\in A_p$.
For any $N\in\mathbb N$ and $x\in\mathbb R$, let
\begin{align*}
f_N(x):=N^{-\frac1q}\sum_{k=N}^{2N-1}
2^{-k(\alpha+u)}e^{i \frac54 2^kx}h(x).
\end{align*}
For any $k\in\mathbb Z$ with $k\geq 3$,
\begin{equation*}
\operatorname{supp} \left[ e^{i \frac54 2^k\cdot} h(\cdot) \right]^\wedge
=\tfrac54 2^k + \operatorname{supp} \widehat{h}
\subset B(\tfrac54 + \tfrac542^k,2^{-5})
\end{equation*}
and
\begin{equation*}
\widehat{\varphi_k}
= \widehat{\varphi}(2^{-k}\cdot)
= 1
\text{ on } 2^k [1,\tfrac32] \supset  B(\tfrac54 (2^k+1),2^{-5}).
\end{equation*}
Thus, for any $j,k\in\mathbb Z$ with $k\geq 3$,
\begin{align} \label{amazing}
\varphi_j*[e^{i \frac54 2^k\cdot} h(\cdot)]=
\begin{cases}
e^{i \frac54 2^k\cdot} h(\cdot) &\text{if } j=k, \\
0 & \text{if } j\neq k.
\end{cases}
\end{align}
Then, for any $N\in\mathbb N$ with $N\ge 3$,
\begin{align*}
\|f_N\|_{\dot B_{p,q}^{\alpha+u}(w)}
= \left\| \left\{ N^{-\frac1q}
\| h \|_{L^p(w)} \right\}_{j=N}^{2N-1} \right\|_{\ell^q}
= \| h \|_{L^p(w)}.
\end{align*}

Next, we estimate $\|\sigma(x,D)f_N\|_{\dot B_{p,q}^{\alpha}(w)}$.
Note that, for any $k\in\mathbb Z$ with $k\geq 4$,
$$
\operatorname{supp} \left[ e^{i \frac54 2^k\cdot} h(\cdot) \right]^\wedge
\subset B(\tfrac54 + \tfrac542^k,2^{-5})
\subset 2^k B(\tfrac54,2^{-3}).
$$
By this and Lemma \ref{Tf}, we obtain
\begin{align*}
\sigma(x,D) \left( e^{i \frac54 2^k\cdot} h(\cdot) \right)
= 2^{k(u-v)}(2\pi) (-i\partial)^v h(x).
\end{align*}
Thus, for any $N\in\mathbb N$ with $N\ge 4$,
\begin{align*}
G_N(x)
:=\sigma(x,D) f_N
= 2\pi N^{-\frac1q}
\sum_{k=N}^{2N-1} 2^{-k(\alpha+v)}
(-i\partial_x)^v h(x).
\end{align*}
This, together with (b) and \eqref{h}, further implies that
\begin{align*}
\|G_N\|_{\dot B_{p,q}^{\alpha}(w)}
= 2\pi N^{-\frac1q} \sum_{k=N}^{2N-1} 2^{-k(\alpha+v)}
\|(-i\partial)^v h\|_{L^p(w)}
\sim N^{-\frac1q} \sum_{k=N}^{2N-1} 2^{-k(\alpha+v)}.
\end{align*}
Note that $\alpha+v=\alpha+\lfloor-\alpha\rfloor\in(-1,0]$.
If $\alpha+v\in(-1,0)$, retaining only the term with $k=N$ gives
\begin{align*}
\|G_N\|_{\dot B_{p,q}^{\alpha}(w)}
\gtrsim N^{-\frac1q} 2^{-N(\alpha+v)} \to \infty
\end{align*}
as $N\to\infty$. If $\alpha+v=0$ and $q\in(1,\infty]$, then
\begin{align*}
\|G_N\|_{\dot B_{p,q}^{\alpha}(w)}
\gtrsim N^{1-\frac1q}
\to \infty
\end{align*}
as $N\to\infty$.
However, $\|f_N\|_{\dot B_{p,q}^{\alpha+u}(w)}= \| h \|_{L^p(w)}$.
Consequently, $\sigma(x,D)$ is not bounded from
$\dot B_{p,q}^{\alpha+u}(w)$ to $\dot B_{p,q}^{\alpha}(w)$.

It remains only to consider the case where $\alpha+v=0$ and $q\in(0,1]$.
For any $N\in\mathbb N$, let
\begin{align*}
f_N(x):= N^{-\frac12} 2^{-5N(\alpha+u)} e^{i \frac54 2^{5N}x} H_N(x),
\end{align*}
where $H_N(x):=\sum_{k=N}^{2N-1}e^{i \frac54 2^k x}h(x)$.
Then
\begin{align} \label{723}
\operatorname{supp} \widehat{f_N}
&=\operatorname{supp}
\left[ e^{i \frac54 2^{5N}\cdot} H_N(\cdot) \right]^\wedge
= \operatorname{supp} \sum_{k=N}^{2N-1}
\widehat{h} (\cdot - \tfrac54 (2^{5N}+2^k)) \notag \\
&\subset \bigcup_{k=N}^{2N-1}
B(\tfrac54 + \tfrac54 (2^{5N}+2^k),2^{-5})
\subset 2^{5N} [1,\tfrac32].
\end{align}
Note that $\widehat{\varphi_{5N}}=1$ on $2^{5N} [1,\tfrac32]$.
Thus, for any $j\in\mathbb Z$,
\begin{align*}
\varphi_j*f_N=
\begin{cases}
f_N &\text{if } j=5N, \\
0 & \text{if } j\neq 5N,
\end{cases}
\end{align*}
and hence
\begin{align} \label{fN3}
\|f_N\|_{\dot B_{p,q}^{\alpha+u}(w)}
= N^{-\frac12} \left\| e^{i \frac54 2^{5N}\cdot} H_N(\cdot) \right\|_{L^p(w)}
= N^{-\frac12} \| H_N \|_{L^p(w)}.
\end{align}
From \eqref{amazing} and
the weighted Littlewood--Paley square-function theorem
(see, for instance, \cite[Theorem 1]{k80}),
it follows that, for any $N\in\mathbb N$ with $N\geq 3$,
\begin{align*}
\| H_N \|_{L^p(w)}
\sim \left\| \left[ \sum_{j=N}^{2N-1}
\left| e^{i \frac54 2^j \cdot} h(\cdot) \right|^2
\right]^{\frac12} \right\|_{L^p(w)}
=N^{\frac12} \|h\|_{L^p(w)}.
\end{align*}
Substituting this estimate into \eqref{fN3}, we obtain
$\|f_N\|_{\dot B_{p,q}^{\alpha+u}(w)}\sim\|h\|_{L^p(w)}$.

Next, we estimate $\|\sigma(x,D)f_N\|_{\dot B_{p,q}^{\alpha}(w)}$.
By \eqref{723}, we obtain
$$
\operatorname{supp} \left[ e^{i \frac54 2^{5N}\cdot} H_N(\cdot) \right]^\wedge
\subset \bigcup_{k=N}^{2N-1}
B(\tfrac54 + \tfrac54 (2^{5N}+2^k),2^{-5})
\subset 2^{5N} B(\tfrac54, 2^{-3}).
$$
This, together with Lemma \ref{Tf}, further implies that
\begin{align*}
\sigma(x,D) \left( e^{i \frac54 2^{5N}\cdot} H_N(\cdot) \right)
= 2^{5N(u-v)}(2\pi) (-i\partial)^v H_N(x).
\end{align*}
From this and $\alpha+v=0$,
we deduce that, for any $N\in\mathbb N$,
\begin{align*}
G_N(x)
&:=\sigma(x,D) f_N
= N^{-\frac12} 2^{-5N(\alpha+v)} (2\pi) (-i\partial)^v H_N(x) \\
&\phantom{:}= 2\pi N^{-\frac12} \sum_{k=N}^{2N-1}
(-i\partial)^v \left( e^{i \frac54 2^k \cdot}h(\cdot) \right) (x).
\end{align*}
This, together with \eqref{amazing}, further implies that,
for any $N\in\mathbb N$ with $N\geq 3$,
\begin{align} \label{GN5}
\|G_N\|_{\dot B_{p,q}^{\alpha}(w)}
\gtrsim N^{-\frac12} \left\| \left\{ 2^{j\alpha}
\left\|(-i\partial)^v \left( e^{i \frac54 2^j \cdot}h(\cdot) \right) \right\|_{L^p(w)}
\right\}_{j=N}^{2N-1} \right\|_{\ell^q}.
\end{align}
For any $j\in\mathbb Z$ and $x\in\mathbb R$,
\begin{align*}
(-i\partial)^v \left( e^{i \frac54 2^j \cdot}h(\cdot) \right) (x)
- (\tfrac54 2^j)^v e^{i \frac54 2^j x} h(x)
= \sum_{l=1}^{v} \binom vl
(\tfrac54 2^j)^{v-l} e^{i \frac54 2^j x} (-i\partial)^l h(x),
\end{align*}
and hence
\begin{align*}
2^{-jv}\left\|(-i\partial)^v \left( e^{i \frac54 2^j \cdot}h(\cdot) \right)
- (\tfrac54 2^j)^v e^{i \frac54 2^j \cdot} h \right\|_{L^p(w)}
\lesssim \sum_{l=1}^{v} 2^{-jl}
\| (-i\partial)^l h \|_{L^p(w)}
\to 0
\end{align*}
as $j\to\infty$. This, combined with $\alpha+v=0$ and the fact that
$2^{-jv} \| (\tfrac54 2^j)^v e^{i \frac54 2^j \cdot} h\|_{L^p(w)}
= (\tfrac54)^v \|h\|_{L^p(w)}$, further implies that
\begin{align*}
2^{j\alpha}
\left\|(-i\partial)^v \left( e^{i \frac54 2^j \cdot}h(\cdot) \right) \right\|_{L^p(w)}
= 2^{-jv}
\left\|(-i\partial)^v \left( e^{i \frac54 2^j \cdot}h(\cdot) \right) \right\|_{L^p(w)}
\to (\tfrac54)^v \|h\|_{L^p(w)}
\end{align*}
as $j\to\infty$.
Substituting this estimate into \eqref{GN5}, we obtain
\begin{align*}
\|G_N\|_{\dot B_{p,q}^{\alpha}(w)}
\gtrsim N^{-\frac12} \left\| \{ \|h\|_{L^p(w)}
\}_{j=N}^{2N-1} \right\|_{\ell^q}
= \|h\|_{L^p(w)} N^{\frac1q-\frac12}
\to \infty
\end{align*}
as $N\to \infty$.
However, $\|f_N\|_{\dot B_{p,q}^{\alpha+u}(w)}\sim \| h \|_{L^p(w)}$.
Consequently, $\sigma(x,D)$ is not bounded from
$\dot B_{p,q}^{\alpha+u}(w)$ to $\dot B_{p,q}^{\alpha}(w)$.
This completes the proof of Lemma \ref{sharpcounterexample}.
\end{proof}

Now, we can prove Theorems \ref{sharppse2} and \ref{sharppse1}.

\begin{proof}[Proof of Theorem \ref{sharppse2}]
The implication (ii) $\Longrightarrow$ (i) follows as a special case of Theorem \ref{pseudo}.
Conversely, if $\alpha\leq\theta_{\lambda,p}$,
then Lemma \ref{sharpcounterexample} shows that
there exists $\sigma\in\dot S^u_{1,1}$
such that the operator $\sigma(x,D)$ fails to be bounded
from $\dot B_{p,q}^{\alpha+u}(w)$ to $\dot B^{\alpha}_{p,q}(w)$.
This further implies that (i) $\Longrightarrow$ (ii),
which completes the proof of Theorem \ref{sharppse2}.
\end{proof}

\begin{proof}[Proof of Theorem \ref{sharppse1}]
The implication (ii) $\Longrightarrow$ (i) follows as a special case of Theorem \ref{pseudo}.
Conversely, if $L<\lfloor\theta_{\lambda,p}-\alpha\rfloor$,
then Lemma \ref{sharpcounterexample} shows that
there exists $\sigma\in\dot S^u_{1,1}$ such that
\begin{equation*}
\sigma(x, D)^*(x^r)=0 \in \mathcal{S}_\infty'
\quad\text{for all } r\in\mathbb Z\cap [0,L],
\end{equation*}
and the operator $\sigma(x,D)$ fails to be bounded
from $\dot B_{p,q}^{\alpha+u}(w)$ to $\dot B^{\alpha}_{p,q}(w)$.
This further implies that (i) $\Longrightarrow$ (ii),
which completes the proof of Theorem \ref{sharppse1}.
\end{proof}

\bigskip

\noindent Fan Bu

\medskip

\noindent Department of Mathematics, Faculty of Arts and Sciences,
Beijing Normal University, Zhuhai 519087, The People's Republic of China

\smallskip

\noindent{\it E-mail:} \texttt{fanbu@bnu.edu.cn}

\bigskip

\noindent Shuaijun Feng, Qingying Xue, Dachun Yang (Corresponding author) and Wen Yuan

\medskip

\noindent Laboratory of Mathematics and Complex Systems (Ministry of Education of China),
School of Mathematical Sciences,  Institute for Advanced Study,
Beijing Normal University, Beijing 100875, The People's Republic of China

\smallskip

\noindent{\it E-mails:} \texttt{sjfeng@mail.bnu.edu.cn} (S. Feng)

\noindent\phantom{{\it E-mails:}} \texttt{qyxue@bnu.edu.cn} (Q. Xue)

\noindent\phantom{{\it E-mails:}} \texttt{dcyang@bnu.edu.cn} (D. Yang)

\noindent\phantom{{\it E-mails:}} \texttt{wenyuan@bnu.edu.cn} (W. Yuan)

\end{document}